\documentclass[12pt,reqno]{amsart}
\usepackage[latin1]{inputenc}
\usepackage{amsmath}
\usepackage{amsfonts}
\usepackage{amssymb,esint}
\usepackage{mathrsfs}
\usepackage{hyperref}
\usepackage{stmaryrd}
\usepackage{multirow}
\usepackage{verbatim}
\usepackage{cancel}

\numberwithin{equation}{section}

\theoremstyle{plain}
\newtheorem{theorem}{Theorem}[section]
\newtheorem{lemma}[theorem]{Lemma}
\newtheorem{corollary}[theorem]{Corollary}

\newtheorem{proposition}[theorem]{Proposition}
\theoremstyle{definition}
\newtheorem{definition}[theorem]{Definition}
\newtheorem{remark}[theorem]{Remark}
\newtheorem{example}[theorem]{Example}

\renewcommand{\theequation}{\arabic{section}.\arabic{equation}}

\newcommand{\R}{\mathbb R}
\newcommand{\N}{\mathbb N}
\renewcommand{\H}{\mathbb H}
\newcommand{\G}{\mathbb G}

\newcommand{\galg}{\mathfrak g}

\newcommand{\al}{\alpha}
\newcommand{\be}{\beta}
\newcommand{\ga}{\gamma}
\newcommand{\de}{\delta}
\newcommand{\ep}{\varepsilon}

\newcommand{\la}{\lambda}
\newcommand{\si}{\sigma}

\usepackage{color}

\newcommand{\res}
{\mathop{\hbox{\vrule height 7pt width .5pt depth 0pt \vrule
height .5pt width 6pt depth 0pt}}\nolimits}

\newcommand{\Shaus}{\mathscr S}
\newcommand{\Haus}{\mathscr H}

\newcommand{\W}{\mathbb W}

\newcommand{\f}{\phi}

\newcommand{\leb}[1]{\mathscr L^{#1}}
\newcommand{\ci}{{C}}
\newcommand{\pa}{\partial}
\newcommand{\V}{\mathbb V}
\newcommand{\F}{\Phi}

\newcommand{\h}{\mathfrak h}
\newcommand{\qtaq}{\quad\text{and}\quad}
\newcommand{\qqtaqq}{\qquad\text{and}\qquad}

\newcommand{\bwl}{\text{\Large$\wedge$}}
\renewcommand{\DH}{\mathcal D_\H}
\newcommand{\nf}{\nabla^\f}
\newcommand{\curr}[1]{\llbracket{#1}\rrbracket}
\newcommand{\clcurr}[1]{\llparenthesis{#1}\rrparenthesis}

\newcommand{\unoenne}{\{1,\dots,n\}}
\newcommand{\unodueenne}{\{1,\dots,2n\}}
\newcommand{\Pab}{{\mathscr P_{a,b}}}

\newcommand{\gr}{\mathrm{gr}}
\newcommand{\eexp}{\overrightarrow{\exp}}
\newcommand{\Tpr}{\Tcurr_{p,r}}
\newcommand{\Lpr}{\mathcal L_{p,r}}
\newcommand{\heis}{\mathbb H}
\newcommand{\C}{\mathscr C}
\newcommand{\gradh}{\nabla_\H}
\DeclareMathOperator*{\Span}{span}
\newcommand{\cono}{\mathscr C}
\newcommand{\spt}{\mathrm{spt}\:}

\newcommand{\Tcurr}{\mathsf T}
\newcommand{\tauT}{\!\vec{\,\mathsf T}}
\newcommand{\Tan}{{\mathrm{Tan}}}

\author{Davide Vittone}
\address{Dipartimento di Matematica ``T. Levi-Civita'', via Trieste 63, 35121 Padova, Italy.}
\email{vittone@math.unipd.it}
\thanks{The author is supported by the   STARS Project SUGGESTION of the University of Padova and by GNAMPA of INdAM (Italy).
}



\begin{document}


\title{Lipschitz graphs and currents in Heisenberg groups}

\begin{abstract}
The main result of the present paper is  a Rademacher-type theorem for intrinsic Lipschitz graphs of  codimension $k\leq n$ in sub-Riemannian Heisenberg groups $\H^n$.
For the purpose of proving such a result we settle several related questions pertaining  both to the theory of intrinsic Lipschitz graphs and to the one of currents.
First, we prove  an extension result for intrinsic Lipschitz graphs as well as a uniform approximation theorem by means of smooth graphs:  these results stem both from a new definition (equivalent to the one introduced by F.~Franchi, R.~Serapioni and F.~Serra Cassano) of intrinsic Lipschitz graphs and are valid for a more general class of intrinsic Lipschitz graphs in Carnot groups. 
Second, our proof of Rademacher's Theorem heavily uses the language of currents in Heisenberg groups: one key result  is, for us, a version of the celebrated Constancy Theorem. Inasmuch as Heisenberg currents are defined in terms of Rumin's complex of differential forms, we also provide a convenient basis of Rumin's spaces.
Eventually, we provide some applications of  Rademacher's Theorem including a Lusin-type result for intrinsic Lipschitz graphs, the equivalence between $\H$-rectifiability and ``Lipschitz'' $\H$-rectifiability, and an area formula for intrinsic Lipschitz graphs in Heisenberg groups.
\end{abstract}

\maketitle

\section{Introduction}
The celebrated Rademacher's Theorem~\cite{Rademacher} states that a Lipschitz continuous function $f:\R^h\to\R^k$ is differentiable almost everywhere in $\R^h$; in particular, the graph of $f$ in $\R^{h+k}$ has a  $h$-dimensional tangent plane at almost all  of its points. 
One of the consequences of Rademacher's Theorem is the following Lusin-type result, which stems from Whitney's Extension Theorem~\cite{WhitneyExtension34}: for every $\ep>0$, there exists $g\in C^1(\R^h,\R^k)$ that coincides with $f$ out of a set of measure at most $\ep$. From the viewpoint of Geometric Measure Theory, this means that Lipschitz-regular objects (functions, submanifolds...) are  essentially as nice as $C^1$-smooth ones and has profound implications, for instance, in the theory of rectifiable sets and currents~\cite{federer,mattila,Simon}.

The present paper aims at developing a similar theory for submanifolds with (intrinsic) Lipschitz regularity in  sub-Riemannian Heisenberg groups: before adequately introducing our results,  we feel the need to list them at least quickly. We believe that our main result is  a Rademacher-type theorem for intrinsic Lipschitz graphs, that was the main open problem since the beginning of this theory. 
Some applications -- namely, a Lusin-type result and an area formula for intrinsic Lipschitz graphs -- are  provided here as well; however, we believe that further consequences are yet to come concerning, for instance,  rectifiability and minimal submanifolds in  Heisenberg groups. 
Some of the tools we develop for proving our main result are worth mentioning:  in fact, we prove an extension result for intrinsic Lipschitz graphs as well as the fact that they can be uniformly approximated by smooth graphs. Both results stem from what can be considered as another contribution of the present paper, i.e., a new definition of intrinsic Lipschitz graphs that is equivalent to the original one, introduced by B.~Franchi, R.~Serapioni and F.~Serra Cassano and now widely accepted: recall in fact that intrinsic Lipschitz graphs in Heisenberg groups played a fundamental role in the recent proof by A.~Naor and R.~Young~\cite{NaorYoungAnnals2018} of the ``vertical versus horizontal''  isoperimetric inequality in $\H^n$  that allowed to settle the longstanding  question of determining  the approximation ratio of the Goemans-Linial algorithm for the Sparsest Cut Problem.
Let us also say that our proof of Rademacher's Theorem heavily uses the language of currents in Heisenberg groups: a key result  is for us (a version of) the celebrated Constancy Theorem~\cite{federer,Simon,KrantzParks}. From the technical point of view, the use of currents constitutes the  hardest part of the paper: in fact,  currents in Heisenberg groups are defined in terms of the complex of differential forms introduced by M.~Rumin in~\cite{RuminComptesRendus1990,Rumin}, that is not easy to handle. Among other things, we had to provide a convenient basis of Rumin's covectors that could be fruitfully employed in the computation of Rumin's exterior derivatives; we were surprised  by the fact that the use of {\em standard Young tableaux} from combinatorics revealed crucial in performing this task. 

It is  time to introduce  and discuss our results more appropriately. 

\subsection{Heisenberg groups and intrinsic graphs} 
For $n\geq 1$, the Heisenberg group $\H^n$ is the connected, simply connected and nilpotent Lie group associated with the Lie algebra $\h$ with $2n+1$ generators $X_1,\dots,X_n,Y_1,\dots,Y_n,T$; all Lie brackets between these generators are null  except for
\[
[X_j,Y_j]=T\quad\text{for every }j=1,\dots,n.
\] 
The algebra $\h$ is stratified, as it can be decomposed as $\h=\h_1\oplus\h_2$ with $\h_1:=\Span\{X_j,Y_j:j=1,\dots,n\}$ and $\h_2:=\Span\{T\}$. The first layer $\h_1$ in the stratification is called {\em horizontal}. 

It will be often convenient to identify $\H^n$ with $\R^{2n+1}$ by exponential coordinates
\[
\R^n\times\R^n\times\R\ni (x,y,t)\longleftrightarrow \exp(x_1X_1+\dots +x_nX_n+y_1Y_1+\dots+y_nY_n+tT)\in\H^n,
\]
where $\exp:\h\to\H^n$ is the exponential map; 0 is the group identity. 
The Heisenberg group is a homogeneous group according to~\cite{FS}: indeed, for $\la>0$ the maps 
$\de_\la(x,y,t):=(\la x,\la y,\la^2 t)$ determine  a one-parameter family of group automorphisms of $\H^n$ called {\em dilations}. We endow $\H^n$ with a left-invariant and homogeneous distance $d$, so that
\[
d(qp,qp')=d(p,p')\quad\text{and}\quad d(\de_\la p,\de_\la q)=\la d(p,q)\qquad\forall\: p,p',q\in\H^n,\la>0.
\]
It will be convenient to assume that $d$ is rotationally invariant, i.e., that
\[
\|(x,y,t)\|_\H=\|(x',y',t)\|_\H\qquad\text{whenever }|(x,y)|=|(x',y')|
\]
where we set $\|p\|_\H:=d(0,p)$ for every $p\in\H^n$. Relevant examples of rotationally invariant distances are the well-known Carnot-Carath\'eodory and Kor\'anyi (or Cygan-Kor\'anyi) distances.

An intensive search for a robust  intrinsic notion of  $C^1$ or Lipschitz regularity for submanifolds was conducted in the last two decades: in fact (see~\cite{AmbKirMathAnn2000}) the Heisenberg group $\H^1$ is purely $k$-unrectifiable, in the sense of~\cite{federer}, for $k=2,3,4$. 
It can however be stated that the theory of {\em $\H$-regular} submanifolds (i.e., submanifolds with intrinsic $C^1$ regularity) is well-established,  see for instance the beautiful paper~\cite{FSSCAIM}. 
It turns out that $\H$-regular submanifold in $\H^n$ of {\em low dimension} $k\in\{1,\dots,n\}$ are  $k$-dimensional submanifolds of class $C^1$ (in the  Euclidean sense) that are tangent to the horizontal bundle $\h_1$. 
On the contrary, $\H$-regular submanifolds of {\em low codimension} $k\in\{1,\dots,n\}$ are more complicated:  they are (locally) noncritical level sets of $\R^k$-valued maps on $\H^n$ with continuous horizontal derivatives (see \S\ref{subsec:Hreg,Hrettif} for precise definitions) and, as a matter of fact, they can have fractal Euclidean dimension~\cite{KirchhSC}.

A key tool for the study of $\H$-regular submanifold is provided by {\em intrinsic graphs}. Assume that $\V,\W$ are homogeneous  complementary subgroups of $\H^n$, i.e., that they are invariant under dilations, $\V\cap\W=\{0\}$ and $\H^n=\W\V=\V\W$; given $A\subset \W$ and a map $\f:A\to\V$, the {\em intrinsic graph} of $\f$ is $\gr_\f:=\{w\f(w):w\in A\}\subset \H^n$. It is worth recalling that, if $\V,\W$ are homogeneous and  complementary subgroups, then one of the two is necessarily horizontal (i.e., contained in $\exp(\h_1)$), Abelian, and of dimension  $k\leq n$, while the other has dimension $2n+1-k\geq n+1$, is normal, and  contains the group center $\exp(\h_2)$; see \cite[Remark 3.12]{FSSCAIM}. The first appearance of intrinsic graphs is most likely to be attributed to the Implicit Function Theorem of the fundamental paper~\cite{FSSCMathAnn2001}, 
where the authors prove a $\H$-rectifiability result for (boundaries of) sets with finite perimeter in $\H^n$. As a matter of fact, $\H$-regular submanifolds are locally intrinsic graphs whose properties have been studied in many papers, see e.g.~\cite{ASCV,ArenaSer,BigolinSC2,BigolinSC1,CittiManfredini,Corni,CorniMagnani,
DiDonato2018,DiDonato2019,DDFaessOrpo2019,FSSCAIM,JNGV,MagnaniTowardsDiffCalc,MVMathZ12}.

Intrinsic graphs also provide the language for introducing a theory of Lipschitz submanifolds in $\H^n$. Observe that, while for the case of low dimensional submanifolds one could simply consider Euclidean Lipschitz submanifolds that are a.e. tangent to the horizontal distribution, for submanifolds of low codimension there is no immediate way of modifying the ``level set definition'' of $\H$-regularity into a Lipschitz one. Intrinsic Lipschitz graphs in $\H^n$ first appeared in~\cite{FSSCJNCA}; their definition is stated in terms of a suitable  cone property. Given $\al>0$, consider the homogeneous cone of axis $\V$ and aperture $\al$
\[
\cono_\al:=\{wv\in\H^n:w\in\W,v\in\V,\|w\|_\H\leq\al\|v\|_\H\}.
\]
We say that a map $\f:A\subset \W\to\V$ is {\em intrinsic Lipschitz} if there exists $\al>0$ such that
\[
\gr_\f \cap (p \cono_\al)=\{p\}\qquad\text{for every }p\in\gr_\f.
\]
Intrinsic Lipschitz graphs can be introduced in the more general framework of {\em Carnot groups}: apart from the elementary basics contained in Section~\ref{sec:gruppi}, we refer to~\cite{FSJGA} for a beautiful introduction to the topic and to~\cite{BigolinCaravennaSC,ChouFaessOrpoAJM2019,CittiMPSC,DDFaessOrpo2019,DLDMV,FaessOrpoRigot,FMarchiS,FPensoS,NGSC,SerapioniDiffQuot,SCSomeTopics} for several facets of the theory. 

\subsection{Rademacher's Theorem for intrinsic Lipschitz graphs and consequences}
One of the main questions about intrinsic Lipschitz graphs concerns their almost everywhere ``intrinsic'' differentiability. Consider an intrinsic Lipschitz map $\f:A\to\V$ defined on some relatively open subset $A\subset\W$. If $\W$ has low dimension $k\leq n$, then (see \cite{AntonelliMerlo} or also \cite[Remark~3.11]{FSSCJNCA},~\cite[Proposition~3.7]{FSJGA}) $\gr_\f$ is a $k$-dimensional submanifold with Euclidean Lipschitz regularity that is a.e. tangent to the horizontal bundle $\h_1$; therefore, the  problem reduces to the case of $\H$-regular graphs with low codimension $k=\dim\V\leq n$. 
A positive answer (\cite{FSSCJGA-Diff}) is known only for the case of codimension $k=1$: in fact, in this case $\gr_\f$ is (part of) the boundary of a set with finite $\H$-perimeter (\cite{CapognaDanielliGarofaloCAG1994,FSSCHouston}) in $\H^n$ and one can use the rectifiability result~\cite{FSSCMathAnn2001} available for such sets. A Rademacher-type theorem for intrinsic Lipschitz functions of codimension 1 was proved in Carnot groups of type $\star$, see~\cite{FMarchiS}. In this paper we provide a full solution to the problem in $\H^n$, as stated in our main result.

\begin{theorem}\label{thm:rademacher}
If $A\subset\W$ is open and $\f:A\to\V$ is  intrinsic Lipschitz, 
then $\f$ is intrinsically differentiable at almost every point of $A$.
\end{theorem}

In Theorem~\ref{thm:rademacher}, ``almost every'' must be understood with respect to a Haar measure on the subgroup $\W$ -- for instance, the  Hausdorff measure of dimension $2n+2-k$. Concerning the notion of intrinsic differentiability (see \S\ref{subsec:differenziabilita_blowups}), recall that left-translations and dilations of intrinsic Lipschitz graphs are intrinsic Lipschitz graphs; in particular, for every $w\in A$ and every $\la>0$ there exists an intrinsic Lipschitz $\f_w^\la:B\to\V$, defined on some open subset $B\subset\W$, such that
\[
\de_\la((w\f(w))^{-1}\gr_\f)=\gr_{\f_w^\la}.
\]
One then says (\cite[\S3.3]{ArenaSer}) that $\f$ is {\em intrinsically differentiable} at $w$ if, as $\la\to+\infty$, the blow-ups $\f_w^\la$ converge locally uniformly on $\W$ to an intrinsic linear map, i.e., to a map  $\psi:\W\to\V$ such that $\gr_\psi$ is a homogeneous subgroup of $\H^n$ with codimension $k$. This subgroup, that  is necessarily vertical (i.e., it contains the center of $\H^n$) and  normal, is called {\em  tangent subgroup} to $\gr_\f$ at $w\f(w)$ and is denoted by $\Tan^\H_{\gr_\f}(w\f(w))$. 

For the reader's convenience, the proof of Theorem~\ref{thm:rademacher} is sketched at the end of the Introduction. We are now going to introduce a few consequences of our main result: the first one is a Lusin-type theorem for intrinsic Lipschitz graphs.

\begin{theorem}\label{thm:Lusinbreve}
Let $A\subset\W$ be an open set and $\f:A\to\V$ an intrinsic Lipschitz function. Then for every $\ep>0$ there exists an intrinsic Lipschitz function $\psi:A\to\V$ such that $\gr_\psi$ is a $\H$-regular submanifold and 
\[
\Shaus^{Q-k}((\gr_\f\,\Delta\, \gr_\psi)\cup\{p\in\gr_\f\cap \gr_\psi:\Tan^\H_{\gr_\f}(p)\neq \Tan^\H_{\gr_\psi}(p)\})<\ep.
\]
\end{theorem}

As customary, the integer $Q:=2n+2$ denotes the homogeneous dimension of $\H^n$ and $\Shaus^{Q-k}$ is the spherical Hausdorff measure of dimension $Q-k$; by $A\Delta B:=(A\setminus B)\cup(B\setminus A)$ we denote the symmetric difference of sets $A,B$. Theorem~\ref{thm:Lusinbreve} is part of Theorem~\ref{thm:Lusin}; the latter stems from the equivalent definition of intrinsic Lipschitz graphs provided by Theorem~\ref{teo:equivdefLip} below and is proved by  an adaptation of the classical argument of Whitney's Extension Theorem, see also~\cite{FSSCMathAnn2001,FSSCstep2,VittonePhDThesis,VodoPupy}. Theorem~\ref{thm:Lusinbreve}  implies that, as in the Euclidean case, the notion of $\H$-rectifiability (Definition~\ref{def:Hrettificabili})  can be equivalently defined in terms of  either $\H$-regular submanifolds or  intrinsic Lipschitz graphs; see Corollary~\ref{cor:HrectifiableC1vsLip}. 
We refer to~\cite{ChouFaessOrpoAJM2019,ChouMagTyson,DDFaessOrpo2019,FaessOrpoRigot,MatSerSC,Merlo_Geom1codim,Merlo_MM} for more  about rectifiability in Heisenberg groups.

We  stress the fact that Rademacher's Theorem~\ref{thm:rademacher} also allows to define a canonical {\em current} $\curr{\gr_\f}$  carried by the  graph of an intrinsic Lipschitz map $\f:\W\to\V$. This current turns out to have zero boundary, see Proposition~\ref{prop:grf0bdry}.

A further consequence of Theorem~\ref{thm:Lusinbreve} is an area formula for intrinsic Lipschitz graphs of low codimension. For $\H$-regular submanifolds, area formulae are proved in~\cite{FSSCMathAnn2001,FSSCAIM,ASCV} for submanifolds of codimension 1 and in~\cite{CorniMagnani} for higher codimension (see also~\cite{MagnaniTowardsTheoryArea}). For intrinsic Lipschitz graphs of low dimension, an area formula is proved in~\cite[Theorem~1.1]{AntonelliMerlo}. Our area formula is stated in Theorem~\ref{thm:formulaarea} and, once  Lusin's Theorem~\ref{thm:Lusinbreve} is available,  it is a quite simple consequence of~\cite[Theorem~1.2]{CorniMagnani}, where a similar area formula is proved for intrinsic graphs that are also $\H$-regular submanifolds. 
As in~\cite{CorniMagnani}, the symbol $J^\f\f(w)$ denotes the intrinsic Jacobian of $\f$ at $w$, see Definition~\ref{def:differenziabilita}, while $C_{n,k}$ denotes a positive constant, depending only on $n,k$ and the distance $d$, that will be introduced later in Proposition~\ref {prop:correnti_cl_vs_H_INTRO}.

\begin{theorem}\label{thm:formulaarea}
Assume that the subgroups $\W,\V$ are orthogonal\footnote{By {\em orthogonal} we mean that $\W,\V$ are orthogonal as linear subspaces of $\H^n\equiv\R^{2n+1}$.} and let  $\f:A\to\V$ be an intrinsic Lipschitz map defined on some Borel subset $A\subset\W$; then for every Borel function $h:\gr_\f\to[0,+\infty)$ there holds
\[
\int_{\gr_\f} h \:d\Shaus^{Q-k} =C_{n,k}\int_A (h\circ\F)J^\f\f\:d\leb{2n+1-k},
\]
where $\Phi$ denotes the graph map $\F(w):=w\f(w).$
\end{theorem}

By abuse of notation, $\leb{2n+1-k}$ denotes the Haar measure on $\W$ associated with the canonical identification of $\W$ with $\R^{2n+1-k}$ induced by exponential coordinates.  It is  worth observing that Theorem~\ref{thm:formulaarea} and Proposition~\ref{prop:correnti_cl_vs_H_INTRO} are the only points where we use the rotational invariance of the distance $d$: in case of  general distances, area formulae for intrinsic Lipschitz graphs can be easily deduced using Theorem~\ref{thm:Lusinbreve} and~\cite[Theorem~1.2]{CorniMagnani}, but they are slightly more complicated than ours, as they involve a certain {\em area factor} that depends on the  tangent plane to the graph.

\subsection{Equivalent definition, extension and approximation of intrinsic Lipschitz graphs}
We now introduce two of the ingredients needed in the proof of  Theorem~\ref{thm:rademacher} that are of independent interest: namely, an extension theorem for intrinsic Lipschitz graphs  in the spirit of the classical McShane-Whitney Theorem, and an approximation result by means of smooth graphs. They are stated in Theorems~\ref{thm:estensione} and~\ref{teo:approssimiamoooooooo} and are both based on a new, equivalent definition of intrinsic Lipschitz graphs, Theorem~\ref{teo:equivdefLip}, that can be regarded as another contribution of this paper. 

Our alternative definition of intrinsic Lipschitz graphs appeared in~\cite{VAnnSNS} for graphs of codimension 1; it can be seen as a generalization of the original level-set definition of $\H$-regular submanifolds. Observe however that it is not immediate  to give a level-set definition even for Lipschitz submanifolds of codimension 1 in $\R^n$: in fact, {\em every} closed set $S\subset\R^n$ is the level set of some Lipschitz function -- for instance,  the distance from $S$. Anyway, we leave as an exercise to the reader the following observation, that was actually the starting point of~\cite{VAnnSNS}: a set $S\subset\R^n=\R^{n-1}\times\R$ is (contained in) the graph of a Lipschitz function $\f:\R^{n-1}\to\R$ if and only if there exist $\de>0$ and a Lipschitz function $f:\R^n\to\R$ such that  $S\subset\{x\in\R^n:f(x)=0\}$ and $\frac{\partial f}{\partial{x_n}}\geq\de$ a.e. on $\R^n$.

Since their proofs present no extra difficulty  with respect to the Heisenberg case, Theorems~\ref{teo:equivdefLip},~\ref{thm:estensione} and~\ref{teo:approssimiamoooooooo} are stated in the more general setting of a  {\em Carnot group} $\G$ where two homogeneous complementary subgroups $\W,\V$ are fixed with $\V$ horizontal: this means that $\V\subset\exp(\galg_1)$, where $\galg_1$ is the first layer in the stratification of the Lie algebra of $\G$. When $\V$ is horizontal, we say that an intrinsic Lipschitz graph $\f:\W\to\V$ is {\em co-horizontal}, see~\cite{ADDDLD}. Observe that $\V$ is necessarily Abelian and there exists a homogeneous isomorphism according to which we can identify $\V$ with  $\R^k$, see~\eqref{eq:VVVVVVVV}: this identification is understood in the scalar product appearing in~\eqref{eq:(b)2} below.

\begin{theorem}\label{teo:equivdefLip}
Assume that a splitting $\G=\W\V$ is fixed in such a way that the subgroup $\V$ is horizontal; set $k:=\dim\V$. 
If $S\subset\G$ is not empty, then the following statements are equivalent:
\begin{itemize}
\item[(a)] there exist $A\subset\W$ and an intrinsic Lipschitz map $\f:A\to\V$  such that $S=\gr_\f$;
\item[(b)] there exist $\de>0$ and a Lipschitz map $f:\G\to\R^k$ such that
\begin{align}
& S\subset \{x\in\G:f(x)=0\}
\label{eq:(b)1}
\\
\text{and}\quad& \langle f(xv)-f(x),v\rangle \geq \de|v|^2\qquad\text{for every }v\in\V\text{ and }x\in\G.
\label{eq:(b)2}
\end{align}
\end{itemize}
\end{theorem}

It is worth remarking that, if $X_1,\dots, X_k\in\galg_1$ are such that $\V=\exp(\Span\{X_1,\dots,X_k\})$, then statement~\eqref{eq:(b)2} is equivalent to the a.e.~uniform ellipticity (a.k.a.~coercivity) of the matrix col$[\,X_1f(x)\,|\,\dots\,|\,X_k f(x)\,]$; see Remark~\ref{rem:coercivoderivate}. 
In case $k=1$, Theorem~\ref{teo:equivdefLip} was proved in \cite[Theorem~3.2]{VAnnSNS}. 

Let us underline two of the most interesting features of this alternative definition. First, it allows for a definition of co-horizontal intrinsic Lipschitz submanifolds in the more general setting of Carnot-Carath\'eodory spaces, as in~\cite{VAnnSNS}. Second, it gives  {\em gratis} an extension result for intrinsic Lipschitz maps: in fact (Remark~\ref{rem:graficoglobale}), the level set $\{x\in\G:f(x)=0\}$ appearing in~\eqref{eq:(b)1} is the graph of an intrinsic Lipschitz map that is defined on the whole $\W$ and extends $\f$. We can then state the following result.

\begin{theorem}\label{thm:estensione}
Let $A\subset\W$ and $\f:A\to\V$ be a co-horizontal intrinsic Lipschitz graph; then there exists an intrinsic Lipschitz extension $\tilde\f:\W\to\V$ of $\f$. Moreover, $\tilde \f$ can be chosen in such a way that its intrinsic Lipschitz constant is controlled in terms of the intrinsic Lipschitz constant of $\f$.
\end{theorem}

Theorem~\ref{thm:estensione} was proved in \cite[Proposition~3.4]{VAnnSNS} for the case of codimension $k=1$; see also \cite{FSSCJGA-Diff,FSJGA,NaorYoungAnnals2018,Rigot2019quantitative}. 

In Proposition~\ref{prop:equivdefLipCinfty} we use a standard approximation argument based on group convolutions (see e.g.~\cite[\S1.B]{FS}) to show that the function $f$ appearing in Theorem~\ref{teo:equivdefLip} can be chosen with the additional property that $f\in C^\infty(\{x\in\G:f(x)\neq 0\})$. This fact has the following consequence.

\begin{theorem}\label{teo:approssimiamoooooooo}
Let $A\subset\W$ and $\f:A\to\V$ be a co-horizontal intrinsic Lipschitz graph. Then there exists a sequence $(\f_i)_{i\in\N}$  of  $C^\infty$-regular and  intrinsic Lipschitz maps $\f_i:\W\to\V$ such that
\[
\f_i\to\f\text{ uniformly in }A\text{ as }i\to\infty\,.
\]
Moreover, the intrinsic Lipschitz constant of $\f_i$ is bounded, uniformly in $i$, in terms of the intrinsic Lipschitz constant of $\f$.
\end{theorem}
 
A similar result has been proved in~\cite{CittiMPSC} for intrinsic Lipschitz graphs of codimension 1 in Heisenberg groups; see also~\cite{ASCV,ADDDLD,MVMathZ12,VAnnSNS}.

\subsection{Currents and the Constancy Theorem}\label{subsec:currentsconstancy_intro}
As in the classical setting, currents in Heisenberg groups are defined in duality with spaces of smooth forms with compact support; here, however, the De Rham complex must be replaced by the complex introduced by M.~Rumin~\cite{RuminComptesRendus1990,Rumin} in the setting of contact manifolds. The construction of the spaces $\Omega_\H^m$ of {\em Heisenberg differential $m$-forms} is detailed in \S\ref{subsec:Rumin}; here we only recall that, for $1\leq k\leq n$,  Heisenberg forms of codimension $k$ are smooth functions on $\H^n$ with values in a certain subspace $\mathcal J^{2n+1-k}$
of $(2n+1-k)$-covectors.
We denote by $\mathcal J_{2n+1-k}$ the (formal) dual of $\mathcal J^{2n+1-k}$; clearly, every $(2n+1-k)$-vector $t$ canonically induces an element $[t]_\mathcal J\in\mathcal J_{2n+1-k}$ defined by $[t]_\mathcal J(\la):=\langle\, t\mid\la\,\rangle$, where $\langle \,\cdot\mid\cdot\,\rangle$ is the standard pairing vectors-covectors. 
See \S\ref{subsec:algebraHeis} and \S\ref{subsec:Rumin} for more details.

The starting point of the theory of Heisenberg currents is the existence of a linear second-order   operator $D:\Omega_\H^n\to\Omega_\H^{n+1}$ such that the sequence
\[
0\to\R\to\Omega_\H^0\stackrel{d}{\to}\Omega_\H^1\stackrel{d}{\to}\dots\stackrel{d}{\to}\Omega_\H^n\stackrel{D}{\to}\Omega_\H^{n+1}\stackrel{d}{\to}  \dots   \stackrel{d}{\to}\Omega_\H^{2n+1}\to 0
\]
is exact, where $d$ is (the operator induced by) the standard exterior derivative. A {\em Heisenberg $m$-current} $\Tcurr$ is by definition a continuous linear functional on the space $\DH^m\subset\Omega_\H^m$ of Heisenberg $m$-forms with compact support. The {\em boundary} $\partial\Tcurr$ of  $\Tcurr$ is the Heisenberg $(m-1)$-current defined, for every $\omega\in \DH^{m-1}$, by
\[
\begin{array}{ll}
\partial \Tcurr(\omega):=\Tcurr(d\omega)\quad&\text{if }m\neq n+1\\
\partial \Tcurr(\omega):=\Tcurr(D\omega)\quad&\text{if }m= n+1.
\end{array}
\]
We say that $\Tcurr$ is  {\em locally normal} if both $\Tcurr$ and $\partial\Tcurr$ have locally finite mass, i.e., if they have order 0 in the sense of distributions. Recall that, if $\Tcurr$ has locally finite mass, then there exist a Radon  measure $\mu$ on $\H^n$ and a locally $\mu$-integrable function $\tau$ with values in a suitable space of multi-vectors (that, for $m\geq n+1$, is precisely $\mathcal J_m$) such that $\Tcurr=\tau\mu$, where
\[
\tau\mu(\omega):=\int_{\H^n}\langle\,\tau(p)\mid\omega(p)\,\rangle\,d\mu(p)\qquad\text{for every }\omega\in\DH^m.
\]
One can also assume that $|\tau|=1$ $\mu$-a.e., where $|\cdot|$ denotes some fixed norm on multi-vectors\footnote{More precisely, when $m\geq n+1$ one needs  a norm on  $\mathcal J_m$.}: in this case we write $\tauT$  
and $\|\Tcurr\|$ in place of $\tau$ and $\mu$, respectively.

Relevant examples of currents will be for us those concentrated on $\H$-rectifiable sets of low codimension. Recall that a set $R\subset\H^n$  is {\em locally $\H$-rectifiable} of codimension $k\in\unoenne$ if $\Shaus^{Q-k}\res R$ is locally finite and $R$  can be covered by countably many $\H$-regular submanifolds of codimension $k$ plus a $\Shaus^{Q-k}$-negligible set. In this case, a (unit) {\em approximate  tangent $(2n+1-k)$-vector} $t^\H_R(p)$ to $R$ can be defined at $\Shaus^{Q-k}$-a.e. $p\in\R$, see \S\ref{subsec:Hreg,Hrettif}; we denote by $\curr R$ the  Heisenberg current $[t^\H_R]_\mathcal J\Shaus^{Q-k}\res R$ naturally associated with $R$. 

A fundamental result in the classical theory of currents is the {\em Constancy Theorem} (see e.g.~\cite[4.1.7]{federer} and~\cite[Theorem~26.27]{Simon}) which states that, if $\Tcurr$ is an  $n$-dimensional current  in $\R^n$ such that $\partial\Tcurr=0$, then $\Tcurr$ is constant, i.e., there exists $c\in\R$ such that $\Tcurr(\omega)=c\int_{\R^n}\omega$ for every smooth $n$-form $\omega$ with compact support. 
A more general version of the Constancy Theorem can be proved for currents supported on an $m$-dimensional plane $\mathscr P\subset\R^n$: if $\Tcurr$ is an $m$-current with support in $\mathscr P$ and such that $\partial \Tcurr=0$, then there exists $c\in\R$ such that $\Tcurr(\omega)=c\int_\mathscr P\omega$ for every smooth $m$-form $\omega$ with compact support. 
For this statement, see e.g.~\cite[Proposition 7.3.5]{KrantzParks}. The following result can be considered as the Heisenberg analogue of this more general Constancy Theorem.

\begin{theorem}\label{thm:Constancy}
Let $k\in\unoenne$ be fixed and let  $\Tcurr$ be a Heisenberg $(2n+1-k)$-current supported on a  vertical plane $\mathscr P\subset\H^n$ of dimension $2n+1-k$. Assume that $\partial\Tcurr=0$; then there exists a constant $c\in\R$ such that $\Tcurr=c\curr{\mathscr P}$.
\end{theorem}

Using a procedure involving projection on planes (see~\cite[Theorem~4.2]{Silhavy}; let us   mention also~\cite[\S5]{AlbertiMarchese} and~\cite{AlbMassStep} for  some related results), 
the (version on planes of the) Constancy Theorem in $\R^n$ has the following consequence: if $R\subset\R^n$ is an $m$-rectifiable set and  $\Tcurr=\tau\mu$ is a  normal $m$-current, where $\mu$ is a Radon measure and $\tau$ is a locally $\mu$-integrable $m$-vectorfield with $\tau\neq0$ $\mu$-a.e., then 
\begin{enumerate}
\item[(i)] $\mu\res R$ is absolutely continuous with respect to the Hausdorff measure $\Haus^m\res R$, and 
\item[(ii)] $\tau$ is tangent to $R$ at $\mu$-almost every point of $R$.
\end{enumerate}
A consequence of this fact, which might help explaining its geometric meaning, is the following one: if $\Tcurr=\tau\mu$ is a normal  current concentrated\footnote{By {\em concentrated} we mean that $\mu(\H^n\setminus R)=0$.} on a rectifiable set $R$, then  $\mu\ll\Haus^m\res R$ and $\tau$ is necessarily tangent to $R$ $\mu$-almost everywhere\footnote{Equivalently: there exists a $\Haus^m\res R$-measurable function $f:R\to\R$ such that $\Tcurr(\omega)=\int_Rf\omega$.}.

In our proof of  Rademacher's Theorem~\ref{thm:rademacher} we will utilize the following result, which is the Heisenberg counterpart of the ``tangency'' property (ii) above; we were not able to deduce any ``absolute continuity'' statement analogous to (i) because no good notion of projection on  planes is available in $\H^n$. Notice however that, in the special case when $\Tcurr$ is concentrated on a vertical plane, Theorem~\ref{thm:Constancy} allows to deduce a complete  result including absolute continuity.

\begin{theorem}\label{teo:correntinormalirettificabili}
Let $k\in\unoenne$ and let a locally normal Heisenberg $(2n+1-k)$-current $\Tcurr$ and a locally $\H$-rectifiable set $R\subset\H^n$ of codimension $k$ be fixed. 
Then
\[
\tauT(p)\text{ is a multiple of }[t^\H_R(p)]_\mathcal J\qquad \text{for $\|\Tcurr\|_a$-a.e. }p.
\] 
\end{theorem}

In  Theorem~\ref{teo:correntinormalirettificabili} we  decomposed $\|\Tcurr\|=\|\Tcurr\|_a+\|\Tcurr\|_s$ as the sum of the absolutely continuous and singular part of $\|\Tcurr\|$ with respect to $\Shaus^{Q-k}\res R$. Observe that $t^\H_R$ is defined only $\Shaus^{Q-k}$-almost everywhere on $R$, hence it could be undefined on a set with positive $\|\Tcurr\|_s$-measure. 
The geometric content of Theorem~\ref{teo:correntinormalirettificabili} is again clear: for a current $\Tcurr$ concentrated on $R$ to be normal, it is necessary that  $\tauT$ is almost everywhere  tangent to $R$. 

The proof of Theorem~\ref{teo:correntinormalirettificabili} follows a blow-up strategy according to which one can prove that, at $\Shaus^{Q-k}$-a.e. $p\in R$, the current $\tauT(p)\Shaus^{Q-k}\res \Tan^\H_R(p)$ has zero boundary, where $\Tan^\H_R(p)=\exp(\Span t^\H_R(p))$ is the approximate tangent plane to $R$ at $p$; Proposition~\ref{prop:correntipiane} shows that this is possible only if $\tauT(p)$ is a multiple of $[t^\H_R(p)]_\mathcal J$. 
Proposition~\ref{prop:correntipiane} is essentially a simpler version of Theorem~\ref{thm:Constancy}; its classical counterpart can be found for instance in~\cite[Lemma~1 in \S3.3.2]{GiaqModSouc_I}. 
The proof of Proposition~\ref{prop:correntipiane}\footnote{It is worth pointing out that we cannot deduce Proposition~\ref{prop:correntipiane} from Theorem~\ref{thm:Constancy}: in fact, Proposition~\ref{prop:correntipiane} is needed for proving Theorem~\ref{teo:correntinormalirettificabili}, which in turn is needed for the proof of Theorem~\ref{thm:Constancy}.} consists in feeding the given boundaryless current with (the differential of) enough test forms in order to eventually deduce the desired ``tangency'' property. Apart from the computational difficulties pertaining to the second-order  operator $D$ (at least in case $k=n$), one demanding task we had to face was the search for a convenient basis of $\mathcal J^{2n+1-k}$, see the following \S\ref{subsec:baseRuminINTRO}. 

We conclude this section with an important observation. Assume that $S$ is an oriented submanifold of codimension $k$ that is (Euclidean) $C^1$-regular; in particular, the  tangent vector $t^\H_S$ is defined except at characteristic points of $S$, which however are $\Shaus^{Q-k}$-negligible~\cite{balogh,MagJEMS}. Then, on the one side, $S$ induces the natural Heisenberg current $\curr S=[t^\H_S]_\mathcal J \Shaus^{Q-k}\res S$; on the other side, associated to $S$ is also the classical current $\clcurr S$ defined by $\clcurr S(\omega) :=\int_S\omega$ for every $(2n+1-k)$-form $\omega$ with compact support. The following fact holds true provided the homogeneous distance $d$ is rotationally invariant.

\begin{proposition}\label{prop:correnti_cl_vs_H_INTRO}
Let $k\in\unoenne$; then there exists a positive constant $C_{n,k}$, depending on $n,k$ and the rotationally invariant distance $d$, such that for every $C^1$-regular submanifold $S\subset\H^n$ of codimension $k$ 
\begin{equation}\label{eq:curr=clcurr}
\curr S(\omega)=C_{n,k}\clcurr S(\omega)\qquad\text{for every }\omega\in \DH^{2n+1-k}.
\end{equation}
In particular, if $S$ is a submanifold without boundary, then $\partial \curr S=0$ as a Heisenberg $(2n-k)$-current.
\end{proposition}

In other words, $\curr S$ and $\clcurr S$ coincide, as Heisenberg currents, up to a multiplicative constant. This is remarkable. The first part of the statement of Proposition~\ref{prop:correnti_cl_vs_H_INTRO} is proved in Lemma~\ref{lem:correntiR=H}, while the second one is a consequence of the fact that the operator $D$ is the composition of the differential $d$ with another operator; see Corollary~\ref{cor:senzabordo}.  For the exact value of $C_{n,k}$, see Remark~\ref{rem:valoreCnk}. Proposition~\ref{prop:correnti_cl_vs_H_INTRO} is crucial in the proof of our main result Theorem~\ref{thm:rademacher}.

\subsection{A basis for Rumin's spaces \texorpdfstring{$\mathcal J^{2n+1-k}$}{}}\label{subsec:baseRuminINTRO}
We believe it is worth introducing, at least quickly, the basis of $\mathcal J^{2n+1-k}$ that we use; we need some preliminary notation. Assume that the elements of a finite subset $M\subset\N$ with cardinality $|M|=m$ are arranged (each element of $M$ appearing exactly once) in a tableau  with 2 rows, the first row displaying $\ell\geq \frac m2$ elements $R^1_1,\dots,R^1_\ell$ and the second one displaying $m-\ell\geq 0$ elements $R^2_1,\dots,R^2_{m-\ell}$, as follows
\begin{equation*}
R=\begin{tabular}{|c|c|c|c|ccc}
\cline{1-7}
$R^{1^{\vphantom{2^2}}}_1$ & $R^1_2$ &$\ \cdots\ $& $R^1_{m-\ell}$ &\multicolumn{1}{c|}{$ R_{m-\ell+1}^1$} & \multicolumn{1}{c|}{$\ \cdots\ $} & \multicolumn{1}{c|}{$R_\ell^1$}\\
\cline{1-7}
$R^2_1$ & $R^2_2$ &$\ \cdots\ $& $R^{2^{\vphantom{2^2}}}_{m-\ell}$
\\
\cline{1-4}
\end{tabular}.
\end{equation*}
Such an $R$ is called {\em Young tableau}, see e.g.~\cite{Fulton}. Clearly, $R$ has to be read as a $(2\times \ell)$ rectangular tableau when $\ell=m/2$ while, in case $\ell=m$, we agree that the second row is empty.  Given such an $R$, define the $2\ell$-covector
\[
\al_R:=(dxy_{R^1_1}- dxy_{R^2_1})\wedge(dxy_{R^1_2}- dxy_{R^2_2})\wedge\dots\wedge(dxy_{R^1_{m-\ell}}- dxy_{R^2_{m-\ell}})\wedge dxy_{R^1_{m-\ell+1}}\wedge\dots\wedge dxy_{R^1_\ell},
\]
where for shortness we set $dxy_i:=dx_i\wedge dy_i$; when $\ell=m$ (i.e., when the second row of $R$ is empty) we agree that $\al_R=dxy_{R^1_{1}}\wedge\dots\wedge dxy_{R^1_\ell}$ . One key observation is the fact that,
\begin{equation}\label{eq:wedgiarea0}
\al_R\wedge\sum_{i\in M}dxy_i=0,
\end{equation}
which is essentially a consequence of the equality  $(dxy_i-dxy_j)\wedge(dxy_i+dxy_j)=0$.

Before stating Proposition~\ref{prop:baseJintro} we need some further notation. First, we say that $R$ is a {\em standard Young tableau}  when the elements in each row and each column of $R$ are in increasing order, i.e., when $R^i_j<R^i_{j+1}$ and $R^1_j<R^2_j$. Second, given $I=\{i_1,\dots,i_{|I|}\}\subset\unoenne$ with $i_1<i_2<\dots<i_{|I|}$ we write
\[
dx_I:=dx_{i_1}\wedge\dots\wedge dx_{i_{|I|}},\qquad dy_I:=dy_{i_1}\wedge\dots\wedge dy_{i_{|I|}}
\]
Eventually, we denote by $\theta:=dt+\frac12\sum_{i=1}^n(y_idx_i-x_idy_i)$ the contact form on $\H^n$, which is left-invariant and then can be thought of as a covector in $\bwl^1\h$. Observe that  $\theta$ vanishes on horizontal vectors.

\begin{proposition}\label{prop:baseJintro}
For every $k\in\unoenne$, a basis of $\mathcal J^{2n+1-k}$ is provided by the elements of the form $dx_I\wedge dy_J\wedge \al_R\wedge\theta$ where $(I,J,R)$ ranges among those triples such that
\begin{itemize}
\item $I\subset\unoenne$, $J\subset\unoenne$, $|I|+|J|\leq k$ and  $I\cap J=\emptyset$;
\item $R$ is a standard Young tableau containing the elements of $\unoenne\setminus(I\cup J)$ arranged in two rows of length, respectively, $(2n-k-|I|-|J|)/2$  and   $(k-|I|-|J|)/2$. 
\end{itemize}
\end{proposition}

Proposition~\ref{prop:baseJintro} follows from Corollary~\ref{cor:baseJ}. 
Observe that  the tableaux $R$ appearing in the statement are rectangular exactly in  case $k=n$. In this case it might happen that $I\cup J=\unoenne$, i.e., that $R$ is the empty table: if so, we agree that $\al_R=1$. It  is also worth observing that the covectors $\la_{I,J,R}=dx_I\wedge dy_J\wedge \al_R\wedge\theta$ appearing in Proposition~\ref{prop:baseJintro} indeed belong to $\mathcal J^{2n+1-k}$ because $\la_{I,J,R}\wedge\theta=0$ (by definition) and $\la_{I,J,R}\wedge d\theta=0$, which comes as a consequence of~\eqref{eq:wedgiarea0} and the fact that $d\theta=-\sum_{i=1}^n dxy_i$ is, up to a sign, the standard symplectic form.

During the preparation of this paper  we became aware  that a basis of $\mathcal J^{2n+1-k}$ is provided also in the paper~\cite{BaldiBarnabeiFranchi}:
however, the basis in~\cite{BaldiBarnabeiFranchi} is presented by induction on $n$, while ours is given directly and is somewhat  manageable in the computations we need.

\subsection{Sketch of the proof of Rademacher's Theorem~\ref{thm:rademacher}}
For the reader's convenience we provide a sketch of the proof of our main result. 
Let  $\f:A\subset\W\to\V$ be intrinsic Lipschitz; by Theorem~\ref{thm:estensione} we can assume that $A=\W$. We now use Theorem~\ref{teo:approssimiamoooooooo} to produce a sequence of smooth maps $\f_i:\W\to\V$ converging uniformly to $\f$: it can be easily proved that the associated Heisenberg currents $\curr{\gr_{\f_i}}$ converge (possibly up to a subsequence) to a current $\Tcurr$ supported on $\gr_\f$ and, actually, that $\Tcurr=\tau\Shaus^{Q-k}\res\gr_\f$ for some bounded  function $\tau:\gr_\f\to\mathcal J_{2n+1-k}\setminus\{0\}$. 
Moreover we have $\partial\curr{\gr_{\f_i}}=0$ for every $i$ because of  Proposition~\ref{prop:correnti_cl_vs_H_INTRO}, therefore also $\partial\Tcurr=0$: as we will see, this equality carries the relevant geometric information.

Our aim is to prove that, at a.e. $w\in\W$, the blow-up of $\f$ at $w$ (i.e., the limit as $r\to+\infty$ of $\de_{r}((w\f(w))^{-1}\gr_\f)$) is the graph of an intrinsic linear map; a priori, however, there could exist many possible blow-up limits $\psi$ associated with different diverging scaling sequences $(r_j)_j$. In Lemma~\ref{lem:blowuptinvariante} we prove the following: for a.e. $\bar w\in\W$, {\em all} the possible blow-ups $\psi$ of $\f$ at $\bar w$ are $t$-invariant, i.e., 
\[
\psi(w\exp(tT))=\psi(w(0,0,t))=\psi(w)\qquad\text{for every }t\in\R,w\in\W.
\]
The proof of Lemma~\ref{lem:blowuptinvariante} makes use of the  Rademacher's Theorem   proved for the case of codimension 1 in~\cite{FSSCJGA-Diff}. 

Let then $\bar w$ be such a point and fix a $t$-invariant blow-up $\psi$ of $\f$ at $\bar w$ associated with a scaling sequence $(r_j)_j$. It is a good point to notice that, being both intrinsic Lipschitz (because it is the limit of uniformly intrinsic Lipschitz maps) and $t$-invariant, $\psi$ is necessarily Euclidean Lipschitz, see Lemma~\ref{lem:tinvariantLipschitz}. Consider now  the current $\Tcurr_\infty$ defined (up to passing to a subsequence) as the blow-up limit along $(r_j)_j$  of $\Tcurr$ at $\bar p:=\bar w\f(\bar w)\in\gr_\f$, namely,
\[
\Tcurr_\infty:=\lim_{j\to\infty}(\de_{r_j}\circ L_{\bar p^{-1}})_\#\Tcurr
\]
where $L_{\bar p^{-1}}$ denotes left-translation by $\bar p^{-1}$ and the subscript $\#$ denotes push-forward. If one assumes that $\bar p$ is also a Lebesgue point (in a suitable sense) of the function $\tau$, then the following properties hold for $\Tcurr_\infty$:
\begin{itemize}
\item $\Tcurr_\infty=f\,\tau(\bar p) \Shaus^{Q-k}\res \gr_\psi$ for some positive and bounded function $f$ on $\gr_\psi$;
\item $\gr_\psi$ is locally Euclidean rectifiable, in particular it is locally $\H$-rectifiable;
\item $\partial\Tcurr_\infty=0$, because $\Tcurr_\infty$ is limit of boundaryless currents.
\end{itemize}
We can then apply  Theorem~\ref{teo:correntinormalirettificabili} to deduce  that $[t^\H_{\gr_\psi}(p)]_\mathcal J$ is a multiple of $\tau(\bar p)$ for a.e. $p\in\gr_\psi$. By $t$-invariance, the unit  tangent vector $t_{\gr_\psi}(p)$ coincides with  $t^\H_{\gr_\psi}(p)$. 
Summarizing, we have a $t$-invariant Euclidean Lipschitz submanifold $\gr_\psi$ whose unit tangent vector $t_{\gr_\psi}$ is always vertical (i.e., of the form $t_{\gr_\psi}=t'\wedge T$ for a suitable multi-vector $t'$) and has the property that, for a.e. point $p$,  $[t_{\gr_\psi}(p)]_\mathcal J$ is a multiple of  $\tau(\bar p)\in\mathcal J_{2n+1-k}\setminus\{0\}$. If we could guarantee that there is a unique (up to a sign) unit simple  vector $\bar t$ that is vertical  and such that $[\,\bar t\,]_\mathcal J$ is a multiple of $\tau(\bar p)$, then  we would conclude that $\gr_\psi$ is always tangent to that particular $\bar t$, i.e., that $\gr_\psi$ is a vertical plane $\mathscr P$. Since $\bar t$ (and then $\mathscr P$) depends only on $\bar p$ and not on the particular sequence $(r_j)_j$, then the blow-up $\mathscr P$ is unique and is the graph of an intrinsic linear map $\psi$: this would conclude the proof.

Unluckily, this is not always the case: in fact, in the second Heisenberg group $\H^2$ the unit simple vertical 3-vectors $X_1\wedge Y_1\wedge T$ and $-X_2\wedge Y_2\wedge T$ have the property that
\[
[ X_1\wedge Y_1\wedge T]_\mathcal J = [- X_2\wedge Y_2\wedge T]_\mathcal J.
\]
This, however, is basically the worst-case scenario. A key, technically demanding result is  Proposition~\ref{prop:almassimo2}, where we prove that there exist at most two linearly independent unit simple vertical  vectors $\bar t_1,\bar t_2$ such that $[ \bar t_1]_\mathcal J = [\bar t_2]_\mathcal J$ are multiples of $\tau(\bar p)$; moreover,  the planes $\mathscr P_1,\mathscr P_2$ associated (respectively) with $\pm\bar t_1,\pm\bar t_2$  are not {\em rank-one connected}, i.e., $\dim\mathscr P_1\cap\mathscr P_2$ has codimension at least 2 in $\mathscr P_1$ (equivalently, in $\mathscr P_2$).  This means that the vertical Euclidean Lipschitz submanifold $\gr_\psi$ has at most two possible   tangent planes $\mathscr P_1,\mathscr P_2$; however (see e.g. \cite[Proposition 1]{BallJamesARMA87} or \cite[Proposition 2.1]{MullerCetraro}) the fact that these two planes are not rank-one connected forces $\gr_\psi$ to be a plane (either $\mathscr P_1$ or $\mathscr P_2$) itself.

This is not the  conclusion yet: we have for the moment proved that, for a.e. $\bar w\in\W$, all the possible blow-ups of $\f$ at $\bar w$ are either the map $\psi_1$ parameterizing  $\mathscr P_1$, or the map $\psi_2$ parameterizing $\mathscr P_2$; both  are determined by $\tau(\bar p)$ (i.e. by $\bar w$) only. However, it is not difficult to observe that   the family of all possible blow-ups of $\f$ at a fixed point must enjoy a suitable connectedness property, hence it cannot consist of the two  points $\psi_1,\psi_2$ only. This proves the uniqueness of blow-ups and concludes the proof of our main result.

\subsection{Structure of the paper}
In Section~\ref{sec:gruppi} we introduce intrinsic Lipschitz graphs in Carnot groups and prove Theorems~\ref{teo:equivdefLip},~\ref{thm:estensione} and~\ref{teo:approssimiamoooooooo}.  
Heisenberg groups are introduced in Section~\ref{sec:Heisenberg}, where we focus on the algebraic preliminary material, in particular about multi-linear algebra and the Rumin's complex. We also provide the basis of Rumin's spaces of Proposition~\ref{prop:baseJintro}, introduce Heisenberg currents and prove Proposition~\ref{prop:correnti_cl_vs_H_INTRO}. Eventually, we state  Proposition~\ref{prop:almassimo2}, that we use in the proof of Rademacher's Theorem~\ref{thm:rademacher}, and whose long and tedious proof is postponed to Appendix~\ref{app:max2}.
In Section~\ref{sec:intrLipgrHeis} we deal with intrinsic Lipschitz graphs of low codimension: in particular, we define intrinsic differentiability and we prove the crucial Lemma~\ref{lem:blowuptinvariante}. We also introduce $\H$-regular submanifold and $\H$-rectifiable sets and  we study (Euclidean)  $C^1$-regular intrinsic  graphs. 
Section~\ref{sec:constancytheorem} is devoted to the proof of the Constancy-type Theorems~\ref{thm:Constancy} and~\ref{teo:correntinormalirettificabili}. 
The proof of Rademacher's Theorem~\ref{thm:rademacher} is provided in Section~\ref{sec:dimrademacher}. 
Eventually, Section~\ref{sec:applic} contains the applications of our main result concerning Lusin's Theorem~\ref{thm:Lusinbreve}, the equivalence between $\H$-rectifiability and ``Lipschitz'' $\H$-rectifiability (Corollary~\ref{cor:HrectifiableC1vsLip}) and the area formula of Theorem~\ref{thm:formulaarea}.\medskip

{\em Acknowledgments.} 
We are grateful to G.~Alberti and A.~Marchese for pointing out  reference~\cite{Silhavy} to us. 
The author  wishes to thank F.~Boarotto, A.~Julia and S.~Nicolussi Golo for providing a pleasant  environment and for their patience during the long preparation of this work.




\section{Intrinsic Lipschitz graphs in Carnot groups: extension and approximation results}\label{sec:gruppi}
In this section we introduce Carnot groups and intrinsic Lipschitz graphs; our  goal is to prove the extension and approximation results stated in Theorems~ \ref{thm:estensione} and \ref{teo:approssimiamoooooooo}. These two results are used  later in the paper for intrinsic Lipschitz graphs in Heisenberg groups; however,  they can be proved with no extra effort in the wider setting  of {\em Carnot groups} and  we will therefore operate in this framework, that also allows for some  simplifications in the notation. The presentation of Carnot groups will be only minimal and we refer to \cite{FS,BLU,GromovWithin,LeDonneAPrimer,SCSomeTopics} for a more comprehensive treatment. The reader looking for a thorough  account on intrinsic Lipschitz graphs might instead  consult \cite{FSJGA}.

\subsection{Carnot groups: algebraic and metric preliminaries}
A Carnot (or {\em stratified})  group is a connected, simply connected and nilpotent Lie group whose Lie algebra $\galg$ is {\em stratified}, i.e., it possesses a decomposition  $\galg=\galg_1\oplus\dots\oplus\galg_s$ such that
\[
\forall\ j=1,\dots,s-1\quad \galg_{j+1}=[\galg_j,\galg_1],\qquad\galg_s\neq\{0\}\qquad\text{and}\qquad [\galg_s,\galg]=\{0\}.
\]
We refer to the integer $s$ as the {\em step} of $\G$ and to $m:=\,$dim $\galg_1$ as its {\em rank}; we also denote by $d$ the topological dimension of $\G$. The group identity is denoted by $0$ and, as customary, we identify $\galg$, $T_0\G$ and the algebra of left-invariant vector fields on $\G$. The elements of $\galg_1$ are referred to as {\em horizontal}.

The exponential map $\exp:\galg\to\G$ is a diffeomorphism and, given a basis $X_1,\dots,X_d$ of $\galg$, we will often identify $\G$ with $\R^d$ by means of exponential coordinates:
\[
\R^d\ni x=(x_1,\dots,x_d)\longleftrightarrow \exp\left( x_1X_1+\dots+x_dX_d\right)\in\G.
\]
We will also assume that the basis is adapted to the stratification, i.e., that
\begin{align*}
&  X_1,\dots,X_m \text{ is a basis of $\galg_1$, and}\\
& \forall\ j=2,\dots,s,\ X_{\dim (\galg_1\oplus\cdots\oplus\galg_{j-1})+1},\dots, X_{\dim (\galg_1\oplus\cdots\oplus\galg_{j})}\text{ is a basis of $\galg_j$.} 
\end{align*}
In these coordinates, one has 
\begin{equation}\label{eq:formacampiorizzontali}
X_i(x)=\partial_{x_i}+\sum_{j=m+1}^d P_{i,j}(x)\partial_{x_j}\quad\text{for every }i=1,\dots,m
\end{equation}
for suitable polynomial functions $P_{i,j}$. A one-parameter family $\{\de_\la\}_{\la>0}$ of {\em dilations} $\de_\la:\galg\to\galg$  is defined by (linearly extending)  
\[
\text{$\de_\la(X):=\la^j X$ for any $X\in\galg_j$;}
\]
notice that dilations are Lie algebra homomorphisms and $\de_{\la\mu}=\de_\la\circ\de_\mu$. By composition with $\exp$ one can then define a one-parameter family, for which we use the same symbol, of group isomorphisms $\de_\la:\G\to\G$.

We fix a left-invariant homogeneous distance $d$ on $\G$, so that
\[
d(xy,xz)=d(y,z)\qtaq d(\de_\la x,\de_\la y)=\la d(x,y)\qquad\text{for all }x,y,z\in\G,\la>0.
\]
We use $d$ to denote both the distance on $\G$ and its topological dimension, but no confusion  will ever arise. We denote by $B(x,r)$ the open ball of center $x\in\G$ and radius $r>0$; it will also be convenient to denote by $\|\cdot\|_\G$ the homogeneous norm defined for $x\in\G$ by $\|x\|_\G:=d(0,x)$.  Recall that $\leb d$ is a Haar measure on $\G\equiv\R^d$ and that the homogeneous dimension of $\G$ is the integer $Q:=\sum_{j=1}^s j\dim\galg_j$. One has
\[
\leb d (B(x,r)) = r^Q \leb d (B(0,1))\qquad\text{for all }x\in\G,r>0.
\]
The number $Q$ is always greater than $d$ (apart from the Euclidean case $s=1$) and it coincides with the Hausdorff dimension of $\G$. Since also the Hausdorff $Q$-dimensional measure is a Haar measure, it coincides with $\leb d$ up to a constant.  

Given a measurable function $f:\G\to\R$ we denote by $\nabla_\G f=(X_1f,\dots,X_mf)$ its  horizontal derivatives in the sense of distributions. It is well-known  that, if $f$ is Lipschitz continuous, then it is {\em Pansu differentiable} almost everywhere~\cite{PansuAnnals89} and, in particular, the pointwise horizontal gradient $\nabla_\G f$ exists almost everywhere on $\G$. Moreover (see e.g. \cite{FSSCBUMI,GNLip}) we have 
\begin{equation}\label{eq:W1inf=Lip}
\text{if $f:\G\to\R$ is continuous, then $f$ is Lipschitz  if and only if $\nabla_\G f\in L^\infty(\G)$}
\end{equation}
where Lipschitz continuity, of course, is meant with respect to the homogeneous distance $d$ on $\G$. It is worth mentioning that  the Lipschitz constant of $f$ is bounded by $\|\nabla_\G f\|_{L^\infty(\G)}$, apart from multiplicative constants that depend only on the distance $d$,  both from below and from above.

We will need later the following result, proved in \cite[Lemma 2.2]{VAnnSNS}, where we denote by $\eexp(X)(x)$ the point reached in unit time by the integral curve of a vector field $X$ starting at a point $x$.

\begin{lemma}\label{lem:derivatecoercive}
Let $f:\R^d\to\R$ be a continuous function and let $Y$ be a smooth vector field in $\R^d$. Assume that $Yf\geq \de$ holds, in the sense of distributions, on an open set $U\subset\R^d$ and for a suitable   $\de\in\R$. If $x\in U$ and $T>0$ are such that $\eexp(hY)(x)\in U$ for every $h\in[0,T)$, then
\[
f(\eexp(tY)(x))\geq f(x)+ \de t\quad  \text{ for every }t\in[0,T).
\]
\end{lemma}

\subsection{Intrinsic Lipschitz graphs}\label{subsec:intrLipgr}

Following \cite{FSJGA}, we fix a  splitting $\G=\W\V$ in terms of a couple $\W,\V$ of homogeneous (i.e., invariant under dilations) and complementary (i.e., $\W\cap\V=\{0\}$ and $\G=\W\V$) Lie subgroups  of $\G$. In exponential coordinates, $\W,\V$ are linear subspaces of $\G\equiv\R^d$. Clearly, the splitting induces for every $x\in\G$ a unique decomposition $x=x_\W x_\V$ such that $x_\W\in\W$ and $x_\V\in\V$; we will sometimes refer to the maps $x\mapsto x_\W$ and $x\mapsto x_\V$ as the {\em projections} of $\G$ on $\W$ and on $\V$, respectively.

Given $A\subset\W$ and a map $\f:A\to\V$, the {\em intrinsic graph} $\gr_\f$ of $\f$ is the set
\[
\gr_\f:=\{w\f(w):w\in A\}\subset\G.
\]
The notion of intrinsic Lipschitz continuity  for maps $\f$ from $\W$ to $\V$ was introduced  by B.~Franchi, R.~Serapioni and F.~Serra Cassano \cite{FSSCJNCA} in terms of a cone property for $\gr_\f$. The {\em intrinsic cone} $\cono_\al$ of aperture $\al> 0$ and axis $\V$ is 
\[
\cono_\al:=\{x\in\G:\|x_\W\|_\G\leq\al\|x_\V\|_\G \}\,.
\]
Observe that $\cono_\al$  is homogeneous (invariant under dilations) and that $\V\subset \cono_\al$. For $x\in\G$ we also introduce the cone  $\cono_\al(x):=x\cono_\al$ with vertex $x$.

\begin{definition}\label{def:grafLipCarnot}
Let $A\subset\W$; we say that  $\f:A\to\V$ is {\em intrinsic Lipschitz} if there exists $\al>0$ such that
\begin{equation}\label{eq:defLipCarnot}
\forall\:x\in\gr_\f\qquad \gr_\f\cap \cono_\al(x)=\{x\}.
\end{equation}
The {\em intrinsic Lipschitz constant} of $\f$ is  $\inf\{\frac1\al:\al>0\text{ and~\eqref{eq:defLipCarnot} holds}\}$.
\end{definition}

Since all homogeneous distances on $\G$ are equivalent, Definition~\ref{def:grafLipCarnot} is clearly independent from the fixed distance $d$ on the group. It was proved in \cite[Theorem 3.9]{FSJGA} that, if $\f$ is intrinsic Lipschitz, then the Hausdorff dimension of $\gr_\f$ is the same as the Hausdorff dimension of the domain $\W$; actually, the corresponding Hausdorff measure on $\gr_\f$ is  Ahlfors regular, and then also locally finite on $\gr_\f$. In particular  one always has
\begin{equation}\label{eq:grLiptrascurabili}
\leb d(\gr_\f)=0
\end{equation}
provided, of course,  we are not in the trivial case $\W=\G$, $\V=\{0\}$.

\begin{remark}\label{rem:nointersezalloragrafico}
For the purpose of future references, we observe the following easy fact. Let $S\subset \G$ and $\al>0$ be fixed; if
\[
\forall\:x\in\ S\qquad S\cap \cono_\al(x)=\{x\},
\] 
then $S=\gr_\f$ for suitable $\f:A\to\V$ (which is clearly intrinsic Lipschitz) and $A\subset \W$. See e.g.~\cite[\S2.2.3]{FSJGA}.
\end{remark}

\subsection{A level set definition of co-horizontal intrinsic Lipschitz graphs}
From now on we assume that the splitting $\W\V$ of $\G$ is fixed in such a way that $\V$ is not trivial ($\V\neq\{0\}$) and it is horizontal, i.e.,  $\V\subset\exp(\galg_1)$. Of course, this poses some algebraic restrictions: for instance, $\V$ is forced to be Abelian.  Moreover, it can be easily checked that  {\em free} Carnot groups (of step at least $2$) have no  splitting such that $\V$ is horizontal and $\dim\V\geq 2$. Nonetheless, the theory we are going to develop here is rich enough to include  intrinsic Lipschitz graphs of codimension 1 in any Carnot group (in fact, every 1-dimensional horizontal subgroup $\V$ of a Carnot group provides a splitting  $\W\V$ for some $\W$) and intrinsic Lipschitz graphs of  codimension at most $n$ in the Heisenberg group $\H^n$, which are the main object of study of the present paper. 

With such assumptions on the splitting $\W\V$, intrinsic Lipschitz graphs $\F:A\subset\W\to\V$ will be called {\em co-horizontal} (see~\cite{ADDDLD}). We denote by $k$ the topological dimension of $\V$ 
and we assume without loss of generality that the adapted basis $X_1,\dots,X_d$ of $\galg$ has been fixed in such a way that
\[
\V=\exp(\Span\{X_1,\dots,X_k\}).
\]
We consequently identify  $\V$ with $\R^k$ through the map
\begin{equation}\label{eq:VVVVVVVV} 
\R^k\ni (v_1,\dots,v_k)\longleftrightarrow \exp(v_1X_1+\dots+v_kX_k)\in\V
\end{equation}
and  we accordingly write $ v=(v_1,\dots,v_k)\in\V$.  The map in \eqref{eq:VVVVVVVV} turns out to be a group isomorphism as well as a biLipschitz map between $(\R^k,|\cdot|)$ and $(\V,d)$: this proves that the Hausdorff dimension of $\V$ equals the topological dimension $k$. We observe that, since the flow of a left-invariant vector field corresponds to right multiplication, we have
\[
xv=\eexp (v_1X_1+\dots+v_kX_k)(x).
\]
In particular, the projections on the factors $\W,\V$ can be written as
\[
x_\V=\exp(x_1X_1+\dots+x_kX_k),\qquad x_\W =x\,x_\V^{-1}=\eexp(-(x_1X_1+\dots+x_kX_k))(x)
\]
and are therefore smooth maps.

Our first goal is to provide the equivalent characterization of co-horizontal intrinsic Lipschitz graphs  stated in Theorem \ref{teo:equivdefLip}. We however need some preparatory lemmata as well as some extra  convention about notation. First, we introduce the homogeneous (pseudo)-norm
\[
\|x\|_*:=\left(\sum_{j=1}^s \sum_{i:X_i\in\galg_j}|x_i|^{\frac{2\;\!s!}{j}}\right)^{\frac1{2\;\!s!}},\qquad x\in\G,
\]
that is equivalent to $\|\cdot\|_\G$ in the sense that there exists $C_*\geq 1$ such that
\begin{equation}\label{eq:introducoC*}
\|x\|_\G/C_*\leq \|x\|_*\leq C_*  \|x\|_\G\qquad\forall\ x\in\G.
\end{equation}
Observe that $x\mapsto\|x\|_*$ is of class $C^\infty$ in $\G\setminus\{0\}$. 
Second, given $i\in\{1,\dots,k\}$, $\be>0$ and $\ep>0$ we introduce the homogeneous cone
\[
\begin{split}
\cono_{i,\be,\ep} :=& \left\{wv:w\in\W,\ v\in\V,\ |v_i|+\ep\sum_{j\in\{1,\dots,k\}\setminus\{i\}}|v_j|\geq \beta\|w\|_*\right\}
\\
= & \left\{x\in\G:|x_i|+\ep\sum_{j\in\{1,\dots,k\}\setminus\{i\}}|x_j|\geq \beta\|x_\W\|_*\right\},
\end{split}
\]
where we used the fact that $x_\V=(x_1,\dots,x_k)$. 
Third, if $t\in\R$ and $f:D\to\R$ is a real-valued function defined on some set $D$,  we denote by $\{f\geq t\}$ the set $\{x\in D:f(x)\geq t\}$. Similar conventions are understood when writing $\{f>t\},\ \{f<t\},\ \{f=t\},\ \{t_1<f<t_2\}$, etc.

\begin{lemma}\label{lem:lemmacono}
For every $i\in\{1,\dots,k\}$, $\be>0$ and $\ep\in(0,1)$ there exists a 1-homogeneous Lipschitz function $f_{i,\be,\ep}:\G\to\R$ such that
\begin{align}
& f_{i,\be,\ep}(0)=0\label{eq:0}\\
& 1\leq X_if_{i,\be,\ep}\leq 3\quad\leb d\text{-a.e. on }\G\label{eq:1}\\
& \ep\leq X_\ell f_{i,\be,\ep}\leq 3\ep\quad\leb d\text{-a.e. on }\G\quad \forall\ \ell\in\{1,\dots,k\}\setminus\{i\}\label{eq:2}\\
& \{f_{i,\be,\ep}\geq 0\} \subset \cono_{i,\be,\ep}.  \label{eq:3}
\end{align}
Moreover, if  $0<\be\leq\bar\be$, then the Lipschitz constant of $f_{i,\be,\ep}$ can be controlled in terms of $\ep$ and $\bar\be$ only.
\end{lemma}
\begin{proof}
We can without loss of generality assume that $i=1$. For $x\in\G$ define
\[
f(x):=\left\{
\begin{array}{ll}
2[x_1+\ep(x_2+\dots+x_k)-\be\| x_\W\|_*]\quad & \text{if }|x_1+\ep(x_2+\dots+x_k)|\leq 2\be\| x_\W\|_*\vspace{.1cm}\\
(x_1+\ep(x_2+\dots+x_k)) & \text{if }x_1+\ep(x_2+\dots+x_k)> 2\be\| x_\W\|_*\vspace{.1cm}\\
3[x_1+\ep(x_2+\dots+x_k)] & \text{if }x_1+\ep(x_2+\dots+x_k)<- 2\be\| x_\W\|_*.
\end{array}
\right.
\]
We prove that $f_{1,\be,\ep}:=f$ satisfies all the claimed statements. Property \eqref{eq:0} and the homogeneity of $f$ are immediate. Property \eqref{eq:3} is equivalent 
to the implication
\[
|x_1|+\ep(|x_2|+\dots+|x_k|)< \beta\|x_\W\|_*\ \Longrightarrow\ f(x)<0
\]
that one can easily check.  The function $f$ is continuous on $\G$ and smooth on $\G\setminus D$, where 
\[
D:=\{x\in\G:|x_1+\ep(x_2+\dots+x_k)| = 2\be\| x_\W\|_*  \}.
\]
Since $D$ is $\leb d$-negligible, statements \eqref{eq:1} and \eqref{eq:2} follow if we prove that for every $\ell=2,\dots,k$
\begin{equation}\label{eq:claim4}
1\leq X_1f\leq 3\quad\text{and}\quad \ep\leq X_\ell f\leq 3\ep\qquad\text{on }\G\setminus D.
\end{equation}
Using~\eqref{eq:formacampiorizzontali} one gets
\begin{equation}\label{eq:gradfbe}
\begin{array}{ll}
\nabla_\G f(x)=(1,\ep,\dots,\ep,0,\dots,0) & \text{if }x_1+\ep(x_2+\dots+x_k)> 2\be\| x_\W\|_*\vspace{.1cm}\\
\nabla_\G f(x)=(3,3\ep,\dots,3\ep,0,\dots,0)\quad & \text{if }x_1+\ep(x_2+\dots+x_k)<- 2\be\| x_\W\|_*\,.
\end{array}
\end{equation}
We now notice that, for any $x\in\G$, the map $y\mapsto y_\W$ is constant on the coset $x\V$, which is a smooth submanifold tangent to $X_1,\dots,X_k$: this implies that
\[
(X_1f,\dots,X_kf)(x)=(2,2\ep,\dots,2\ep)\qquad\text{if }|x_1+\ep(x_2+\dots+x_k)|<2\be\| x_\W\|_*
\]
which, together with \eqref{eq:gradfbe}, implies \eqref{eq:claim4}. 

We have only  to check that $f$ is Lipschitz continuous on $\G$ and that a bound on the Lipschitz constant can be given in terms of $\ep$ and $\bar\be$. Taking into account \eqref{eq:gradfbe} and the continuity of $f$ on $\G$, by \eqref{eq:W1inf=Lip} it is enough to prove the function $g:\G\to\R$ defined by
\[
g(x):=2[x_1+\ep(x_2+\dots+x_k)-\be\| x_\W\|_*]
\]
satisfies
\begin{equation}\label{bastano1-1}
|\nabla_\G g|\leq C\quad\text{on }\{x\in\G:|x_1+\ep(x_2+\dots+x_k)|<2\be\| x_\W\|_*\}
\end{equation}
for some positive $C$. Since $x\mapsto x_\W$ is  smooth on $\G$ and $\|\cdot\|_*$ is smooth on $\G\setminus\{0\}$, we get that  $g$ is smooth on $\G\setminus\V$. Moreover $g$ is 1-homogeneous, thus $\nabla_\G g$ is 0-homogeneous  (i.e., invariant under dilations) and continuous on $\G\setminus\V$. Inequality \eqref{bastano1-1} will then follow if we prove that
\[
|\nabla_\G g|\leq C\quad\text{on } \pa B(0,1)\cap\{x\in\G:|x_1+\ep(x_2+\dots+x_k)|\leq\be\| x_\W\|_*\};
\]
in turn, this inequality and the bound (in terms of $\ep,\bar\be$) on the Lipschitz constant of $f$ follow by proving that
\begin{equation}\label{eq:uffa}
|\nabla_\G g|\leq C\quad\text{on } \pa B(0,1)\cap\{x\in\G:|x_1+\ep(x_2+\dots+x_k)|\leq\bar\be\| x_\W\|_*\}.
\end{equation}
The set $\V$ is closed, while $\pa B(0,1)\cap\{x\in\G:|x_1+\ep(x_2+\dots+x_k)|\leq\bar\be\| x_\W\|_*\}$ is compact; since they are disjoint,  they have positive distance and the continuity of $\nabla_\G g$ on $\G\setminus\V$ ensures that
\[
\sup \big\{|\nabla_\G g(x)|:x\in\pa B(0,1)\text{ and }|x_1+\ep(x_2+\dots+x_k)|\leq \bar\be\| x_\W\|_*\big\} < +\infty,
\]
which is~\eqref{eq:uffa} and allows to conclude.
\end{proof}

\begin{lemma}\label{lem:lemmafunzione}
Let $A\subset\W$ be nonempty and let $\f:A\to\V$ be intrinsic Lipschitz. Then for every $\ep\in(0,1)$ and $i\in\{1,\dots,k\}$ there exists a Lipschitz function $f_{i,\ep}:\G\to\R$ such that
\begin{align}
& \gr_\f\subset\{f_{i,\ep}=0\}\label{eq:funz1}\\
& 1\leq X_if_{i,\ep}\leq 3 \qquad\leb d\text{-a.e. on }\G\label{eq:funz2}\\
& \ep \leq X_\ell f_{i,\ep}\leq 3\ep\qquad\leb d\text{-a.e. on }\G\quad \forall\ell\in\{1,\dots,k\}\setminus\{i\}. \label{eq:funz3}
\end{align}
Moreover, if the intrinsic Lipschitz constant of $\f$ is not greater than $\bar\al>0$, then  the Lipschitz constant of $f_{i,\ep}$ can be bounded in terms of $\ep$ and $\bar\al$ only. 
\end{lemma}
\begin{proof}
Assume that the intrinsic Lipschitz constant of $\f$ is not greater than some $\bar\Lambda>0$ and define $\al:=(2\Lambda)^{-1}$; then \eqref{eq:defLipCarnot} holds for such $\al$.  Recalling that the constant $C_*>0$ was introduced in~\eqref{eq:introducoC*}, we set $\be:=k\:\!C_*^2/\al$. Taking into account the inequalities
\begin{align*}
& |x_i|+\ep\!\!\!\sum_{j\in\{1,\dots,k\}\setminus\{i\}}\!\!\!\!|x_j|\ \leq |x_1|+\dots+|x_k|\leq k\|x_\V\|_*\leq kC_*\|x_\V\|_\G,\\
& \beta\|x_\W\|_* \geq \be\|x_\W\|_\G/{C_*}
\end{align*}
we obtain the inclusion $\cono_{i,\be,\ep}\subset \cono_\al$. For $y\in\G$ set $f_y(x):=f_{i,\be,\ep}(y^{-1}x)$, where $f_{i,\be,\ep}$ is the function  provided by Lemma \ref{lem:lemmacono}, and define
\[
f(x):=\sup_{y\in \gr_\f} f_y(x).
\]
We prove that $f_{i,\ep}:=f$ satisfies the claimed statement.

Let $x\in \gr_\f$; then $f(x)\geq f_x(x)=0$, while for every $y\in \gr_\f\setminus\{x\}$ one has $f_y(x)<0$ because of \eqref{eq:3} and
\[
\gr_\f\cap \{f_y\geq0\}\    =\   \gr_\f\cap y \{f_{i,\be,\ep}\geq0\}  
\subset\   \gr_\f\cap y \cono_{i,\be,\ep}\   \subset\  \gr_\f\cap y\cono_\al\  =\ \{y\}.
\]
This proves that $f(x)=0$, which is \eqref{eq:funz1}.

The functions $f_y$ are uniformly Lipschitz continuous, hence $f$ shares the same Lipschitz continuity. Let $x\in\G$ be fixed; then for every $\eta>0$ there exists $y\in \gr_\f$ such that
\[
f_y(x)\geq f(x)-\eta.
\]
Since $X_i f_y\geq 1$, by Lemma \ref{lem:derivatecoercive} we have for every $t\geq 0$
\[
f(\eexp(tX_i)(x))\geq f_y(\eexp(tX_i)(x))\geq f_y(x) +t\geq f(x)+t -\eta.
\]
By the arbitrariness of $\eta$ one obtains
\[
f(\eexp(tX_i)(x))\geq f(x)+t\quad\text{for every }t\geq 0,
\]
i.e., $X_if\geq 1$ a.e. on $\G$. A similar argument, using $g:=-f$ and the inequality $X_ig\geq -3$, shows that $f(\eexp(tX_i)(x))\leq f(x)+3t$ for every $t\geq 0$, i.e., that $X_if\leq 3$ a.e. on $\G$. This proves \eqref{eq:funz2}. 

This proof of \eqref{eq:funz3} is completely analogous and we omit it.
\end{proof}

The following lemma is most likely well-known; we however provide a proof for the sake of completeness.

\begin{lemma}\label{lem:zeriLip}
Let $f:\R^k\to\R^k$ be a Lipschitz map such that there exists $\de>0$ for which
\begin{equation}\label{eq:conddelta}
\langle f(x+v)-f(x),v\rangle \geq \de|v|^2\qquad\text{for every }x,v\in\R^k.
\end{equation}
Then there exists a unique $\bar x\in\R^k$ such that $f(\bar x)=0$.
\end{lemma}
\begin{proof}
We reason by induction on $k$ and leave the case $k=1$ as an  exercise to the reader. We assume that the lemma holds for some $k\geq 1$ and we prove it for $k+1$. 

By the one-dimensional case of the lemma, for every $x\in\R^k$ there exists a unique $g(x)\in\R$ such that $f_{k+1}(x,g(x))=0$; we claim that $g$ is Lipschitz continuous. Letting $L$ denote the Lipschitz constant of $f$ we indeed have for every $x,y\in\R^k$
\begin{align*}
f_{k+1}(y,g(x)+\tfrac L\de|y-x|) \stackrel{\eqref{eq:conddelta}}{\geq} &  f_{k+1}(y,g(x)) + L|y-x|\\
\stackrel{\phantom{\eqref{eq:conddelta}}}{\geq} & f_{k+1}(x,g(x)) - L|y-x| + L|y-x|\ =\ 0
\end{align*}
and 
\begin{align*}
f_{k+1}(y,g(x)-\tfrac L\de|y-x|) \stackrel{\eqref{eq:conddelta}}{\leq} &  f_{k+1}(y,g(x)) - L|y-x|\\
\stackrel{\phantom{\eqref{eq:conddelta}}}{\leq} & f_{k+1}(x,g(x)) + L|y-x| - L|y-x|\ =\ 0.
\end{align*}
The last two displayed formulae imply that
\[
g(x)-\tfrac L\de|y-x| \leq g(y) \leq g(x)+\tfrac L\de|y-x|
\]
and the Lipschitz continuity of $g$ follows. In particular, the function $h:\R^k\to\R^k$ defined by $h(z):=(f_1,\dots,f_k)(z,g(z))$ is Lipschitz continuous; since 
\[
\langle h(z+v)-h(z),v\rangle \geq \de|v|^2\qquad\text{for every }z,v\in\R^k,
\]
by inductive assumption there is a unique $\bar z\in\R^k$ such that $h(\bar z)=0$. It follows that $\bar x:=(\bar z,g(\bar z))$ is the unique zero of $f$, which concludes the proof.
\end{proof}

Before passing to the  the main proof of this section, we recall once more that $\V$ is identified with $\R^k$ by
\[
\R^k\ni(v_1,\dots,v_k)\longleftrightarrow \exp(v_1X_1+\dots+v_kX_k)\in \V.
\]
This identification is understood, in particular, when considering scalar products between elements of $\R^k$ and $\V$ as in \eqref{eq:(b)2}.


\begin{remark}\label{rem:coercivoderivate}
It is easily seen that, for a given Lipschitz map $f:\G\to\R^k$,  statement~\eqref{eq:(b)2} is equivalent to the uniform ellipticity (a.k.a.~coercivity) of the matrix col$[X_1f|\dots|X_k f]$, i.e., to the fact that
\begin{equation}\label{eq:b(2)equiv}
\text{col$[X_1f|\dots|X_k f]$}(x)\geq \de\: I\qquad\text{for $\leb d$-a.e. }x\in\G
\end{equation}
in the sense of bilinear forms, where $I$ denotes the $k\times k$ identity matrix. Observe that such a matrix is defined a.e. on $\G$ by Pansu's Theorem~\cite[Th\'eor\`eme 2]{PansuAnnals89}.
\end{remark}

\begin{proof}[Proof of Theorem \ref{teo:equivdefLip}]
{\em Step 1.} We prove the implication (a)$\Rightarrow$(b). Consider the map 
\[
f:=(f_{1,\ep},\dots,f_{k,\ep}):\G\to\R^k,
\]
where $\ep\in(0,1)$ will be determined later and the functions $f_{i,\ep}$ are provided by Lemma \ref{lem:lemmafunzione}. The inclusion \eqref{eq:(b)1} follows from \eqref{eq:funz1}. In order to prove \eqref{eq:(b)2} we first observe that for $\leb d$-a.e. $x\in\G$ and $\leb k$-a.e. $v\in\R^k\equiv\V$ one has
\begin{align*}
\langle f(xv)-f(x),v\rangle =& \int_0^1 \left\langle \sum_{j=1}^kv_jX_jf(x\delta_tv),v\right\rangle\:dt,
\end{align*}
where we used the fact that $x\delta_tv=\eexp(t(v_1X_1+\dots+v_kX_k))(x)$. Therefore
\begin{align*}
\langle f(xv)-f(x),v\rangle = & \int_0^1  \sum_{i,j=1}^kv_iv_jX_jf_i(x\delta_tv) \:dt\\
= & \int_0^1  \sum_{i=1}^kv_i^2X_if_i(x\delta_tv) \:dt + \int_0^1  \sum_{\substack{i,j=1,\dots, k\\i\neq j}}v_iv_jX_jf_i(x\delta_tv) \:dt\\
\geq & (1-3(k^2-k)\ep) |v|^2.
\end{align*}
where, in the last inequality, we used \eqref{eq:funz2} and \eqref{eq:funz3}. If $k=1$, this inequality is  \eqref{eq:(b)2} with $\de=1$; if $k\geq 2$, \eqref{eq:(b)2} follows with $\de=1/2$ provided we choose $ \ep=(6(k^2-k))^{-1}$. This proves the implication (a)$\Rightarrow$(b).
%
%

{\em Step 2.} We now prove the converse implication (b)$\Rightarrow$(a); it is enough to prove that $Z_f:=\{f=0\}$ is the intrinsic graph of some intrinsic Lipschitz function $\f:\W\to\V$. For every $w\in\W$ define $f_w:\V\equiv\R^k\to\R^k$ as $f_w(v):=f(wv)$. By Lemma \ref{lem:zeriLip} there is a unique $\bar v=\bar v(w)$ such that $f_w(\bar v)=0$; we define $\f:\W\to\V$ by $\f(w):=\bar v$. For $\la\in(0,1)$, that will be fixed later, we introduce the homogeneous cone
\[
D_\la:=\bigcup_{v\in\V} \overline{B(v,\la \|v\|_\G)}=\bigcup_{v\in\V} v\overline{B(0,\la \|v\|_\G)}.
\]
By a simple topological argument (see e.g. \cite[Remark A.2]{DMV}) there exists $\al=\al(\la)>0$ such that $\cono_\al\subset D_\la$; in order to prove that $\f$ is intrinsic Lipschitz it is sufficient to show that
\begin{equation}\label{eq:ccccc}
Z_f\cap xD_\la =\{x\}\qquad\forall\; x\in Z_f.
\end{equation}
To this aim, for every $x\in Z_f$ and every $y\in xD_\la\setminus\{x\}$ one has by definition
\[
y=xvz\qquad \text{for some }v\in\V\setminus\{0\}\text{ and $z\in\G$ such that }d(0,z)\leq\la d(0,v).
\]
Denoting by $L$ the Lipschitz constant of $f$ we obtain
\begin{align*}
\langle f(y),v\rangle = & \langle f(xvz)-f(xv),v\rangle +  \langle f(xv)-f(x),v\rangle\\
\geq & - L\,d(0,z)|v| + \de|v|^2 \\
\geq & -L\la\,d(0,v)|v|+ \de|v|^2\\
\geq & -L\la\, C_*\|v\|_*|v|+ \de|v|^2\\
\geq & (\de-L\la C_*c_k)|v|^2
\end{align*}
for some positive constant $c_k$ depending on $k$ only. It follows that, provided $\la$ is chosen small enough, one has $\langle f(y),v\rangle>0$, hence $f(y)\neq 0$ and  $y\notin Z_f$. This proves \eqref{eq:ccccc} and concludes the proof of the theorem.
\end{proof}

\begin{remark}\label{rem:graficoglobale}
In Step 2 of the previous proof we showed that, if $f$ is as in Theorem \ref{teo:equivdefLip} (b), then the level set $\{f=0\}$ is an {\em entire} intrinsic Lipschitz graph, i.e., it is the intrinsic graph of a $\V$-valued map $\f$ defined on the whole $\W$.
\end{remark}

\begin{remark}\label{rem:dipendenzaalfa}
It is worth pointing out that, in the implication (b)$\Rightarrow$(a), the aperture $\al$ depends, apart from geometric quantities, only on the Lipschitz  and coercivity constants $L,\de$ of $f$. More precisely: if $f$ is as in Theorem \ref{teo:equivdefLip} (b), the Lipschitz constant $L$ of $f$ is not greater than some $\bar L>0$ and the coercivity constant $\de$ is not smaller than some $\bar\de>0$, then the aperture $\al$ (and hence the intrinsic Lipschitz constant of $\f$) can be controlled in terms of $\bar L$ and $\bar\de$ only.

A similar remark applies at the level of the implication (a)$\Rightarrow$(b): in fact, if  $\f$ is as in Theorem \ref{teo:equivdefLip} (a) and the intrinsic Lipschitz constant  of $\f$ is not greater than some positive $\Lambda$, then statement (b) in Theorem \ref{teo:equivdefLip}  holds with $\de=1/2$ and (by the second part of Lemma~\ref{lem:lemmafunzione}) a function $f$ with Lipschitz constant bounded in terms of $\Lambda$ only. 
\end{remark}

\subsection{Extension  and smooth approximation of co-horizontal intrinsic Lipschitz maps}
Given Remark~\ref{rem:graficoglobale}, Theorem~\ref{thm:estensione} is an  immediate consequence of  Theorem \ref{teo:equivdefLip}.

\begin{proof}[Proof of Theorem~\ref{thm:estensione}]
Let $S:=\gr_\f$ and consider $f:\G\to\R$ as given by Theorem \ref{teo:equivdefLip} (b); by Remark~\ref{rem:graficoglobale}, the level set $Z_f:=\{f=0\}$ is the graph of an intrinsic Lipschitz function  $\tilde\f:\W\to\V$ defined on the whole $\W$. Since $\gr_\f=S\subset Z_f=\gr_{\tilde\f}$, $\tilde\f$ is an extension of $\f$. The bound on the intrinsic Lipschitz constant of $\tilde \f$ follows from Remark~\ref{rem:dipendenzaalfa}.
\end{proof}

We now state a technical improvement of Theorem \ref{teo:equivdefLip} that will provide the key tool in the proof of the approximation result stated in Theorem \ref{teo:approssimiamoooooooo}. In case $k=1$, Proposition~\ref{prop:equivdefLipCinfty} should be compared with~\cite[Lemma~4.3]{VAnnSNS}.

\begin{proposition}\label{prop:equivdefLipCinfty}
Let $A\subset\W$ be nonempty and $\f:A\to\V$ be intrinsic Lipschitz. Then there exist $\de>0$ and a Lipschitz map $f:\G\to\R^k$ such that \eqref{eq:(b)1} (with $S:=\gr_\f$) and \eqref{eq:(b)2} hold  together with
\begin{align}
& f\in C^\infty(\G\setminus\{f=0\}).\label{eq:(b)3}
\end{align}
Moreover, if the intrinsic Lipschitz constant  of $\f$ is not greater than some positive $\Lambda$, then the statement holds with $\de=1/4$ and  a function $f:\G\to\R$ with Lipschitz constant bounded in terms of $\Lambda$ only.
\end{proposition}
\begin{proof}
By Theorem \ref{teo:equivdefLip} and Remark \ref{rem:coercivoderivate} there exist $\de>0$ and a Lipschitz function $g:\G\to\R^k$ such that $S$ is contained in the level set $Z_g:=\{g=0\}$ and
\begin{align}
& \widehat\nabla_\G g(x):=\text{col}[X_1g|\dots|X_kg](x) \geq 2\de I\qquad\text{for $\leb d$-a.e. }x\in\G.\label{eq:nablahatg}
\end{align}
By Remark~\ref{rem:dipendenzaalfa}, one can also assume  $\de=1/4$ and that the Lipschitz constant of $g$ is bounded in terms of $\Lambda$ only. For $j\in\N$ choose 
\begin{itemize}
\item[(i)] bounded open sets $(U_j)_{j\in\N}$ such that $\overline{U_j}\subset\G\setminus Z_g$ and $\G\setminus Z_g=\cup_j U_j$
\item[(ii)] positive numbers $\ep_j$ such that $\ep_j<d(U_j,Z_g)$
\item[(iii)] non-negative functions $u_j\in C^\infty_c(U_j)$ forming a partition of the unity on $\G\setminus Z_g$, i.e., $\sum_j u_j=1\text{ on }\G\setminus Z_g$.
\end{itemize}  
We can also assume that $\sum_j \chi_{U_j} \leq M$ for some $M>0$, where $\chi_{U_j} $ denotes the characteristic function of  $U_j$. Notice, in particular, that the sum in (iii) is locally finite. 

We are going to use the group convolution $\star$, see e.g. \cite[Chapter~1]{FS}: here we only recall  that, given $G:\G\to\R^k$ and  $H\in C^\infty_c(\G)$, the group convolution
\[
(G\star H)(x):=\int_{\G} G(y^{-1}x)H(y)d\leb d(y)= \int_\G G(y)H(xy^{-1})d\leb d(y)
\]
is a smooth function satisfying
\[
X(G\star H)=(XG)\star H\qquad\text{for every }X\in\galg.
\]
We fix a positive kernel $K\in C^\infty_c(B(0,1))$ such that $\int_\G K\,d\leb d=1$ and, for $\ep>0$, we set $K_\ep:=\ep^{-Q}K\circ\de_{1/\ep}$. Possibly reducing $\ep_j>0$, as  specified later in~\eqref{eq:wurstel1},~\eqref{eq:wurstel2} and~\eqref{eq:wurstel3}, define
\[
f:=\begin{cases}
\sum_{j} u_j (g\star K_{\ep_j})\qquad &\text{on }\G\setminus Z_g\\
0 &\text{on } Z_g.
\end{cases}
\]
Notice that the sum above is locally finite. We also observe that \eqref{eq:(b)1} holds because $S\subset Z_g\subset\{f=0\}$; actually, the equality $Z_g=\{f=0\}$ will come as a byproduct of what follows.

We first check that $f$ is continuous on $\G$; it is clearly smooth on $\G\setminus Z_g$, hence continuity has to be checked only at points of $Z_g$. Up to reducing $\ep_j$, we can assume that
\begin{equation}\label{eq:wurstel1}
|(g\star K_{\ep_j})-g| \leq d(U_j,Z_g)\qquad\text{ on }U_j
\end{equation}
so that
\[
|f(x)-g(x)|\leq\sum_{j} u_j(x)\:| (g\star K_{\ep_j})(x)-g(x)|\leq  d(x,Z_g)\qquad\forall\ x\in\G\setminus Z_g.
\]
This implies that $f$ is continuous also on $Z_g=\{g=0\}$ because, for  every $\bar x\in Z_g$, one has
\[
\lim_{x\to\bar x} |f(x)-f(\bar x)|
=\lim_{x\to\bar x} |f(x)|
\leq\lim_{x\to\bar x}  |g(x)|+d(x,\bar x)=0.
\] 

We prove that $f$ is Lipschitz continuous on $\G$. Since $Z_g$ is an intrinsic Lipschitz graph, by \eqref{eq:grLiptrascurabili} we have $\leb d(Z_g)=0$; by \eqref{eq:W1inf=Lip},  it is then enough to prove that $|\nabla_\G f|$ is bounded on $\G\setminus Z_g$. Using the properties of convolutions  and the fact that $\sum_j \nabla_\G u_j=0$ one gets
\begin{align}
\nabla_\G f = & \sum_j(g\star K_{\ep_j})\otimes(\nabla_\G u_j) + \sum_j u_j((\nabla_\G g)\star K_{\ep_j})\nonumber\\\
=& \sum_j(g\star K_{\ep_j}-g)\otimes(\nabla_\G u_j) + \sum_j u_j((\nabla_\G g)\star K_{\ep_j}).\label{eq:derivconvol}
\end{align}
The second sum in \eqref{eq:derivconvol} is  bounded by $\|\nabla_\G g\|_{L^\infty(\G)}$ and then by a multiple of the Lipschitz constant $L$ of $g$; possibly reducing $\ep_j$ we can assume that
\begin{equation}\label{eq:wurstel2}
|g\star K_{\ep_j}-g|\leq (\sup |\nabla_\G u_j|)^{-1}\qquad \text{on }U_j
\end{equation}
so that the first sum in \eqref{eq:derivconvol} is bounded by $M$. This proves that $f$ is Lipschitz continuous on $\G$ with Lipschitz constant bounded in terms of the Lipschitz constant of $g$ (and then in terms of $\Lambda$ only).

Eventually, we have to check that $f$ satisfies \eqref{eq:(b)2}; we prove this by checking the equivalent inequality \eqref{eq:b(2)equiv}. Writing $\widehat\nabla_\G f:=\text{col}[X_1f|\dots|X_k f]$ and reasoning as in \eqref{eq:derivconvol} we obtain
\[
\widehat\nabla_\G f =  \sum_j(g\star K_{\ep_j}-g)\otimes(\widehat\nabla_\G u_j) + \sum_j u_j((\widehat\nabla_\G g)\star K_{\ep_j})=: A+B.
\]
By \eqref{eq:nablahatg} we have $B\geq2\de I$ a.e. on $\G$. Given $\eta>0$, possibly reducing $\ep_j$ we can assume that
\begin{equation}\label{eq:wurstel3}
|g\star K_{\ep_j}-g|\leq \eta(\sup |\widehat\nabla_\G u_j|)^{-1}\qquad \text{on }U_j
\end{equation}
so that $\|A\|_{L^\infty(\G)}\leq M\eta$.  Inequality \eqref{eq:b(2)equiv} immediately follows  provided one chooses $\eta=\eta(\de)$ small enough. The proof is complete.
\end{proof}

We have all the tools needed for the proof of the approximation result stated in Theorem~\ref{teo:approssimiamoooooooo}.  

\begin{proof}[Proof of Theorem~\ref{teo:approssimiamoooooooo}]
By Theorem~\ref{thm:estensione} we can assume without loss of generality that $A=\W$. Let $f:\G\to\R^k$ be as in Proposition \ref{prop:equivdefLipCinfty}; for $i\in\N$ we consider
\[
Z_i:=\{f=(1/i,0,\dots,0)\}.
\]
By Theorem \ref{teo:equivdefLip} (applied to the function $f-(1/i,0,\dots,0)$) and Remark \ref{rem:graficoglobale}, for every $i\in\N$ the level set $Z_i$ is the graph of an intrinsic Lipschitz function $\f_i:\W\to\V$  defined on the whole $\W$. By Remark \ref{rem:dipendenzaalfa}, the intrinsic Lipschitz constant of $\f_i$ is bounded (uniformly in~$i$)  in terms of the intrinsic Lipschitz constant of $\f$. Recalling once more the identification $\V\equiv\R^k$, which gives the equality $\f_i(w)=\f(w)(\f_i(w)-\f(w))$, we have for every $w\in\W$
\begin{align*}
\de |\f_i(w)-\f(w)|^2 & \leq  \left\langle f(w\f_i(w))-f(w\f(w)), \f_i(w)-\f(w)\right\rangle\\
&\leq  |f(w\f_i(w))-f(w\f(w))|\:|\f_i(w)-\f(w)|\\
&=  \frac{|\f_i(w)-\f(w)|}{i}\,,
\end{align*}
which proves that $\f_i\to\f$ uniformly as $i\to\infty$.

We have only to show that each $\f_i$ is $C^\infty$ smooth. The quickest way is probably to reason in exponential coordinates {\em of the second type}, i.e., to identify $\G\equiv\R^d$ by
\begin{align*}
\R^d\ni x=(x_1,\dots,x_d)\longleftrightarrow & \exp(x_{k+1}X_{k+1}+\dots+x_dX_d)\exp(x_1X_1+\dots+x_kX_k)\in\G
\end{align*}
Since $\V$ is Abelian we have for every $j=1,\dots,k$
\begin{align*}
&\phantom{=}\exp(x_{k+1}X_{k+1}+\dots+x_dX_d)\exp(x_1X_1+\dots+x_kX_k)\\
&= \exp(x_{k+1}X_{k+1}+\dots+x_dX_d)\exp(x_1X_1+\dots+\widehat{x_jX_j}+\dots+x_kX_k)\exp(x_jX_j)\\
&=\eexp(x_jX_j)\Big(\eexp(x_1X_1+\dots+\widehat{x_jX_j}+\dots+x_kX_k)\big(\eexp(x_{k+1}X_{k+1}+\dots+x_dX_d)(0)\big)\Big),
\end{align*}
which proves that in these  coordinates $X_j=\partial_{x_j}$ for all $j=1,\dots,k$. Since $f\in C^\infty(\R^d\setminus\{f=0\})$, the classical Implicit Function Theorem ensures that the level set $Z_i=\{f=(1/i,0,\dots,0)\}$ is the graph of a $C^\infty$ smooth function $\psi_i=\psi_i(x_{k+1},\dots,x_d):\R^{d-k}\to\R^k$. Writing $x=(x',x'')\in\R^k\times\R^{d-k}$ we observe that
\begin{align*}
\phantom{\Longleftrightarrow }& x=(x_1,\dots,x_d)\in Z_i\\
\Longleftrightarrow & x'=(x_1,\dots,x_k)=\psi_i(x'')\\
\Longleftrightarrow & x=\eexp\Big((\psi_i(x''))_1X_1+\dots+(\psi_i(x''))_kX_k)(\eexp(x_{k+1}X_{k+1}+\dots+x_dX_d)(0)\Big)\\
\Longleftrightarrow & x= \exp(x_{k+1}X_{k+1}+\dots+x_dX_d)\exp((\psi_i(x''))_1X_1+\dots+(\psi_i(x''))_kX_k),
\end{align*}
which in turn proves that $\f_i$ and $\psi_i$ coincide as maps from $\W$ to $\V$. This proves that $\f_i$ is smooth, as wished.
\end{proof}




\section{The Heisenberg group}\label{sec:Heisenberg}
From now on we will always work in  {\em Heisenberg groups}, that provide the most notable examples of Carnot groups; we refer to~\cite{CDPTbook} for a  thorough introduction. 
For every $n\geq 1$, the  $n$-th Heisenberg group $\H^n$ is the connected, simply connected and nilpotent Lie group $\H^n$ associated with the $(2n+1)$-dimensional stratified Lie algebra $\h$ generated by   elements $X_1,\dots,X_n,Y_1,\dots,Y_n,T$  whose Lie brackets  vanish except for 
\[
[X_j,Y_j]=T\qquad\text{for every }j=1,\dots,n.
\]
The algebra stratification is given by $\h=\h_1\oplus\h_2$, where
\[
\text{$\h_1:=\Span\{X_1,\dots,X_n,Y_1,\dots,Y_n\},\qquad \h_2:=\Span\{T\}$;}
\]
the second layer $\h_2$ is the center of the algebra. 

We   identify $\H^n$ with $\R^{2n+1}=\R^n_x\times\R^n_y\times\R_t$ by means of exponential coordinates
\[
\H^n\ni\exp(x_1X_1+\dots+x_nX_n+y_1Y_1+\dots+y_nY_n+tT)\longleftrightarrow(x,y,t)\in\R^{2n+1}.
\]
In these coordinates, the group law reads as
\[
(x,y,t)(x',y',t')=(x+x',y+y',t+t'+ \tfrac12 \langle x,y'\rangle-\tfrac12\langle x',y\rangle)
\]
and left-invariant vector fields have the form
\[
X_j=\partial_{x_j}-\frac{y_j}2 \partial_t,\qquad Y_j=\partial_{y_j}+\frac{x_j}2\partial_t,\qquad T=\partial_t.
\]
We also observe that, in exponential coordinates, we have  $p^{-1}=-p$ for every $p\in\H^n$. As in Section~\ref{sec:gruppi}, a one-parameter family $(\de_\la)_{\la>0}$ of group automorphisms is provided by the dilations $\de_\la(x,y,t):=(\la x,\la y,\la^2 t)$. 

We fix a homogeneous distance $d$ on $\H^n$ and we denote by $B(p,r)$, $p\in\H^n$, $r>0$, the associated open balls; for $p\in\H^n$ we also write $\|p\|_\H:=d(0,p)$. It will be convenient to assume that $d$ is rotationally invariant, i.e., that
\begin{equation}\label{eq:drotinv}
\|(x,y,t)\|_\H=\|(x',y',t)\|_\H\qquad\text{whenever }|(x,y)|=|(x',y')|.
\end{equation}
Relevant examples of rotationally invariant homogeneous distances  are provided by the Carnot-Carath\'eodory distance $d_{cc}$ 
\[
d_{cc}(p,q):=\inf\left\{\|h\|_{L^1([0,1],\R^{2n})} : 
\begin{array}{l}
\text{the curve $\ga_h:[0,1]\to\H^n$ defined by}\\
\ga_h(0)=p,\ \dot\ga_h=\sum_{j=1}^n(h_jX_j+h_{j+n}Y_j)(\ga_h)\\
\text{has final point }\ga_h(1)=q
\end{array}
\right\},
\]
by the left-invariant distance $d_\infty$ (see~\cite[Proposition~2.1]{FSSCMathAnn2001}) such that
\begin{equation}\label{eq:dinfty}
d_\infty(0,(x,y,t)):=\max\{|(x,y)|,2|t|^{1/2}\},
\end{equation}
and by the Kor\'anyi (or Cygan-Kor\'anyi, see~\cite{Cygan}) distance  $d_{K}$ defined by
\begin{equation}\label{eq:dKoranyi}
d_K(0,(x,y,t)):=\bigl((|x|^2+|y|^2)^2+16t^2\bigr)^{1/4}.
\end{equation}

The Lebesgue measure $\leb{2n+1}$ is a Haar measure on $\H^n\equiv\R^{2n+1}$ and  $Q:=2n+2$ is the homogeneous dimension of $\H^n$, in particular
\[
\leb{2n+1}(B(p,r))=r^Q\leb{2n+1}(B(0,1))\qquad\text{for every $p\in\H^n$ and } r>0.
\]
Actually, also the Hausdorff and spherical Hausdorff measures $\Haus^Q,\Shaus^{Q}$ are Haar measure in $\H^n$, hence $\leb{2n+1},\Haus^Q$ and $\Shaus^Q$ coincide up to multiplicative constants. Recall in particular that the spherical Hausdorff measure $\Shaus^k$ of dimension $k\geq 0$ is defined by
\[
\Shaus^k(E):=\lim_{\de\to 0^+} \inf\left\{ \sum_{i=0}^\infty (2r_i)^k:E\subset\bigcup_{i=0}^\infty B(p_i,r_i)\text{ for some }p_i\in\H^n,r_i\in(0,\de)\right\}.
\]

One of the aims of this section is to introduce  Rumin's complex of differential forms in $\H^n$, for which we follow the presentations in \cite{FSSCAIM,Rumin}.

\subsection{Multi-linear algebra}\label{subsec:algebraHeis}
The Heisenberg group $\H^n$ is a  contact manifold: the {\em contact form} $\theta$ defined by $\theta_{|\h_1}=0$  and $\theta(T)=1$ satisfies $\theta\wedge(d\theta)^n\neq 0$ and, actually, the $2n+1$-form $\theta\wedge(d\theta)^n$ is a left-invariant volume form. In coordinates 
\[
\theta=dt+\frac12\sum_{j=1}^n(y_jdx_j-x_jdy_j),
\]
while 
\[
d\theta=-\sum_{j=1}^n dx_j\wedge dy_j
\]
is the standard symplectic 2-form in $\R^{2n}$, up to a sign. Notice that the basis $dx_1,\dots,dy_n,\theta$ is the dual basis to $X_1,\dots,Y_n,T$; observe also that here we are  identifying the dual of $\h$ with left-invariant 1-forms on $\H^n$, in the same way as the algebra $\h$ can be identified with left-invariant vector fields.

It will be sometimes convenient to denote the family $X_1,\dots,Y_n,T$ by $W_1,\dots,W_{2n+1}$, where 
\begin{equation}\label{eq:defW}
\begin{aligned}
& W_j:=X_j &&\text{if }1\leq j\leq n\\
& W_j:=Y_{j-n}\qquad &&\text{if }n+1\leq j\leq 2n\\
& W_{2n+1}:=T.
\end{aligned}
\end{equation}
Analogously, we use $\theta_1,\dots,\theta_{2n+1}$ to denote, respectively, $dx_1,\dots,dy_n,\theta$. Given a subset $I\subset\{1,\dots,2n+1\}$ we write $I=\{i_1,\dots,i_k\}$ for $i_1<i_2<\dots<i_k$ and  we denote by $W_I,\theta_I$ the (formal) exterior products
\begin{equation}\label{eq:defW_I}
W_I:=W_{i_1}\wedge\dots\wedge W_{i_k},\qquad \theta_I:=\theta_{i_1}\wedge\dots\wedge \theta_{i_k}.
\end{equation}
We also denote by $|I|$ the cardinality of $I$. 
We can now introduce  the exterior algebras $\bwl_\ast\h$ and $\bwl^\ast\h$ of (multi-)vectors and (multi-)covectors as
\[
\bwl_\ast\h:=\bigoplus_{k=0}^{2n+1}\bwl_k\h,\qquad
\bwl^\ast\h:=\bigoplus_{k=0}^{2n+1}\bwl^k\h
\]
where $\bwl_0\h=\bwl^0\h=\R$ and
\begin{align*}
& \bwl_k\h:=\Span\{W_I:I\subset\{1,\dots,2n+1\}\text{ with }|I|=k\}\\
& \bwl^k\h:=\Span\{\theta_I:I\subset\{1,\dots,2n+1\}\text{ with }|I|=k\}.
\end{align*}
The elements of $\bwl_k\h$ and $\bwl^k\h$ are called, respectively, $k$-vectors and $k$-covectors. The inner product on $\h$ making $W_1,\dots,W_{2n+1}$ an orthonormal frame can be naturally extended to $\bwl_k\h$ in such a way that the elements $W_I$ form an orthonormal frame. In this way one can define an explicit isomorphism
\[
\bwl_k\h\ni v\longmapsto v^\ast\in \bwl^k\h
\]
by requiring $\langle w\mid  v^\ast\rangle:= \langle w,v\rangle$ for every $w\in \bwl_k\h$, where $\langle \,\cdot\mid \cdot\,\rangle$ denotes duality pairing between vectors and covectors.

We analogously introduce the exterior algebras of {\em horizontal} vectors and covectors
\[
\bwl_\ast\h_1:=\bigoplus_{k=0}^{2n}\bwl_k\h_1,\qquad
\bwl^\ast\h_1:=\bigoplus_{k=0}^{2n}\bwl^k\h_1
\]
where
\begin{align*}
& \bwl_k\h_1:=\Span\{W_I:I\subset\{1,\dots,2n\}\text{ with }|I|=k\}\\
& \bwl^k\h_1:=\Span\{\theta_I:I\subset\{1,\dots,2n\}\text{ with }|I|=k\}.
\end{align*}

\begin{remark}\label{rem:parteorizz}
Given $\tau\in\bwl_k\h$ and $\la\in\bwl^k\h$, $1\leq k\leq 2n$, we denote by $\tau_{\h_1}$ and $\la_{\h_1}$ their horizontal component, i.e., the unique $\tau_{\h_1}\in\bwl_k\h_1$ and $\la_{\h_1}\in\bwl^k\h_1$ such that $\tau=\tau_{\h_1}+\si\wedge T$ and $\la=\la_{\h_1}+\mu\wedge\theta$ for some (unique) $\si\in\bwl_{k-1}\h_1,\mu\in\bwl^{k-1}\h_1$. 
\end{remark}

Of special importance for us are vertical planes, that we now introduce. As customary, given a non-zero simple $k$-vector $\tau=\tau_1\wedge\dots\wedge \tau_k$ we denote by  $\Span\tau$ the linear space generated by $\tau_1,\dots, \tau_k$; equivalently,  $\Span\tau=\{v:v\wedge\tau=0\}$.

\begin{definition}\label{def:pianiverticali}
A set $\mathscr P\subset\H^n$ is a {\em vertical plane} of dimension $k$ if there exists a non-zero $\tau\in\bwl_{k-1}\h_1$ such that 
\[
\mathscr P=\exp(\Span(\tau\wedge T)).
\]
\end{definition}

In exponential coordinates, a vertical plane $\mathscr P$ is a $k$-dimensional linear subspace of $\H^n\equiv\R^{2n+1}=\R^{2n}\times\R$  of the form $V\times\R$ for some $(k-1)$-dimensional subspace $V\subset\R^{2n}$. A vertical plane is always a normal subgroup of $\H^n$.

\subsection{Differential forms and Rumin's complex}\label{subsec:Rumin}
The spaces $\h$, $\h_1$, $\bwl_k\h$, $\bwl^k\h$, $\bwl_k\h_1$, $\bwl^k\h_1$, as well as the spaces  $\mathcal I^k$ and $\mathcal J^{2n+1-k}$
introduced below in~\eqref{eq:IkJk}
,  canonically induce several bundles on $\H^n$, that we will denote by using the same symbol. The same convention applies to dual and quotient spaces of such spaces. As customary, we denote by $\Omega^k$ the space of smooth differential $k$-forms on $\H^n$, i.e., the  space of smooth sections of $\bwl^k\h$ (seen as a bundle on $\H^n$).

We now recall some of the spaces of differential forms introduced by M. Rumin in \cite{RuminComptesRendus1990,Rumin}. As before, we identify the left-invariant 2-form $d\theta$ in $\H^n$ with a 2-covector in $\bwl^2\h$ and we define
\begin{equation}\label{eq:IkJk}
\begin{aligned}
 \mathcal I^k:=&\{\lambda\wedge\theta+\mu\wedge d\theta:\textstyle\lambda\in\bwl^{k-1}\h,\mu\in\bwl^{k-2}\h\},
 \\
 \mathcal J^{k}:=&\{\textstyle\lambda\in\bwl^{k}\h:\lambda\wedge\theta=\lambda\wedge d\theta=0\},
\end{aligned}
\end{equation}
where we adopted the convention that $\bwl^i\h=\{0\}$ if $i<0$. Observe that $\mathcal I^0=\{0\}$. The space $\mathcal I^*:=\bigoplus_{k=0}^{2n+1}\mathcal I^k$ is the graded ideal generated by $\theta$, while $\mathcal J^*:=\bigoplus_{k=0}^{2n+1}\mathcal J^k$ is the annihilator of $\mathcal I^*$. As observed in \cite{Rumin} (see also~\cite[page~166]{FSSCAIM}), for $1\leq k\leq n$ we have $\mathcal J^k=\{0\}$ and $\mathcal I^{2n+1-k}=\bwl^{2n+1-k}\h$; these equalities are, in essence, consequences of the fact that the {\em Lefschetz operator}\footnote{Our operator $L$ actually differs by a sign from the Lefschetz one (wedge product by the standard symplectic form).}  $L:\bwl^{h}\h_1\to\bwl^{h+2}\h_1$ defined by $L(\la):=\la\wedge d\theta$ is injective for $h\leq n-1$ and surjective for $h\geq n-1$, see for instance the beautiful proof in~\cite[Proposition~1.1]{BryantGG}.

\begin{remark}\label{rem:J2n+1-k}
Recalling the notation introduced in Remark~\ref{rem:parteorizz}, it can be easily proved that $\la_{\h_1}=0$ for every $\lambda\in\mathcal J^{k}$. In particular, there exists a unique $\lambda_H\in\bwl^{k-1}\h_1$ such that $\lambda=\lambda_H\wedge\theta$, hence
\[
\text{$\langle\tau\mid\lambda\rangle=0$\qquad for every $\lambda\in\mathcal J^{k}$, $\tau\in \bwl_{k}\h_1$.}
\]
We also notice  that $\lambda_H\wedge d\theta=0$. 
\end{remark}

\begin{remark}\label{rem:introducoJ_k}
For $k\geq n+1$ it is convenient to introduce the (formal) dual space $\mathcal J_k:=(\mathcal J^k)^*$. 
Observe that every multi-vector $\tau\in\bwl_{k}\h$ canonically induces an element $[\tau]_\mathcal J\in\mathcal J_k$ defined by
\[
\langle [\tau]_\mathcal J\mid \la\rangle:=\langle \tau\mid \la\rangle\qquad \text{for every }\la\in\mathcal J^k.
\]
Equivalently, $[\tau]_\mathcal J$ is the equivalence class of $\tau$ in the quotient of $\bwl_k\h$ by its subspace $(\mathcal J^k)^\perp$.
\end{remark}

Let us introduce the spaces $\Omega_\H^k$ of {\em Heisenberg differential $k$-forms} 
\begin{align*}
& \Omega_\H^k:= C^\infty\big(\H^n,\tfrac{\text{\large$\wedge$}^k\h}{\mathcal I^k}\big)=\left\{\text{smooth sections of }\tfrac{\text{\large$\wedge$}^k\h}{\mathcal I^k}\right\},&&\text{if }0\leq k\leq n\\
& \Omega_\H^k:= C^\infty(\H^n,\mathcal J^k)=\left\{\text{smooth sections of }\mathcal J^k\right\},&&\text{if }n+1\leq k\leq 2n+1.
\end{align*}
Clearly, $\Omega_\H^0=C^\infty(\H^n)$. Denoting exterior differentiation by\footnote{We use $d$ also to denote  the distance on $\H^n$; of course, no confusion will ever arise.} $d$, we observe that, if $k\geq n+1$, then $d(\Omega_\H^k)\subset\Omega_\H^{k+1}$. If $ k\leq  n-1$ we have $d(\mathcal I^k)\subset \mathcal I^{k+1}$, hence $d$ passes to the quotient. All in all, for $k\in\{0,\dots,2n\}\setminus\{n\}$  a well-defined operator $d:\Omega_\H^k\to\Omega_\H^{k+1}$ is induced by  exterior  differentiation. The following fundamental result was proved by M.~Rumin \cite{Rumin}, see also~\cite[Theorem 5.9]{FSSCAIM}.

\begin{theorem}
There exists a second-order  differential operator $D:\Omega_\H^n\to\Omega_\H^{n+1}$ such that the sequence
\[
0\to\R\to\Omega_\H^0\stackrel{d}{\to}\Omega_\H^1\stackrel{d}{\to}\dots\stackrel{d}{\to}\Omega_\H^n\stackrel{D}{\to}\Omega_\H^{n+1}\stackrel{d}{\to}  \dots   \stackrel{d}{\to}\Omega_\H^{2n+1}\to 0
\]
is exact
\end{theorem}

The construction of the operator $D$ is crucial for our purposes and we recall it here. First, as already mentioned, the  Lefschetz operator $L:\bwl^{n-1}\R^{2n}\to\bwl^{n+1}\R^{2n}$ defined by $L(\la)=\la\wedge d\theta$ is bijective. Second, observe that
\begin{equation}\label{eq:equivbis}
\frac{\bwl^n\h}{\mathcal I^n}=\frac{\bwl^n\h_1}{\{\mu\wedge d\theta:\mu\in\bwl^{n-2}\h_1\}}.
\end{equation}
We are going to define an operator $D$ on smooth sections  of $\bwl^n\h_1$ and prove that it passes to the quotient modulo smooth sections of $\{\mu\wedge d\theta:\mu\in\bwl^{n-2}\h_1\}$. Given a smooth section $\al$ of $\bwl^n\h_1$ we set
\begin{align*}
D\al:=&d\big(\al-\theta\wedge L^{-1}((d\al)_{\h_1})\big),\\
=& d\big(\al+(-1)^n L^{-1}((d\al)_{\h_1})\wedge\theta\big),
\end{align*}
where we used the notation in Remark \ref{rem:parteorizz}. We have to prove that $D(\be\wedge d\theta)=0$ for every smooth section $\be $ of $\bwl^{n-2}\h_1$. Inasmuch as
\begin{align*}
& L^{-1}((d(\be\wedge d\theta))_{\h_1} )=L^{-1}((d\be\wedge d\theta)_{\h_1})=L^{-1}((d\be)_{\h_1}\wedge d\theta)=(d\be)_{\h_1}\\
& d(\be\wedge d\theta)=(-1)^{n-1}d(d\be\wedge\theta)
\end{align*}
we deduce that
\begin{align*}
D(\be\wedge d\theta)
& = d(\be\wedge d\theta + (-1)^n(d\be)_{\h_1}\wedge\theta)\\
&= (-1)^{n-1} d(d\be\wedge\theta - (d\be)_{\h_1}\wedge\theta)\\
&=  (-1)^{n-1} d((d\be)_{\h_1}\wedge\theta - (d\be)_{\h_1}\wedge\theta)=0,
\end{align*}
as wished. 

This proves that $D$ is well-defined as a linear operator $\Omega_\H^n\to \Omega^{n+1}$. We have to  check that $D(\Omega_\H^n)\subset \Omega_\H^{n+1}$, i.e., that  $D\al\wedge\theta=D\al\wedge d\theta=0$ for every $\alpha\in \Omega_\H^n$. To this aim, let us write $(d\al)_\mathfrak v:=d\al-(d\al)_{\h_1}$ to denote the ``vertical'' part of $d\al$, which can be written as $(d\al)_\mathfrak v=\theta\wedge\be_\al$ for  a suitable smooth section $\beta_\al$ of $\bwl^{n}\h_1$. Then
\begin{equation}\label{eq:Dalcondal_v}
\begin{aligned}
D\al&
= d\big(\al+(-1)^n L^{-1}((d\al)_{\h_1})\wedge\theta\big)\\
&=\cancel{(d\al)_{\h_1}} + (d\al)_\mathfrak v +  (-1)^n d(L^{-1}((d\al)_{\h_1}))\wedge\theta- \cancel{L^{-1}((d\al)_{\h_1})\wedge d\theta}\\
& = \theta\wedge \big(\be_\al +d(L^{-1}((d\al)_{\h_1}))\big)
\end{aligned}
\end{equation}
This implies that $D\al\wedge\theta=0$ and, as a consequence, $0=d(D\al\wedge\theta)=(-1)^{n+1}D\al\wedge d\theta$, as wished.

\begin{example}
Let us compute the operator $D:\Omega_\H^1\to\Omega_\H^2$ in the  first Heisenberg group $\H^1$. Given $\omega\in\Omega_\H^1$ we choose a smooth section $\al$ of $\bwl^1\h_1$ that is a representative of $\omega$ in the quotient~\eqref{eq:equivbis}. Writing $\al=f dx+g dy$ we have
\[
(d\al)_{\h_1}=(Xg-Yf)dx\wedge dy=-(Xg-Yf)d\theta
\]
and, clearly, $L^{-1}((d\al)_{\h_1})=Yf-Xg$. Therefore
\begin{align*}
D\omega =\,& d(f dx+g dy - (Yf-Xg)\,\theta)\\
=\,& \cancel{(-Yf+Xg)dx\wedge dy} - (Tf) dx\wedge\theta - (Tg) dy\wedge\theta\\
& -(XYf-XXg)dx\wedge\theta - (YYf-YXg)dy\wedge\theta-\cancel{(Yf-Xg)d\theta}\\
=\,& (-2XYf+YXf+XXg)dx\wedge\theta + (2YXg-XYg-YYf)dy\wedge\theta.
\end{align*}
See also \cite[Example~B.2]{BaldiFranchiTchouTesi} and~\cite[Example~3.11]{BaldiFranchiDivCurl}. 

We refer to~\cite[Example~3.12]{BaldiFranchiDivCurl} for the computation of the operator $D$ in $\H^2$. 
\end{example}

\begin{remark}\label{rem:dzero}
It is sometimes convenient to have a unique symbol to denote the differential operators in Rumin's complex. Following \cite{FranchiTripaldi} we then define $d_C:\Omega^k_\H\to\Omega^{k+1}_\H$ by $d_C=d$, if $k\neq n$, and $d_C=D$, if $k=n$. 
\end{remark}

\begin{remark}
For $k\leq n$, an interesting interpretation of  $\tfrac{\text{\large$\wedge$}^k\h}{\mathcal I^k}$ and $\mathcal J^{2n+1-k}$ as spaces of {\em integrable covectors} is provided in~\cite[Theorem~2.9]{FSSCAIM}.
\end{remark}

\subsection{\texorpdfstring{$\H$}{H}-linear maps}
We now introduce the notion of $\H$-linear maps in $\H^n$, that are among the most simple examples of {\em contact maps} (see e.g.~\cite{KoranyiReimann}), and study their behaviour on Heisenberg forms. For a finer study of the relations between Heisenberg forms and contact maps see e.g.~\cite[Chapter~3]{CanarecciPhdThesis}. The results in this section will be applied especially, but not exclusively, to left-translations and dilations in $\H^n$ and  their compositions.

\begin{definition}
A map $\mathcal L:\H^n\to\H^n$ is {\em $\H$-linear} if it is a group homomorphism that is also homogeneous, i.e., $\de_r(\mathcal Lp)=\mathcal L(\de_r p)$ for any $p\in\H^n$ and any $r>0$.
\end{definition}

It is not difficult to prove  that $\mathcal L$ is $\H$-linear if and only if 
$\mathfrak l:=\exp^{-1}\circ \mathcal L\circ\exp:\h\to\h$ is a Lie algebra homomorphism, see e.g. \cite[\S3.1]{Magthesis} and the references therein. In particular
\begin{equation}\label{eq:contenHlin}
\mathfrak l(\h_1)\subset\h_1\qqtaqq\mathfrak l(\h_2)\subset\h_2.
\end{equation}
Moreover, if for every $p\in\H^n$ one canonically identifies $T_p\H^n\equiv T_0\H^n=\h$ by means of the differential of the left-translation by $p^{-1}$, then the differential $d\mathcal L$ of $\mathcal L$ is constant, i.e.
\[
d\mathcal L_p=d\mathcal L_0=\mathfrak l\qquad\forall\ p\in\H^n.
\]
This easily follows by computing the differential at $p$ of $\mathcal L(q)=\mathcal L(p)\mathcal L(p^{-1}q)$.

By abuse of notation, we use a unique symbol $\mathcal L_*$ to denote any of $d\mathcal L_p$, $d\mathcal L_0$ and $\mathfrak l$. Similarly, the symbol  $\mathcal L^*$  denotes any of the associated pull-back actions $(d\mathcal L_p)^*:T_{\mathcal L(p)}^*\H^n\to T_{p}^*\H^n$, $(d\mathcal L_0)^*:T_{0}^*\H^n\to T_{0}^*\H^n$ and $\mathfrak l^*:\bwl^1\h\to\bwl^1\h$. 
We use the symbols $\mathcal L_*,\mathcal L^*$ also to denote the induced maps $\mathcal L_*:\bwl_*\h\to\bwl_*\h$ and $\mathcal L^*:\bwl^*\h\to\bwl^*\h$ defined by
\[
\mathcal L_*(v_1\wedge\dots\wedge v_k):=\mathcal L_*(v_1)\wedge\dots\wedge\mathcal L_*(v_k),\qquad\forall\ v_1,\dots,v_k\in\h,
\]
and
\[
\mathcal L^*(\la_1\wedge\dots\wedge \la_k):=\mathcal L^*(\la_1)\wedge\dots\wedge\mathcal L^*(\la_k),\qquad\forall\ \la_1,\dots,\la_k\in\bwl^1\h.
\]

\begin{proposition}\label{prop:dthetainvariante}
Let   $\mathcal L:\H^n\to\H^n$ be a $\H$-linear isomorphism; then the pull-backs  of $\theta$ and $d\theta$ satisfy
\begin{equation}\label{eq:pullbacktheta}
\mathcal L^*(\theta)=c_\mathcal L\theta \qtaq \mathcal L^*(d\theta)=c_\mathcal L\,d\theta
\end{equation}
where $c_\mathcal L\neq 0$ is defined\footnote{The number $c_\mathcal L$ is well-defined because of \eqref{eq:contenHlin}; it is not zero because $\mathcal L_*$ is an isomorphism.} by $\mathcal L_*(T)=c_\mathcal L T$.
\end{proposition}
\begin{proof}
Since pull-back and exterior derivative commute, it is sufficient to prove the first assertion in \eqref{eq:pullbacktheta}. It is proved in \cite[\S3.15]{Warner} that, since $\theta$ is left-invariant, then  $d\mathcal L^*(\theta)$  is left-invariant as well. By \eqref{eq:contenHlin}, we have $\mathcal L^*(\theta)=0$ on $\h_1$, hence $\mathcal L^*(\theta)$ is a scalar multiple of $\theta$. The statement follows because $\mathcal L^*(\theta)(T)=\theta (\mathcal L_*(T))=\theta (c_\mathcal L T)=c_\mathcal L$.
\end{proof}

\begin{proposition}\label{prop:JtoJ}
Let $\mathcal L:\H^n\to\H^n$ be a $\H$-linear isomorphism; then 
\begin{itemize}
\item[(i)] if $1\leq k\leq n$, $\mathcal L^*:\bwl^k\h\to\bwl^k\h$ passes to the quotient and defines an isomorphism $\mathcal L^*:\bwl^k\h/\mathcal I^k\to\bwl^k\h/\mathcal I^k$;
\item[(ii)] if $n+1\leq k\leq 2n$, $\mathcal L^*(\mathcal J^k)=\mathcal J^k$.
\end{itemize}
\end{proposition}
\begin{proof}
(i) If $1\leq k\leq n$ it is enough to show that $\mathcal L^*(\mathcal I^k)=\mathcal I^k$. For every $\la\in \bwl^{k-1}\h$ and $\mu\in\bwl^{k-2}\h$ we have by Proposition~\ref{prop:dthetainvariante}
\begin{align*}
& \mathcal L^*(\la\wedge\theta+\mu\wedge d\theta)=c_\mathcal L (\mathcal L^*(\la)\wedge\theta + \mathcal L^*(\mu)\wedge d\theta    ),
\end{align*}
hence $\mathcal L^*(\mathcal I^k)\subset \mathcal I^k$. The equality $\mathcal L^*(\mathcal I^k)= \mathcal I^k$ follows because $\mathcal L^*$ is a isomorphism.

(ii) 
 Given $\la\in\mathcal J^k$ we have
\[
\mathcal L^*(\la)\wedge\theta= c_\mathcal L^{-1} \mathcal L^*(\la\wedge \theta)=0\qtaq 
\mathcal L^*(\la)\wedge d\theta= c_\mathcal L^{-1} \mathcal L^*(\la\wedge d\theta)=0,
\]
hence $\mathcal L^*(\la)\in\mathcal J^k$. This proves the inclusion $\mathcal L^*(\mathcal J^k)\subset\mathcal J^k$, and the equality follows again because $\mathcal L^*$ is an isomorphism.
\end{proof}

Proposition~\ref{prop:JtoJ} has the following  consequence that we will later use when $\mathcal L$ is (a composition of) dilations and left-translations.  Compare with~\cite[Theorem~3.2.1]{CanarecciPhdThesis}.

\begin{corollary}\label{cor:d_Ccommuta}
Let $\mathcal L:\H^n\to\H^n$ be a $\H$-linear isomorphism; then the pull-back $\mathcal L^*$  commutes with the differential operators in Rumin's complex, i.e., for every $k\in\unodueenne$
\[
d_C(\mathcal L^*\omega)=\mathcal L^*(d_C\omega)\qquad\text{for every } \omega\in\Omega^k_\H.
\]
\end{corollary}
\begin{proof}
When $k\neq n$ this is a simple consequence of Proposition~\ref{prop:JtoJ} and the fact that $\mathcal L^*$ commutes with exterior differentiation $d=d_C$. 

When $k=n$ we have $d_C=D$  and some computations are needed. For every $\omega\in \Omega^k_\H$ we fix a representative $\al\in C^\infty(\H^n,\bwl^n\h_1)$ of the equivalence class $\omega$ in the quotient in the right-hand side of~\eqref{eq:equivbis}. By Proposition~\ref{prop:JtoJ}, $\mathcal L^*\al$ is a representative of $\mathcal L^*\omega$ and we can compute
\begin{align*}
D(\mathcal L^*\omega)
& = d\big(\mathcal L^*\al-\theta\wedge L^{-1}((d\mathcal L^*\al)_{\h_1})\big)\\
& = d\big(\mathcal L^*\al-\theta\wedge L^{-1}((\mathcal L^*(d\al))_{\h_1})\big)\\
& = d\big(\mathcal L^*\al-\mathcal L^*(\theta/c_\mathcal L)\wedge L^{-1}(\mathcal L^*((d\al)_{\h_1}))\big),
\end{align*}
where we used Proposition~\ref{prop:dthetainvariante}. The latter also gives
\begin{align*}
\mathcal L^*((d\al)_{\h_1})
&=\mathcal L^*(L^{-1}(d\al)_{\h_1}\wedge d\theta)\\
&=\mathcal L^*(L^{-1}(d\al)_{\h_1})\wedge\mathcal L^*( d\theta)\\
&=c_\mathcal L \mathcal L^*(L^{-1}(d\al)_{\h_1})\wedge d\theta
\end{align*}
so that
\[
L^{-1}(  \mathcal L^*((d\al)_{\h_1})  ) = c_\mathcal L \mathcal L^*(L^{-1}(d\al)_{\h_1}).
\]
Therefore
\begin{align*}
D(\mathcal L^*\omega)
& = d\big(\mathcal L^*\al-\mathcal L^*(\theta/c_\mathcal L)\wedge c_\mathcal L \mathcal L^*(L^{-1}(d\al)_{\h_1})\big)\\
&= d\big(\mathcal L^*(\al-\theta\wedge(L^{-1}(d\al)_{\h_1}))\big)\\
&=\mathcal L^*\big(d(\al-\theta\wedge(L^{-1}(d\al)_{\h_1}))\big)\\
&= \mathcal L^*(D\omega)
\end{align*}
and the proof is concluded.
\end{proof}

We recall that the notion of vertical plane was introduced in Definition~\ref{def:pianiverticali}. Given natural numbers $a,b$ such that $1\leq a+b\leq n$, we denote by $\Pab\subset \H^n$ the $(2a+b+1)$-dimensional vertical plane 
\begin{equation}\label{eq:defPab}
\begin{aligned}
\!\!\:
\Pab:=\:& \{(x,y,t)\in\H^n:x_i=y_j=0\text{ for all }a+b+1\leq i\leq n \text{ and }a+1\leq j\leq n\}\\
=\:&
\begin{cases}
\{(x_1,\dots,x_{a+b},0,\dots,0,y_1,\dots,y_a,0,\dots,0,t)\}\quad & \text{if }a\geq 1\text{ and } a+b<n\\
\{(x_1,\dots,x_n,y_1,\dots,y_a,0,\dots,0,t)\}\quad & \text{if }a\geq 1\text{ and } a+b=n\\
\{(x_1,\dots,x_{b},0,\dots,0,t)\}\quad & \text{if }a=0.
\end{cases}
\end{aligned}
\end{equation}
As shown by the following proposition, the planes $\Pab$ can be considered as canonical models for vertical planes in $\H^n$.

\begin{proposition}\label{prop:ruotiamosupianicanonici}
Let $\mathscr P\subset\R^{2n+1}\equiv\H^n$ be a vertical plane; then there  exist non-negative integers $a,b$ and a $\H$-linear isomorphism $\mathcal L:\H^n\to\H^n$ such that 
\[
\begin{aligned}
& a+b\leq n\qtaq \dim \mathscr P=2a+b+1\\
& \mathcal L^*(\theta)=\theta,\quad \mathcal L^*(d\theta)=d\theta\qtaq
\mathcal L(\mathscr P)=\Pab.
\end{aligned}
\]
\end{proposition}
\begin{proof}
The set $V:=\exp^{-1}(\{(x,y,t)\in \mathscr P:t=0\})$ is a vector subspace of $\h_1$. If $\mathfrak l:\h\to\h$ is the Lie algebra isomorphism provided by the following Lemma \ref{lem:Artin}, then $\mathcal L:=\exp\circ\mathfrak l\circ\exp^{-1}$ is a $\H$-linear isomorphism which, by Proposition~\ref{prop:dthetainvariante}, satisfies the statement of the present proposition. 
\end{proof}

\begin{lemma}\label{lem:Artin}
Let $V\subset\h_1$ be a linear subspace  with $\dim V\geq 1$. Then there exist non-negative integers $a,b$ and a Lie algebra isomorphism $\mathfrak l:\h\to\h$ such that $\dim V=2a+b$, $\mathfrak l(T)=T$ and
\[
\mathfrak l(V)=
\begin{cases}
\Span\{X_1,\dots,X_{a+b},Y_1,\dots,Y_a\}\qquad& \text{if }a\geq 1\\
\Span\{X_1,\dots,X_{b}\}\qquad& \text{if }a=0.
\end{cases}
\]
Moreover, $a=0$ if and only if $[V,V]=\{0\}$, i.e., if and only if $V$ is an Abelian sub-algebra of $\h$.
\end{lemma}
\begin{proof}
We first recall the  canonical symplectic structure on $\h_1$, that is provided by the bilinear skew-symmetric form
\[
B(X,Y):=\langle[X,Y],T\rangle,\qquad X,Y\in\h_1.
\]
Notice that $d\theta$ (seen as a 2-covector at $0$, i.e., as an element of $\bwl^2\h$) satisfies
\[
\langle X\wedge Y\mid d\theta\rangle=-B(X,Y),\qquad \forall\ X,Y\in\h_1\subset\h\equiv T_0\H^n
\]
which can be easily checked by testing $d\theta$ on a basis $X_i\wedge X_j,X_i\wedge Y_j,Y_i\wedge Y_j$ of $\bwl^2\h_1$.

We borrow some  language and notation  from \cite[Chapter 3]{Artin}. Let
\[
\text{rad}(V):=\{X\in V:B(X,X')=0\text{ for every }X'\in V\}=\{X\in V:[X,V]=0\}
\]
be the {\em radical} of $V$ and let $b:=\dim\text{rad}(V)$. Choose a subspace $U$ of $V$ such that $V=$rad$(V)\oplus U$;  $U$ is clearly  {\em non-singular}, i.e., for every $X\in U$ there exists $X'\in U$ such that $[X,X']\neq0$.
By \cite[Theorem 3.7]{Artin}, the dimension of $ U$ is even and we set $\dim U=2a$.

If $a\geq 1$, by \cite[Theorem 3.7]{Artin} there exists a basis $\widetilde X_1,\dots,\widetilde X_a,\widetilde Y_1,\dots, \widetilde Y_a$ of $U$ such that
\[
B(\widetilde X_i,\widetilde X_j)=0,\quad B(\widetilde Y_i,\widetilde Y_j)=0\qtaq B(\widetilde X_i,\widetilde Y_j)=\de_{ij},\qquad i,j\in\{1,\dots,a\},
\]
i.e.,
\[
[\widetilde X_i,\widetilde X_j]=0,\quad [\widetilde Y_i,\widetilde Y_j]=0\qtaq [\widetilde X_i,\widetilde Y_j]=\de_{ij}T,\qquad i,j\in\{1,\dots,a\}.
\]

If $b\geq 1$, fix a basis $\widetilde X_{a+1},\dots,\widetilde X_{a+b}$ of rad$(V)$. 

We can now define $\mathfrak l':V\to\h_1$ by 
\[
\mathfrak l'(\widetilde X_i)=X_i\qtaq \mathfrak l'(\widetilde Y_j)=Y_j
\]
for all integers $1\leq i\leq a+b$ and $1\leq j\leq b$. Notice that $\mathfrak l'$ is an {\em  isometry} of $V$ into $\h_1$, i.e., 
\[
B(\mathfrak l'(X),\mathfrak l'(Y))=B(X,Y)\qquad \text{for every }X,Y\in V.
\]
By Witt's Theorem (see \cite[Theorem 3.9]{Artin}), $\mathfrak l'$ can be extended to an  isometry $\mathfrak l'':\h_1\to\h_1$, i.e., a map satisfying $[\mathfrak l''(X),\mathfrak l''(Y)]=[X,Y]$ for every $X,Y\in \h_1$.
Finally, we extend $\mathfrak l''$ to $\mathfrak l:\h\to\h$ by setting $\mathfrak l_{|\h_1}:=\mathfrak l''$ and $\mathfrak l(T):=T$. The map $\mathfrak l$ provides the desired Lie algebra isomorphism.
\end{proof}

\begin{remark}\label{rem:isometrie_dell'algebra}
Under the same assumptions of Lemma~\ref{lem:Artin}: if $[V,V]=0$, then the Lie algebra isomorphism $\mathfrak l:\h\to\h$ provided by Lemma~\ref{lem:Artin} can be chosen in such a way that $\mathfrak l$ is an isometry of $\h$ when endowed with the scalar product making $X_1,\dots,X_n,Y_1,\dots,Y_n,T$ orthonormal. Observe in particular that here the term {\em isometry} has a different meaning  than in the proof of Lemma~\ref{lem:Artin}.  

Let us prove our statement. As in the proof of Lemma~\ref{lem:Artin}, fix a basis $\widetilde X_1,\dots,\widetilde X_b$ of $V=\text{rad}(V)$.  
Consider the linear isomorphism $J:\h_1\to\h_1$ defined by 
\[
J(X_i)=-Y_i\qtaq J(Y_i)=X_i\qquad\text{for every }i=1,\dots,n.
\]
We observe that $B(X,Y)=\langle X,JY\rangle$ for $X,Y\in\h_1$ and that $J$ is an isometry of $\h_1$. Set $\widetilde Y_i:=J(\widetilde X_i)$ and observe that, since $V=\text{rad}(V)\subset (J(V))^\perp$, the elements $\{\widetilde X_i,\widetilde Y_i:i=1,\dots,b\}$ form an orthonormal basis of $V\oplus J(V)$.

Consider $W:=(V\oplus J(V))^\perp$; then $W=W'\oplus\Span\{T\}$ where  $W':=(V\oplus J(V))^\perp\cap\h_1$. We claim that $J(W')=W'$; by dimensional reasons, it suffices to prove that $J(W')\subset W'$ and this follows because, for every $X\in W'$ and $Y\in V$, there holds
\begin{align*}
& \langle JX,Y\rangle = -\langle X,JY\rangle =0\\
& \langle JX,JY\rangle = \langle X,Y\rangle =0.
\end{align*}

We can now exhibit a Lie algebra isomorphism $\mathfrak l:\h\to\h$ that is also an isometry. First, we define $\mathfrak l (T)=T$ and
\[
\mathfrak l (\widetilde X_i)=X_i\qtaq \mathfrak l (\widetilde Y_i)=Y_i\qquad\text{for every }i=1,\dots,b.
\]
If $W'=\{0\}$, the proof is concluded. Otherwise, we fix a unit element $\widetilde X_{b+1}\in W'$ and, defining $\widetilde Y_{b+1}:=J(\widetilde X_{b+1})\in W'$, we set
\[
\mathfrak l (\widetilde X_{b+1})=X_{b+1}\qtaq \mathfrak l (\widetilde Y_{b+1})=Y_{b+1}.
\]
If $W'=\Span\{\widetilde X_{b+1},\widetilde Y_{b+1}\}$, the proof is concluded; otherwise, we fix a unit element $\widetilde X_{b+2}\in W'\cap \Span\{\widetilde X_{b+1},\widetilde Y_{b+1}\}^\perp$ and, defining $\widetilde Y_{b+2}:=J(\widetilde X_{b+2})\in W'\cap \Span\{\widetilde X_{b+1},\widetilde Y_{b+1}\}^\perp$, we set
\[
\mathfrak l (\widetilde X_{b+2})=X_{b+2}\qtaq \mathfrak l (\widetilde Y_{b+2})=Y_{b+2}.
\]
It is clear that this  construction can be iterated and that it eventually stops providing as a final outcome  the desired isometric Lie algebra isomorphism $\mathfrak l:\h\to\h$.
%
\end{remark}

\begin{remark}\label{rem:isometrie_delgruppo}
Under the same assumptions of Proposition~\ref{prop:ruotiamosupianicanonici}: if the vertical plane $\mathscr P$ is an Abelian subgroup of $\H^n$ (i.e., if $a=0$), then the $\H$-linear isomorphism $\mathcal L$ provided by Proposition~\ref{prop:ruotiamosupianicanonici} can be chosen to be an isometry of $\H^n$. This is an easy consequence of Remark~\ref{rem:isometrie_dell'algebra} and the rotational invariance of the distance $d$.

In particular, if $\V\subset\exp(\h_1)$ is a horizontal (and necessarily Abelian) subgroup of $\H^n$, we can consider the vertical subgroup $\mathscr P$ generated by $\V$ and $\exp(\Span\{T\})$ and deduce that there exists an isometric $\H$-linear isomorphism $\mathcal L$ of $\H^n$ such that
\[
\mathcal L(\V)=\exp(\Span\{X_1,\dots,X_k\}),
\]
where $k:=\dim \V$.
\end{remark}

\subsection{A basis for Rumin's spaces.}\label{subsec:baseRumin}
In this section we provide a basis for  Rumin's spaces $\mathcal J^{n+1},\dots,\mathcal J^{2n}$; since later in the paper we will denote by $k$ the codimension of the involved objects, we fix $k$ such that $1\leq k\leq n$ and study $\mathcal J^{2n+1-k}$. By Remark~\ref{rem:J2n+1-k}, this space coincides with $\{\la\wedge\theta:\la\in\bwl^{2n-k}\h_1,\la\wedge d\theta=0\}$, hence one is lead to the study of the kernel of the Lefschetz operator $\la\mapsto\la\wedge d\theta$. 

We  write $h:=2n-k$ and, identifying $\h_1\equiv\R^{2n}=\R^n_x\times\R^n_y$, we denote by $L$ the operator
\[
L(\la):=\la\wedge d\theta=-\la\wedge (dx_1\wedge dy_1+\dots+dx_n\wedge dy_n),\qquad\la\in\bwl^\ast\R^{2n}.
\]
Since $h\geq n$, the operator $L:\bwl^h\R^{2n}\to\bwl^{h+2}\R^{2n}$ is surjective (\cite[Proposition~1.1]{BryantGG}).

We need some preliminary notation. For every $i\in \{1,\dots,n\}$ we use the compact notation $dxy_i:=dx_i\wedge dy_i$, so that $d\theta=-(dxy_1+\dots+dxy_n)$. Moreover, for every $I\subset\{1,\dots,n\}$ we denote by $|I|$ the cardinality of $I$ and, if $I=\{i_1,i_2,\dots,i_{|I|}\}$ with $i_1<\dots<i_{|I|}$, we define
\begin{align*}
& dx_I:=dx_{i_1}\wedge dx_{i_2}\wedge\dots\wedge dx_{i_{|I|}}\in\bwl^{|I|}\R^{2n}\\
& dy_I:=dy_{i_1}\wedge dy_{i_2}\wedge\dots\wedge dy_{i_{|I|}}\in\bwl^{|I|}\R^{2n}\\
& dxy_I:=dxy_{i_1}\wedge dxy_{i_2}\wedge\dots\wedge dxy_{i_{|I|}}\in\bwl^{2|I|}\R^{2n}.
\end{align*}
For $I=\emptyset$ we agree that $dx_\emptyset=dy_\emptyset=dxy_\emptyset=1$. It will be also convenient to set for $i\in\{1,\dots,2n\}$
\[
dz_i:=dx_i\text{ if }i\leq n,\quad dz_i:=dy_{i-n}\text{ if }i\geq n+1 
\]
and accordingly
\begin{equation}\label{eq:defdzI}
dz_I:=dz_{i_1}\wedge dz_{i_2}\wedge\dots\wedge dz_{i_{|I|}}
\end{equation}
whenever $I=\{i_1,i_2,\dots,i_{|I|}\}$ with $1\leq i_1<\dots<i_{|I|}\leq 2n$.

A basis of $\bwl^h\R^{2n}$ is given by the family $\{dx_I\wedge dy_J\wedge dxy_K\}_{(I,J,K)}$, where   $(I,J,K)$ range among all (ordered) triples of pairwise disjoint subsets $I,J,K$ of $\{1,\dots,n\}$ such that $|I|+|J|+2|K|=h$. In particular one can write
\[
\bwl^h\R^{2n}=\bigoplus_{(I,J)} dx_I\wedge dy_J\wedge\bwl^{h, I,J}\R^{2n},
\]
where the sum ranges among all ordered pairs $(I,J)$ of disjoint $I,J\subset\{1,\dots,n\}$ such that $0\leq |I|+|J|\leq h$ and $\bwl^{h, I,J}\R^{2n}$ is defined by
\[
\bwl^{h, I,J}\R^{2n}:=\Span\{dxy_K:K\subset\{1,\dots,n\}\setminus(I\cup J)\text{ and }|I|+|J|+2|K|=h\}.
\] 
The parity of $|I|+|J|$ is necessarily the same of $h$.
Observe also that
\[
L(dx_I\wedge dy_J\wedge\bwl^{h, I,J}\R^{2n})\subset dx_I\wedge dy_J\wedge\bwl^{h+2, I,J}\R^{2n};
\]
more precisely,  for every $\al\in \bwl^{h, I,J}\R^{2n}$ one has
\[
L(dx_I\wedge dy_J\wedge\al)=-dx_I\wedge dy_J\wedge \al\wedge\sum_{i\in\unoenne\setminus( I\cup J)}dx_i\wedge dy_i.
\]
In particular 
\begin{equation}\label{eq:kersplit}
\ker L=\bigoplus_{(I,J)}dx_I\wedge dy_J\wedge \ker L_{h,I,J},
\end{equation}
where $L_{h,I,J}:\bwl^{h, I,J}\R^{2n}\to\bwl^{h+2, I,J}\R^{2n}$ is defined by 
\[
L_{h,I,J}(\al):=-\al\wedge\sum_{i\in\unoenne\setminus( I\cup J)}dx_i\wedge dy_i. 
\]
There is a canonical isomorphism $\iota:\bwl^{h,I,J}\R^{2n}\to\bwl_D^{2\ell}\R^{2m}$, where
\begin{equation}\label{eq:ell,m}
\begin{split}
&\ell:=({h-|I|-|J|})/2,\qquad m:=n-|I|-|J|,\\
&\bwl^{2\ell}_D\R^{2m}:=\Span\{dxy_K:K\subset\{1,\dots,m\},|K|=\ell\},
\end{split}
\end{equation}
according to which $\iota\circ L_{h,I,J}= L\circ \iota$. The study of $\ker L_{h,I,J}$ (which, by~\eqref{eq:kersplit}, determines $\ker L$) is thus equivalent to the study of the kernel of  the restriction 
\[
L_D:=L_{{\big|}\text{\large$\wedge$}^{2\ell}_D\text{\small$\R^{2m}$}}.
\]

\begin{remark}
The objects introduced so far are well-defined unless  $|I|+|J|=n$, i.e., if $m=0$: this can happen only if $h=n$ and gives that also $\ell=0$. In this case, we agree that $\bwl^{h,I,J}\R^{2n}=\bwl^{2\ell }\R^{2m}=\R$, $\bwl^{h+2,I,J}\R^{2n}=\bwl^{2\ell +2}\R^{2m}=\{0\}$ and $L_D=0$. It is immediate to check that, for every such $I,J$, one has $L(dx_I\wedge dy_J)=0$.
\end{remark}

Observe that $L(\bwl^{2\ell}_D\R^{2m})\subset \bwl^{2\ell+2}_D\R^{2m}$, hence $L_D$ maps $\bwl^{2\ell}_D\R^{2m}$ on $\bwl^{2\ell+2}_D\R^{2m}$. It is understood that, when $\ell=m$, then $\bwl^{2\ell+2}_D\R^{2m}=\{0\}$ and $L_D=0$. Notice also that the inequality $2\ell\geq m$ holds because $h\geq n$.

\begin{lemma}\label{lem:dimKerLD}
The operator $L_D:\bwl^{2\ell}_D\R^{2m}\to\bwl^{2\ell+2}_D\R^{2m}$ is surjective for every integer $\ell$ such that $m\leq2\ell\leq 2m$. In particular
\[
\dim\ker L_D=\dim \bwl^{2\ell}_D\R^{2m}-\dim\bwl^{2\ell+2}_D\R^{2m}=\binom m\ell-\binom m{\ell+1},
\]
where we agree that $\binom m{m+1} =0$.
\end{lemma}
\begin{proof}
When $2\ell=2m$ there is nothing to prove. For the remaining cases, it suffices to prove that, for every $i=1,\dots, m$, the operator
\[
L_D^i:=\underbrace{L_D\circ\dots\circ L_D}_{\text{$i$ times}}:\bwl_D^{m-i}\R^{2m}\to\bwl_D^{m+i}\R^{2m}
\]
is bijective. Since $\dim\bwl_D^{m-i}\R^{2m} = \dim\bwl_D^{m+i}\R^{2m}$, this is an immediate consequence of the bijectivity of $L^i:\bwl^{m-i}\R^{2m}\to\bwl^{m+i}\R^{2m}$, see~\cite[Proposition~1.1]{BryantGG}.
\end{proof}

We now provide a basis of $\ker L_D:\bwl_D^{2\ell}\R^{2m}\to \bwl_D^{2\ell+2}\R^{2m}$ for $m\leq2\ell\leq 2m$. Assume that the numbers $1,\dots,m$ have been arranged (each number appearing exactly once) in a tableau $ R$ with 2 rows, the first row having $\ell$ elements $R^1_1,\dots,R^1_\ell$ and the second having $m-\ell$ elements $R^2_1,\dots,R^2_{m-\ell}$, as follows
\begin{equation}\label{eq:Tabella}
R=\begin{tabular}{|c|c|c|c|ccc}
\cline{1-7}
$R^{1^{\vphantom{2^2}}}_1$ & $R^1_2$ &$\ \cdots\ $& $R^1_{m-\ell}$ &\multicolumn{1}{c|}{$ R_{m-\ell+1}^1$} & \multicolumn{1}{c|}{$\ \cdots\ $} & \multicolumn{1}{c|}{$R_\ell^1$}\\
\cline{1-7}
$R^2_1$ & $R^2_2$ &$\ \cdots\ $& $R^{2^{\vphantom{2^2}}}_{m-\ell}$
\\
\cline{1-4}
\end{tabular}
\end{equation}
where, clearly, $R$ has to be read as a $(2\times \ell)$ rectangular matrix when $2\ell=m$. As customary, we call {\em Young tableau} such a tableau, see~\cite{Fulton}. When $m=\ell=0$, we agree that the tableau is empty. When $m=\ell\geq 1$ we agree that $R$ has one row only (the second row is empty).

Given such a Young tableau $R$ consider the covector $\al_R\in\bwl^{2\ell}_D\R^{2m}$ defined by
\begin{equation}\label{eq:DefAlfaT}
\al_R:=(dxy_{R^1_1}- dxy_{R^2_1})\wedge(dxy_{R^1_2}- dxy_{R^2_2})\wedge\dots\wedge(dxy_{R^1_{m-\ell}}- dxy_{R^2_{m-\ell}})\wedge dxy_{R^1_{m-\ell+1}}\wedge\dots\wedge dxy_{R^1_\ell}.
\end{equation}
When $m=\ell=0$ we agree that $\al_R:=1$, while if $m=\ell\geq 1$ we set $
\al_R:=dxy_{R^1_1}\wedge\dots\wedge dxy_{R^1_m}$.
Using the  equality $(dxy_i-dxy_j)\wedge(dxy_i+dxy_j)=0$, valid for every $i,j=1,\dots,m$, one can easily check that
\begin{equation}\label{eq:tuttetabelleok}
L(\al_R)=\al_R\wedge d\theta=0.
\end{equation}
It follows that, for every Young tableau $R$, $\al_R$ belongs to $\ker L_D$.

Consider now the set $\mathscr R$ of  {\em standard Young tableaux}, i.e., of those Young tableaux $R$   such that the entries in each row and in each  column are in increasing order:
\[
\begin{array}{ccccccccccccc}
R^1_1 & < & R^1_2 & < & \ \cdots\  & < & R^1_{m-\ell} & <  & R^1_{m-\ell+1} & < & \ \cdots\  & < & R^1_\ell\\
\wedge && \wedge && &&\wedge\\
R^2_1 & < & R^2_2 & < & \ \cdots\  & < & R^2_{m-\ell}. 
\end{array}
\]
Notice that a standard Young tableau $R$ always satisfies $R^1_1=1$. The following lemma shows that the family $\{\al_R\}_{R\in\mathscr R}\subset\ker L_D$ is made by linearly independent covectors.

\begin{lemma}\label{lem:YoungTabLinInd}
Let $m,\ell$ be integers such that $0\leq m\leq 2\ell\leq 2m$ and let $\mathscr R$ be the set of standard Young tableaux with two rows of width $\ell$ and $m-\ell$, as in \eqref{eq:Tabella}. Then the elements $\{\al_R\}_{R\in\mathscr R}$ defined in \eqref{eq:DefAlfaT} are linearly independent.
\end{lemma}
\begin{proof}
When $m=0$ also $\ell=0$ and the family $\{\al_R\}_{R\in\mathscr R}$ is made by a unique element $\al_R=1$; there is nothing to prove. When $\ell=m$, the family $\mathscr R$ is made by the single standard Young tableau 
$
R=\begin{tabular}{|c|c|c|c|}
\cline{1-4}
$1$ & $2$ &$\ \cdots\ $& $m$
\\
\cline{1-4}
\end{tabular}.
$
Therefore $\{\al_R\}_{R\in\mathscr R}=\{dxy_1\wedge\dots\wedge dxy_m\}$ is made by linearly independent vectors, as wished. 

We  therefore assume that $1\leq m\leq 2\ell\leq 2m-2$. Let us begin with some preliminary considerations. Given a Young tableau $R$, we denote by $\si(R)$ the sum $R^1_1+R^1_2+\dots+ R^1_\ell$ of the elements in the first row of $R$. Clearly, there exists an integer $M=M(m,\ell)$ such that 
\[
1+\dots+\ell\leq \si(R)\leq M\qquad\text{for every $R\in\mathscr R$;}
\]
moreover, there is a unique $R_{min}\in\mathscr R$ such that $\si(R)=1+\dots+\ell$, namely
\[
R_{min}:=
\begin{tabular}{|c|c|c|c|ccc}
\cline{1-7}
$1$ & $2$ &$\ \cdots\ $& $m-\ell$ &\multicolumn{1}{c|}{$ m-\ell+1$} & \multicolumn{1}{c|}{$\ \cdots\ $} & \multicolumn{1}{c|}{$\ell^{\phantom{2}}$}\\
\cline{1-7}
$\ell+1$ & $\ell+2^{\phantom{2}}$ &$\ \cdots\ $& $m$
\\
\cline{1-4}
\end{tabular}.
\]

Let $\overline R\in\mathscr R$ be fixed and let $\overline K:=\{\overline R^1_1,\overline R^1_2,\dots,\overline R^1_\ell\}$ be the subset of $\{1,\dots,m\}$ containing the $\ell$ elements in the first row of $\overline R$. We claim that, if
\begin{equation}\label{eq:Giorgia}
\text{$R\in\mathscr R$ is such that   each column of $R$ contains one element of $\overline K$}
\end{equation}
(in particular, by the pigeonhole principle, each column of $R$ contains exactly one element of $\overline K$), then
\begin{equation}\label{eq:Riccardo}
\text{$\si(R)\leq\si(\overline R)$, and the equality $\si(R)\leq\si(\overline R)$ holds if and only if $R=\overline R$.}
\end{equation}
Indeed, for every  tableau $R$ as in~\eqref{eq:Giorgia} there exist a permutation $\pi$ of $\{1,\dots,\ell\}$ and a function $u:\{1,\dots,\ell\}\to\{1,2\}$ such that
\[
R^{u(j)}_j=\overline R^1_{\pi(j)}\qquad\text{for all $j=1,\dots,\ell$.}
\]
Therefore
\begin{equation}\label{eq:sigmaminore}
\si(R)=\sum_{j=1}^\ell R^1_j\leq \sum_{j=1}^\ell R^{u(j)}_j=\sum_{j=1}^\ell\overline R^1_{\pi(j)}=\si(\overline R);
\end{equation}
notice that equality holds if and only if $R^1_j=R^{u(j)}_j$ for every $j=1,\dots,\ell$, i.e., if and only if
\[
\{ R^1_1, R^1_2,\dots, R^1_\ell\} =
\{\overline R^1_1,\overline R^1_2,\dots,\overline R^1_\ell\}.
\]
Since a standard Young tableau is uniquely determined once one fixes the set of elements of the first row,  equality in \eqref{eq:sigmaminore} holds if and only if $R=\overline R$.

We now prove the lemma. Assume that there are real coefficients $(b_R)_{R\in\mathscr R}$, such that $\sum_{R\in\mathscr R}b_R\al_R=0$; we prove that all the coefficients $b_R$ are null. We perform this task reasoning by induction on $\bar s=1+\dots+\ell,\dots,M$ and showing that
\[
b_R=0\qquad\text{for every $R$ such that $\si(R)=\bar s$.}
\]

Consider first the case $\bar s=1+\dots+\ell$; we need to prove that $b_{R_{min}}=0$. For every  $R\in\mathscr R$ let us write $\al_R=\sum_K c_{R,K} dxy_K$, where the sum ranges among all $K\subset\{1,\dots,m\}$ with $|K|=\ell$ and $c_{R,K}$ are suitable real numbers. We claim that, defining $\overline K:=\{1,2,\dots,\ell\}$, then
\begin{equation}\label{eq:Tbarra}
c_{R,\overline K}=0\qquad\text{for every }R\in\mathscr R\setminus\{R_{min}\}.
\end{equation}
This would be enough to conclude: indeed, one would have
\[
0=  \sum_{R\in\mathscr R}b_R\al_R = \sum_{K} \sum_{R\in\mathscr R}b_Rc_{R,K} dxy_K= b_{R_{min}} dxy_{\overline K} + \sum_{K\neq \overline K} \sum_{R\in\mathscr R}b_Rc_{R,K} dxy_K
\]
which gives $b_{R_{min}}=0$, as desired. 

\noindent We prove \eqref{eq:Tbarra}: by the very definition~\eqref{eq:DefAlfaT} of $\al_R$, we see that for every $K$ one has $c_{R,K}=0$ unless each column of $R$ contains (exactly) one element of $K$ (in which case $c_{R,K}\in\{1,-1\}$). Using this observation with $K=\overline K$, \eqref{eq:Tbarra} follows because, by the implication~\eqref{eq:Giorgia}~$\Rightarrow$~\eqref{eq:Riccardo}, the only standard Young tableau such that each of its columns contains one element of $\overline K$ is $R_{min}$ itself, because any other standard Young tableau $R$ with such a  property would satisfy $\si(R)<1+\dots+\ell$.  

Assume now that $b_R=0$ for every $R\in\mathscr R$ such that $\si(R)\leq \overline s-1$; we prove that $b_R=0$ for every (fixed) $\overline R$ such that $\si(\overline R)=\overline s$. Let $\overline K:=\{\overline R^1_1,\dots,\overline R^1_\ell\}$ be the set formed by the elements in the first row of $\overline R$; we claim that
\begin{equation}\label{eq:Tbarra222}
c_{R,\overline K}=0\qquad\text{for every }R\in\mathscr R\setminus \{\overline R\}\text{ such that }\si(R)\geq\overline s.
\end{equation}
This would be enough to conclude: indeed,  one would have
\[
0=  \sum_{R\in\mathscr R}b_R\al_R=  \sum_{\substack{R\in\mathscr R\\\si(R)\geq \overline s}}b_R\al_R
 = \sum_{K} \sum_{\substack{R\in\mathscr R\\\si(R)\geq \overline s}}b_Rc_{R,K} dxy_K= b_{\overline R} dxy_{\overline K} + \sum_{K\neq \overline K} \sum_{\substack{R\in\mathscr R\\\si(R)\geq \overline s}}b_Rc_{R,K} dxy_K
\]
which would give $b_{\overline R}=0$, as desired. 

\noindent Claim \eqref{eq:Tbarra222} can be proved similarly as before: by the  definition of $\al_R$, for a standard Young tableau $R$ one has $c_{R,\overline K}=0$ unless each column of $R$ contains (exactly) one element of $\overline K$; in particular, \eqref{eq:Tbarra222} follows from the implication \eqref{eq:Giorgia}$\Rightarrow$\eqref{eq:Riccardo}. 
%
\end{proof}

The cardinality of $\mathscr R$ can be computed using the {\em Hook length formula}, also known as {\em Frame-Robinson-Thrall formula}. We refer to~\cite[page~53]{Fulton}
: such a formula states that the cardinality of $\mathscr R$ equals
\begin{align*}
&\frac{m!}{(\ell+1)\ell(\ell-1)\cdots(2\ell-m+2)\;1\cdot\;(2\ell-m)\cdots 2\cdot 1\;\cdot\; (m-\ell)(m-\ell-1)\cdots1}\\
=\:& \frac{m!}{(\ell+1)!(m-\ell)!}(2\ell-m+1)
\ =\ \frac{m!}{(\ell+1)!(m-\ell)!}(\ell+1-(m-\ell))\\
=\:&\binom m\ell-\binom m{\ell+1}
\end{align*}
provided $\ell<m$.
If $\ell=m$, then the cardinality of $\mathscr R$ is 1. In both cases, Lemma~\ref{lem:dimKerLD} implies that the cardinality of $\mathscr R$ is equal to $\dim\ker L_D$: together with Lemma \ref{lem:YoungTabLinInd}, this proves that $\{\al_R\}_{R\in\mathscr R}$ is a basis of $\ker L_D$.  

We can  summarize the discussion above as follows.

\begin{proposition}\label{prop:BaseKerLD}
Consider integer numbers $n,h$ such that $n\geq 1$  and $n\leq h\leq 2n$. Then a basis of the kernel of $L:\bwl^h\R^{2n}\to\bwl^{h+2}\R^{2n}$ is given by the elements of the form $dx_I\wedge dy_J\wedge\al_R$ where
\begin{itemize}
\item $I,J$ are disjoint subsets of $\{1,\dots,n\}$;
\item $\al_R$ is defined as in \eqref{eq:DefAlfaT} and $R$ is a standard Young tableau where the elements of $\{1,\dots,n\}\setminus(I\cup J)$ are arranged in two rows, the first one having $(h-|I|-|J|)/2$ elements and the second one having $(2n-h-|I|-|J|)/2$ elements. 
\end{itemize}
\end{proposition} 

\begin{remark}\label{rem:tabellenonYoung}
It follows from~\eqref{eq:tuttetabelleok} that, given $I,J\subset\unoenne$ disjoint and a (non-necessarily standard) Young tableau $Q$ containing the elements of $\unoenne\setminus(I\cup J)$, then the covector $dx_I\wedge dy_J\wedge\al_Q$ is in the kernel of $L$. Moreover, by~\eqref{eq:kersplit} $dx_I\wedge dy_J\wedge\al_Q$ can be written  as a linear combination of covectors of the form $dx_I\wedge dy_J\wedge\al_R$ where $R$ ranges among all standard Young tableaux with the same shape  and containing the same elements as $Q$.
\end{remark}

Proposition \ref{prop:BaseKerLD} has the following immediate consequence. 

\begin{corollary}\label{cor:baseJ}
Let $n\geq 1$ and $1\leq k\leq n$ be integers. Then a basis of $\mathcal J^{2n+1-k}$ is given by the elements of the form $dx_I\wedge dy_J\wedge\al_R\wedge \theta$ where $I,J,R$ are given by Proposition \ref{prop:BaseKerLD} with  $h:=2n-k$.
\end{corollary}
\begin{proof}
We observed  at the beginning of the section that
\[
\mathcal J^{2n+1-k}=\{\la\wedge\theta:\la\in\bwl^{2n-k}\h_1\text{ and }\la\wedge d\theta=0\}.
\]
The statement now easily follows from Proposition~\ref{prop:BaseKerLD}.
\end{proof}

Proposition~\ref{prop:baseJintro} is now an easy consequence of Corollary~\ref{cor:baseJ}.

We feel it might be useful to provide some simple examples before continuing our analysis.

\begin{example}\label{example:baseJ3inH2}
When $n=2$, the space $\mathcal J^{n+1}=\mathcal J^3$ in $\H^2$ is 5-dimensional. The corresponding basis $dx_I\wedge dy_J\wedge\al_R\wedge\theta$ provided by Proposition~\ref{prop:baseJintro} is displayed on the left column of the following table, while on the right column the corresponding triple $(I,J,R)$ appears:
\begin{align*}
& dx_1\wedge dx_2\wedge\theta &&\hspace{-2cm}\longleftrightarrow\hspace{1cm} (\{1,2\},\emptyset,\emptyset)\\ 
& dx_1\wedge dy_2\wedge\theta&&\hspace{-2cm}\longleftrightarrow\hspace{1cm} (\{1\},\{2\},\emptyset)\\ 
& dx_2\wedge dy_1\wedge\theta&&\hspace{-2cm}\longleftrightarrow\hspace{1cm} (\{2\},\{1\},\emptyset)\\ 
& dy_1\wedge dy_2\wedge\theta&&\hspace{-2cm}\longleftrightarrow\hspace{1cm}(\emptyset,\{1,2\},\emptyset)\\ 
& (dx_1\wedge dy_1 - dx_2\wedge dy_2)\wedge\theta
&&\hspace{-2cm}\longleftrightarrow\hspace{1cm} 
\left(\emptyset,\emptyset,
\begin{tabular}{|c|}
\cline{1-1}
$1$\\
\cline{1-1}
$2$
\\
\cline{1-1}
\end{tabular}\ 
\right),
\end{align*}
where $\emptyset$ denotes either the empty set or the empty tableau. See also~\cite[Example~3.12]{BaldiFranchiDivCurl}.
\end{example}

\begin{example}\label{example:baseJ5inH3}
When $n=3$, the space $\mathcal J^5$ in $\H^3$ is 14-dimensional; let us write the basis provided by Proposition~\ref{prop:baseJintro}. Observe that here $n=3$, $k=2$ and $h=4$.  We first determine the triples $(I,J,R)$: since  $|I|+|J|$ is not greater than $ k=2$ and has the same parity  as $h=4$. then either $|I|+|J|=0$ or $|I|+|J|=2$. 

If $|I|+|J|=0$, then $m=3$, $\ell=2$ and all the indices 1, 2, 3 appear in the tableau $R$, whose rows have length $2$ and $1$, respectively. It is immediate to check that $R$ can be only one the following two tableaux
\[
\begin{tabular}{|c|c|}
\cline{1-2}
$1$ & $2$\\
\cline{1-2}
$3$
\\
\cline{1-1}
\end{tabular}\,
\qquad\qquad
\begin{tabular}{|c|c|}
\cline{1-2}
$1$ & $3$\\
\cline{1-2}
$2$
\\
\cline{1-1}
\end{tabular}
\]
which, respectively, provide  the elements
\begin{equation}\label{eq:sheriff}
(dx_1\wedge dy_1- dx_3\wedge dy_3)\wedge dx_2\wedge dy_2\wedge\theta,\qquad (dx_1\wedge dy_1- dx_2\wedge dy_2)\wedge dx_3\wedge dy_3\wedge\theta
\end{equation}
of $\mathcal J^5$.

If $|I|+|J|=2$, then  $I\cup J=\{a,b\}$ for some $a,b\in\{1,2,3\}$ and the tableau $R$ contains the unique remaining index $c\in\{1,2,3\}\setminus\{a,b\}$; in particular, the second row of $R$ is empty (in fact, here $\ell=m=1$) and 
$R=\begin{tabular}{|c|}
\cline{1-1} $c$ \\ \cline{1-1} \end{tabular}$. This produces the following elements of $\mathcal J^5$
\begin{align}
& dx_a\wedge dx_b\wedge dx_c\wedge dy_c\wedge\theta\qquad\text{provided $a<b$}\nonumber\\
& dx_a\wedge dy_b\wedge dx_c\wedge dy_c\wedge\theta\label{eq:deputy}\\
& dy_a\wedge dy_b\wedge dx_c\wedge dy_c\wedge\theta\qquad\text{provided $a<b$}\nonumber
\end{align}
and, as $a,b,c$ vary in $\{1,2,3\}$,  each of the covectors in~\eqref{eq:deputy} provides, respectively, 3, 6 and 3 elements of $\mathcal J^5$. All in all, a basis of $\mathcal J^5$ is  provided by the 2+12 elements displayed in~\eqref{eq:sheriff} and~\eqref{eq:deputy}.
\end{example}

\begin{remark}\label{rem:Rrettangolare}
The difference between the lengths of the first and second rows of every standard Young tableau $R$ appearing in Proposition~\ref{prop:baseJintro} (respectively, in Proposition~\ref{prop:BaseKerLD}) is determined by $k$ (resp., by $h$) and it is equal to $n-k$ (resp., to $h-n$). In particular, the standard Young tableaux $R$ appearing in Proposition~\ref{prop:baseJintro} (resp., in Proposition~\ref{prop:BaseKerLD}) are rectangular (the two rows have the same length) if and only if $k=n$ (resp., if $h=n$). 
\end{remark}

\begin{remark}\label{rem:tuttetabelleok}
It follows from Remark~\ref{rem:tabellenonYoung} that, given $I,J\subset\unoenne$ disjoint and a (non-necessarily standard) Young tableau $Q$ containing the elements of $\unoenne\setminus(I\cup J)$,  the covector $dx_I\wedge dy_J\wedge\al_Q\wedge\theta$ belongs to $\mathcal J^{|I|+|J|+2\ell+1}$, where $\ell$ denotes the length of the first row of $Q$, and it can be written as a linear combination of the elements $dx_I\wedge dy_J\wedge\al_R\wedge\theta\in \mathcal J^{|I|+|J|+2\ell+1}$ where $R$ ranges among the standard Young tableaux with the same shape and containing the same elements as $Q$.
\end{remark}


\begin{remark}\label{rem:cosafailvettorestandard}
It is a good point to state, for future references, the following property of the basis  provided by Proposition~\ref{prop:baseJintro}.

Let $a,b$ be fixed non-negative integers such that $a+b\leq n$ and $n\leq 2a+b\leq 2n$. 
Let $(I,J,R)$ range among the triples such that $\{dx_I\wedge dy_J\wedge\al_R\wedge \theta\}_{(I,J,R)}$ is the basis of $\mathcal J^{2a+b+1}$ provided by Proposition~\ref{prop:baseJintro}, i.e.,
\begin{itemize}
\item $I,J$ are disjoint subsets of $\{1,\dots,n\}$ such that $|I|+|J|\leq 2n-2a-b$;
\item $R$ is a standard Young tableau where the elements of $\{1,\dots,n\}\setminus(I\cup J)$ are arranged in two rows, the first one having $(2a+b-|I|-|J|)/2$ elements and the second one having $(2n-2a-b-|I|-|J|)/2$ elements.
\end{itemize}
Then either
\begin{equation}\label{eq:IJR}
\left\{
\begin{array}{l}
I=\{a+1,\dots,a+b\}\\
J=\emptyset\\
R=
\begin{tabular}{|c|c|c|c|ccc}
\cline{1-7}
$1$ & $2$ &$\ \cdots\ $& $n-a-b$ 
&\multicolumn{1}{c|}{$ n-a+1$} 
& \multicolumn{1}{c|}{$\ \cdots\ $} 
& \multicolumn{1}{c|}{$a^{\phantom{2}}$}\\
\cline{1-7}
$a+b+1$ & $a+b+2^{\phantom{2}}$ &$\ \cdots\ $& $n$
\\
\cline{1-4}
\end{tabular}
\end{array}
\right.
\end{equation}
%
or
\begin{equation}\label{eq:IJR222}
\langle X_1\wedge\dots\wedge X_{a+b}\wedge Y_1\wedge\dots\wedge Y_a\wedge T\mid  dx_I\wedge dy_J\wedge \al_R\wedge\theta\rangle=0
\end{equation}
where, in~\eqref{eq:IJR}
\begin{itemize}
\item $I=\emptyset$ if $b=0$,
\item $R$ is the empty tableau if $a=0$,
\item $R$  has to be interpreted as a  rectangular $2\times a$ matrix if $n=2a+b$ and $a\geq 1$.
\end{itemize}

Let us prove what claimed. If either $ I\neq\{a+1,\dots,a+b\}$ or $ J\neq\emptyset$, then~\eqref{eq:IJR222} holds. If instead $I=\{a+1,\dots,a+b\}$ and $J=\emptyset$, then the shape of $R$ (i.e., the lengths of its rows) is necessarily the same of the tableau displayed in~\eqref{eq:IJR} and $R$ contains precisely the elements $\{1,\dots,a,a+b+1,\dots,n\}$. We have
\begin{equation}\label{eq:pesaro}
\begin{aligned}
&\langle X_1\wedge\dots\wedge X_{a+b}\wedge Y_1\wedge\dots\wedge Y_a\wedge T\mid  dx_I\wedge dy_J\wedge \al_R\wedge\theta\rangle\\
=\: &
\pm\langle X_1\wedge\dots\wedge X_{a}\wedge Y_1\wedge\dots\wedge Y_a\mid  \al_R\rangle.
\end{aligned}
\end{equation}
where the sign is determined by $a$ and $b$. 
By its definition~\eqref{eq:DefAlfaT}, $\al_R$ can be written as a sum $\sum_S \si(S) dxy_S$, where $\si(S)\in\{1,-1\}$ is a proper sign and the sum ranges among all subsets $S\subset\{1,\dots,a,a+b+1,\dots,n\}$ (i.e., $S$ is a subset of the set of the entries of $R$) of cardinality $a$ that contain exactly one element from each column of $R$. It is clear that 
\[
\text{if $S\neq\{1,\dots,a\}$, then }\langle X_1\wedge\dots\wedge X_{a}\wedge Y_1\wedge\dots\wedge Y_a\mid  dxy_S\rangle=0.
\]
If $R$ is such that at least one of its columns contains no elements of $\{1,\dots,a\}$, then none of the $S$'s appearing in the sum $\al_R=\sum_S \si(S) dxy_S$ is $\{1,\dots,a\}$, and by~\eqref{eq:pesaro}\[
\langle X_1\wedge\dots\wedge X_{a+b}\wedge Y_1\wedge\dots\wedge Y_a\wedge T\mid  dx_I\wedge dy_J\wedge \al_R\wedge\theta\rangle=0.
\]
If all the $a$ columns of $R$ contain at least one element of $\{1,\dots,a\}$, then each of these columns contains exactly one such element. Since $R$ is a standard Young tableau, the sum of the elements of the first row of $R$ is at most $1+\dots+a$, but this sum is clearly also at least $1+\dots+a$. It follows that the first line of $R$ is made by the elements $1,\dots,a$ (in increasing order) and that the remaining elements $a+b+1,\dots,n$ (not belonging to $I$, $J$ nor the first line of $R$) have to be placed, in increasing order, in the second line of $R$. This proves that $R$ must be the one in~\eqref{eq:IJR} and concludes the proof.
\end{remark}

\subsection{Heisenberg currents}\label{subsec:correnti}
For every $k=0,\dots,2n+1$ we introduce the spaces $\DH^k\subset\Omega_\H^k$ of compactly supported smooth Heisenberg $k$-forms, i.e.,
\begin{align*}
& \DH^k:= C_c^\infty\big(\H^n,\tfrac{\text{\large$\wedge$}^k\h}{\mathcal I^k}\big)&&\hspace{-2cm}\text{if }0\leq k\leq n\\
& \DH^k:= C_c^\infty(\H^n,\mathcal J^k)&&\hspace{-2cm}\text{if }n+1\leq k\leq 2n+1
\end{align*}
and we observe that $d_C$  maps $\DH^k$ to $\DH^{k+1}$.  We endow the space $\DH^k$ with the natural topology induced by the topology of the space $\mathcal D^k$ of compactly supported $k$-forms on $\H^n$. 

\begin{definition}
Given $k\in\{0,\dots,2n+1\}$, we denote by $\mathcal D_{\H,k}$ the space of continuous linear functionals on $\DH^k$. An element of $\mathcal D_{\H,k}$ is called {\em Heisenberg $k$-dimensional current} or, for shortness, {\em Heisenberg $k$-current}. 
\end{definition}

For every $k\in\{1,\dots,2n+1\}$ and every Heisenberg $k$-current $\Tcurr\in\mathcal D_{\H,k}$ we denote by $\partial\Tcurr$ the Heisenberg $(k-1)$-current defined, for every $\omega\in\DH^{k-1}$, by $\partial\Tcurr(\omega):=\Tcurr(d_C\omega)$ (recall Remark~\ref{rem:dzero}). Namely,
\[
\begin{array}{ll}
\partial\Tcurr(\omega):=\Tcurr(d\omega)\quad& \text{if }k\neq n+1\\
\partial\Tcurr(\omega):=\Tcurr(D\omega)&\text{if }k= n+1.
\end{array}
\]

It is not our aim to introduce here the {\em mass} of a current (see~\cite[Definition~5.12]{FSSCAIM} and~\cite[Definitions~2.5 and~2.6]{Canarecci_Currents}), which would require to introduce a notion of {\em comass} (\cite[4.1.7]{federer}) on Rumin's spaces of covectors. For our purposes it will be enough to introduce the notion of  {\em  current with finite mass}, see Definition~\ref{def:correntinormali}, and to this end any choice of (co)mass on covectors is equivalent. We denote by $|\cdot|$ the standard norm on $\bwl^*\h$ (in particular, $|\cdot|$ is defined on $\mathcal J^k$ for $k\geq n+1$) and we agree that, for every $1\leq k\leq n$  and $\la\in\bwl^k\h/\mathcal I^k$,
\begin{equation}\label{eq:normasuJn}
|\la|:=\min\left\{|\nu|:\nu\in\bwl^k\h_1,\ [\nu]=\la\right\}
\end{equation}
where $[\nu]$ is the equivalence class of $\nu$ in the quotient ${\bwl^k\h_1}/{\{\mu\wedge d\theta:\mu\in\bwl^{k-2}\h_1\}}$ (recall~\eqref{eq:equivbis}). The quantity $|\cdot|$ is a norm on ${\bwl^k}\h/{\mathcal I^k}$.

\begin{definition}\label{def:correntinormali}
Let $k\in\unodueenne$ be fixed. We say that a current $\Tcurr\in\mathcal D_{\H,k}$ has {\em finite mass} if there exists $M_{\Tcurr}\in\R$ such that
\[
|\Tcurr(\omega)|\leq M_\Tcurr \sup_{p\in\H^n}|\omega(p)|\qquad\text{for every }\omega\in\DH^k.
\]
The current $\Tcurr$ has {\em locally finite mass} if, for each compact set $K\subset\H^n$, there exists $M_{\Tcurr,K}\in\R$ such that
\[
|\Tcurr(\omega)|\leq M_{\Tcurr,K} \sup_{p\in\H^n}|\omega(p)|\qquad\text{for every }\omega\in\DH^k\text{ with support in }K.
\]
Finally, $\Tcurr$ is {\em normal} (respectively, {\em locally normal}) if both $\Tcurr$ and $\partial\Tcurr$ have finite mass (resp. locally finite mass).
\end{definition}

\begin{remark}\label{rem:distriborder0}
The reader familiar with the theory of distributions will realize that a Heisenberg $k$-current has locally finite mass if and only if it has order 0 in the sense of distributions; equivalently, if it is a measure taking values in a proper space of $k$-vectors. More precisely, if $\Tcurr$ is a $k$-current with locally finite mass, then there exist a locally finite (non-negative) measure $\mu$ and a $\mu$-measurable function $\tau$, taking values in $\mathcal J_k$ (if $k\geq n+1$; recall Remark~\ref{rem:introducoJ_k}) or in the dual space $({\bwl^k\h}/{\mathcal I^k})^*$ (if $k\leq n$), such that $\Tcurr=\tau\mu$, i.e.,
\[
\Tcurr(\omega)=\int\langle\tau\mid \omega\rangle\,d\mu\qquad\text{for every }\omega\in\DH^k.
\]
As done in the Introduction,  one can also assume that $|\tau|=1$ $\mu$-a.e.: in this case we denote $\tau$ and $\mu$ by, respectively, $\tauT$  
and $\|\Tcurr\|$ and write $\Tcurr=\tauT \|\Tcurr\|$.
\end{remark}

As in \S\ref{subsec:baseRumin}, for the rest of the present section we denote by $k$ the codimension, rather than the dimension, of a current or submanifold. We  focus on the low-codimensional case and we fix $k\in\{1,\dots,n\}$.

Recall that every $C^1$-regular oriented submanifold $S\subset\H^n$ of codimension $k$ naturally induces a $(2n+1-k)$-dimensional classical current $\clcurr S$ defined, for every smooth and compactly supported $(2n+1-k)$-form $\omega$ in $\H^n$, by
\[
\clcurr S(\omega):=\int_S\omega=\int_S\langle t_S\mid \omega\rangle\,d\text{vol}_S,
\]
where $t_S$ is a unit $(2n+1-k)$-vector tangent to $S$ and with positive orientation and $d\text{vol}_S$ is the  surface measure on $S$ induced by the left-invariant Riemannian metric on $\H^n$ making $X_1,\dots,Y_n,T$ orthonormal. 

Using the notation in Remark~\ref{rem:parteorizz}, for every $p\in S$ we can write
\begin{equation}\label{eq:tauSV}
t_S(p)=(t_S(p))_{\h_1} + \eta_S(p)\wedge T
\end{equation}
for a unique $\eta_S(p)\in\bwl_{2n-k}\h_1$. Notice that $\eta_S(p)= 0$ if and only if $p$ is a {\em characteristic point of $S$}, i.e., $T_pS\subset\h_{1}$. 

Assume now that $p\in S$ is not a characteristic point; then, the intersection $T_pS\cap\h_{1}$ is a  $(2n-k)$-dimensional plane and it is immediate to check that $T_pS\cap\h_{1}=\Span\eta_S(p)$. Therefore, the unit vector
\begin{equation}\label{eq:deftauHS}
\tau^\H_S(p):=\frac{\eta_S(p)}{|\eta_S(p)|}
\end{equation}
is canonically associated with the linear subspace $T_pS\cap\h_{1}$. We denote by $t^\H_S:=\tau^\H_S\wedge T$ the {\em (horizontal)  tangent vector} to $S$. Geometrically, $t^\H_S$ characterizes the blow-up limit of $S$ at $p$: in fact,
\[
\lim_{r\to +\infty} \de_{r} (p^{-1}S) = \exp(\Span t^\H_S(p)),
\]
where the limit is taken with respect to local Hausdorff convergence of sets. See e.g.~\cite{MagCEJM}. 

With this notation, we define the Heisenberg current $\curr S\in\mathcal D_{\H,2n+1-k}$  by
\[
\curr S(\omega):=\int_S \langle t^\H_S\mid \omega\rangle\,d\Shaus^{Q-k},\qquad\omega\in\DH^{2n+1-k}.
\]
The definition is well-posed because the Hausdorff measure $\Shaus^{Q-k}$ is locally finite on $S$ (see e.g. \cite{MagCEJM}) and the set of characteristic points of $S$, where in principle $t^\H_S$ is not defined, is $\Shaus^{Q-k}$-negligible (\cite{balogh,MagJEMS}).

The next  result, though very simple, has to our knowledge never been noticed in the literature.  Lemma~\ref{lem:correntiR=H} and the subsequent Corollary~\ref{cor:senzabordo} prove Proposition~\ref{prop:correnti_cl_vs_H_INTRO}, which will play a crucial role in the sequel.

\begin{lemma}\label{lem:correntiR=H}
Let $n\geq 1$ and $k\in\{1,\dots,n\}$ be fixed. Then there exists $C_{n,k}>0$ such that, for every oriented and $C^1$-smooth submanifold $S$ of $\H^n$ of dimension $2n+1-k$, one has
\[
\curr S(\omega)=C_{n,k}\clcurr S(\omega)\qquad\text{for every }\omega\in \DH^{2n+1-k}.
\]
\end{lemma}
\begin{proof}

By \cite[Theorem 1.2]{MagCEJM} (see also~\cite[Theorem~8.1 and Proposition~8.7]{MagnaniTowardsTheoryArea}),
there exists a positive constant $C_{n,k}$ such that
\[
\Shaus^{Q-k}\res S= C_{n,k}|\eta_S\wedge T|\,\text{vol}_S= C_{n,k} |\eta_S|\,\text{vol}_S,
\]
where $\eta_S(p)$ is as in~\eqref{eq:tauSV}. We remark that one is allowed to apply  \cite[Theorem 1.2]{MagCEJM} because of the rotational invariance~\eqref{eq:drotinv} of the distance $d$, see  \cite[Proposition 4.5]{MagCEJM}. Therefore
\[
\curr S(\omega) 
 = \int_S \langle t^\H_S\mid \omega\rangle\,d\Shaus^{Q-k}
 =\int_S\left\langle \left. \frac{\eta_S}{|\eta_S|}\wedge T \right|\omega \right\rangle d\Shaus^{Q-k}
 =  C_{n,k}\int_S \langle\eta_S\wedge T\mid \omega\rangle\, d\text{vol}_S
\]
and by Remark~\ref{rem:J2n+1-k}
\[
\curr S(\omega) 
= C_{n,k}\int_S \langle(t_S)_{\h_1}+\eta_S\wedge T\mid \omega\rangle\, d\text{vol}_S = C_{n,k}\int_S \langle t_S\mid \omega\rangle\, d\text{vol}_S.
\] 
This concludes the proof.
\end{proof}

\begin{remark}
The constant $C_{n,k}$ provided by Lemma~\ref{lem:correntiR=H} actually depends also on the distance $d$. We  however omit this dependence.
\end{remark}

\begin{remark}
Let $S$ be as in Lemma~\ref{lem:correntiR=H}. The definition of the Heisenberg current $\curr S$ depends on $\Shaus^{Q-k}$, i.e., on the choice of the rotationally invariant distance $d$; on the contrary, the classical current $\clcurr S$ is a purely {\em differential} object -- there is no metric involved. Therefore, Lemma~\ref{lem:correntiR=H}  suggests that $\Shaus^{Q-k}_*:=\Shaus^{Q-k}/C_{n,k}$, which does not depend on the choice of $d$, might be the correctly normalized spherical Hausdorff measure on $\H^n$.
\end{remark}

The following result is very simple.

\begin{corollary}\label{cor:senzabordo}
Let $n\geq 1$ and $k\in\{1,\dots,n\}$ be fixed. If $S$ is an oriented  $C^1$-smooth submanifold of $\H^n$ of codimension $k$ and without boundary, then $\partial \curr S=0$.
\end{corollary}
\begin{proof}
If $k<n$, by Lemma \ref{lem:correntiR=H} and Stokes' theorem we have for every $\omega\in \DH^{2n-k}$
\[
\partial \curr S(\omega)= \curr S(d\omega)= C_{n,k}\clcurr S(d\omega)=0.
\]
Similarly, when $k=n$ we have for every $\omega\in \DH^{n}$
\[
\partial \curr S(\omega)= \curr S(D\omega)=C_{n,k}\clcurr S(D\omega)=0
\]
because, by definition of $D$, the form $D\omega$ is exact.
\end{proof}

\subsection{Rank-one connection and a property of tangent vectors to vertical planes}\label{subsec:rankoneconnect}
We conclude the present Section~\ref{sec:Heisenberg} by stating a technical result, Proposition~\ref{prop:almassimo2}, which will be of the utmost relevance in the proof of  Rademacher's Theorem~\ref{thm:rademacher}. Since the proof of Proposition~\ref{prop:almassimo2} is quite long and involved, for the moment we only state it and postpone its proof to Appendix~\ref{app:max2}. 

Let us start with a preliminary definition.

\begin{definition}\label{def:rank1conn}
Let $V$ be a real vector space of dimension $\ell$ and let $m$ be an integer such that $1\leq m<\ell$. We say that two $m$-dimensional vector subspaces $\mathscr P_1,\mathscr P_2$ of $V$ are {\em rank-one connected} if $\dim \mathscr P_1\cap\mathscr P_2\geq m-1$.
\end{definition}

The terminology chosen in Definition~\ref{def:rank1conn} is borrowed from some classical problems in the Calculus of Variations, see e.g. \cite{BallJamesARMA87,MullerCetraro}. Some motivations are provided by the following  Remark~\ref{rem:applicazionilinearirankone}, that we state for future references and where we identify linear applications and matrices.

\begin{remark}\label{rem:applicazionilinearirankone}
Let $W,V$ be real vector spaces of dimensions $m\geq 1$ and $\ell\geq 1$, respectively. Let also $L_1,L_2:W\to V$ be linear maps and, for $i=1,2$, let $\mathscr P_i:=\{(w,L_i(w))\in W\times V:w\in W\}$ be the graph of $L_i$. Then  the following statements are equivalent:
\begin{enumerate}
\item[(a)] the  vector subspaces $\mathscr P_1,\mathscr P_2 $ of $W\times V$ are rank-one connected;
\item[(b)] either $L_1=L_2$ or rank $(L_1-L_2)=1$.
\end{enumerate} 
\end{remark}

Also the subsequent Lemma~\ref{lem:whyrank1} motivates the terminology in Definition~\ref{def:rank1conn}: recall in fact that simple multi-vectors are sometimes  called {\em rank-one} multi-vectors.  Before stating  Lemma~\ref{lem:whyrank1} let us fix some standard language. 
If  $\mathscr P\subset V$ is an $m$-vector subspace and $t\in\bwl_m V$ is not null, we say that $t$ is {\em tangent} to $\mathscr P$ if $t$ is simple and it can be written as $t=v_1\wedge\dots\wedge v_m$ for some basis $v_1,\dots, v_m$ of $\mathscr P$. Equivalently, $t$ is tangent to $\mathscr P$ if and only if $\Span t:=\{v\in V:v\wedge t=0\}$ coincides with $\mathscr P$. Needless to say, if $t$ and $s$ are both tangent to $\mathscr P$, then $t$ and $s$ are linearly dependent, i.e., $t$ is a multiple of $s$ (and vice versa).

\begin{lemma}\label{lem:whyrank1}
Let $V$ be a real vector space of dimension $\ell$ 
and let $m$ be an integer such that $1\leq m \leq\ell$; let also $\mathscr P_1,\mathscr P_2$ be fixed $m$-dimensional vector subspaces  of $V$. Then the following statements are equivalent:
\begin{enumerate}
\item[(a)] $\mathscr P_1,\mathscr P_2$ are rank-one connected;
\item[(b)] for every couple of simple vectors $t_1,t_2\in\bwl_m V$ tangent to  $\mathscr P_1,\mathscr P_2$ (respectively), the difference $t_1-t_2$ is a simple $m$-vector;
\item[(c)] there exists a couple of simple vectors $t_1,t_2\in\bwl_m V$ tangent to  $\mathscr P_1,\mathscr P_2$ (respectively) such that the difference $t_1-t_2$ is a simple $m$-vector.
\end{enumerate}
\end{lemma}

The proof of Lemma~\ref{lem:whyrank1} is quite simple, nonetheless it is  provided in  Appendix~\ref{app:max2} for the sake of completeness.

After recalling Definition~\ref{def:pianiverticali} (vertical planes) and the notation $[\,\cdot\,]_\mathcal J$ introduced in Remark~\ref{rem:introducoJ_k}, we can eventually state the following result.

\begin{proposition}\label{prop:almassimo2}
Let $k\in\unoenne$ and  $\zeta\in \mathcal J_{2n+1-k}$ be fixed with $\zeta\neq 0$. Then
\begin{enumerate}
\item[(i)] if $k<n$, there exists at most one $(2n+1-k)$-dimensional vertical plane $\mathscr P$ whose unit tangent vector $t_{\mathscr P}=t_{\mathscr P}^\H$ is such that $[t_{\mathscr P}^\H]_\mathcal J$ is a multiple of  $\zeta$;
\item[(ii)] if $k=n$, there exist at most two vertical $(n+1)$-planes $\mathscr P$ whose unit tangent vectors $t_{\mathscr P}=t_{\mathscr P}^\H$ are such that $[t_{\mathscr P}^\H]_\mathcal J$ is a multiple of  $\zeta$. Moreover, if $\mathscr P_1,\mathscr P_2$ are two different such planes, then ${\mathscr P_1}$ and ${\mathscr P_2}$ are not rank-one connected.
\end{enumerate}
\end{proposition}

As we said, the long proof of Proposition~\ref{prop:almassimo2} is postponed to Appendix~\ref{app:max2}. It is however worth observing that the non-uniqueness phenomenon allowed for by statement (ii) above can indeed occur. In fact, consider the vertical 3-planes in $\H^2$ 
\[
\mathscr P_1:=\{(x,y,t)\in\H^2:x_2=y_2=0\}\qtaq
\mathscr P_2:=\{(x,y,t)\in\H^2:x_1=y_1=0\}.
\]
Unit tangent vectors are provided by
\[
t^\H_{\mathscr P_1}=X_1\wedge Y_1\wedge T\qtaq
t^\H_{\mathscr P_2}=-X_2\wedge Y_2\wedge T
\]
and, using for instance Example~\ref{example:baseJ3inH2}, one can easily check that  $[t^\H_{\mathscr P_1}]_{\mathcal J}=[t^\H_{\mathscr P_2}]_{\mathcal J}$. Observe that $\mathscr P_1$ and $\mathscr P_2$ are not rank-one connected, as stated by Proposition~\ref{prop:almassimo2}. The latter also guarantees that no vertical plane $\mathscr P\neq\mathscr P_1$ such that $[t^\H_{\mathscr P}]_{\mathcal J}=[t^\H_{\mathscr P_1}]_{\mathcal J}$ exists other than $\mathscr P_2$. See also Remarks~\ref{rem:nonuniqueness1} and~\ref{rem:nonuniqueness2}.


\section{Intrinsic Lipschitz graphs in Heisenberg groups}\label{sec:intrLipgrHeis}
Though already introduced in Section~\ref{sec:gruppi}, for the reader's convenience we now recall  the notion of intrinsic Lipschitz graph. Assume that a splitting of $\H^n$ is fixed: namely, let $\W,\V$ be homogeneous (i.e., invariant under dilations) and complementary (i.e., $\W\cap\V=\{0\}$ and $\H^n=\W\V$) subgroups of $\H^n$. Each $p\in\H^n$ possesses a unique decomposition $p=p_\W p_\V$ as  product of  elements $p_\W\in\W$, $p_\V\in\V$.  For $\al\geq 0$, the homogeneous cone $C_\al$  along $\V$ is 
\[
C_\al:=\{p\in\H:\|p_\V\|_\H\geq\al\|p_\W\|_\H \}=\{wv:w\in\W,v\in\V,\|v\|_\H\geq\al\|w\|_\H\}\,.
\]
The set $C_\al$  is homogeneous (i.e., invariant under dilations) and $0\in C_\al$; actually, $\V\subset C_\al$. For $p\in\H^n$ we  set $C_\al(p):=pC_\al$.
Observe  that $C_\al$ coincides with  the cone $\cono_{1/\al}$ defined in \S\ref{subsec:intrLipgr}: the reason for this change in notation is that, from now on, we will more frequently deal with the intrinsic Lipschitz constants of maps rather than with the apertures of the associated cones. 

Given $A\subset\W$ and a map $\f:\W\to\V$, the intrinsic graph of $\f$ is the set
\[
\gr_\f:=\{w\f(w):w\in A\}\subset\H^n.
\]
We hereafter adopt the convention that, whenever a map $\f:A\to\V$ is introduced, we denote by $\F:A\to\H^n$ the associated graph map $\F(w):=w\f(w)$; in particular, $\gr_\f=\F(A)$. 

 \begin{definition}\label{def:grafLip}
Let $A\subset\W$; we say that a map $\f:A\to\V$ is {\em intrinsic Lipschitz continuous} if there exists $\al>0$ such that
\begin{equation}\label{eq:defLip}
\forall\:p\in\gr_\f\qquad \gr_\f\cap C_\al(p)=\{p\}.
\end{equation}
We call {\em intrinsic Lipschitz constant} of $\f$ the   infimum of those positive $\al$ for which \eqref{eq:defLip} holds.
\end{definition}

Intrinsic Lipschitz maps of low dimension ($\dim \W\leq n$) are Euclidean Lipschitz continuous and the Hausdorff dimension of their graphs equals the topological one, i.e., $\dim\W$. See~\cite[Remark~3.11]{FSSCJNCA},~\cite[Proposition~3.7]{FSJGA} or~\cite{AntonelliMerlo}.  
On the contrary, intrinsic Lipschitz maps of low codimension $k=\dim\V\leq n$ are not better than Euclidean $1/2$-H\"older continuous, see Remark~\ref{rem:1/2Holder}. Despite this fractal behaviour (see also~\cite{KirchhSC}) they enjoy good metric properties: for instance, the Hausdorff dimension of their graphs is the same, $Q-k$, as $\W$, and the $(Q-k)$-dimensional Hausdorff measure on their graph is even $(Q-k)$-Ahlfors regular: see Remark~\ref{rem:Ahlfors} and the references therein.

\subsection{Intrinsic Lipschitz graphs of low codimension}
We are interested in intrinsic Lipschitz graphs of codimension at most $n$; we then assume that $k:=\dim\V$ is a positive integer not greater than $n$. It can be easily checked (see e.g. \cite[Remark 3.12]{FSSCAIM}) that this  forces $\V$ to be Abelian; by   Remark~\ref{rem:isometrie_delgruppo}, up to an isometric group isomorphism of $\H^n$ one can always assume that 
\begin{equation}\label{eq:Vfissato}
\V=\exp(\Span\{X_1,\dots,X_k\}).
\end{equation}
Moreover, it follows from \cite[Proposition 3.1]{FSJGA} (alternatively, from Theorem~\ref{teo:equivdefLip})  that, if $\W,\W'$ are complementary to $\V$ and $S\subset\H^n$ is such that $S=\gr_\f$ for some intrinsic Lipschitz  $\f:A\to\V$ with $A\subset\W$, then there exists $ A'\subset\W'$ and an intrinsic Lipschitz map $\f': A'\to\V$ such that $S=\gr_{\f'}$. In particular, it will not be restrictive (see also Remark~\ref{rem:Rademachersuffpianicanon}) to assume that 
\begin{equation}\label{eq:Wfissato}
\begin{aligned}
&\W=\exp(\Span\{X_{k+1},\dots,X_n,Y_1,\dots,Y_n,T\}) &&\text{if }1\leq k\leq n-1\\
&\W=\exp(\Span\{Y_1,\dots,Y_n,T\}) &&\text{if }k=n.
\end{aligned}
\end{equation}

Hence, from now on we   work with  the subgroups $\W,\V$ defined in~\eqref{eq:Vfissato} and~\eqref{eq:Wfissato}.
In coordinates
\begin{align}
&\V=\{(x,y,t)\in\R^n\times\R^n\times\R:y=0,\ x_{k+1}=\ldots=x_n=t=0\} &&\text{if }1\leq k\leq n-1\nonumber\\
&\V=\{(x,y,t)\in\R^n\times\R^n\times\R:y=0,\ t=0\} &&\text{if }k=n\label{eq:pianiVWincoordinate}\\
&\W=\{(x,y,t)\in\R^n\times\R^n\times\R:x_1=\ldots=x_k=0\}.\nonumber
\end{align}
For simplicity we will write $v\in\V$ and $w\in\W$ as
\begin{align}
& v=(x_1,\dots,x_k)\in\R^k\nonumber\\
& w=(x_{k+1},\dots,x_n,y_1,\dots,y_n,t)\in\R^{2n+1-k} &&\text{if }1\leq k\leq n-1\label{eq:vwincoordinate}\\
& w=(y_1,\dots,y_n,t)\in\R^{n+1} &&\text{if }k=n.\nonumber
\end{align}
Notice that, if $p=(x,y,t)$, then $p_\V=(x_1,\dots,x_k)$.  Observe also that the measure $\leb{2n+1-k}$ induced on $\W$ by the identification $\W\equiv\R^{2n+1-k}$ in \eqref{eq:vwincoordinate} is a  Haar measure on $\W$ that is also $(Q-k)$-homogeneous, i.e.,
\begin{equation}\label{eq:lebW}
\leb{2n+1-k}(B(w,r)\cap\W) = r^{Q-k} \leb{2n+1-k}(B(0,1)\cap\W)\qquad\forall\ w\in\W,r>0.
\end{equation}
In particular, $\leb{2n+1-k}$ coincides, up to multiplicative factors, with the Hausdorff $\Haus^{Q-k}$ and spherical Hausdorff $\Shaus^{Q-k}$ measures on $\W$.

We now write the intrinsic graph map $\F$ in coordinates: writing $w\in\W\equiv\R^{2n+1-k}$ and  $\f(w)\in\V\equiv\R^k$ as in~\eqref{eq:vwincoordinate} one gets
\begin{equation}\label{eq:graficoincoordinate}
\begin{array}{ll}
\F(w)=(\f(w),x_{k+1},\dots,x_n,y_1,\dots,y_n,t-\tfrac 12\langle \f(w),(y_1,\dots,y_k)\rangle)&\text{if }1\leq k\leq n-1\\
 \F(w)=(\f(w),y_1,\dots,y_n,t-\tfrac 12\langle \f(w),y\rangle) &\text{if }k=n,
\end{array}
\end{equation}
where the scalar products appearing in \eqref{eq:graficoincoordinate} are those of $\R^k$.
Let us write an equivalent analytic formulation for~\eqref{eq:defLip}. Clearly, the latter is equivalent to require that
\begin{equation}\label{eq:qwerty}
\F(w')^{-1}\F(w)=(-\f(w'))(-w')w\f(w)\notin C_\al\qquad\text{for all }w,w'\in A,\ w\neq w'.
\end{equation}
Since $\W$ is a normal subgroup and $\V$ is commutative we have
\[
(-\f(w'))(-w')w\f(w) = \underbrace{(-\f(w'))(-w')w\f(w')}_{\in\W}\underbrace{(\f(w)-\f(w'))}_{\in\V\equiv\R^k}
\]
and \eqref{eq:qwerty} is equivalent to
\begin{equation*}
\|\f(w)-\f(w')\|_\H < \al\| (-\f(w'))(-w')w\f(w')\|_\H
\qquad\text{for all }w,w'\in A,\ w\neq w'.
\end{equation*}
After a boring  computation (that we omit) and writing  $w=(x_{k+1},\dots,y_n,t)$ and  $w'=(x_{k+1}',\dots,y_n',t')$, when $k<n$ we obtain that~\eqref{eq:defLip} is equivalent to
\begin{multline}\label{eq:Lipanalitica<n}
|\f(w)-\f(w')|\\
< \al \left\|\left(x_{k+1}-x_{k+1}',\dots,y_n-y_n',t-t'-\sum_{j=1}^k\f_j(w')(y_j-y_j')+\frac 12\sum_{j=k+1}^n(x_jy_j'-x_j'y_j)\right) \right\|_\H
\end{multline}
for all points $w=(x_{k+1},\dots,y_n,t)$, $w'=(x_{k+1}',\dots,y_n',t')\in A$ with $w\neq w'$. If $k=n$, the formula above reads as
\begin{equation}\label{eq:Lipanalitica=n}
|\f(w)-\f(w')|
< \al \left\|\left(y_1-y_1',\dots,y_n-y_n',t-t'-\sum_{j=1}^k\f_j(w')(y_j-y_j')\right) \right\|_\H
\end{equation}
for all  $w=(y_1,\dots,y_n,t)$, $w'=(y_1',\dots,y_n',t')\in A$, $w\neq w'$. 

It is convenient to point out some basic facts about intrinsic Lipschitz functions.

\begin{remark}\label{rem:sephie'0}
Let $\f:A\subset\W\to\V$ be intrinsic Lipschitz with intrinsic Lipschitz constant not greater than $\al$ and assume that $\f(0)=0$; then, $|\f(w)|\leq\al\|w\|_\H$  for every $w\in\W$. In order to prove this statement, it is enough to plug $w'=0$ in~\eqref{eq:Lipanalitica<n} and~\eqref{eq:Lipanalitica=n}.
\end{remark}

\begin{remark}\label{rem:sealfaminore1/2}
Assume that $\H^n$ is endowed with the distance $d_\infty$ defined in~\eqref{eq:dinfty} and let $\al\leq 1/2$ be fixed. Then, for every intrinsic Lipschitz $\f:\W\to\V$ with intrinsic Lipschitz constant not greater than $\al$ there holds
\begin{equation}\label{eq:peperoni}
|p_\V-q_\V|\leq 2\al d_\infty(p,q)\qquad\text{for every }p,q\in\gr_\f.
\end{equation}
We can assume without loss of generality that $q=0$; using Remark~\ref{rem:sephie'0}
\begin{align*}
|p_\V| &=|\f(p_\W)| \leq \al d_\infty(0,p_\W)\\
&\leq\al(d_\infty(0,p)+ d_\infty(p,p_\W))\leq \al d_\infty(0,p) + \tfrac 12 d_\infty(p_\V,0) =\al d_\infty(0,p) + \tfrac 12 |p_\V|,
\end{align*}
which is~\eqref{eq:peperoni}.
\end{remark}

\begin{remark}\label{rem:epsilonalfa}
Assume that $\H^n$ is endowed with the distance $d_\infty$ defined in~\eqref{eq:dinfty}. Then, for every $\ep>0$ there exists $\bar\al=\bar\al(\ep,n,k)>0$ such that the following holds: for every intrinsic Lipschitz map $\f:\W\to\V$ with intrinsic Lipschitz constant not greater than $\bar\al$
\begin{equation}\label{eq:lapizzaall'ananasfaschifo}
(1-\ep)|p_\V-\f(p_\W)|\leq d_\infty(p,\gr_\f)\leq|p_\V-\f(p_\W)|\qquad\forall\: p\in\H^n.
\end{equation}
The second equality in~\eqref{eq:lapizzaall'ananasfaschifo} is trivial (and, actually, it holds for {\em every} $\f:\W\to\V$) because
\[
d_\infty(p,\gr_\f)\leq d_\infty(p,\F(p_\W))=d_\infty(p_\W p_\V,p_\W \f(p_\W))=d_\infty(0,(p_\V)^{-1}\f(p_\W))=|p_\V-\f(p_\W)|.
\]
In order to prove the first inequality in~\eqref{eq:lapizzaall'ananasfaschifo} we argue by contradiction. Assume that there exists $\bar\ep>0$ such that, for every $i\in\N$, there exist $p_i\in\H^n$ and  $\f_i:\W\to\V$ such that $\f_i$ is intrinsic Lipschitz  with intrinsic Lipschitz constant not greater than $1/i$ and
\begin{equation}\label{eq:lapizzaall'ananasfaschifo2}
\text{for every }i\qquad d_\infty(p_i,\gr_{\f_i}) < (1-\bar\ep) |(p_i)_\V-\f_i((p_i)_\W)|.
\end{equation}
Up to a left translation we can assume that $(p_i)_\W=0$ and $\f_i(0)=0$ for all $i$; in particular $p_i\in\V$ for all $i$ and, up to a dilation, we can also assume that $d_\infty(p_i,0)=1$, so that~\eqref{eq:lapizzaall'ananasfaschifo2} becomes
\begin{equation}\label{eq:lapizzaall'ananasfaschifo3}
\text{for every } i\qquad d_\infty(p_i,\bar p_i)=d_\infty(p_i,\gr_{\f_i}) < 1-\bar\ep,
\end{equation}
where the points $\bar p_i\in\gr_{\f_i}$ are chosen so that $d_\infty(p_i,\bar p_i)=d_\infty(p_i,\gr_{\f_i})$. Observe that 
$d_\infty(0,\bar p_i)\leq d_\infty(0, p_i) +d_\infty(p_i,\bar p_i)<2$; in particular, up to extracting a subsequence there exist $p\in\V$ and $\bar p\in B(0,2) $ such that $p_i\to p\in\V$ and $\bar p_i\to\bar p$ as $i\to+\infty$; clearly, $d_\infty(p,0)=1$. By Remark~\ref{rem:sephie'0} we have that $\f_i\to 0$ uniformly on compact sets of $\W$ (equivalently: $\gr_{\f_i}\to \W$ with respect to the local Hausdorff convergence of closed sets in $\H^n$), in particular $\bar p\in\W$. Letting $i\to+\infty$ in~\eqref{eq:lapizzaall'ananasfaschifo3} provides
\[
 d_\infty(p,\bar p)\leq 1-\bar\ep.
\]
However, writing $p=(p_1,\dots,p_k,0,\dots,0)$ and $\bar p=(0,\dots,0,\bar p_{k+1},\dots,\bar p_{2n+1})$ we notice that
\[
d_\infty(p,\bar p)\geq |(p_1,\dots,p_k)|=d_\infty(0,p)=1,
\]
a contradiction.
\end{remark}

\begin{remark}\label{rem:1/2Holder} 
It follows from \eqref{eq:Lipanalitica<n}--\eqref{eq:Lipanalitica=n} that, if $K\subset\W$ is compact and $\f:\W\to\V$ is intrinsic Lipschitz with intrinsic Lipschitz constant $\al$, then there exists $M=M(\al,\f(0),K)$ such that $|\f|\leq M$ on $K$.

In particular, one can apply \cite[Proposition 4.8]{FSSCJGA-Diff} to get the following: for every $\al>0$ and every compact set $K\subset\W$  there exists $C=C(\al,K)>0$ such that, for every $w,w'\in K$ and every intrinsic Lipschitz $\f:\W\to\V$ with $\f(0)=0$ and intrinsic Lipschitz constant not greater than $\al$, the $1/2$-H\"older estimate
\[
|\f(w)-\f(w')|\leq C|w-w'|^{1/2}
\]
holds, where $|\cdot|$ denote the Euclidean norm in $\W\equiv\R^{2n+1-k}$.
\end{remark}

\begin{remark}\label{rem:Ahlfors}
It was proved in~\cite[Theorem 3.9]{FSJGA} that the $(Q-k)$-dimensional Hausdorff measure on intrinsic Lipschitz graphs graphs is $(Q-k)$-Ahlfors regular; in particular, intrinsic Lipschitz graphs have the same  Hausdorff dimension $Q-k$ of the domain $\W$.
Actually, the statement of~\cite[Theorem 3.9]{FSJGA} is more quantitative: in fact, it states that for every $\al>0$ there exists $C_1=C_1(\al)>0$ such that, for every intrinsic Lipschitz function $\f:\W\to\V$ whose intrinsic Lipschitz constant is not greater than $\al$, one has 
\begin{equation}\label{eq:Ahlfors}
{C_1}^{-1} r^{Q-k} \leq \Shaus^{Q-k}(\gr_\f\cap B(p,r)) \leq C_1 r^{Q-k}\qquad\forall\, p\in\gr_\f,\ r>0.
\end{equation}

Let us point out for future references one further consequence that is implicitly proved in~\cite[Theorem 3.9]{FSJGA}. Denote by $\pi_\W:\H^n\to\W$ the projection $\pi_\W(p):=p_\W$; by~\cite[formula (44)]{FSJGA} there exists a constant $C_2=C_2(\al)\in(0,1)$ such that, for every $\f$ as above,
\[
\pi_\W(B(p,C_2r))\subset\pi_\W(\gr_\f\cap B(p,r))\subset \pi_\W(B(p,r))\qquad\forall\, p\in\gr_\f,\ r>0.
\]
By~\cite[Lemma~2.20]{FSJGA}, which states that there exists $C_3>0$ such that
\[
\leb{2n+1-k}(\pi_\W(B(p,r))=C_3r^{Q-k}\qquad\forall\:p\in\H^n,\ r>0,
\]
we deduce that
\[
{C_4}^{-1}\Shaus^{Q-k}\res\gr_\f
\leq
\F_\#(\leb{2n+1-k}\res\W)
\leq 
{C_4}\Shaus^{Q-k}\res\gr_\f
\]
for a suitable $C_4=C_4(\al)>0$, where  $\F_\#$ denotes push-forward of measures. Since $\Shaus^{Q-k}\res\gr_\f$ is a doubling measure, one can differentiate the measure $\F_\#(\leb{2n+1-k}\res\W)$ with respect to $\Shaus^{Q-k}\res\gr_\f$ (see e.g.~\cite{Rigot}) to get the existence of  $g:\gr_\f\to[{C_4}^{-1},C_4]$ such that $\F_\#(\leb{2n+1-k}\res\W)=g\Shaus^{Q-k}\res\gr_\f$. Equivalently, there exists a measurable function $J_\f:\W\to[{C_4}^{-1},C_4]$ such that
\[
\Shaus^{Q-k}\res\gr_\f = \F_\#(J_\f\,\leb{2n+1-k}\res\W).
\]
Theorem~\ref{thm:formulaarea} (proved later in Section~\ref{sec:applic}) will show that $J_\f$ coincides with the intrinsic Jacobian determinant $J^\f\f$ of $\f$ (see Definition~\ref{def:differenziabilita} below) up to a multiplicative constant. We will of course need to consider  $J_\f$ and $J^\f\f$ as separate notions until Theorem~\ref{thm:formulaarea} is proved and we therefore ask the reader to remember that  the two objects are distinguished even though quite similar  in notation.
\end{remark}

\subsection{Intrinsic differentiability and blow-ups of intrinsic Lipschitz maps}\label{subsec:differenziabilita_blowups}
Left-translations of  intrinsic Lipschitz graphs are also intrinsic Lipschitz graphs. When  $A\subset\W$, $\bar w\in A$ and an intrinsic Lipschitz map $\f:A\to\V$ are fixed and one sets $\bar p:=\bar w\f(\bar w)$, then (see \cite[Proposition 3.6]{ArenaSer} or \cite[Proposition 2.21]{FSJGA}) $\bar p^{-1}\gr_\f$ is the intrinsic Lipschitz graph of the map $\f_{\bar w}:\bar p^{-1}A\f(\bar w)\to\V$ defined by
\[
\f_{\bar w}(w):=\f(\bar w)^{-1}\f(\bar p w \f(\bar w)^{-1}).
\]
We observe that  the domain $\bar p^{-1}A\f(\bar w)=\f(\bar w)^{-1}\bar w^{-1} A\f(\bar w)$ of $\f_{\bar w}$ is a subset of $\W$ because $\W$ is a normal subgroup; moreover, $\f_{\bar w}(0)=0$ by construction. Clearly, $\f_{\bar w}$ has the same intrinsic Lipschitz constant as $\f$.

Dilations of intrinsic Lipschitz graphs   are intrinsic Lipschitz graphs too: if $A$ and $\f$  are as above, $r>0$ and $\f(0)=0$, then $\de_r(\gr_\f)$ is the intrinsic Lipschitz graph of the function $\f^r:\de_r A\to\V$ defined by
\[
\f^r(w):=\de_r \f(\de_{1/r}w)=r\f(\de_{1/r}w).
\]
The intrinsic Lipschitz constant of $\f^r$ equals the one of $\f$.

\begin{definition}\label{def:blowup}
Let $\f:A\to\V$ be a map defined on a (relatively) open subset $A\subset\W$. We say that $\bar\f:\W\to\V$ is a {\em blow-up} of $\f$ at $\bar w\in A$ if there exists a sequence $(r_j)_j$ such that $r_j\to +\infty$ as $j\to+\infty$ and
\[
\lim_{j\to\infty} (\f_{\bar w})^{r_j}=\bar\f\qquad\text{locally uniformly on $\W$.}
\]
\end{definition}

\begin{remark}\label{rem:blowupssonoLip}
Clearly, every blow-up $\bar\f$ of $\f$ is such that $\bar\f(0)=0$. Blow-ups of $\f$ at $\bar w$ are in general not unique. The functions $(\f_{\bar w})^r$, $r>0$, have the same intrinsic Lipschitz constant as $\f$; in particular, every blow-up of $\f$ in intrinsic Lipschitz continuous with intrinsic Lipschitz constant not greater than the one of $\f$.
\end{remark}

We say that $\psi:\W\to\V$ is  {\em intrinsic linear}  if its  graph $\gr_\psi$ is a homogeneous subgroup of $\H^n$; in coordinates, this is equivalent to requiring that $\gr_\psi$ is a vertical plane (recall Definition~\ref{def:pianiverticali}) of dimension $2n+1-k$. 
Another characterization can be given as follows.  For every $w\in\W$  define $w_H\in\R^{2n+1-k}$ as
\begin{equation}\label{eq:defw_H}
\begin{array}{ll}
w_H:=(x_{k+1},\dots,y_n)\qquad & \text{if $k<n$ and $w=(x_{k+1},\dots,y_n,t)$}\\
 w_H:=(y_1,\dots,y_n) & \text{if $k=n$ and $w=(y_1,\dots,y_n,t)$.}
\end{array}
\end{equation}
Then, $\psi$ is intrinsic linear if and only if there exists a $k\times(2n-k)$ matrix $M$ (here identified with a linear map $M:\R^{2n-k}\to\R^k\equiv\V$) such that, for every $w\in\W$, $\psi(w)=M\,w_H$. 

We can now state the following definition.

\begin{definition}\label{def:differenziabilita}
Let $A\subset\W$ be open  and $\f:A\to\V$ be given; we say that $\f$ is {\em intrinsically differentiable} at $\bar w\in A$ if there exists an intrinsic linear map $d\f_{\bar w}:\W\to\V$ such that
\[
\lim_{s\to0}\left( \sup\left\{
\frac{d(\f_{\bar w}(w),d\f_{\bar w}(w))}{d(0,w)}:w\in\W\cap B(0,s)\right\}\right)=0.
\]
The map $d\f_{\bar w}$ is called {\em intrinsic differential} of $\f$ at $\bar w$; the intrinsic graph of $d\f_{\bar w}$ is  called {\em tangent plane} to $\gr_\f$ at $\F(\bar w)$ and is denoted by $\Tan^\H_{\gr_\f}(\F(\bar w))$.

The {\em intrinsic gradient} $\nabla^\f\f(\bar w) $ is the unique $k\times(2n-k)$ matrix such that $d\f_{\bar w}(w)=\nabla^\f\f(\bar w)\, w_H$ for every $w\in \W$. We also define the {\em intrinsic Jacobian determinant} $J^\f\f(\bar w)$ of $\f$ at $\bar w$ as
\[
J^\f\f(\bar w):=\left( 1+\sum_M (\det M)^2 \right)^{1/2},
\]
where the sum ranges on all minors (of any size) of the matrix $\nabla^\f\f$.\end{definition}

\begin{remark}\label{rem:WVortoghonal}
The notions introduced in  Definition~\ref{def:differenziabilita} (and, in particular, that of intrinsic Jacobian needed in Theorem~\ref{thm:formulaarea}) make sense also when the subgroups $\W,\V$ are orthogonal, i.e., when they  are orthogonal as linear subspaces of $\H^n\equiv\R^{2n+1}$. In fact, as Remark~\ref{rem:isometrie_delgruppo} (see also~\cite[\S2.4]{CorniMagnani}) in this case there exists an isometric $\H$-linear isomorphism sending $\W,\V$ to the subgroups defined in~\eqref{eq:Wfissato} and~\eqref{eq:Vfissato}.
\end{remark}

\begin{remark}
It will be convenient to denote the components $(\nabla^\f\f)_{ij}$ of the intrinsic gradient  using indices that vary in the  ranges $i=1,\dots, k$ and $j= k+1,\dots,2n$. This choice might seem a bit unusual for what concerns the index $j$, but it is somehow suggested by the definition of $w_H$. Further justification is provided in \S\ref{subsec:correntiindottedagraficilipschitz}.
\end{remark}

In the following Proposition~\ref{prop:differ->convtoTanH} we collect several statements that are equivalent to intrinsic differentiability: the equivalences among (a), (b) and (c) are straightforward while for the equivalence with (d) we refer to~\cite[Theorem~4.15]{FSSCJGA-Diff}.

\begin{proposition}\label{prop:differ->convtoTanH}
Consider an open set $A\subset\W$, a map $\f:\W\to\V$ and a point $\bar w\in A$. Then the following statements are equivalent:
\begin{enumerate}
\item[(a)] $\f$ is  intrinsically differentiable at $\bar w$;
\item[(b)] there exists an intrinsic linear map $\psi:\W\to\V$ such that $(\f_{\bar w})^r\to \psi$ locally uniformly on $\W$ as $r\to+\infty$;
\item[(c)] the blow-up of $\f$ at $\bar w$ is unique and it is an intrinsic linear map;
\item[(d)] there exists a $(2n+1-k)$-dimensional vertical plane $\mathscr P$ that is complementary to $\V$ and such that, as $r\to+\infty$, the sets $\de_{r}(\F(\bar w)^{-1}\gr_\f)$ converge to $\mathscr P$ with respect to the local Hausdorff convergence of  sets.
, i.e.,
\begin{equation}\label{eq:convtoTanH}
\lim_{s\to 0^+} \left(\sup\left\{\frac{d(\F(\bar w)^{-1}p,\mathscr P)}{d(\F(\bar w),p)}: p\in\gr_\f\cap B(\F(\bar w),s)\right\}\right)=0.
\end{equation}
\end{enumerate}
Moreover, the plane $\mathscr P$ in {\rm (d)} coincides with $\Tan^\H_{\gr_\f}(\F(\bar w))$.
\end{proposition}

It is worth observing that intrinsic graphs parameterizing $\H$-regular submanifolds are intrinsically differentiable: see Remark~\ref{rem:Hreg-->intrdiff} for a  precise statement.

As one can easily guess, intrinsic Lipschitz functions with small Lipschitz constant have small intrinsic gradient at differentiability points; the following lemma provides a quantitative version of this statement.

\begin{lemma}\label{lem:nablaffijminoredialfa}
Assume that $\H^n$ is endowed with the distance $d=d_\infty$ introduced in~\eqref{eq:dinfty}. Let $\f:A\to\V$ be an intrinsic Lipschitz function defined on an open set $A$ of $\W$ and let $\al$ be the intrinsic Lipschitz constant of $\f$. Then for every point $w\in A$ where $\f$ is intrinsically differentiable we have
\begin{equation}\label{eq:nablaffij}
|(\nabla^\f\f(w))_{ij}|\leq\al\qquad\forall\: i=1,\dots,k,\ \forall\:j=k+1,\dots,2n.
\end{equation}
\end{lemma}
\begin{proof}
By Remark~\ref{rem:blowupssonoLip}, the intrinsic differential $d\f_w$ is intrinsic Lipschitz with Lipschitz constant not greater than $\al$. Let $i,j$ be as in~\eqref{eq:nablaffij}; then, by Remark~\ref{rem:sephie'0} we have
\[
|(\nabla^\f\f(w))_{ij}| = |(d\f_w(\exp(W_j)))_i|\leq\al\|\exp(W_j)\|_\H=\al\, d_\infty(0,\exp(W_j))=\al,
\]
where $(d\f(\exp(W_j)))_i$ is the $i$-component of $d\f_w(\exp(W_j))\in\V\equiv\R^k$.
\end{proof}

\subsection{Blow-ups of intrinsic Lipschitz maps are almost always \texorpdfstring{$t$}{t}-invariant}\label{subsec:quasisempreverticali}
The aim of this section is the proof of the following Lemma~\ref{lem:blowuptinvariante}, a very first step towards Theorem~\ref{thm:rademacher}. Given  $h\in\R$, we write $\vec h:=\exp(hT)$ and we observe that, in coordinates, $(x,y,t)\vec h=(x,y,t+h)$.

We say that $\f:\W\to\V$ is {\em $t$-invariant} if
\[
\f(w\vec h)=\f(w)\qquad\text{for every }w\in\W,\ h\in\R.
\]
With the notation introduced in~\eqref{eq:defw_H}, $\f:\W\to\V$ is $t$-invariant if and only if there exists $f_\f:\R^{2n-k}\to\R^k\equiv\V$ such that \begin{equation}\label{eq:phif_phi}
\f(w)=f_\f(w_H)\qquad\text{for every }w\in\W.
\end{equation}
Clearly, every intrinsic linear map $\psi:\W\to\V$ is $t$-invariant: a simple consequence of this fact is contained in the following observation.

\begin{remark}\label{rem:piccoloHolder}
If $\f$ is intrinsically differentiable at $\bar w$, then 
\[
|\f(\bar w\vec h)-\f(\bar w)|=o(|h|^{1/2})\qquad\text{as }h\to0.
\]
This  is a simple consequence of the fact that, as $r\to+\infty$, $(\f_{\bar w})^r$ converges to an intrinsic linear (and then $t$-invariant) map.
\end{remark}

Let us collect  some basic facts about intrinsic Lipschitz maps that are also $t$-invariant: thought very simple, they will be useful in the sequel.

\begin{lemma}\label{lem:tinvariantLipschitz}
Let $\f:\W\to\V$ be intrinsic Lipschitz continuous and $t$-invariant and let $f_\f$ be as in~\eqref{eq:phif_phi}. Then
\begin{enumerate}
\item[(i)] $f_\f:\R^{2n-k}\to\R^k$ and $\f:\W\equiv\R^{2n-k}\to\V\equiv\R^k$  are Euclidean Lipschitz continuous;
\item[(ii)] the  intrinsic graph $\gr_\f$ coincides with
\[
\{(f_\f(w),w,t)\in\R^k\times\R^{2n-k}\times\R\equiv\H^n:w\in\R^{2n-k},t\in\R\}
\]
\item[(iii)] $\f$ is Euclidean differentiable at $\bar w$ if and only if $f_\f$ is Euclidean differentiable at $\bar w_H$;
\item[(iv)] $\f$ is intrinsically differentiable at $\bar w$ if and only if it is Euclidean differentiable at $\bar w$. In this case, $\nabla^\f\f(\bar w)=\nabla f_\f(\bar w_H)$.
\end{enumerate}
\end{lemma}
\begin{proof}
The second part of statement (i) is a direct consequence of the first one, that we now prove only in case $k<n$ as the case $k=n$ requires only minor adjustment in the notation. For every $u=(x_{k+1},\dots,y_n)\in\R^{2n-k}$ and $u'=(x_{k+1}',\dots,y_n')\in\R^{2n-k}$  consider $w:=(x_{k+1},\dots,y_n,0)\in\W$ and $w':=(x_{k+1}',\dots,y_n',t')\in\W$, where $t'=t'(u,u')$ is defined by
\[
t':=\frac 12\sum_{j=k+1}^n(x_jy_j'-x_j'y_j)-\sum_{j=1}^k (f_{\f}(u'))_j(y_j-y_j').
\]
From \eqref{eq:Lipanalitica<n} we deduce that for a suitable $\al>0$ and a positive $C$ depending on the  distance $d$
\begin{align*}
|f_\f(u) - f_\f(u') |
& = |\f(w)-\f(w')|\\
&\leq \al \|(x_{k+1}-x_{k+1}',\dots,y_n-y_n',0) \|_\H \\
&\leq C\al d_\infty(0,(x_{k+1}-x_{k+1}',\dots,y_n-y_n',0))\\
&= C\al \|u-u'\|_{\R^{2n-k}}.
\end{align*}
This proves (i).

Statements (ii) and (iii) are trivial. Concerning (iv), if $f_\f$ is Euclidean differentiable at $\bar w\in\W$ one has
\begin{equation}\label{eq:riunione}
\begin{aligned}
\f_{\bar w}(w)& = \f(\bar w)^{-1}\f(\bar w\f(\bar w) w \f(\bar w)^{-1})
= f_\f((\bar w\f(\bar w) w \f(\bar w)^{-1})_H) - f_\f(\bar w_H) \\
&= f_\f( \bar w_H+w_H) - f_\f(\bar w_H) =  \nabla f_\f(\bar w) w_H + o(|w_H|).
\end{aligned}
\end{equation}
and the intrinsic differentiability of $\f$ at $\bar w$ follows because $|w_H|\leq d_\infty(0,w)\leq C d(0,w)$ for a suitable $C>0$. Conversely, assume that $\f$ is intrinsically differentiable at $\bar w$; for every $u\in \R^{2n-k}$ we define $w=(u,0)\in\W$ and, as in~\eqref{eq:riunione}, we obtain
\begin{align*}
f_\f( \bar w_H+u) - f_\f(\bar w_H) &= f_\f( \bar w_H+w_H) - f_\f(\bar w_H)\\
& = \f_{\bar w}(w) = d\f_{\bar w}(w)+o(d_\infty(0,w))=\nabla^\f\f(\bar w)w_H + o(|u|).
\end{align*}
This proves that $f_\f$ is Euclidean differentiable at $\bar w_H$ and, by (iii), that $\f$ is Euclidean differentiable at $w\in\W$.
\end{proof}

We  now state and prove the following result, that will play a distinguished role in the proof of Theorem~\ref{thm:rademacher}.

\begin{lemma}\label{lem:blowuptinvariante}
Let $\f:\W\to\V$ be intrinsic Lipschitz. Then there exists a  $\leb{2n+1-k}$-negligible set $E\subset\W$ such that, for every $\bar w\in\W\setminus E$ and every blow-up $\bar\f:\W\to\V$ of $\f$ at $\bar w$, $\bar\f$ is $t$-invariant. 
\end{lemma}
\begin{proof}
We prove the statement assuming $k<n$,  the case $k=n$ requiring only straightforward  modifications in the notation.

We claim that 
\begin{equation}\label{eq:claim1}
\text{for a.e. $\bar w\in\W$}\qquad|\f(\bar w\vec h)-\f(\bar w)|=o(|h|^{1/2})\text{ as }h\to0.
\end{equation}
To prove this we write $\f=(\f_1,\dots,\f_k)\in\V\equiv\R^k$ and, for every fixed $i\in\{1,\dots,k\}$ and every fixed $(\bar x_{k+1},\dots, \bar x_n,\bar y_1,\dots,\bar y_{i-1},\bar y_{i+1},\dots\bar y_n)\in\R^{2n-k-1}$, 
we consider the map $\psi_i:\R^2\to\R$ defined by
\[
\psi_i(y,t):=\f_i(\bar x_{k+1},\dots, \bar x_n,\bar y_1,\dots,\bar y_{i-1},y,\bar y_{i+1},\dots\bar y_n,t)
\]
Using \eqref{eq:Lipanalitica<n} we obtain that for every $y,y',t,t'\in\R$ 
\[
|\psi_i(y,t)-\psi_i(y',t')| \leq \al \left\| (0,\dots,0,y-y',0,\dots,0,t-t'-\psi_i(y',t')(y-y') \right\|_\H
\]
where $\al$ is the intrinsic Lipschitz constant of $\f$. This ensures that $\psi_i$ is intrinsic Lipschitz in $\H^1$, i.e., when seen as a map $\psi_i:\W'\to\V'$ where $\W':=\{(0,y,t):y,t\in\R\}$ and $\V':=\{(x,0,0):x\in\R\}$. By the Rademacher's theorem for intrinsic Lipschitz graph of codimension one \cite{FSSCJGA-Diff}, $\psi_i$ is intrinsically differentiable at $\leb 2$-a.e. $(y,t)\in\W'$ and, by Remark \ref{rem:piccoloHolder}, for every such $(y,t)$ one has
\[
|\psi_i(y,t+h)-\psi_i(y,t)|=o(|h|^{1/2})\qquad\text{as }h\to 0.
\]
The claim \eqref{eq:claim1} easily follows.

In order to prove the statement of the lemma it suffices to prove that, for every fixed $\ep>0$, there exists $E_\ep\subset\W$ such that
\begin{itemize}
\item[(a)] $\leb{2n+1-k}(E_\ep)<\ep$ 
\item[(b)] for every $\bar w\in \W\setminus E_\ep$ and every blow-up $\bar\f:\W\to\V$ of $\f$ at $\bar w$, $\bar\f$ is $t$-invariant. 
\end{itemize}
By \eqref{eq:claim1} and Severini-Egorov Theorem, for every $\ep>0$ there exists $E_\ep\subset\W$ such that 
\begin{enumerate}
\item $\leb{2n+1-k}(E_\ep)<\ep$
\item there exists a sequence $(\de_i)_i$ such that
\[
|\f(\bar w\vec h)-\f(\bar w)|<\frac{|h|^{1/2}}i \qquad\text{for every $\bar w\in\W\setminus E_\ep$,  $i\in\N$ and  $h\in(-\de_i,\de_i)$}
\]
\end{enumerate}
By~\eqref{eq:lebW}, the Lebesgue measure $\leb{2n+1-k}$ is doubling and the Lebesgue theorem holds in the metric measure space $(\W,d,\leb{2n+1-k})$ (see e.g. \cite[Chapter 1]{SteinHarmAnal}). Therefore, up to modifying $E_\ep$ on a negligible set we can also assume that
\begin{enumerate}
\item[(3)] for every $\bar w\in\W\setminus E_\ep$ 
\[
\lim_{s\to 0}\frac{\leb{2n+1-k}(B(\bar w,s)\cap\W\setminus E_\ep)}{\leb{2n+1-k}(B(\bar w,s)\cap \W)}=1.
\]
\end{enumerate}
By construction, $E_\ep$ satisfies (a) above; we are going to prove it also satisfies (b), thus completing the proof.

Let then $\bar w\in\W\setminus E_\ep$ and a blow-up $\bar\f:\W\to\V$ of $\f$ at $\bar w$ be fixed; up to a left-translation, we can assume without loss of generality that $\bar w=0$. Let $r_j\to+\infty$ be a sequence such that
\[
\lim_{j\to+\infty}\f^{r_j}=\bar\f\qquad\text{in }L^\infty_{loc}(\W).
\]
Let $R>0$ be fixed; we prove that
\[
\bar\f( x_{k+1},\dots, y_n, t)=\bar\f( x_{k+1},\dots, y_n,0)\quad\forall\ ( x_{k+1},\dots, y_n, t)\in[-R,R]^{2n+1-k}\subset\W\equiv\R^{2n+1-k},
\]
which would immediately give (b). In turn, it is enough to prove that for every $\eta>0$ there exists $\bar\jmath\in\N$ such that
\begin{equation}\label{eq:(b)equiv}
|\f^{r_j}( x_{k+1},\dots, y_n, t)-\f^{r_j}( x_{k+1},\dots, y_n,0)| <\eta\quad\forall\ j\geq\bar\jmath,\forall\ ( x_{k+1},\dots, y_n, t)\in[-R,R]^{2n+1-k}.
\end{equation}
Observe that, by property (3), for $j$ large enough the set $\de_{r_j}(\W\setminus E_\ep)$ is $\eta$-dense in the box $[-R,R]^{2n+1-k}$: namely, for every $w=( x_{k+1},\dots, y_n, t)\in[-R,R]^{2n+1-k}$ there exists $w'\in\de_{r_j}(\W\setminus E_\ep)\cap [-R,R]^{2n+1-k}$ such that $|w-w'|<\eta$. We write $w'=( x_{k+1}',\dots, y_n', t')$; setting $i:=\lfloor R^{1/2}/\eta\rfloor+1$, for large enough $j$ (and, namely, for $r_j^2>R/\de_i$) we have
\begin{align*}
&|\f^{r_j}( x_{k+1}',\dots, y_n', t')-\f^{r_j}( x_{k+1}',\dots, y_n', 0)|\\
=\,& r_j\left|\f( \tfrac{x_{k+1}'}{r_j},\dots, \tfrac{y_n'}{r_j}, \tfrac{t'}{r_j^2})- \f( \tfrac{x_{k+1}'}{r_j},\dots, \tfrac{y_n'}{r_j}, 0)\right|\ 
\leq\ r_j\,\frac{|t'|^{1/2}}{i\,r_j}\ \leq \frac{R^{1/2}}{i}\ <\ \eta
\end{align*}
where we used $|\tfrac{t'}{r_j^2}|\leq \tfrac{R}{r_j^2}<\de_i$ and the fact that  $( \tfrac{x_{k+1}'}{r_j},\dots, \tfrac{y_n'}{r_j}, \tfrac{t'}{r_j^2})=\de_{1/r_j}(w')\in\W\setminus E_\ep$ satisfies (2).
%
By Remark \ref{rem:1/2Holder} there exists $C>0$ (depending only on $R$ and the Lipschitz constant of $\f^{r_j}$, which is the same as the Lipschitz constant of $\f$ and is thus independent of $j$) we finally obtain for large enough $j$ 
\begin{align*}
|\f^{r_j}( x_{k+1},\dots, y_n, t)-\f^{r_j}( x_{k+1},\dots, y_n,0)|
\leq\, & |\f^{r_j}( x_{k+1},\dots, y_n, t)-\f^{r_j}( x_{k+1}',\dots, y_n',t')| \\
&+
|\f^{r_j}( x_{k+1}',\dots, y_n', t')-\f^{r_j}( x_{k+1}',\dots, y_n',0)| \\
&+  |\f^{r_j}( x_{k+1}',\dots, y_n',0)-\f^{r_j}( x_{k+1},\dots, y_n,0)|\\
\leq\, & 2C|w-w'|^{1/2} + \eta\\
\leq\, & 2C\eta^{1/2}+\eta.
\end{align*}
Since the requirements made on $j$ depend on $\eta$ and $R$ but not on $( x_{k+1},\dots, y_n, t)\in[-R,R]^{2n+1-k}$, we have proved the existence of $\bar\jmath$ such that \eqref{eq:(b)equiv} holds. This concludes the proof.
\end{proof}

\subsection{\texorpdfstring{$\H$}{H}-regular submanifolds and \texorpdfstring{$\H$}{H}-rectifiable sets}\label{subsec:Hreg,Hrettif}
In this section we briefly introduce submanifolds  with intrinsic $C^1$ regularity in Heisenberg groups together with the  notion of $\H$-rectifiability. We refer to~\cite{FSSCAIM} for a more comprehensive presentation. 

Given an open set $U\subset\H^n$, we say that $f:U\to\R$ is  of class $C^1_\H$ if $f$ is continuous and its horizontal derivatives 
\[
\nabla_\H f:=(X_1f,\dots,X_nf,Y_1f,\dots,Y_nf)
\]
are represented by continuous functions on $U$. In this case we write $f\in C^1_\H(U)$. We agree that, for every $p\in U$,  $\nabla_\H f(p)\in\R^{2k}$ is identified with the horizontal vector
\[
\nabla_\H f(p):=X_1f(p)X_1+\dots+Y_nf(p)Y_n\in\h_1.
\]

\begin{definition}
Let $k\in\unoenne$ be fixed. We say that $S\subset \H^n$ is a {\em $\H$-regular submanifold} (or a {\em $C^1_\H$-submanifold}) of codimension $k$ if, for every $p\in S$, there exist an open neighbourhood $U\subset\H^n$ of $p$ and $f\in C^1_\H(U,\R^k)$ such that
\[
S\cap U=\{q\in U:f(q)=0\}\qtaq \text{$\nabla_\H f(q)$ has rank $k$ for all $q\in U$.}
\] 
We also define the {\em horizontal normal} $n_S^\H(p)$ to $S$ at $p$ as the horizontal $k$-vector  
\[
n_S^\H(p):=\frac{\nabla_\H f_1(p)\wedge\dots\wedge\nabla_\H f_k(p)}{|\nabla_\H f_1(p)\wedge\dots\wedge\nabla_\H f_k(p)|}\in\bwl_k\h_1
\]
and the {\em (horizontal) tangent} $t^\H_S(p):=*n_S^\H(p)\in \bwl_{2n+1-k}\h$.
\end{definition}

In the definition of the tangent multi-vector $t^\H_S$ the symbol $*$ denotes the Hodge operator. The latter, recalling the notation introduced in~\eqref{eq:defW} and~\eqref{eq:defW_I},  can be defined as the linear isomorphism $\ast:\bwl_k\h\to \bwl_{2n+1-k}\h$ such that
\[
\ast W_I:=(-1)^{\sigma(I)}W_{I^\ast}\qquad \text{for every $I\subset\{1,\dots,2n+1\}$ such that }|I|=k
\]
where $I^\ast:=\{1,\dots,2n+1\}\setminus I$ and $\sigma(I)$ denotes the number of couples $(i,i^\ast)\in I\times I^\ast$ such that $i>i^\ast$. Equivalently, the sign $(-1)^{\sigma(I)}$ can be defined by requiring that $W_I\wedge W_{I^\ast}=(-1)^{\sigma(I)}W_{\{1,\dots,2n+1\}}$ and, all in all, this amounts to requiring that
\[
v\wedge *v=|v|^2X_1\wedge\dots\wedge Y_n\wedge T\qquad \forall\ v\in\bwl_*\h
\]
where the norm $|\cdot|$ is the one associated with the canonical scalar product on multi-vectors making the basis $W_I$ orthonormal. Notice that, if $v=v_1\wedge\dots\wedge v_k\in \bwl_k\h_1$ is a simple horizontal $k$-vector, then $\ast v=w\wedge T$ for some $w\in\bwl_{2n-k}\h_1$. In particular, the horizontal tangent $t^\H_S(p)$ is in fact a vertical multi-vector, i.e., it can be written as $t^\H_S(p)=\tau^\H_S(p)\wedge T$ for a unique  unit vector in $\tau^\H_S(p)\in\bwl_{2n-k}\h_1$. Observe that, when $S$ is of class $C^1$, the definition of $\tau^\H_S$ is consistent with ~\eqref{eq:deftauHS}.


Both $n_S^\H$ and $t_S^\H$ are unit simple vectors. Observe that they are well-defined (even though only up to a sign), i.e., independent from the choice of the defining function $f$. One way of proving this fact is by considering the blow-up of $S$ at $p$: indeed one has
\begin{equation}\label{eq:pianotangenteHregolare}
\lim_{r\to0^+} \de_{1/r}(p^{-1}S)=\Tan^\H_S(p),
\end{equation}
where the limit is taken with respect to the local Hausdorff convergence and  $\Tan^\H_S(p):=\exp(\Span\: t_S^\H(p))$. See e.g.~\cite{FSSCAIM}. The  $(2n+1-k)$-plane $\Tan^\H_S(p)$ is a vertical plane according to Definition~\ref{def:pianiverticali} and it is called {\em tangent plane} to $S$ at $p$. As a consequence of Theorem~\ref{teo:Hreg_grafici_area} below (see also~\cite[Lemma~3.4]{JNGV}), we  have the weak convergence of measures
\begin{equation}\label{eq:convdebolealtangente}
\Shaus^{Q-k}\res\de_{1/r}(p^{-1}S) \rightharpoonup \Shaus^{Q-k}\res \Tan^\H_S(p).
\end{equation}
Also the vector $\tau^\H_S$ is defined only up to a sign; its geometric meaning is provided by the equality
\[
\exp(\Span\: \tau^\H_S(p)) = \Tan^\H_S(p)\cap\exp(\h_1).
\]

\begin{remark}\label{rem:Hreg-->intrdiff}
Proposition~\ref{prop:differ->convtoTanH} implies that the notation $\Tan^\H_S$ introduced in~\eqref{eq:pianotangenteHregolare} is consistent with the notation $\Tan^\H_{\gr_\f}$ of Definition~\ref{def:differenziabilita}. As a consequence, if $A\subset\W$ is open and the function $\f:A\to\V$ parameterizes a $\H$-regular submanifold $S=\gr_\f$ such that $\Tan^\H_S(p)$ is complementary to $\V$ for every $p\in S$, then $\f$ is intrinsically differentiable at every point of $A$.
\end{remark}

It is well-known that  $\H$-regular submanifolds are locally intrinsic Lipschitz graphs and that an integral formula can be provided for their spherical Hausdorff measure $\Shaus^{Q-k}$. We resume these facts in the following statement, which summarizes several results available in the (quite vast) literature and in particular~\cite[Theorem~4.2]{ArenaSer},~\cite[Theorem~4.1]{FSSCAIM} and~\cite[formula~(43)]{CorniMagnani}; see also~\cite{ASCV,BigolinSC2,BigolinSC1,CittiManfredini,Corni,DiDonato2018,DiDonato2019,JNGV,
MagnaniTowardsDiffCalc,MVMathZ12}.
It is worth recalling that the intrinsic Jacobian determinant $J^\f\f$ was introduced in Definition~\ref{def:correntinormali}.

\begin{theorem}\label{teo:Hreg_grafici_area}
Let $S\subset\H^n$ be a $\H$-regular submanifold of codimension $k\leq n$. Then for every $p\in S$ there exist
an open neighbourhood $U$ of $p$,
an open set $A\subset\W$ and an intrinsic Lipschitz $\f:A\to\V$
such that, up to an isometric $\H$-linear isomorphism of $\H^n$,
\begin{align*}
& S\cap U=\gr_\f\\
& \text{$\f$ is intrinsically differentiable on $A$}\\
&\text{ $\nabla^\f\f$ is continuous.}
\end{align*}
Moreover
\begin{equation}\label{eq:areaareaareaformula}
\Shaus^{Q-k}(E)=C_{n,k}\int_{\F^{-1}(E)} J^\f\f\:d\leb{2n+1-k}\qquad\text{for every Borel set }E\subset S\cap U
\end{equation}
where  $\F(w):=w\f(w)$ and $C_{n,k}>0$ is the same constant as in Proposition~\ref{prop:correnti_cl_vs_H_INTRO}.
\end{theorem}

\begin{remark}\label{rem:areaCorniMagnaniorthogonal}
As pointed out in~\cite{CorniMagnani}, the area formula~\eqref{eq:areaareaareaformula} holds more generally when $\W,\V$ are orthogonal (recall Remark~\ref{rem:WVortoghonal}).
\end{remark}

\begin{remark}\label{rem:valoreCnk}
As explained in~\cite{MagnaniTowardsTheoryArea} and~\cite{CorniMagnani}, the exact value of the constant $C_{n,k}$  (which, from the philological point of view, in the present paper is introduced for the first time in Lemma~\ref{lem:correntiR=H}) in Proposition~\ref{prop:correnti_cl_vs_H_INTRO}  and Theorem~\ref{teo:Hreg_grafici_area} is 
\[
C_{n,k}=\Big( \sup\big\{ \leb{2n+1-k}(\W\cap B(p,1)) :   p\in B(0,1)\big\}  \Big)^{-1}.
\]
The  rotational invariance of the distance $d$ plays an important role, see~\cite[Theorem~2.12]{CorniMagnani}.
\end{remark}

We now introduce intrinsic rectifiable sets in Heisenberg groups.

\begin{definition}\label{def:Hrettificabili}
Let $k\in\unoenne$; we say that $R\subset\H^n$ is a {\em $\H$-rectifiable} set of codimension $k$ if $\Shaus^{Q-k}(R)<\infty$ and there exists a finite or countable family $(S_j)_{j}$ of $\H$-regular submanifolds of codimension $k$  such that
\[
\Shaus^{Q-k}\left(R\setminus\bigcup_j S_j\right)=0.
\]
We say that $R\subset\H^n$ is a {\em locally $\H$-rectifiable} set of codimension $k$ if $R\cap B(0,r)$ is $\H$-rectifiable of codimension $k$ for every $r>0$.
\end{definition}

We will later use the well-known fact that sets that are rectifiable  (see e.g.~\cite{federer}) in the Euclidean sense are also $\H$-rectifiable. As a matter of terminology, we say that a set $R\subset\H^n\equiv\R^{2n+1}$ is {\em locally Euclidean rectifiable} of codimension $k$ if  ${\Shaus_{|\cdot|}}^{\!\!\!2n+1-k}\res R$ is a locally finite measure and there exists a finite or countable family $(S_j)_j$ of Euclidean Lipschitz submanifolds of codimension $k$ such that
\[
{\Shaus_{|\cdot|}}^{\!\!\!2n+1-k}(R\setminus\bigcup_j S_j)=0,
\]
where  ${\Shaus_{|\cdot|}}^{\!\!\!2n+1-k}$ denotes the spherical Hausdorff measure with respect to the Euclidean distance on $\H^n\equiv\R^{2n+1}$. If ${\Shaus_{|\cdot|}}^{\!\!\!2n+1-k}(R)<\infty$ we say that $R$ is {\em Euclidean rectifiable}.

\begin{proposition}\label{prop:Euclrectarerect}
Let $R\subset\H^{2n+1}$ be locally Euclidean rectifiable of codimension $k$, $1\leq k\leq n$. Then $R$ is also locally $\H$-rectifiable of codimension $k$.
\end{proposition}

A proof of Proposition~\ref{prop:Euclrectarerect} can be found for instance in~\cite[Proposition~5.4]{FSSCAIM}.

We also recall that classical rectifiable sets of dimension $m$ in $\R^n$  can be equivalently defined as those sets with finite $\Shaus^m_{|\cdot|}$-measure that can be covered, up to $\Shaus^m_{|\cdot|}$-negligible sets, by a countable family of (possibly rotated or translated)  graphs of Lipschitz maps $\R^m\to\R^{n-m}$.  As we will  prove later in  Corollary~\ref{cor:HrectifiableC1vsLip}, a similar statement holds in Heisenberg groups: namely, $R\subset\H^n$ is $\H$-rectifiable of codimension $k\in\{1,\dots,n\}$ if and only if $\Shaus^{Q-k}(R)<\infty$ and there exists a countable family $(\f_j)_{j}$ of intrinsic Lipschitz maps $\f_j:\W_j\to\V_j$, where $\W_j,\V_j $ are homogeneous complementary subgroups of $\H^n$ with $\dim \V_j=k$, such that
\[
\Shaus^{Q-k}\left(R\setminus\bigcup_j \gr_{\f_j}\right)=0.
\]

\begin{definition}\label{def:normaletangenteHrettif}
If $R\subset\H^n$ is locally $\H$-rectifiable and $(S_j)_j$ is a family of $\H$-regular submanifolds as in Definition~\ref{def:Hrettificabili}, we define the {\em horizontal normal} $n^\H_R(p)\in\bwl_k\h_1$ at $p\in R$ by
\[
n^\H_R(p):=n^\H_{S_j}(p)\qquad\text{if }p\in R\cap S_j.
\]
Accordingly, we set $t^\H_R(p):=*n_R^\H(p)\in\bwl_{2n+1-k}\h$ and $\Tan^\H_R(p):=\exp(\Span(t^\H_R(p))$. Eventually, we define $\tau^\H_R(p)\in\bwl_{2n-k}\h_1$ by requiring that $t^\H_R(p)=\tau^\H_R(p)\wedge T$. 
\end{definition}

The objects introduced in Definition~\ref{def:normaletangenteHrettif}  are well-defined $\Shaus^{Q-k}$-almost everywhere on $R$ (as usual, up to a sign) because of the following well-known lemma, whose proof we sketch for the sake of completeness.

\begin{lemma}
Let $S_1,S_2\subset\H^n$ be $\H$-regular submanifolds of codimension $k\in\unoenne$; then
\[
\Shaus^{Q-k}(\{p\in S_1\cap S_2:n^\H_{S_1}(p)\notin\{\pm n^\H_{S_2}(p)\}\})=0.
\] 
\end{lemma}
\begin{proof}
Let $E:=\{p\in S_1\cap S_2:n^\H_{S_1}(p)\notin\{\pm n^\H_{S_2}(p)\}\}$. For every $p\in E$ 
we have
\[
\limsup_{r\to 0^+}\de_{1/r}(p^{-1}E)\subset \limsup_{r\to 0^+}\de_{1/r}(p^{-1}(S_1\cap S_2))\subset \Tan^\H_{S_1}(p)\cap \Tan^\H_{S_2}(p),
\]
where the $\limsup$  are taken with respect to the local Hausdorff topology. Since the right-hand side is a vertical plane of dimension at most $2n-k$, the statement now follows from~\cite[Lemma B.3]{MVAPDE}.
\end{proof}

\subsection{Currents induced by \texorpdfstring{$C^1$}{C1} intrinsic graphs}\label{subsec:correntiindottedagraficilipschitz}
We now want to study currents induced by $C^1$ regular intrinsic graphs of codimension $k$ in $\H^n$ with $1\leq k\leq n$. Assume  that a $C^1$ map $\f:\W\to\V$ is fixed; we introduce the family $\nf=(\nf_{k+1},\dots,\nf_{2n})$ of vector fields on $\W$ defined by
\begin{equation}\label{eq:nablaphi_i}
\nf_i:=\begin{cases}
X_i & \text{if }k+1\leq i\leq n\\
Y_{i-n}+\f_{i-n}T=\pa_{y_{i-n}}+\f_{i-n}\pa_t\quad& \text{if }n+1\leq i\leq n+k\\
Y_{i-n}& \text{if }n+k+1\leq i\leq 2n.
\end{cases}
\end{equation}
The vectors $\nf_i$ are tangent to $\W$ because so are $X_{k+1},\dots,Y_n,T$. The  family $\nf$ was introduced in~\cite{ASCV} in the case of codimension 1 and in~\cite{Corni,SCSomeTopics} for codimension $k\leq n$. 

\begin{remark}\label{rem:nablaphiphi=sestesso}
The notation introduced in~\eqref{eq:nablaphi_i} is consistent with the one  in Definition~\ref{def:differenziabilita}: in fact (see e.g.~\cite[Proposition~3.7]{Corni}), when $\f$ is of class $C^1$ the components of the matrix $\nf\f(w)$ associated with the intrinsic differential $d\f_w$ are precisely the derivatives $\nf_i\f_j(w)$ of  $\f=(\f_1,\dots,\f_k)$  along the directions $\nf_i$.
\end{remark}

Recalling the notation $W_i$ and $\F$ introduced in \eqref{eq:defW} and \eqref{eq:graficoincoordinate}, one can differentiate the graph map $\F$ along the directions $\nf_i$ to obtain that,
for every $w\in\W$, the vectors 
\begin{equation}\label{eq:defcampizeta}
\nf_i\F(w)=\left(W_i+\sum_{h=1}^k\nf_i\f_h(w)X_h\right)(\F(w)),\qquad i=k+1,\dots,2n
\end{equation}
(which should be thought of as vectors in $\F(w)$ that are continuous with respect to $w$) are horizontal and tangent to the submanifold $S:=\gr_\f$ at the point $\F(w)$. The  equality in~\eqref{eq:defcampizeta} comes from a boring computation that we omit. Since the vectors $\nf_i\F(w)$, $i=k+1,\dots,2n$, are also linearly independent, they generate the $(2n-k)$-dimensional subspace $T_{\F(w)}S\cap\h_1$, hence the multi-vector 
\begin{equation}\label{eq:defZETA}
\nf\F(w):=\nf_{k+1}\F(w)\wedge\dots\wedge\nf_{2n}\F(w)\in\bwl_{2n-k}\h_1
\end{equation}
is a multiple of the unit multi-vector $\tau_{S}^\H(\F(w))$  defined in \S\ref{subsec:correnti}. 

\begin{remark}\label{rem:Jacob=norm}
It is worth noticing that the intrinsic Jacobian determinant $J^\f\f(w)$ equals the norm $|\nf\F(w)|$ of the multi-vector $\nf\F(w)$. As usual, the norm on multi-vectors is the one induced by the left-invariant scalar product making $X_1,\dots,Y_n,T$ orthonormal.
\end{remark}

We state for future references the following result, which is essentially a restatement of Theorem~\ref{teo:approssimiamoooooooo}.

\begin{proposition}\label{prop:approssimazionelisciaunifLipHeisenberg}
Let $A\subset\W$ and $\f:A\to\V$ be intrinsic Lipschitz. Then there exists a sequence $(\f_i)_{i\in\N}$  of $C^\infty$ smooth and uniformly intrinsic Lipschitz maps $\f_i:\W\to\V$ such that
\[
\f_i\to\f\text{ uniformly in }A\text{ as }i\to\infty\,.
\]
Moreover, there exists $C>0$, depending only on the intrinsic Lipschitz constant of $\f$ and the distance $d$, such that
\[
|\nabla^{\f_i}\F_i(w)|\leq C \qquad\text{for every $i\in\N$ and }w\in\W,
\] 
where $\F_i$ is the graph map $\W\ni w\mapsto w\f_i(w)\in\H^n$.
\end{proposition}
\begin{proof}
The first part of the statement is  Theorem~\ref{teo:approssimiamoooooooo}, while the second one is a consequence of Lemma~\ref{lem:nablaffijminoredialfa}, Remark~\ref{rem:nablaphiphi=sestesso} and~\eqref{eq:defcampizeta}.
\end{proof}

Similarly as in~\eqref{eq:defcampizeta} one gets
\begin{equation}\label{eq:TPhi}
T\F(w)=\left(T+\sum_{h=1}^k T\f_h(w)X_h\right)(\F(w)).
\end{equation}
The vector fields in~\eqref{eq:defcampizeta} and~\eqref{eq:TPhi} generate the tangent space to $S$ at $\F(w)$. We then fix  the orientation of $S$ is such a way that 
\begin{equation}\label{eq:comeorientareungrafico}
t_S(\F(w)):=\frac{\nf\F\wedge T\F}{|\nf\F\wedge T\F|}\quad\text{is positively oriented for every $w\in\W$.}
\end{equation}
Observe that $\langle t_S,W_{k+1}\wedge\dots\wedge W_{2n}\wedge T\rangle\neq 0$ on $S$. Actually, our choice of the orientation for $S$ corresponds to declaring that a unit tangent vector $t_S$ is positively oriented if and only if $\langle t_S,W_{k+1}\wedge\dots\wedge W_{2n}\wedge T\rangle>0$ on $S$.

Recalling the notation introduced in~\eqref{eq:tauSV}, we deduce from~\eqref{eq:defcampizeta},~\eqref{eq:TPhi} and~\eqref{eq:comeorientareungrafico} that 
\[
\eta_S(\F(w))=\frac{\nf\F(w)}{|\nf\F(w)\wedge T\F(w)|},
\]
which implies $\tau_{S}^\H(\F(w))={\nf\F(w)}/{|\nf\F(w)|}$ and, eventually,
\begin{equation}\label{eq:tSHgrafici}
t_{S}^\H(\F(w))=\frac{\nf\F(w)}{|\nf\F(w)|}\wedge T.
\end{equation}

The multi-vector $\nf\F$ also allows to characterize the Heisenberg current associated with a $C^1$ intrinsic graph. The following lemma is an important tool used in the proof of Theorem~\ref{thm:rademacher}.

\begin{lemma}\label{lem:correntevistaintegrandosuW}
Let $\f:\W\to\V$ be a $C^1$ map and let the graph $S:=\gr_\f$ be oriented as in~\eqref{eq:comeorientareungrafico}. Then
\[
\curr S(\omega)=C_{n,k}\int_\W\langle \nf\F(w)\wedge T\mid \omega(\F(w))\rangle\:d\leb{2n+1-k}(w)\qquad\text{for every }\omega\in\DH^{2n+1-k},
\]
where the constant $C_{n,k}>0$ is the one provided by Proposition~\ref{prop:correnti_cl_vs_H_INTRO}.
\end{lemma}
\begin{proof}
We have for every $\omega\in\DH^{2n+1-k}$
\begin{align*}
\curr S(\omega)
&=C_{n,k}\int_S\langle  [t^\H_S(p)]_\mathcal J\mid \omega(p)\rangle\:d\Shaus^{Q-k}(p)\\
&=C_{n,k} \int_\W \langle  t^\H_S(\F(w))\mid \omega(\F(w))\rangle\:J^\f\f(w)\:d\leb{2n+1-k}(w)\\
&=C_{n,k} \int_\W \left\langle\left. \frac{\nf\F(w)}{|\nf\F(w)|}\wedge T\right| \omega(\F(w))\right\rangle\:J^\f\f(w)\:d\leb{2n+1-k}(w)\\
&=C_{n,k}\int_\W\langle \nf\F(w)\wedge T\mid \omega(\F(w))\rangle\:d\leb{2n+1-k}(w),
\end{align*}
where the first equality comes from Theorem~\ref{teo:Hreg_grafici_area} and a change of variable, the second one is justified by~\eqref{eq:tSHgrafici} and the last one by Remark~\ref{rem:Jacob=norm}.
\end{proof}



\section{The Constancy Theorem for Heisenberg currents}\label{sec:constancytheorem}
The classical Constancy Theorem (see e.g.~\cite[4.17]{federer} or~\cite[Theorem~26.27]{Simon}) states that, if $\Tcurr$ is an  $n$-dimensional current  in a connected open set $U\subset\R^n$ such that $\partial\Tcurr=0$, then $\Tcurr$ is constant, i.e., there exists $c\in\R$ such that $\Tcurr(\omega)=c\int_U\omega$ for every smooth $n$-form $\omega$ with compact support in $U$. 
The Constancy Theorem can be generalized (see e.g.~\cite[Proposition 7.3.5]{KrantzParks}) to currents supported on an $m$-dimensional plane $\mathscr P\subset\R^n$: if $\Tcurr$ is an $m$-current with support in $\mathscr P$ and such that $\partial \Tcurr=0$, then there exists $c\in\R$ such that $\Tcurr(\omega)=c\int_\mathscr P\omega$ for every smooth $m$-form $\omega$ with compact support.  

As mentioned in the Introduction,
the version for planes of the Constancy Theorem implies the following fact (see~\cite[Theorem~4.2]{Silhavy}): if $R\subset\R^n$ is an $m$-rectifiable set and  $\Tcurr=\tau\mu$ is a  normal $m$-current, where $\mu$ is a Radon measure and $\tau$ is a locally $\mu$-integrable $m$-vectorfield with $\tau\neq0$ $\mu$-a.e., then 
\begin{enumerate}
\item[(i)] $\mu\res R$ is absolutely continuous with respect to the Hausdorff measure $\Haus^m\res R$, and 
\item[(ii)] $\tau$ is tangent to $R$ at $\mu$-almost every point of $R$.
\end{enumerate}
See also~\cite[4.1.31]{federer} and~\cite[Example 10 at page~146]{GiaqModSouc_I} for simpler cases, and \cite[\S5]{AlbertiMarchese} and ~\cite{AlbMassStep} for  similar-in-spirit results.

A similar program is developed in the present section for currents  in Heisenberg groups. First, in \S\ref{subsec:TangencyThmPiani} we  prove Proposition~\ref{prop:correntipiane}, where a partial version of the Constancy Theorem~\ref{thm:Constancy} is proved.  Proposition~\ref{prop:correntipiane} can also be seen as a particular case of Theorem~\ref{teo:correntinormalirettificabili}, but actually the proof of Theorem~\ref{teo:correntinormalirettificabili} follows from Proposition~\ref{prop:correntipiane} by a blow-up argument. 
The proof of Theorem~\ref{teo:correntinormalirettificabili} is developed in Section~\ref{subsec:corr_norm_rettif}: observe that we are not able to prove any ``absolute continuity'' statement analogous to (i) above, but only a ``tangency'' statement corresponding to (ii). As we said in the Introduction, this is due to the absence of a good notion of projection on planes. 
The proof of the Constancy Theorem~\ref{thm:Constancy} for  Heisenberg currents without boundary and supported on vertical planes  is contained in Section~\ref{subsec:ConstThmPlanes}. Observe that here we are  able to prove not only the ``tangency'' property (ii), but also the ``absolute continuity'' one (i): in fact, by using group convolutions on vertical planes one can  reduce to the case in which $\|\Tcurr\|$ is already absolutely continuous.

\subsection{A partial version of Theorem~\ref{thm:Constancy}}\label{subsec:TangencyThmPiani}
For the reader's convenience we recall  some preliminary notation. First, given  a Radon measure $\mu$ on $\H^n$ and a locally $\mu$-integrable function $\tau:\H^n\to\mathcal J_{2n+1-k}$, we denote by $\tau\mu$  the Heisenberg $(2n+1-k)$-current defined by
\[
\tau\mu(\omega):=\int_{\H^n} \langle\tau(p)\mid \omega(p)\rangle\:d\mu(p),\qquad \omega\in \DH^{2n+1-k}.
\]

Given natural numbers $a,b$ such that $1\leq a+b\leq n$, the $(2a+b+1)$-dimensional vertical plane $\Pab$ was introduced in~\eqref{eq:defPab} as
\begin{equation}\label{eq:defPab222}
\begin{aligned}
\Pab:=\:& \{(x,y,t)\in\H^n:x_i=y_j=0\text{ for all }a+b+1\leq i\leq n \text{ and }a+1\leq j\leq n\}
\\
=\:&
\begin{cases}
\{(x_1,\dots,x_{a+b},0,\dots,0,y_1,\dots,y_a,0,\dots,0,t)\}\quad & \text{if }a\geq 1\text{ and } a+b<n\\
\{(x_1,\dots,x_n,y_1,\dots,y_a,0,\dots,0,t)\}\quad & \text{if }a\geq 1\text{ and } a+b=n\\
\{(x_1,\dots,x_{b},0,\dots,0,t)\}\quad & \text{if }a=0.
\end{cases}
\end{aligned}
\end{equation}
We are interested in the case in which the codimension $2n-2a-b$ of $\Pab$ equals a given $k\in\unoenne$, hence we also assume that $n\leq 2a+b\leq 2n$. Observe that,  if $a=0$, then necessarily  $b=n$. 
A unit tangent vector to $\Pab$ is
\[
t^\H_\Pab=X_1\wedge\dots\wedge X_{a+b}\wedge Y_1\wedge\dots\wedge Y_a\wedge T
\]
and we agree that, when $a=0$, this expression has to be read as $X_1\wedge\dots\wedge X_n\wedge T$. We also notice that \eqref{eq:defPab222} induces a natural identification between the subgroup $\Pab$ and $\R^{2a+b+1}$ according to which the measures $\Shaus^{2a+b+2}$ and $\leb{2a+b+1}$, being Haar measures\footnote{The measure $\leb{2a+b+1}$ is a Haar one because $\Pab$ is canonically isomorphic to $\H^a\times\R^b$.} on $\Pab$, coincide up to a multiplicative constant.

The following Lemma~\ref{lem:vettoretangente,caso_n} is stated in the setting of maximal codimension $k=n$; it however holds also for $1\leq k\leq n-1$, as showed later in Lemma  \ref{lem:vettoretangente,casogenerale}. We  use the following notation: given $i,j\in\{1,\dots,n\}$, by $O(x_i),O(x_i^2),O(x_ix_j)$ we denote smooth differential forms that can be written, respectively, as $x_i\al,x_i^2\be,x_ix_j\ga$ for suitable smooth differential forms $\al,\be,\ga$. When applying exterior differentiation we will freely use straightforward formulae like $d(O(x_i^2))=O(x_i)$, $d(O(x_ix_j))=O(x_i)+O(x_j)$, and similar ones. Eventually, the equivalence class $[\,\cdot\,]_\mathcal J$ is as in Remark~\ref{rem:introducoJ_k}.

\begin{lemma}\label{lem:vettoretangente,caso_n}
Let $a,b$ be natural numbers such that $2a+b=n$ and let $\Pab$ be the $(n+1)$-plane defined in \eqref{eq:defPab222}. Assume there exists  $\tau\in\mathcal J_{n+1}$ such that the Heisenberg $(n+1)$-current $\Tcurr:=\tau\Shaus^{n+2}\res\Pab$ is such that $\partial\Tcurr=0$. Then  there exists $\eta\in\R$ such that
\[
\tau=\eta\:[ X_1\wedge\dots\wedge X_{a+b}\wedge Y_1\wedge\dots\wedge Y_a\wedge T]_\mathcal J.
\]
In particular, $\tau$ is a multiple of $[t^\H_\Pab]_\mathcal J$.
\end{lemma}
\begin{proof}
We assume that $a\geq1$; the case $a=0$ requires only simple modifications at the level of notation.

We have to prove that there exists $\eta\in\R$ such that, for every $\la\in\mathcal J^{n+1}$,
\begin{equation}\label{eq:tauvstangente}
\langle\tau\mid \la\rangle=\eta\,\langle X_1\wedge\dots\wedge X_{a+b}\wedge Y_1\wedge\dots\wedge Y_a\wedge T\mid \la\rangle.
\end{equation}
Clearly, it is enough to check~\eqref{eq:tauvstangente} for $\la$ ranging in the basis of $\mathcal J^{n+1}$ provided by Proposition~\ref{prop:baseJintro} (with $k:=n$).
Fix then $I,J\subset \unoenne$ and a standard Young tableau $R$ such that
\begin{itemize}
\item $I\cap J=\emptyset$ and $|I|+|J|\leq n$
\item $R$ is a $(2\times\frac{n-|I|-|J|}{2})$-rectangular (recall Remark~\ref{rem:Rrettangolare}) tableau  that contains the  integers in the set $\unoenne\setminus(I\cup J)$.
\end{itemize}
Testing $X_1\wedge\dots\wedge X_{a+b}\wedge Y_1\wedge\dots\wedge Y_a\wedge T$ against $\la=dx_I\wedge dy_J\wedge\al_R\wedge\theta$ one  realizes that it is enough to show that
\begin{equation}\label{eq:stampelle}
\begin{aligned}
 \text{either}\quad&I=\{a+1,\dots,a+b\},\  J=\emptyset\text{ and } R=
\begin{tabular}{|c|c|c|c|}
\cline{1-4}
$1$ & $2$ &$\ \cdots\ $& $a$\\
\cline{1-4}
$a+b+1$ & $a+b+2^{\phantom{2}}$ &$\ \cdots\ $& $n$
\\
\cline{1-4}
\end{tabular}\\
 \text{or}\quad&\langle \tau\mid  dx_I\wedge dy_J\wedge\al_R\wedge\theta\rangle=0.
\end{aligned}
\end{equation}

We are going to prove \eqref{eq:stampelle} first under some additional assumptions on $I,J$ or $R$, see the following Claims 1, 2 and 3. We fix an auxiliary function $\psi\in C^\infty_c(\H^n)$ such that 
\[
\int_\Pab \psi\:d\Shaus^{n+2}=1.
\]

\medskip

{\em Claim 1: if $I\cap\{a+b+1,\dots,n\}\neq\emptyset$, then $\langle \tau\mid  dx_I\wedge dy_J\wedge\al_R\wedge\theta\rangle=0$.}

\noindent Fix  $\bar\imath\in I\cap\{a+b+1,\dots,n\}$ and define
\[
\omega(x,y,t):=x_{\bar\imath}^2 \;\psi(x,y,t)\: dx_{I\setminus\{\bar\imath\}}\wedge dy_{J\cup\{\bar\imath\}}\wedge\al_R.
\]
By identifying $\omega\in C^\infty_c(\H^n,\bwl^n \h_1)$ with its equivalence class $[\omega]$ according to the quotient in the right-hand side of~\eqref{eq:equivbis}, we have $\omega\in\DH^{n}$.  
One easily computes
\[
d\omega=\pm 2 x_{\bar\imath} \;\psi(x,y,t) \: dxy_{\bar\imath}\wedge dx_{I\setminus\{\bar\imath\}}\wedge dy_{J}\wedge\al_R + O(x_{\bar\imath}^2)
\]
where the sign $\pm$ depends only on $\bar\imath$ and $I$ and, in the sequel, it can change from line to line. We obtain
\[
L^{-1}((d\omega)_{\h_1} )=\pm 2 x_{\bar\imath} \;\psi(x,y,t) \:  dx_{I\setminus\{\bar\imath\}}\wedge dy_{J}\wedge\al_R + O(x_{\bar\imath}^2)
\]
and in turn
\begin{align*}
D\omega & =  d\big(\omega+(-1)^n L^{-1}((d\omega)_{\h_1})\wedge\theta\big)\\
& =  \pm 2\;\psi(x,y,t) dx_I\wedge dy_J\wedge\al_R\wedge\theta + O(x_{\bar\imath}).
\end{align*}
Since $O(x_{\bar\imath})=0$ on $\Pab$  we obtain
\begin{align*}
0&= \partial\Tcurr (\omega) 
 = \Tcurr(D\omega) 
 = \int_\Pab\langle\tau \mid D\omega\rangle d\Shaus^{n+2}\\
&= \pm 2 \int_{\Pab} \psi\langle\tau \mid   dx_I\wedge dy_J\wedge\al_R\wedge\theta\rangle d\Shaus^{n+2}\\
&=\pm 2\langle\tau \mid   dx_I\wedge dy_J\wedge\al_R\wedge\theta\rangle.
\end{align*}
and Claim 1 follows.\medskip

{\em Claim 2: if $J\cap\{a+1,\dots,n\}\neq\emptyset$, then $\langle \tau\mid  dx_I\wedge dy_J\wedge\al_R\wedge\theta\rangle=0$.}

\noindent The proof is analogous to that of Claim 1. Fix  $\bar\jmath\in J\cap\{a+1,\dots,n\}$ and define $\omega\in\DH^n$ by
\[
\omega(x,y,t):=y_{\bar\jmath}^2 \;\psi(x,y,t)\: dx_{I\cup\{\bar\jmath\}}\wedge dy_{J\setminus\{\bar\jmath\}}\wedge\al_R.
\]
A computation similar to the one in Claim 1 gives
\[
D\omega  =  \pm 2 \;\psi(x,y,t) dx_I\wedge dy_J\wedge\al_R\wedge\theta + O(y_{\bar\jmath})
\]
and again
\[
0=
\partial\Tcurr (\omega)
=
\int_\Pab\langle\tau \mid D\omega\rangle d\Shaus^{n+2} 
=
\pm 2\langle\tau \mid   dx_I\wedge dy_J\wedge\al_R\wedge\theta\rangle
\]
allows to conclude.\medskip

{\em Claim 3: if the first row of $R$ contains an element $\bar\imath$ such that $\bar\imath\geq a+1$, then $\langle \tau\mid  dx_I\wedge dy_J\wedge\al_R\wedge\theta\rangle=0$. }

\noindent We observe that, since $R$ is a standard Young tableau, the assumption of Claim 3 is equivalent to $R$ containing a column made by two elements $\bar\imath,\bar\jmath$ such that $a+1\leq\bar\imath<\bar\jmath\leq n$. Define $\omega\in\DH^n$ by
\[
\omega(x,y,t):=y_{\bar\imath}y_{\bar\jmath} \;\psi(x,y,t)\: dx_{I\cup\{\bar\imath,\bar\jmath\}}\wedge dy_{J}\wedge\al_Q,
\]
where $Q$ is the standard Young tableau obtained by removing from $R$ the column containing $\bar\imath,\bar\jmath$ (for instance, $Q$ is the empty tableau if $R$ consists of the column $\bar\imath,\bar\jmath$ only).  One has
\begin{align*}
d\omega =\: & \psi(x,y,t)( y_{\bar\jmath}\,dy_{\bar\imath} + y_{\bar\imath}\,dy_{\bar\jmath}) \wedge dx_{I\cup\{\bar\imath,\bar\jmath\}}\wedge dy_{J}\wedge\al_Q + O(y_{\bar\imath}y_{\bar\jmath})\\
=\: & \psi(x,y,t)
\big[(-1)^{c+1} y_{\bar\jmath}\,dxy_{\bar\imath}\wedge dx_{I\cup\{\bar\jmath\}} + (-1)^{d+2} y_{\bar\imath}\,dxy_{\bar\jmath}\wedge dx_{I\cup\{\bar\imath\}}\big] \wedge dy_{J}\wedge\al_Q
+ O(y_{\bar\imath}y_{\bar\jmath})
\end{align*}
where
\begin{align*}
& c:=|\{i\in I:i<\bar\imath\}|\\
&d:=|\{i\in I:i<\bar\jmath\}|=|\{i\in I\cup\{\bar\imath\}:i<\bar\jmath\}|-1.
\end{align*}
Then
\[
L^{-1}((d\omega)_{\h_1} )=-
\psi(x,y,t)\big[(-1)^{c+1} y_{\bar\jmath}\,dx_{I\cup\{\bar\jmath\}} + (-1)^{d} y_{\bar\imath}\, dx_{I\cup\{\bar\imath\}}\big] \wedge dy_{J}\wedge\al_Q + O(y_{\bar\imath}y_{\bar\jmath})
\]
and in turn
\begin{align*}
D\omega =\: &   d\big(\omega+(-1)^n L^{-1}((d\omega)_{\h_1})\wedge\theta\big)\\
=\: &  (-1)^{n+1} \;\psi(x,y,t)\big[ (-1)^{c+1} dy_{\bar\jmath}\wedge dx_{I\cup\{\bar\jmath\}} + (-1)^{d} dy_{\bar\imath}\wedge dx_{I\cup\{\bar\imath\}}\big]\wedge dy_{J}\wedge
\al_Q \wedge\theta\\
& + O(y_{\bar\imath})+ O(y_{\bar\jmath})\\
=\: &  (-1)^{n+1} \;\psi(x,y,t)\big[ (-1)^{c+d+2} dxy_{\bar\jmath}\wedge dx_{I} + (-1)^{c+d+1} dxy_{\bar\imath}\wedge dx_{I}\big]\wedge dy_{J}\wedge
\al_Q \wedge\theta\\
& + O(y_{\bar\imath})+ O(y_{\bar\jmath})\\
=\: &  \pm  \;\psi(x,y,t)( dxy_{\bar\jmath}- dxy_{\bar\imath})\wedge dx_{I}\wedge dy_{J}\wedge \al_Q\wedge\theta  + O(y_{\bar\imath})+ O(y_{\bar\jmath})\\
=\: &  \mp  \;\psi(x,y,t) dx_{I}\wedge dy_{J}\wedge \al_R\wedge\theta  + O(x_{\bar\imath})+ O(x_{\bar\jmath}).
\end{align*}
As before we obtain
\[
0=
\partial\Tcurr (\omega)
=
\int_\Pab\langle\tau \mid D\omega\rangle d\Shaus^{n+2} 
=
\mp\langle\tau \mid   dx_I\wedge dy_J\wedge\al_R\wedge\theta\rangle
\]
and the claim is proved.\medskip

{\em Claim 4: \eqref{eq:stampelle} holds.}

\noindent We already  know that \eqref{eq:stampelle} holds in case $I,J,R$ satisfy any one of the assumptions in Claims 1, 2, 3. We then assume that none of such assumptions hold, i.e., that
\begin{equation}\label{eq:perstampelle}
\begin{aligned}
& I\subset\{1,\dots,a+b\}\\
\text{and }& J\subset\{1,\dots,a\}\\
\text{and }& \text{the elements in the first row of $R$ are not greater than $a$},
\end{aligned}
\end{equation}
and we prove that one necessarily has
\begin{equation}\label{eq:medusa}
I=\{a+1,\dots,a+b\},\quad  J=\emptyset\quad\text{and}\quad R=
\begin{tabular}{|c|c|c|c|}
\cline{1-4}
$1$ & $2$ &$\ \cdots\ $& $a$\\
\cline{1-4}
$a+b+1$ & $a+b+2^{\phantom{2}}$ &$\ \cdots\ $& $n$
\\
\cline{1-4}
\end{tabular}\,.
\end{equation}
This would be enough to conclude. 

The first two conditions in \eqref{eq:perstampelle} imply that all the $n-a-b=a$ integers $a+b+1,\dots,n$ appear in $R$. They all belong to the second row of $R$ by the third condition in \eqref{eq:perstampelle}, hence $R$ has at least $a$ columns. Therefore the first row of $R$ contains at least $a$ elements, all of them not greater than $a$: it follows that the first row of $R$ contains precisely $1,\dots,a$ (displayed in this order), and that in turn  the second row of $R$ contains precisely $a+b+1,\dots,n$ (displayed in this order). In particular, $R$ is the one displayed in \eqref{eq:medusa}. The remaining integers $a+1,\dots,a+b$, not appearing in $R$, have to belong to either $I$ or $J$; the second condition in \eqref{eq:perstampelle} implies that they all belong to $I$, and the proof is concluded.  
\end{proof}

Lemma~\ref{lem:vettoretangente,caso_n} holds also for non-maximal codimension $k<n$, as we now prove. The reader will easily notice the similarity between the two proofs, the main difference lying in the use of the standard exterior differentiation $d$ in place of Rumin's operator $D$.

\begin{lemma}\label{lem:vettoretangente,casogenerale}
Let $a,b$ be natural numbers such that $1\leq a+b\leq n$ and $n+1\leq 2a+b\leq 2n$; let $\Pab$ be the plane defined in \eqref{eq:defPab222}. 
 Assume there exists  $\tau\in\mathcal J_{2a+b+1}$ such that the Heisenberg $(2a+b+1)$-current $\Tcurr:=\tau\Shaus^{2a+b+2}\res\Pab$ is such that $\partial\Tcurr=0$. Then  there exists $\eta\in\R$ such that
\[
\tau=\eta\:[ X_1\wedge\dots\wedge X_{a+b}\wedge Y_1\wedge\dots\wedge Y_a\wedge T]_\mathcal J.
\]
In particular, $\tau$ is a multiple of $[t^\H_\Pab]_\mathcal J$.
\end{lemma}
\begin{proof}
Observe that  necessarily $a\geq 1$. We assume that also $b\geq 1$ and omit the simple modifications one has to perform in order to treat the case $b=0$.
 
As in Lemma~\ref{lem:vettoretangente,caso_n} we have to prove that~\eqref{eq:tauvstangente} holds for every $\la$ in the basis of $\mathcal J^{2a+b+1}$ provided by Proposition~\ref{prop:baseJintro} (with $k:=2n-2a-b$). To this aim, fix $I,J\subset \unoenne$ and a standard Young tableau $R$ such that
\begin{itemize}
\item $I\cap J=\emptyset$ and $|I|+|J|\leq2n-2a-b$
\item $R$ contains the  integers in the set $\unoenne\setminus(I\cup J)$;
\item the first row of $R$ has length $\ell:=(2a+b-|I|-|J|)/2$ and the second one has length $(2n-2a-b-|I|-|J|)/2$. 
\end{itemize}
Let us observe that the lengths of the two rows of $R$ are  never equal; actually, the difference between these lengths is fixed and equal to $2a+b-n\geq 1$, see also Remark~\ref{rem:Rrettangolare}. In particular, the rightmost element in the first row of $R$ belongs to a column of height one (there is no element in the second row ``below it'').

The proof will be accomplished if we prove that 
\begin{align}
\text{either}\quad 
&\begin{cases}
I=\{a+1,\dots,a+b\}\\
J=\emptyset\\
R= \begin{tabular}{|c|c|c|c|ccc}
\cline{1-7}
$1$ & $2$ &$\ \cdots\ $& $n-a-b$ &\multicolumn{1}{c|}{$n-a-b+1$} &\multicolumn{1}{c|}{$ \cdots$} &   \multicolumn{1}{c|}{$a$}
\\
\cline{1-7}
$a+b+1$ & $a+b+2$ &$\ \cdots\ $& $n$
\\
\cline{1-4}
\end{tabular}\vspace{.2cm}
\end{cases}
\label{eq:stampellebis1}\\
\text{or}\quad &\langle\tau\mid  dx_I\wedge dy_J\wedge\al_R\wedge\theta\rangle=0.\label{eq:stampellebis2}
\end{align}
In the following Claims 1, 2 and 3  we prove that \eqref{eq:stampellebis2} holds under some additional assumptions on $I,J$ or $R$; the argument will be later completed in Claim 4. We  fix again an auxiliary function $\psi\in C^\infty_c(\H^n)$ such that 
\[
\int_\Pab \psi\:d\Shaus^{2a+b+2}=1.
\]
\medskip

{\em Claim 1: if $I\cap\{a+b+1,\dots,n\}\neq\emptyset$, then \eqref{eq:stampellebis2} holds.}

\noindent Fix $\bar\imath\in I\cap\{a+b+1,\dots,n\}$ and define   $\omega\in\DH^{2a+b}$ by
\[
\omega(x,y,t):=x_{\bar\imath}\:\psi(x,y,t)\:dx_{I\setminus\{\bar\imath\}}\wedge dy_J\wedge\al_Q\wedge\theta,
\]
where $Q$ is the Young tableau whose first row is equal to that of $R$ and whose second row is made by the second row of $R$ with the addition of the extra element $\bar\imath$ in the rightmost position. Namely, denoting by $\ell$ and $r$, with $\ell>r$, the lengths of the first and second row of $R$, respectively, we have
\[
Q=\begin{tabular}{|c|c|c|c|cc}
\cline{1-6}
$R^1_1$ & $\ \cdots\ $ &$R^1_r$& $R^1_{r+1}$  &\multicolumn{1}{c|}{$ \cdots$} &   \multicolumn{1}{c|}{$R^1_\ell$}
\\
\cline{1-6}
$R^2_1$ & $\ \cdots\ $& $R^2_r$ &$\bar\imath$
\\
\cline{1-4}
\end{tabular}\,.
\]
The Young tableau $Q$ is not necessarily a standard  one, nonetheless $\omega\in\DH^{2a+b}$ by Remark~\ref{rem:tuttetabelleok}. One can compute
\begin{align*}
d\omega &= \pm\:\psi(x,y,t)\:dx_{I}\wedge dy_J\wedge\al_Q\wedge\theta + O(x_{\bar\imath})\\
& = \pm\:\psi(x,y,t)\:dx_{I}\wedge dy_J\wedge\al_R\wedge\theta + O(x_{\bar\imath}),
\end{align*}
where the second equality is justified by the fact that $dx_I\wedge \al_Q$ contains a factor $dx_{\bar\imath}\wedge(dxy_h- dxy_{\bar\imath})=dx_{\bar\imath}\wedge dxy_h$ for a suitable $h$ (namely, $h=R^1_{r+1}$) appearing in the first row of $R$.  
Since $O(x_{\bar\imath})=0$ on $\Pab$, one gets
\begin{align*}
0 & =
 \partial\Tcurr (\omega) = \Tcurr(d\omega) = \int_\Pab\langle\tau \mid d\omega\rangle d\Shaus^{2a+b+2}
=\pm \langle\tau \mid  dx_I\wedge dy_J\wedge\al_R\wedge\theta\rangle.
\end{align*}
and Claim 1 follows.\medskip

{\em Claim 2: if $J\cap\{a+1,\dots,n\}\neq\emptyset$, then \eqref{eq:stampellebis2} holds.}

\noindent The proof is similar to that of Claim 1. Fix $\bar\jmath\in J\cap\{a+1,\dots,n\}$ and define $\omega\in\DH^{2a+b}$ as 
\[
\omega(x,y,t):=y_{\bar\jmath}\:\psi(x,y,t)\:dx_{I}\wedge dy_{J\setminus\{\bar\jmath\}}\wedge\al_Q\wedge\theta,
\]
where $Q$ is the Young tableau whose first row is equal to that of $R$ and whose second row is made by the second row of $R$ with the addition of the extra element $\bar\jmath$ in the rightmost position. Again,    $\omega\in\DH^{2a+b}$ because of Remark~\ref{rem:tuttetabelleok}. One can compute
\begin{align*}
d\omega 
&= \pm\:\psi(x,y,t)\:dx_{I}\wedge dy_J\wedge\al_Q\wedge\theta + O(y_{\bar\jmath}) \\
&= \pm\:\psi(x,y,t)\:dx_{I}\wedge dy_J\wedge\al_R\wedge\theta + O(y_{\bar\jmath}),
\end{align*}
where, as before, the second equality is justified by the fact that $dy_J\wedge \al_Q$ contains a factor $dy_{\bar\jmath}\wedge(dxy_h- dxy_{\bar\jmath})=dy_{\bar\jmath}\wedge dxy_h$ for a suitable $h$ appearing in the first row of $R$.  
We deduce that
\begin{align*}
0 & =
 \partial\Tcurr (\omega) = \Tcurr(d\omega) = \int_\Pab\langle\tau \mid d\omega\rangle d\Shaus^{2a+b+2}
=\pm \langle\tau \mid  dx_I\wedge dy_J\wedge\al_R\wedge\theta\rangle.
\end{align*}
and Claim 2 follows.\medskip

{\em Claim 3: if the first row of $R$ contains an element $\bar\jmath$ such that $\bar\jmath\geq a+1$, then \eqref{eq:stampellebis2} holds.}

\noindent Since $R$ is a standard Young tableau, also the rightmost element in the first row of $R$   is not smaller than $a+1$; we can then assume that $\bar\jmath$ is precisely this element. As already noticed, there is no element in the second row of $R$ ``below'' $\bar\jmath$. Consider the form $\omega\in\DH^{2a+b}$  defined by
\[
\omega(x,y,t):=y_{\bar\jmath}\:\psi(x,y,t)\:dx_{I\cup\{\bar\jmath\}}\wedge dy_J\wedge \alpha_Q\wedge \theta,
\]
where $Q$ is the (possibly empty)  tableau obtained from $R$ by removing the rightmost entry (i.e., $\bar\jmath$) of the first row. Since $Q$ is a (possibly empty) standard Young tableau containing the same elements of $R$ except for $\bar\jmath$, we have $\omega\wedge d\theta=0$ and in particular $\omega\in\DH^{2a+b}$. Since
\begin{align*}
d\omega & =\psi(x,y,t) dy_{\bar\jmath}\wedge dx_{I\cup\{\bar\jmath\}}\wedge dy_J\wedge \alpha_Q\wedge \theta + O(y_{\bar\jmath}) \\
& =  \pm\:\psi(x,y,t)\: dx_{I}\wedge dy_J\wedge\alpha_Q\wedge dxy_{\bar\jmath}\wedge \theta + O(y_{\bar\jmath}) \\
& =  \pm\:\psi(x,y,t)\:  dx_{I}\wedge dy_J\wedge \alpha_R\wedge \theta + O(y_{\bar\jmath}),
\end{align*}
we deduce as before that
\begin{align*}
0 & =
 \partial\Tcurr (\omega) = \Tcurr(d\omega) = \int_\Pab\langle\tau \mid d\omega\rangle d\Shaus^{2a+b+2}
=\pm \langle\tau \mid  dx_I\wedge dy_J\wedge\al_R\wedge\theta\rangle.
\end{align*}
and Claim 3 follows.\medskip

{\em Claim 4: at least one between \eqref{eq:stampellebis1} and \eqref{eq:stampellebis2} holds.}

\noindent We know that \eqref{eq:stampellebis2} holds if $I,J,R$ satisfy any one of the assumptions of Claims 1, 2, 3. We then assume that none of such assumptions holds, i.e., that
\begin{equation}\label{eq:perstampellebis}
\begin{aligned}
& I\subset\{1,\dots,a+b\}\\
\text{and }& J\subset\{1,\dots,a\}\\
\text{and }& \text{the elements in the first row of $R$ are not greater than $a$,}
\end{aligned}
\end{equation}
and we prove that \eqref{eq:stampellebis1} holds.

By \eqref{eq:perstampellebis}, all the integers $a+b+1,\dots,n$ appear in the second row of $R$: the length of such a row is then at least $n-a-b$. Since the difference between the lengths of the rows of $R$ is  equal to $2a+b-n$, the length of the first row of $R$ is  at least $a$. By the third condition in \eqref{eq:perstampellebis}, the first row of $R$ contains at most $a$ elements, hence it contains precisely the $a$ elements $1,\dots,a$ (in this order). In turn,  the second row contains $n-a-b$ elements, that are forced to be the numbers $a+b+1,\dots,n$ (in this order). In particular, $R$ is the one displayed in \eqref{eq:stampellebis1}.  The remaining integers $a+1,\dots,a+b$, not appearing in $R$, have to belong to either $I$ or $J$; the second condition in \eqref{eq:perstampellebis} implies that they all belong to $I$, and the proof is concluded. 
\end{proof}

Proposition~\ref{prop:ruotiamosupianicanonici} allows to extend Lemmata~\ref{lem:vettoretangente,caso_n} and~\ref{lem:vettoretangente,casogenerale} 
to general vertical planes (Definition~\ref{def:pianiverticali}).
 
\begin{proposition}\label{prop:correntipiane}
Let $k\in\unoenne$, $\tau\in\mathcal J_{2n+1-k}$ and a vertical $(2n+1-k)$-plane $\mathscr P\subset\H^n$ be fixed; assume that the current $\Tcurr:=\tau\Shaus^{Q-k}\res\mathscr P$ is such that $\partial\Tcurr=0$. Then there exists $\eta\in\R$ such that
\[
\tau=\eta\:[t^\H_{\mathscr P}]_\mathcal J.
\]
\end{proposition}


\begin{proof}
By Proposition~\ref{prop:ruotiamosupianicanonici} there exists a $\H$-linear isomorphism $\mathcal L:\H^n\to\H^n$ such that $\mathcal L(\mathscr P)=\Pab$ and $\mathcal L^*(d\theta)=d\theta$. Consider the push-forward  $\mathcal L_\#\Tcurr$ of $\Tcurr$, i.e., the Heisenberg $(2n+1-k)$-current defined by
\[
\mathcal L_\#\Tcurr(\omega):=\Tcurr(\mathcal L^*(\omega)),\qquad\omega\in\DH^{2n+1-k}.
\]
Also $\mathcal L_\#\Tcurr$ has zero boundary by Corollary~\ref{cor:d_Ccommuta}. The push-forward $\mathcal L_\#(\Shaus^{Q-k}\res\mathscr P)$ of the measure $\Shaus^{Q-k}\res\mathscr P$, defined by
\[
\mathcal L_\#(\Shaus^{Q-k}\res\mathscr P) (E):=\Shaus^{Q-k}(\mathcal L^{-1}(E)\cap\mathscr P),\qquad E\subset\H^n,
\]
is a Haar measure on $\Pab$ and in particular
\[
\mathcal L_\#(\Shaus^{Q-k}\res\mathscr P) = \gamma \Shaus^{Q-k}\res\Pab
\]
for a suitable $\gamma>0$. It follows that for every $\omega\in\DH^{2n+1-k}$
\begin{equation*}
\mathcal L_\#\Tcurr(\omega) 
 = \int_\mathscr P \langle\tau \mid \mathcal L^*(\omega) \rangle \:d\Shaus^{Q-k} 
= \int_\Pab \langle \gamma\mathcal L_*(\tau) \mid \omega \rangle \:d\Shaus^{Q-k}
\end{equation*}
where $\mathcal L_*:\mathcal J_{2n+1-k}\to\mathcal J_{2n+1-k}$ is the isomorphism defined by
\begin{equation}\label{eq:L_*Rumin}
\langle \mathcal L_*(\tau) \mid \la\rangle:= \langle\tau \mid  \mathcal L^*(\la)\rangle\qquad\forall\: \la\in\mathcal J^{2n+1-k}.
\end{equation}
Observe that we are implicitly using Proposition~\ref{prop:JtoJ}. By Corollary~\ref{cor:d_Ccommuta} the current $\mathcal L_\#\Tcurr=\gamma\mathcal L_*(\tau)\Shaus^{Q-k}\res\Pab$ has zero boundary: therefore,
Lemma~\ref{lem:vettoretangente,caso_n} (if $k=n$) or Lemma~\ref{lem:vettoretangente,casogenerale} (if $1\leq k< n$) imply that there exists $\eta\in\R$ such that
\[
\gamma\:\mathcal L_*(\tau) = \eta[t^\H_\Pab]_\mathcal J = \eta\, C\,[\mathcal L_*(t^\H_\mathscr P)]_\mathcal J = \eta \,C\,\mathcal L_*[t^\H_\mathscr P]_\mathcal J
\]
for a suitable $C\neq 0$ depending on $\mathcal L$ and $\mathscr P$. Since $\mathcal L_*:\mathcal J_{2n+1-k}\to\mathcal J_{2n+1-k}$ is an isomorphism, we obtain that $\tau=\eta \,C\,\gamma^{-1}[t^\H_\mathscr P]_\mathcal J$ and the proof is concluded.
\end{proof}

\subsection{Proof of Theorem~\ref{teo:correntinormalirettificabili}}\label{subsec:corr_norm_rettif}
We now prove Theorem~\ref{teo:correntinormalirettificabili}. Recall once again that, for $k\in\unoenne$, the space $\mathcal J_{2n+1-k}$ was introduced as the dual space to $\mathcal J^{2n+1-k}$.  We  however  need to introduce also the dual space to Rumin's space $\bwl^n\h/\mathcal I^n$ and, for convenience of notation, we will denote such dual space by $\mathcal J_{n}$. 
The spaces $\mathcal J_n$ and $\mathcal J_{n+1},\dots,\mathcal J_{2n}$ are endowed, respectively, with the operator norm $|\cdot|$ arising from either the norm on $\bwl^n\h/\mathcal I^n$ introduced in~\eqref{eq:normasuJn}, or from the standard norm on $\mathcal J^{2n+1-k}\subset\bwl^{2n+1-k}\h$, $k=1,\dots, n$. 

\begin{proof}[Proof of Theorem~\ref{teo:correntinormalirettificabili}]
We can without loss of generality assume that $R$ is a $\H$-regular submanifold $S$.
We have to prove that there exists $\zeta:S\to\R$ such that 
\begin{equation}\label{eq:AAA}
\tauT(p)=\zeta(p)[t^\H_{S}(p)]_\mathcal J\qquad\text{for $\|\Tcurr\|_a$-a.e. $p\in S$.}
\end{equation}
Since $\Shaus^{Q-k}\res S$ is  locally $(Q-k)$-Ahlfors regular, we can differentiate the measure $\|\Tcurr\|$ with respect to $\Shaus^{Q-k}\res S$, see e.g.~\cite[Theorem 4.7 and Remark 4.5]{Simon}. In particular, we can write $\|\Tcurr\|_a=f\Shaus^{Q-k}\res S$ for a suitable $f\in L^1_{\mathrm{loc}}(\Shaus^{Q-k}\res S)$ and, 
for $\Shaus^{Q-k}$-a.e. $p\in S$, one has
\begin{equation}\label{eq:CCC}
\int_{S\cap B(p,r)} |f\tauT -f(p)\tauT(p)|\:d\Shaus^{Q-k}
= o\left( \Shaus^{Q-k}(S\cap B(p,r))\right)=o(r^{Q-k})
\end{equation}
and
\begin{equation}\label{eq:BBB}
\|\Tcurr\|_s(B(p,r))=o(r^{Q-k}).
\end{equation}
Statement~\eqref{eq:AAA} (and then  Theorem~\ref{teo:correntinormalirettificabili}) reduces to proving that there exists $\eta:S\to\R$ such that 
\begin{equation}\label{eq:AAAAAA}
f(p)\tauT(p)=\eta(p)[t^\H_{S}(p)]_\mathcal J\qquad\text{for $\Shaus^{Q-k}$-a.e. $p\in S$.}
\end{equation}
Since also $\partial \Tcurr$ has locally finite mass, by Riesz' theorem (recall also Remark~\ref{rem:distriborder0}) there exist a Radon measure $\nu$ and a locally $\nu$-integrable function $\si:\H^n\to\mathcal J_{2n-k}$ such that $|\si|=1$ $\nu$-a.e. and $\partial\Tcurr=\si\nu$, i.e.,
\[
\partial\Tcurr(\omega)=\int\langle\si\mid \omega\rangle\:d\nu,\qquad\omega\in\DH^{2n-k}.
\]
Differentiating $\nu$ with respect to $\Shaus^{Q-k}\res S$ we obtain that for $\Shaus^{Q-k}$-a.e. $p\in S$
\begin{equation}\label{eq:DDD}
\nu(B(p,r))
= O\left( \Shaus^{Q-k}(S\cap B(p,r))\right)=O(r^{Q-k}).
\end{equation}

We claim that  \eqref{eq:AAAAAA} holds for those $p\in S$ for which~\eqref{eq:CCC},~\eqref{eq:BBB} and~\eqref{eq:DDD} hold: this would be enough to conclude. 

Let then such a $p$ be fixed. For  $r>0$ consider the $\H$-linear isomorphism $\Lpr:\H^n\to\H^n$ defined by $\Lpr(q):=\de_{1/r}(p^{-1}q)$ and the push-forward $\Tpr:={\Lpr}_\#\Tcurr$, i.e., the current 
\[
\Tpr(\omega):=\Tcurr(\Lpr^*\omega),\qquad \omega\in\DH^{2n+1-k}.
\]
Observe that, by homogeneity and left-invariance, the equality $\Lpr^*\omega = r^{-(Q-k)}(\omega\circ\Lpr )$ holds for every $\omega\in\DH^{2n+1-k}$. If $\bar r>0$ is such that  $\spt\omega\subset B(0,\bar r)$ one gets 
\begin{align*}
\lim_{r\to 0^+}\Tpr(\omega) 
&  \stackrel{\phantom{\eqref{eq:BBB}}}{=} 
\lim_{r\to 0^+} \frac{1}{r^{Q-k}}\int_{B(p,r\bar r)} \langle \tauT \mid  \omega\circ\Lpr \rangle\: d\|\Tcurr\|\\
& \stackrel{\eqref{eq:BBB}}{=} 
\lim_{r\to 0^+} \frac{1}{r^{Q-k}}\int_{S\cap B(p,r\bar r)} \langle f\tauT \mid  \omega\circ\Lpr \rangle\: d\Shaus^{Q-k}\\
& \stackrel{\eqref{eq:CCC}}{=} 
\lim_{r\to 0^+} \frac{1}{r^{Q-k}}\int_{S\cap B(p,r\bar r)} \langle f(p)\tauT(p) \mid  \omega\circ\Lpr \rangle\: d\Shaus^{Q-k}\\
\intertext{and a change of variables gives}
\lim_{r\to 0^+}\Tpr(\omega)
&  \stackrel{\phantom{\eqref{eq:BBB}}}{=} 
\lim_{r\to 0^+} \int_{\de_{1/r}(p^{-1}S)\cap B(0,\bar r)} \langle f(p)\tauT(p) \mid  \omega \rangle\: d\Shaus^{Q-k}\\
&  \stackrel{\phantom{\eqref{eq:BBB}}}{=} 
\lim_{r\to 0^+} \int_{\de_{1/r}(p^{-1}S)} \langle f(p)\tauT(p) \mid  \omega \rangle\: d\Shaus^{Q-k}.
\end{align*}
By \eqref{eq:convdebolealtangente} we can define the limit current  $\Tcurr_\infty$ as
\[
\Tcurr_\infty(\omega):=\lim_{r\to 0^+}\Tpr(\omega) 
=
\int_{\Tan^\H_{S}(p)} \langle f(p)\tauT(p) \mid  \omega \rangle\: d\Shaus^{Q-k},\qquad \omega\in\DH^{2n+1-k}.
\]

The current $\Tcurr_\infty$ is supported on the plane $\Tan^\H_{S}(p)$. We now study its boundary and observe that for every $\omega\in\DH^{2n-k}$
\begin{align*}
\partial\Tcurr_\infty(\omega)
 &= \Tcurr_\infty(d_C\omega) = \lim_{r\to 0^+}\Tpr(d_C\omega) 
 =\lim_{r\to 0^+} \Tcurr(\Lpr^*d_C\omega) .
\intertext{and by Corollary~\ref{cor:d_Ccommuta}}
\partial\Tcurr_\infty(\omega)
&  = \lim_{r\to 0^+} \Tcurr(d_C\Lpr^*\omega) 
=  \lim_{r\to 0^+} \partial \Tcurr(\Lpr^*\omega) 
= \int_{\H^n} \langle\si \mid  \Lpr^*\omega\rangle\:d\nu.
\end{align*} 
By homogeneity we have $ \Lpr^*\omega=r^{-\Delta}(\omega\circ\Lpr)$, where $\Delta$ is the homogeneity degree of $\omega$ and namely
\begin{align*}
&\Delta=Q-k-2=n& &\hspace{-2cm}\text{if }k=n\\
&\Delta=Q-k-1&&\hspace{-2cm}\text{if }1\leq k\leq n-1.
\end{align*}
We then obtain
\begin{align*}
|\partial\Tcurr_\infty(\omega)|
&  = \lim_{r\to 0^+} {r^{-\Delta}}\left|\int_{\H^n} \langle\si \mid  \omega\circ\Lpr\rangle\:d\nu\right| \\
&\leq  \lim_{r\to 0^+} r^{-\Delta+Q-k}\:\frac{\nu(B(p,r\bar r))}{r^{Q-k}}\:\sup|\omega|=0,
\end{align*} 
the last equality following from~\eqref{eq:DDD} and the inequality $\Delta< Q-k$.

The current $\Tcurr_\infty=f(p)\tauT(p)\Shaus^{Q-k}\res\Tan^\H_{S}(p)$ is such that $\partial\Tcurr_\infty=0$. By Proposition~\ref{prop:correntipiane} there exists $\eta=\eta(p)\in\R$ such that
\[
f(p)\tauT(p)=\eta[t^\H_{\Tan^\H_{S}(p)}]_\mathcal J =\eta [t^\H_{S}(p)]_\mathcal J 
\]
and the proof is accomplished.
\end{proof}

\subsection{The Constancy Theorem in Heisenberg groups}\label{subsec:ConstThmPlanes}
In this section we prove Theorem~\ref{thm:Constancy}. We start by establishing some standard facts inspired by classical results about mollification of distributions. Given $k\in\unoenne$, let  $a,b$ be fixed non-negative integers such that $1\leq a+b\leq n$ and $2a+b=2n-k$. Consider the vertical plane $\Pab$ defined in~\eqref{eq:defPab222}; let us fix a mollification kernel $\varphi\in C^\infty_c(\Pab)$ such that
\[
\int_\Pab \varphi\:d\Shaus^{Q-k}=1\qtaq \spt \varphi\subset B(0,1)\cap \Pab.
\]
As usual, for every $\ep>0$ we define the rescaled kernels $\varphi_\ep:=\ep^{k-Q}(\varphi\circ\de_{1/\ep})$ and, given a Heisenberg $(2n+1-k)$-current $\Tcurr$ with support in $\Pab$, we define the Heisenberg $(2n+1-k)$-current $\Tcurr_\ep$ as
\begin{equation}\label{eq:mollifT}
\Tcurr_\ep(\omega):=\int_\Pab \varphi_\ep(p)\Tcurr(\mathcal L_p^* \omega)\:d\Shaus^{Q-k}(p),\qquad\omega\in\DH^{2n+1-k}
\end{equation}
where $\mathcal L_p(q):=pq$ denotes left-translation by $p\in\Pab$.

\begin{lemma}\label{lem:mollify}
Let $\Tcurr$ be a Heisenberg $(2n+1-k)$-current  with support in $\Pab$ and with locally finite mass; for $\ep>0$, consider $\Tcurr_\ep$  as in~\eqref{eq:mollifT}. Then the following statements hold:
\begin{enumerate}
\item[(i)] $\Tcurr_\ep$ has support in $\Pab$;
\item[(ii)] $\Tcurr_\ep\rightharpoonup \Tcurr$  as $\ep\to 0^+$, i.e., $\Tcurr_\ep(\omega)\to \Tcurr(\omega)$ for every $\omega\in\DH^{2n+1-k}$;
\item[(iii)] there exists a $C^\infty$-smooth map $\tau_\ep:\Pab\to\mathcal J_{2n+1-k}$ such that $\Tcurr_\ep=\tau_\ep\Shaus^{Q-k}\res\Pab$;
\item[(iv)] if $\partial\Tcurr=0$, then $\partial\Tcurr_\ep=0$.
\end{enumerate}
\end{lemma}
\begin{proof}
The statement in (i) is clear: in fact, if $\omega\in\DH^{2n+1-k}$ is such that $\spt \omega\cap\Pab=\emptyset$, then $\spt\mathcal L_p^*\omega\cap\Pab=\emptyset$ for every $p\in\Pab$ and $\Tcurr_\ep(\omega)=0$.

Concerning (ii), let $\omega\in\DH^{2n+1-k}$ be fixed and let $R>0$ be such that $\spt\omega\subset B(0,R)$. Writing $\Tcurr=\tauT\|\Tcurr\|$ as in Remark~\ref{rem:distriborder0}, we estimate
\begin{align*}
|\Tcurr_\ep(\omega)-\Tcurr(\omega)| 
& = \left| \int_\Pab\varphi_\ep(p)\,\Tcurr(\mathcal L_p^* \omega-\omega)\: d\Shaus^{Q-k}(p)\right|\\
&\leq \sup_{p\in\Pab\cap B(0,\ep)}\|\mathcal L_p^* \omega-\omega\|_{C^0(B(0,R+\ep))}\|\Tcurr\|(B(0,R+\ep))
\end{align*}
and in particular $\Tcurr_\ep(\omega)\to \Tcurr(\omega)$ as $\ep\to 0^+$.

Since $\mathcal L_p^* \omega(q)=\omega(pq)$, statement (iii) follows from
\begin{align*}
\Tcurr_\ep(\omega)
&= \int_\Pab\int_\Pab \varphi_\ep(p)\,\langle \tauT(q)\mid \omega(pq)\rangle \,d\|\Tcurr\|(q)\:d\Shaus^{Q-k} (p)\\
&= \int_\Pab \Big\langle \int_\Pab \varphi_\ep(pq^{-1}) \tauT(q)\,d\|\Tcurr\|(q)\: \Big|\:\omega(p) \Big\rangle\:d\Shaus^{Q-k} (p).
\end{align*}

Eventually, if $\partial\Tcurr=0$ one has
\begin{equation}\label{eq:boundarymollif}
\partial\Tcurr_\ep(\omega)=\Tcurr_\ep(d_C\omega)=\int_\Pab \varphi_\ep(p)\Tcurr(d_C(\mathcal L_p^* \omega))\:d\Shaus^{Q-k}(p)=0,
\end{equation}
where we used Corollary~\ref{cor:d_Ccommuta}.
\end{proof}

\begin{remark}
One can more generally observe that, as in~\eqref{eq:boundarymollif},
\[
\partial\Tcurr_\ep(\omega)=\int_\Pab \varphi_\ep(p)\: \partial\Tcurr(\mathcal L_p^* \omega)\:d\Shaus^{Q-k}(p),
\]
i.e.,  $\partial \Tcurr_\ep=(\partial\Tcurr)_\ep$. 
\end{remark}

For the reader's convenience we separate the proof of Theorem~\ref{thm:Constancy} in the cases $k=n$ and $1\leq k\leq n-1$; as one can expect, the former is computationally more demanding because of the use of the second-order operator $D$. 

It is convenient to fix some notation. If $R$ is a Young tableau and the elements displayed in $R$ are all different, by abuse of notation we write $\sum_{k\in R}$ to denote summation on all the elements $k$ displayed in $R$. Moreover, if $R$ is rectangular and $k$ is an element displayed in $R$, we denote by  $[R\!\setminus\! k]$ the (possibly empty) Young tableau obtained from $R$ by removing the column containing $k$. For instance, if
\[
R=
\begin{tabular}{|c|c|c|}
\cline{1-3}
1 & 2 &3\\
\cline{1-3}
0 & 9 &5 
\\
\cline{1-3}
\end{tabular}
\]
then
\[
\sum_{k\in R} f(k)=f(0)+f(1)+f(2)+f(3)+f(5)+f(9)
\qqtaqq 
[R\!\setminus\!9]= \begin{tabular}{|c|c|}
\cline{1-2}
1 &3\\
\cline{1-2}
0  &5 
\\
\cline{1-2}
\end{tabular}\,.
\]

\begin{proof}[Proof of Theorem~\ref{thm:Constancy}, case $k=n$]
Reasoning as in Proposition~\ref{prop:correntipiane}, by  Proposition~\ref{prop:ruotiamosupianicanonici} one can assume without loss of generality that $\mathscr P=\Pab$ for some non-negative integers $a,b$ such that $2a+b=n$. Moreover, by Lemma~\ref{lem:mollify} it is not restrictive to assume that $\Tcurr=\tau\Shaus^{n+2}\res\Pab$ for a suitable $C^\infty$-smooth $\tau:\Pab\to\mathcal J_{n+1}$. By Theorem~\ref{teo:correntinormalirettificabili}, $\tau$ can be written as $\tau=\varphi_\tau [t^\H_\Pab]_\mathcal J$ for some $\varphi_\tau\in C^\infty(\Pab)$; let us prove that $\varphi_\tau $ is constant on $\Pab$.

We are going to utilize test $n$-forms $\omega=f dx_I\wedge dy_J\wedge\al_R$, where $f\in C^\infty_c(\H^n)$ and the triple $(I,J,R)$ is as in Proposition~\ref{prop:baseJintro} (with $k=n$); in particular, $R$ is rectangular. As usual, $\omega$ is a smooth section of $\bwl^n\h_1$, but we identify $\omega$ with an element in $\DH^n$ as in~\eqref{eq:equivbis}. By~\eqref{eq:Dalcondal_v} we have
\begin{equation}\label{eq:Dal}
D\omega= (d\omega)_\mathfrak v +  \theta\wedge d(L^{-1}((d\omega)_{\h_1})).
\end{equation}
Taking into account that
\begin{align*}
d\omega =\:
& \sum_{i\in I} (Y_if) dy_i\wedge dx_I\wedge dy_J\wedge\al_R + \sum_{j\in J} (X_j f) dx_j\wedge dx_I\wedge dy_J\wedge\al_R\\
&+\sum_{k\in R} ((X_kf)dx_k+(Y_kf)dy_k)\wedge dx_I\wedge dy_J\wedge\al_R + (Tf)\theta\wedge dx_I\wedge dy_J\wedge\al_R
\end{align*}
we obtain
\begin{equation}\label{eq:dal_v}
(d\omega)_\mathfrak v = (Tf)\theta\wedge dx_I\wedge dy_J\wedge\al_R
\end{equation}
and
\begin{equation}\label{eq:L-1dal_h1}
\begin{aligned}
L^{-1}((d\omega)_{\h_1}) =\:
& \sum_{i\in I} \pm (Y_if)  dx_{I\setminus\{i\}}\wedge dy_J\wedge\al_R + \sum_{j\in J} \pm 
(X_j f) dx_I\wedge dy_{J\setminus\{j\}}\wedge\al_R\\
&+\sum_{k\in R} \left( \pm(X_kf) dx_{I\cup\{k\}}\wedge dy_J \pm(Y_kf) dx_{I}\wedge dy_{J\cup\{k\}}\right)\wedge\al_{[R\setminus k]}.
\end{aligned}
\end{equation}
The signs $\pm$ appearing in~\eqref{eq:L-1dal_h1} could be easily specified, but they are in fact  irrelevant for our purposes.

Let us fix
\[
\overline I:=\{a+1,\dots,a+b\},\quad\overline J:=\emptyset \qtaq
\overline R:=
\begin{tabular}{|c|c|c|c|}
\cline{1-4}
$1$ & $2$ &$\ \cdots\ $& $a$\\
\cline{1-4}
$a+b+1$ & $a+b+2^{\phantom{2}}$ &$\ \cdots\ $& $n$
\\
\cline{1-4}
\end{tabular}\:.
\]
If $a=0$, then $b=n$, $\overline I=\unoenne$ and $\overline R$ is the empty tableau. If $b=0$, then $a=n/2$ and $\overline I=\emptyset$.

Assume $a\geq 1$, fix a column of $\overline R$ and let $\al,\ga$ be its elements, with $\al<\ga$; in particular, $\ga=\al+n-a$. On choosing
\begin{align*}
& I:=\overline I\cup\{\al\},\quad J:=\{\ga\}\qtaq R:=[\overline R\!\setminus\! \al]\\
& f(x,y,t):=x_\ga\: g(x,y,t)\quad\text{for an arbitrary }g\in C^\infty_c(\H^n)
\end{align*}
one gets from~\eqref{eq:dal_v} and~\eqref{eq:L-1dal_h1}
\[
(d\omega)_\mathfrak v=O(x_\ga)=0\quad\text{on }\Pab,
\]
where we used the notation  introduced before Lemma~\ref{lem:vettoretangente,caso_n}, and
\begin{align*}
\theta\wedge d(L^{-1}((d\omega)_{\h_1})) 
&= \pm (Y_\al X_\ga f) dx_{\overline I}\wedge dy_{\overline J}\wedge\al_{ R}\wedge dxy_\al\wedge\theta + O(x_\ga) + \si\\
&= \pm (Y_\al g) dx_{\overline I}\wedge dy_{\overline J}\wedge\al_{\overline R}\wedge\theta + O(x_\ga) + \si
\end{align*}
where, here and in the following, $\si$ denotes a form (which may vary from line to line) in the annihilator of $[t^\H_\Pab]_{\mathcal J}$; equivalently, $\langle X_1\wedge \dots\wedge X_{a+b}\wedge Y_1\wedge\dots\wedge Y_a\wedge T\mid \si\rangle=0$. In particular, $\langle\tau\mid \si\rangle=0$ and from~\eqref{eq:Dal} we obtain
\begin{equation}\label{eq:distribuzionaleYa}
0=\Tcurr(D\omega)=\pm\int_\Pab \varphi_\tau(Y_\al g)\:d\leb{n+1}
\quad\text{for every }g\in C^\infty_c(\H^n).
\end{equation}

In a similar way, on choosing
\begin{align*}
& I:=\overline I\cup\{\ga\},\quad J:=\{\al\}\qtaq R:=[\overline R\!\setminus\! \al]\\
& f(x,y,t):=y_\ga\: g(x,y,t)\quad\text{for an arbitrary }g\in C^\infty_c(\H^n)
\end{align*}
one gets 
again $(d\omega)_\mathfrak v=O(y_\ga)=0$ on $\Pab$ and
\begin{align*}
\theta\wedge d(L^{-1}((d\omega)_{\h_1})) 
&= \pm (X_\al Y_\ga f) dx_{\overline I}\wedge dy_{\overline J}\wedge\al_{ R}\wedge dxy_\al\wedge\theta + O(y_\ga) + \si\\
&= \pm (X_\al g) dx_{\overline I}\wedge dy_{\overline J}\wedge\al_{\overline R}\wedge\theta + O(y_\ga) + \si
\end{align*}
where again $\si$ denotes a form in the annihilator of $[t^\H_\Pab]_{\mathcal J}$. From~\eqref{eq:Dal} we obtain
\[
0=\Tcurr(D\omega)=\pm\int_\Pab \varphi_\tau(X_\al g)\:d\leb{n+1}
\quad\text{for every }g\in C^\infty_c(\H^n)
\]
which, together with~\eqref{eq:distribuzionaleYa}, gives
\begin{equation}\label{eq:distribuzionaleXYTa}
\begin{aligned}
& X_\al \varphi_\tau = Y_\al \varphi_\tau = 0\qquad\text{for every }\al=1,\dots,a\\
& T \varphi_\tau =  X_1 Y_1 \varphi_\tau - Y_1 X_1 \varphi_\tau=0
\end{aligned}
\end{equation}
We recall once again that the equalities in~\eqref{eq:distribuzionaleXYTa} are proved only under the assumption $a\geq 1$.

If $b\geq 1$ we fix $\be\in\overline I$ and choose
\begin{align*}
& I:=\overline I\cup\{a+b+1\}\setminus\{\be\},\quad J:=\emptyset\qtaq
 R:=
\begin{tabular}{|c|c|c|c|}
\cline{1-4}
$1$ & $2$ &$\ \cdots\ $& $a$\\
\cline{1-4}
$\be$ & $a+b+2^{\phantom{2}}$ &$\ \cdots\ $& $n$
\\
\cline{1-4}
\end{tabular}\\
& f(x,y,t):=y_{a+b+1}\: g(x,y,t)\quad\text{for an arbitrary }g\in C^\infty_c(\H^n).
\end{align*}
The tableau $R$ is obtained from $\overline R$ on replacing the entry $a+b+1$ with $\be$. Then  
\[
(d\omega)_\mathfrak v=O(y_{a+b+1})=0\quad\text{on }\Pab
\]
and
\begin{align*}
\theta\wedge d(L^{-1}((d\omega)_{\h_1})) 
&= \pm (X_\be Y_{a+b+1} f) dx_{\overline I}\wedge dy_{\overline J}\wedge\al_{R}\wedge \theta + O(y_{a+b+1}) + \si\\
&= \pm (X_\be g) dx_{\overline I}\wedge dy_{\overline J}\wedge\al_{\overline R}\wedge\theta + O(y_{a+b+1}) + \si
\end{align*}
where again $\si$ denotes a form annihilating $[t^\H_\Pab]_{\mathcal J}$ and we used the fact that $dx_{\overline I}\wedge dy_{\overline J}\wedge(\al_R-\al_{\overline R})$ annihilates $[t^\H_\Pab]_{\mathcal J}$. From~\eqref{eq:Dal} we obtain
\[
0=\Tcurr(D\omega)=\pm\int_\Pab \varphi_\tau(X_\be g)\:d\leb{n+1}
\quad\text{for every }g\in C^\infty_c(\H^n)
\]
so that
\begin{equation}\label{eq:distribuzionaleXbeta}
X_\be\varphi_\tau=0\qquad\text{for every }\be=a+1,\dots,a+b.
\end{equation}

If $a\geq 1$,~\eqref{eq:distribuzionaleXYTa} and~\eqref{eq:distribuzionaleXbeta} are enough to conclude that $\varphi_\tau$ is constant on $\Pab$. If $a=0$,~\eqref{eq:distribuzionaleXbeta} still holds and we have only to prove that
\begin{equation}\label{eq:distribuzionaleSOLOT}
T\varphi_\tau=0\qquad\text{on }\Pab.
\end{equation}
We then choose $I=\overline I=\unoenne$, $J=\overline J=\emptyset$, $R=\overline R$ (in this case, the empty tableau) and we fix an arbitrary $f\in C^\infty_c(\H^n)$. By~\eqref{eq:Dal},~\eqref{eq:dal_v} and~\eqref{eq:L-1dal_h1}
\begin{align*}
D\omega
&=\theta\wedge\left((Tf)dx_{\unoenne}+d\Big(\sum_{i=1}^n \pm(Y_if)\,dx_{\unoenne\setminus\{i\}}\Big)\right)\\
& = \left(Tf+\sum_{i=1}^n \pm X_iY_if\right)\theta\wedge dx_{\unoenne} + \si 
\end{align*}
for a suitable $\si$ in the annihilator of $[t^\H_\Pab]_{\mathcal J}$. This gives
\[
0=\Tcurr(D\omega)=\int_{\Pab} \left(Tf+\sum_{i=1}^n \pm X_iY_if\right)\varphi_\tau\:d\leb{n+1}\stackrel{\eqref{eq:distribuzionaleXbeta}}{=}\int_{\Pab} (Tf)\varphi_\tau\:d\leb{n+1}
\]
and~\eqref{eq:distribuzionaleSOLOT} follows from the arbitrariness of $f$. This concludes the proof.
\end{proof}

\begin{proof}[Proof of Theorem~\ref{thm:Constancy}, case $1\leq k\leq n-1$]
Reasoning as in Proposition~\ref{prop:correntipiane}, by  Proposition~\ref{prop:ruotiamosupianicanonici} one can assume without loss of generality that $\mathscr P=\Pab$ for some non-negative integers $a,b$ such that $1\leq a+b\leq n$ and $2a+b=2n-k$. Observe that $a\geq 1$. By Lemma~\ref{lem:mollify} it is not restrictive to assume that $\Tcurr=\tau\Shaus^{Q-k}\res\Pab$ for a suitable $C^\infty$-smooth $\tau:\Pab\to\mathcal J_{2n+1-k}$. By Theorem~\ref{teo:correntinormalirettificabili}, $\tau$ can be written as $\tau=\varphi_\tau [t^\H_\Pab]_\mathcal J$ for some $\varphi_\tau\in C^\infty(\Pab)$; let us prove that $\varphi_\tau $ is constant on $\Pab$.

We are going to consider the Heisenberg $(2n-k)$-form $\omega=f dx_I\wedge dy_J\wedge\al_R\wedge\theta$, where
\begin{itemize}
\item $f\in C^\infty_c(\H^n)$;
\item $I\subset\unoenne$, $J\subset\unoenne$, $|I|+|J|\leq k+1$ and  $I\cap J=\emptyset$;
\item $R$ is a (non-necessarily standard) Young tableau that contains the elements of $\unoenne\setminus(I\cup J)$ arranged in two rows of length, respectively, $(2n-k-1-|I|-|J|)/2$  and   $(k+1-|I|-|J|)/2$.
\end{itemize}
Observe that $\omega\in\DH^{2n-k}$ because of Remark~\ref{rem:tuttetabelleok}.
Then
\begin{equation}\label{eq:domega}
\begin{aligned}
d\omega
=& \sum_{i\in I}\pm (Y_if)dx_{I\setminus\{i\}}\wedge dy_J\wedge\al_R\wedge dxy_i\wedge\theta \\
&+ \sum_{j\in J}\pm (X_jf)dx_{I}\wedge dy_{J\setminus\{j\}}\wedge\al_R\wedge dxy_j\wedge\theta\\
&+ \sum_{k\in R} \big( \pm (X_kf)dx_{I\cup\{k\}}\wedge dy_J  \pm (Y_kf)dx_{I}\wedge dy_{J\cup\{k\}}\big)\wedge\al_R\wedge\theta
\end{aligned}
\end{equation}
where, again, the signs $\pm$ will play no role.

We set
\begin{align*}
&\overline I:=\{a+1,\dots,a+b\},\qquad\overline J:=\emptyset\\
&{\overline R}:=
\begin{tabular}{|c|c|c|c|ccc}
\cline{1-7}
$1$ & $2$ &$\ \cdots\ $& $n-a-b$ &\multicolumn{1}{c|}{$ n-a-b+1$} & \multicolumn{1}{c|}{$\ \cdots\ $} & \multicolumn{1}{c|}{$a$}\\
\cline{1-7}
$a+b+1$ & $a+b+2^{\phantom{2}}$ &$\ \cdots\ $& $n$
\\
\cline{1-4}
\end{tabular}
\end{align*}
and fix $\al\in\{1,\dots, a\}$. Observe that $\overline R$ is never rectangular, a fact that plays a role in the following construction. If $\al\geq n-a-b+1$ we define $R$ by removing $\al$ from $\overline R$, i.e.,
\[
R:=
\begin{tabular}{|c|c|c|c|cccc}
\cline{1-8}
$1$ & $\ \cdots\ $& $n-a-b$  & \multicolumn{1}{c|}{$\ \cdots\ $} & \multicolumn{1}{c|}{$\al-1$} & \multicolumn{1}{c|}{$\al+1$} & \multicolumn{1}{c|}{$\ \cdots\ $} &\multicolumn{1}{c|}{$a$}\\
\cline{1-8}
$a+b+1$ & $\ \cdots\ $& $n$
\\
\cline{1-3}
\end{tabular}\:;
\]
otherwise, if $\al\leq n-a-b$ we define a tableau $R$ by removing  from  $\overline R$ the column containing $\al$ and $\al+a+b$ and placing $\al+a+b$ as the rightmost element in the second row, i.e.,
\[
R:=\begin{tabular}{|c|c|c|c|c|c|c|cc}
\cline{1-9}
$1$ & $\cdots\!$& $\al-1$ & $\al+1$ & $\cdots\!$& $n-a-b$  & \multicolumn{1}{c|}{$n-a-b+1$} & \multicolumn{1}{c|}{$\cdots\!$} &\multicolumn{1}{c|}{$a$}\\
\cline{1-9}
$a+b+1$ & $\cdots\!$& $\al-1+a+b$& $\al+1+a+b$& $\cdots\!$ &$n$ & $\al+a+b$
\\
\cline{1-7}
\end{tabular}
\]

With this choice of $R$, we consider the Heisenberg form $\omega=f dx_I\wedge dy_J\wedge\al_R\wedge\theta\in\DH^{2n-k}$ associated with $I:=\overline I$ and $J:=\{\al\}$; by~\eqref{eq:domega}
\begin{align*}
d\omega&= \pm (X_\al f)\: dx_{\overline I}\wedge\al_R\wedge dxy_\al\wedge\theta + \si\\
&= \pm (X_\al f)\: dx_{\overline I}\wedge\al_{\overline R}\wedge\theta + \si,
\end{align*}
where again $\si$ is a form annihilating $t^\H_\Pab$ that can vary from line to line. This gives
\[
0=\Tcurr(d\omega)=\pm\int_\Pab(X_\al f)\varphi_\tau\:d\leb{2a+b+1}\quad\text{for every }f\in C^\infty_c(\H^n)
\]
and in turn
\begin{equation}\label{eq:disXa}
\text{for every $\al=1,\dots,a$,}\qquad X_\al \varphi_\tau=0\text{ on }\Pab.
\end{equation}

Using the same tableau $R$, but choosing $I:=\overline I\cup\{\al\}$ and $J:=\emptyset$, one gets for $\omega:=f dx_I\wedge dy_J\wedge\al_R\wedge\theta\in\DH^{2n-k}$ that
\begin{align*}
d\omega&= \pm (Y_\al f)\: dx_{\overline I}\wedge\al_R\wedge dxy_\al\wedge\theta + \si\\
&= \pm (Y_\al f)\: dx_{\overline I}\wedge\al_{\overline R}\wedge\theta + \si,
\end{align*}
for  $\si$ annihilating $t^\H_\Pab$. Similarly as before we deduce that
 \begin{equation}\label{eq:disYa}
\text{for every $\al=1,\dots,a$,}\qquad Y_\al \varphi_\tau=0\text{ on }\Pab
\end{equation}
that, together with~\eqref{eq:disXa} and the inequality $a\geq 1$, implies
 \begin{equation}\label{eq:disT}
T \varphi_\tau=0\quad\text{on }\Pab.
\end{equation}

If $b=0$,~\eqref{eq:disXa},~\eqref{eq:disYa} and~\eqref{eq:disT} imply that $\varphi_\tau$ is constant on $\Pab$. If $b\geq 1$, we have only to show that
\begin{equation}\label{eq:disXb}
\text{for every $\be=a+1,\dots,a+b$,}\qquad X_\be \varphi_\tau=0\text{ on }\Pab.
\end{equation}
Let $\be\in\{a+1,\dots,a+b\}$ be fixed.
We consider the test form $\omega=f dx_I\wedge dy_J\wedge\al_R\wedge\theta\in\DH^{2n-k}$ where $I:=\overline I\setminus\{\be\}$, $J:=\emptyset$ and $R$ is the tableau obtained from $R$ on placing an extra entry equal to $\be$ as the rightmost element in the second row, namely, 
\[
R:=
\begin{tabular}{|c|c|c|c|c|cc}
\cline{1-7}
$1$ & $2$ &$\ \cdots\ $& $n-a-b$ &\multicolumn{1}{c|}{$ n-a-b+1$} & \multicolumn{1}{c|}{$\ \cdots\ $} & \multicolumn{1}{c|}{$a$}\\
\cline{1-7}
$a+b+1$ & $a+b+2^{\phantom{2}}$ &$\ \cdots\ $& $n$ & $\be$
\\
\cline{1-5}
\end{tabular}\:.
\]
Then
\begin{align*}
d\omega &=\pm(X_\be f) dx_{I\cup\{\be\}}\wedge\al_R\wedge\theta + \si\\
&=\pm(X_\be f) dx_{\overline I}\wedge\al_{\overline R}\wedge\theta + \si
\end{align*}
and, as before, the arbitrariness of $f$ implies~\eqref{eq:disXb}, as desired.
\end{proof}

\section{Proof of  Rademacher's Theorem~\ref{thm:rademacher}}\label{sec:dimrademacher}
This section is devoted to the proof of our main result.  We start with a  boring (but necessary) preliminary observation.

\begin{remark}\label{rem:Rademachersuffpianicanon}
Let $\W,\V$ be homogeneous complementary subgroups of $\H^n$ and let $\f:A\subset\W\to\V$ be fixed. By Remark~\ref{rem:isometrie_delgruppo} there exists an isometric $\H$-linear isomorphism $\mathcal L:\H^n\to\H^n$ such that $\mathcal L(\V)=\V_0:=\exp(\Span\{X_1,\dots, X_k\})$. Let us write $\W_1:=\mathcal L(\W)$ and $A_1:=\mathcal L(A)$; then $\mathcal L(\gr_\f)=\gr_{\f_1}$ for $\f_1:=\mathcal L\circ \f\circ\mathcal L^{-1}:A_1\to\V_0$. Since also $\W_0:=\exp(\Span\{X_{k+1},\dots,Y_n,T\})$ is complementary to $\V_0$, it follows from \cite[Proposition~3.1]{FSJGA} (or, alternatively, from Theorem~\ref{teo:equivdefLip} of the present paper) that $\gr_{\f_1}=\gr_{\f_0}$ for some map $\f_0:A_0\to\V_0$ defined on a suitable $A_0\subset\W_0$.

Let us check that, if $\f$ is intrinsic Lipschitz, so is $\f_0$. Since $\mathcal L$ is an isometry, also $\f_1$ is intrinsic Lipschitz. Moreover, $\W_1 $ is complementary to $\V_0$, hence $\W_1=\gr_L $ for an intrinsic linear map $L:\W_0\to\V_0$; let $C>0 $ be such that $\|L(w_0)\|_\H\leq C\|w_0\|_\H$ for every $w_0\in\W_0$. We prove that, for every $\al>0$, the inclusion
\begin{align*}
&\{w_0v_0:w_0\in\W_0,v_0\in\V_0,\|v_0\|_\H\geq (\al C+\al+C)\|w_0\|_\H\}\\
\subset\ & 
\{w_1v_0':w_1\in\W_1,v_0'\in\V_0,\|v_0'\|_\H\geq\al\|w_1\|_\H\}
\end{align*}
holds: the intrinsic Lipschitz continuity of $\f_0$ will then easily follow from the intrinsic Lipschitz continuity of $\f_1$. Let  $w_0\in\W_0$ and $v_0\in\V_0$ be such that $\|v_0\|_\H\geq (\al C+\al+C)\|w_0\|_\H$; then 
\begin{align*}
& w_0v_0=w_1 v_0'\qquad\text{for }w_1:=w_0L(w_0)\in\W_1 \text{ and }v_0':=L(w_0)^{-1}v_0\in\V_0
\end{align*}
and
\begin{align*}
& \al\|w_1\|_\H\leq \al(1+C)\|w_0\|_\H \leq   \|v_0\|_\H - C\|w_0\|_\H \leq \|v_0\|_\H - \|L(w_0)\|_\H \leq \| v_0'\|_\H,
\end{align*}
as claimed. 

Eventually, let us observe that $\f_0$ is intrinsically differentiable a.e. if and only if $\f$ is intrinsically differentiable a.e.: this follows from the geometric characterization of intrinsic differentiability provided by Proposition~\ref{prop:differ->convtoTanH} (d) and by Remark~\ref{rem:Ahlfors}, which imply that 
\begin{align*}
& \text{$\f_0$ is intrinsically differentiable a.e. on $\W_0$}\\
\Longleftrightarrow\  & \text{the blow-up of $\gr_{\f_0}$ is a vertical plane at $\Shaus^{Q-k}$-almost every point of $\gr_{\f_0}$}\\
\Longleftrightarrow\  & \text{the blow-up of $\gr_{\f}$ is a vertical plane at $\Shaus^{Q-k}$-almost every point of $\gr_{\f}$}\\
\Longleftrightarrow\  & \text{$\f$ is intrinsically differentiable a.e. on $\W$.}
\end{align*}

This discussion shows that, in order to prove Theorem~\ref{thm:rademacher} for intrinsic Lipschitz graphs of codimension at most $n$, it is not restrictive to assume that $\V$ and $\W$ are those defined in~\eqref{eq:Vfissato} and~\eqref{eq:Wfissato}.
\end{remark}

We can now prove our main result. 
For the reader's convenience, the proof is divided into several steps.

\begin{proof}[Proof of Theorem~\ref{thm:rademacher}]
%
As mentioned in the Introduction, thanks to~\cite{AntonelliMerlo} we have to deal only with the case of intrinsic Lipschitz graphs of low codimension; in particular,  $\V$ is an Abelian horizontal subgroup of $\H^n$ and $k:=\dim\V$ is at most $n$. By Remark~\ref{rem:Rademachersuffpianicanon},  we can without loss of generality assume that $\V=\exp(\Span\{X_1,\dots, X_k\})$ and $\W=\exp(\Span\{X_{k+1},\dots, Y_n,T\})$. By Theorem~\ref{thm:estensione} we can also assume that  $\f$ is defined on the whole $\W$. 

{\em Step 1: definition of a current $\Tcurr$ supported on $\gr_\f$.}
By Proposition~\ref{prop:approssimazionelisciaunifLipHeisenberg} 
we can consider a sequence of smooth functions $\f_i:\W\to\V$ such that
\begin{itemize}
\item $\f_i\to\f$  uniformly on $\W$
\item there exists  $C>0$ such that $|\nabla^{\f_i}\F_i(w)|\leq C$ for all $w\in\W$ and all $i\in\N$,
\end{itemize}
where $\F_i$ is the graph map $\F_i(w):=w\f_i(w)$ and $\nabla^{\f_i}\F_i:\W\to\bwl_{2n-k}\h_1$ is defined as in \eqref{eq:defZETA}.
By Lemma~\ref{lem:correntevistaintegrandosuW}, the Heisenberg $(2n+1-k)$-current $\curr{S_i}$ associated with the intrinsic graph $S_i:=\gr_{\f_i}$ can be written as 
\begin{align*}
\curr{S_i}(\omega)
 & = C_{n,k}\int_\W\langle [\nabla^{\f_i}\F_i(w)\wedge T]_\mathcal J\mid \omega(\F_i(w))\rangle\:d\leb{2n+1-k}(w)\qquad\text{for every }\omega\in\DH^{2n+1-k}.
\end{align*}

Possibly passing to a subsequence,  we can assume that there exists $\zeta\in L^\infty(\W,\mathcal J_{2n+1-k})=(L^1(\W,\mathcal J^{2n+1-k}))^*$ such that
\[
[\nabla^{\f_i}\F_i\wedge T]_\mathcal J\ \stackrel{*}{\rightharpoonup} \ \zeta\quad\text{weakly-$*$ in }L^\infty(\W,\mathcal J_{2n+1-k}).
\]
The uniform convergence  $\F_i\to\F$ implies that for every $\omega\in\DH^{2n+1-k}$
\begin{equation}\label{eq:defTTT}
\Tcurr(\omega):=\lim_{i\to\infty}\ \curr{S_i}(\omega)= C_{n,k}\int_\W\langle \zeta(w)\mid \omega(\F(w))\rangle\:d\leb{2n+1-k}(w).
\end{equation}
The Heisenberg  current $\Tcurr$ is clearly supported on $\gr_\f$. The boundary $\partial\Tcurr$ of $\Tcurr$ is the null current: in fact 
\[
\partial\Tcurr(\omega)=\Tcurr(d_C\omega)=\lim_{i\to\infty}\ \curr{S_i}(d_C\omega)=0\qquad\text{for every }\omega\in\DH^{2n-k},
\]
where $d_C$ is as in Remark~\ref{rem:dzero} and the last equality is due to Corollary~\ref{cor:senzabordo}. The equality $\partial\Tcurr=0$ is the key geometric information we will exploit.

Let us prove  that 
\begin{equation}\label{eq:zetanon0}
\zeta(w)\neq 0\qquad \text{for $\leb{2n+1-k}$-a.e. }w\in\W.
\end{equation}
Let $\be\in \mathcal J^{2n+1-k}$ be defined by
\begin{align*}
&\be:= dx_{k+1}\wedge\dots\wedge dx_n\wedge dy_1\wedge\dots\wedge dy_n\wedge\theta\hspace{-2cm}&&\text{if }k<n\\
&\be:=  dy_1\wedge\dots\wedge dy_n\wedge\theta&&\text{if }k=n.
\end{align*}
Then, for every $\chi\in C^\infty_c(\W)$ we have
\begin{align*}
\int_\W \chi(w) \langle\zeta(w)\mid \be\rangle\,d\leb{2n+1-k}(w)
=&  \lim_{i\to\infty}\int_{\W} \chi(w) \underbrace{\langle  \nabla^{\f_i}\F_i(w)\wedge T \mid \be\rangle}_{\equiv 1}\,d\leb{2n+1-k}(w)\\
=& \int_\W \chi(w)\,d\leb{2n+1-k}(w).
\end{align*}
This implies that $\langle  \zeta(w)\mid \be \rangle=1$ for $\leb{2n+1-k}$-a.e. $w\in\W$ and~\eqref{eq:zetanon0} follows.
 
{\em Step 2: statement of sufficient conditions for differentiability.}
Since $\Shaus^{Q-k}\res \gr_\f$ is   $(Q-k)$-Ahlfors regular (Remark~\ref{rem:Ahlfors}), the Lebesgue Differentiation Theorem applies (see e.g.~\cite[Theorem 1.8]{Heinonen}) and we obtain for $\Shaus^{Q-k}$-a.e. $\bar p\in \gr_\f$
\begin{equation}\label{eq:condizione_a)Lebesgue}
\begin{aligned}
\int_{\gr_\f\cap B(\bar p,r)} 
\left|\zeta(\F^{-1}(p))-\zeta(\F^{-1}(\bar p))\right|
\:d\Shaus^{Q-k}(p)
&=   o\left( \Shaus^{Q-k}(\gr_\f\cap B(\bar p,r))\right)\\
&=o(r^{Q-k})\qquad\text{as }r\to0^+.
\end{aligned}
\end{equation}
By~\eqref{eq:zetanon0},~\eqref{eq:equivalenza} and Lemma~\ref{lem:blowuptinvariante},  the following three properties
\begin{align}
&\label{eq:111}
\text{the condition~\eqref{eq:condizione_a)Lebesgue} holds for $\bar p:=\F(\bar w)\in\gr_\f$}\\
& \label{eq:222}
\text{every blow-up of $\f$ at $\bar w$ is $t$-invariant}\\
& \zeta(\bar w)\neq 0\label{eq:333}
\end{align}
hold for $\leb{2n+1-k}$-a.e. $\bar w\in\W$. 
We claim that $\f$ is intrinsically differentiable at every $\bar w\in\W$ such that~\eqref{eq:111},~\eqref{eq:222} and~\eqref{eq:333} hold: this will be enough to conclude. Let such a $\bar w$ be fixed.

{\em Step 3: blow-up at $\bar w$.}
Let $\f_\infty$ be one of the (possibly many) blow-ups of $\f$ at $\bar w$. Namely, there exists a sequence $(r_j)_j$ of positive numbers such that $r_j\to+\infty$ as $j\to+\infty$ and 
\[
\lim_{j\to\infty} (\f_{\bar w})^{r_j}=\f_\infty\qquad\text{locally uniformly on $\W$,}
\]
where $(\f_{\bar w})^{r_j}(w)=\de_{r_j}\big(\f(\bar w)^{-1}\f(\bar w\f(\bar w)(\de_{1/r_j} w)\f(\bar w)^{-1})\big)$ is as in \S\ref{subsec:differenziabilita_blowups}.   For $r>0$ let us introduce the $\H$-linear isomorphisms $\mathcal L_{\bar w,r}:\H^n\to\H^n$ defined by 
\[
\mathcal L_{\bar w,r}(q):=\de_r(\F(\bar w)^{-1}q),\qquad r>0,\ q\in\H^n;
\]
$\mathcal L_{\bar w,r}$ is defined in such a way that $\gr_{(\f_{\bar w})^{r_j}}=\de_{r_j}(\F(\bar w)^{-1}\gr_\f)=\mathcal L_{\bar w,r_j}(\gr_\f)$. Consider the push-forward $\Tcurr_j:=(\mathcal L_{\bar w,r_j})_\#\Tcurr$, i.e., the Heisenberg current defined  by
\[
\Tcurr_j(\omega):=\Tcurr(\mathcal L_{\bar w,r_j}^*\omega)=\Tcurr(r_j^{Q-k}\ \omega\circ \mathcal L_{\bar w,r_j}),\qquad \omega\in\DH^{2n+1-k}
\]
where $\mathcal L_{\bar w,r_j}^*$ denotes pull-back of forms and the last equality comes from left-invariance and homogeneity. Observe that $\partial T_j=0$ for every $j$ because
\[
\Tcurr_j(d_C\omega)=\Tcurr(\mathcal L_{\bar w,r_j}^*(d_C\omega)) =\Tcurr(d_C(\mathcal L_{\bar w,r_j}^*\omega)) =\partial \Tcurr(\mathcal L_{\bar w,r_j}^*\omega)=0.
\]
By Remark~\ref{rem:Ahlfors}  there exist a  constant $C\geq 1$, depending only on the intrinsic Lipschitz constant of $\f$, and a measurable function $J_\f:\W\to\R$ such that 
\begin{equation}\label{eq:equivalenza}
C^{-1} \leq J_\f \leq C\qtaq \Shaus^{Q-k}\res\gr_\f=\F_\#(J_\f\,\leb{2n+1-k}),
\end{equation}
where $\F_\#$ denotes push-forward of measures. 
Using~\eqref{eq:defTTT} and~\eqref{eq:equivalenza} 
\begin{align*}
\Tcurr_j(\omega)
&= C_{n,k}\: r_j^{Q-k}\int_\W \left\langle\zeta(w)\,\left|\, \frac{\omega(\de_{r_j}(\F(\bar w)^{-1}\F(w)))}{J_\f(w)}\right. \right\rangle J_\f(w)\:d\leb{2n+1-k}(w)\\
&= C_{n,k}\: r_j^{Q-k}\int_{\gr_\f} \left\langle{\zeta(\F^{-1}(p))}\,\left|\, \frac{\omega(\de_{r_j}(\F(\bar w)^{-1}p))}{J_\f(\F^{-1}(p))}\right. \right\rangle d\Shaus^{Q-k}(p)\\
&= C_{n,k}\: r_j^{Q-k}\int_{\gr_\f} \left[\left\langle{\zeta(\bar w)}\,\left|\, \frac{\omega(\de_{r_j}(\F(\bar w)^{-1}p))}{J_\f(\F^{-1}(p))}\right. \right\rangle d\Shaus^{Q-k}(p) + o\big(r_j^{-(Q-k)}\big)\right]\\
&= C_{n,k}\: r_j^{Q-k}\int_\W \left[\left\langle {\zeta(\bar w)}\mid \omega(\de_{r_j}(\F(\bar w)^{-1}\F(w))) \right\rangle \:d\leb{2n+1-k}(w) + o\big(r_j^{-(Q-k)}\big)\right]
\end{align*}
where, in the third equality, we used~\eqref{eq:111} and the fact that, if $\omega$ is supported in $B(0,\bar r)$, then $p\mapsto \omega(\de_{r_j}(\F(\bar w)^{-1}p))/J_\f(\F^{-1}(p))$ is  supported in $B(\F(\bar w),\bar r/r_j)$. Therefore
\begin{align*}
\lim_{j\to\infty}\Tcurr_j(\omega)
& = C_{n,k} \lim_{j\to\infty} r_j^{Q-k}\int_\W \left\langle {\zeta(\bar w)}\mid\omega(\de_{r_j}(\F(\bar w)^{-1}\F(w))) \right\rangle \:d\leb{2n+1-k}(w) 
\end{align*}
provided the limit in the right-hand side exists. We now perform the change of variable $w=\bar w\f(\bar w)(\de_{1/r_j}u)\f(\bar w)^{-1}$, $u\in\W$, according to which 
\[
\de_{r_j}(\F(\bar w)^{-1}\F(w))=\F_{\bar w}^{r_j}(u),\qquad\text{where }\F_{\bar w}^{r_j}(u):=u\:( \f_{\bar w})^{r_j}(u).
\]
Therefore
\begin{align*}
\lim_{j\to\infty}\Tcurr_j(\omega)
& = C_{n,k} \lim_{j\to\infty} \int_\W \left\langle   {\zeta(\bar w)}\mid \omega(\F^{r_j}_{\bar w}(u))\right\rangle \:d\leb{2n+1-k}(u) \\
& = C_{n,k}  \int_\W \left\langle  {\zeta(\bar w)}\mid \omega(\F_\infty(u))\right\rangle \:d\leb{2n+1-k}(u)
\end{align*}
due to the uniform convergence of $\F^{r_j}_{\bar w}(u) = u( (\f_{\bar w})^{r_j}(u))$ to the graph map $\F_\infty(u):= u\f_\infty (u)$. 

We obtained that the Heisenberg $(2n+1-k)$-current $\Tcurr_\infty$ defined by
\[
\Tcurr_\infty(\omega):=\lim_{j\to\infty}\Tcurr_j(\omega)= C_{n,k}  \int_\W \left\langle   {\zeta(\bar w)}\mid \omega\circ\F_\infty\right\rangle \:d\leb{2n+1-k},\qquad\omega\in\DH^{2n+1-k}
\]
is supported on $\gr_{\f_\infty}$. Since $\partial\Tcurr_j=0$ for every $j$, than also $\partial\Tcurr_\infty=0$. Moreover, $\f_\infty$ is a uniform limit  of uniformly intrinsic Lipschitz maps, hence it is intrinsic Lipschitz and there exists a measurable function $J_{\f_\infty}:\W\to\R$ such that
\[
C^{-1} \leq J_{\f_\infty} \leq C\qtaq \Shaus^{Q-k}\res\gr_{\f_\infty}={\F_\infty}_\#(J_{\f_\infty}\leb{2n+1-k}).
\]
In particular
\[
\Tcurr_\infty(\omega)= C_{n,k}  \int_{\gr_{\f_\infty}} \left\langle  \left. \frac{\zeta(\bar w)}{J_{\f_\infty}({\F_{\infty}}^{-1}(p))}
\,\right|\,\omega(p)\right\rangle \:d\Shaus^{Q-k}(p),\qquad\omega\in\DH^{2n+1-k},
\]
i.e.,
\[
\Tcurr_\infty=\frac{\zeta(\bar w)}{J_{\f_\infty}\circ\F^{-1}_\infty}\:\Shaus^{Q-k}\res\gr_{\f_\infty}.
\]
By~\eqref{eq:222}, the intrinsic Lipschitz map $\f_\infty$ is $t$-invariant: by Lemma~\ref{lem:tinvariantLipschitz} and Proposition~\ref{prop:Euclrectarerect}, $\f_\infty$ is  Euclidean Lipschitz and $\gr_{\f_\infty}$ is locally $\H$-rectifiable of codimension $k$. By Theorem~\ref{teo:correntinormalirettificabili} and~\eqref{eq:333} we deduce that there exists $\eta:\gr_{\f_\infty}\to\R\setminus\{0\}$ such that
\begin{equation}\label{eq:tHgrinftycostante}
\zeta(\bar w)=\eta(p)[t^\H_{\gr_{\f_\infty}}(p)]_\mathcal J\qquad\text{for $\Shaus^{Q-k}$-a.e.~}p\in\gr_{\f_\infty}.
\end{equation}

{\em Step 4: every blow-up is linear.}
We claim that $\f_\infty$ is intrinsic linear. Let $f_{\f_\infty}:\R^{2n-k}\to\R^k$ be defined by $\f_\infty(w)=f_{\f_\infty}(w_H)$ for every $w\in\W$, where the notation $w_H$ is the one introduced in~\eqref{eq:defw_H}; $f_{\f_\infty}$ is Euclidean Lipschitz continuous by Lemma~\ref{lem:tinvariantLipschitz} (i). By Lemma~\ref{lem:tinvariantLipschitz} (iv),  the gradient $\nabla f_{\f_\infty}(w_H)$ is defined for a.e. $w\in\W$ and it uniquely determines $\Tan^\H_{\gr_{\f_\infty}}(\F_\infty(w))$; since $\f_\infty$ is intrinsic linear if and only if $f_{\f_\infty}$ is linear, $\f_\infty$ is intrinsic linear if and only if $\Tan^\H_{\gr_{\f_\infty}}$ is constant on $\gr_{\f_\infty}$. Assume that $\Tan^\H_{\gr_{\f_\infty}}$ is not constant: then, by~\eqref{eq:tHgrinftycostante} and Proposition~\ref{prop:almassimo2}, there exist  two vertical planes $\mathscr P_1,\mathscr P_2$ that are not rank-one connected and such that 
\[
\text{$\Tan^\H_{\gr_{\f_\infty}}(p)\in\{\mathscr P_1,\mathscr P_2\}$ for $\Shaus^{Q-k}$-a.e.~$p\in\gr_{\f_\infty}$.}
\]
In particular, there exist two $k\times(2n-k)$ matrices $M_1,M_2$ such that $\nabla f_{\f_\infty}\in\{M_1,M_2\}$. By Lemma~\ref{lem:tinvariantLipschitz} (ii) and Remark~\ref{rem:applicazionilinearirankone}, the rank of the matrix $M_1-M_2$ is at least 2. However, a well-known result proved in \cite[Proposition 1]{BallJamesARMA87} (see also \cite[Proposition 2.1]{MullerCetraro}) states that, if a Lipschitz map has only two possible gradients $M_1,M_2$ and rank $(M_1-M_2)\geq 2$, then the map is affine. This proves that $f_{\f_\infty}$ is linear and, actually, that
\begin{equation}\label{eq:M1oppureM2}
\text{either $f_{\f_\infty}\equiv M_1$ or $f_{\f_\infty}\equiv M_2$.}
\end{equation}

{\em Step 5: uniqueness of blow-ups.}
We have proved that every blow-up $\f_\infty$ is intrinsic linear. We now prove that the blow-up of $\f$ at $\bar w$ is unique: by Proposition~\ref{prop:differ->convtoTanH} (c), this is equivalent to the intrinsic differentiability of $\f$ at $\bar w$. 
Assume on the contrary that there exist two different blow-ups $\f_\infty^1,\f_\infty^2$ of $\f$ at $\bar w$; 
observing that the matrices $M_1,M_2$ introduced in Step 4  are uniquely determined by $\zeta(\bar w)$ (via the vertical planes $\mathscr P_1,\mathscr P_2$ provided by Proposition~\ref{prop:almassimo2}), we deduce from~\eqref{eq:M1oppureM2} that, possibly renaming $M_1$ and $M_2$,
\[
\f^1_\infty(w)=M_1 w_H\qtaq \f^2_\infty(w)=M_2 w_H\qquad\forall\: w\in\W.
\]
Let $w\in\W$ be such that $\f^1_\infty(w)\neq\f^2_\infty(w)$, i.e., $M_1 w_H\neq M_2 w_H$. By definition of blow-up, there exist two  diverging sequences $(r_j^1)_{j}$, $(r_j^2)_{j}$ such that
\[
(\f_{\bar w})^{r_j^1}\to \f^1_\infty\qtaq (\f_{\bar w})^{r_j^2}\to \f^2_\infty
\]
and, up to passing to suitable subsequences, we can assume without loss of generality that $r_j^1< r_j^2 < r_{j+1}^1$ for every $j$. Let $\de:=d( \f^1_\infty(w),\f^2_\infty(w))>0$; since the map $r\mapsto (\f_{\bar w})^r(w)$ is continuous and bounded (see Remark~\ref{rem:1/2Holder}), for every large enough $j$ we can find $r_j^3\in(r_j^1,r_j^2)$ such that $d((\f_{\bar w})^{r_j^3}(w),\f^1_\infty(w))\geq \de/3$ and $d((\f_{\bar w})^{r_j^3}(w),\f^2_\infty(w))\geq \de/3$. By Remark~\ref{rem:1/2Holder} and Ascoli-Arzel\`a's Theorem, up to passing to a subsequence we have that $(\f_{\bar w})^{r_j^3}\to\psi$ locally uniformly on $\W$ for some $\psi:\W\to\V$. Observe that $d(\psi(w),\f^1_\infty(w))\geq \de/3$ and $d(\psi(w),\f^2_\infty(w))\geq \de/3$; thus, $\psi$ is a blow-up of $\f$ at $\bar w$ (in particular, it is  $t$-invariant) that is different from both $\f^1_\infty$ and $\f^2_\infty$. This contradicts~\eqref{eq:M1oppureM2}, and the proof is concluded.
\end{proof}

We conclude this section with an observation. Let $\f:A\subset\W\to\V$ be intrinsic Lipschitz. 
Theorem~\ref{thm:rademacher}, together with Remark~\ref{rem:Ahlfors}, implies that a tangent\footnote{In the sense of blow-up limits as in Proposition~\ref{prop:differ->convtoTanH} (d).} plane $\Tan^\H_{\gr_\f}$ to $\gr_\f$ exists $\Shaus^{Q-k}$-a.e. on $\gr_\f$: let us  denote by $t^\H_{\gr_\f}(p)\in\bwl^{2n+1-k}\h$  the unit tangent vector associated with $\Tan^\H_{\gr_\f}(p)$ at $p\in\gr_\f$.
Again, $t^\H_{\gr_\f}(p)$ is defined only up to a sign, hence $p\mapsto t^\H_{\gr_\f}(p)$ could  be not even measurable with respect to $\Shaus^{Q-k}\res\gr_\f$. As one can expect, there are however two canonical choices for the orientation: we present below one of the two, the other being of course the opposite one. See also~\cite{CanarecciOrientability} for some related issues. With this consistent choice of  orientation for $\gr_\f$, one will be allowed to  define the $(2n+1-k)$-Heisenberg  current $\curr{\gr_\f}$ canonically associated with $\gr_\f$ by
\begin{equation}\label{eq:defcanonicalcurrLip}
\curr{\gr_\f}(\omega):=\int_{\gr_\f}\langle t^\H_{\gr_\f}\mid\omega\rangle\:d\Shaus^{Q-k},\qquad\omega\in\DH^{2n+1-k}.
\end{equation}
As a further property, we will prove in Proposition~\ref{prop:grf0bdry} that for entire graphs ($A=\W$) the equality $\partial\curr{\gr_\f}=0$ holds. 

Let us fix our choice of $t^\H_{\gr_\f}$. As in Remark~\ref{rem:Rademachersuffpianicanon}, up to a $\H$-linear map $\mathcal L$ we can assume that $\W,\V$ are as in~\eqref{eq:Vfissato} and~\eqref{eq:Wfissato}. In this case our choice for the orientation of $\gr_\f$ is (compare with~\eqref{eq:tSHgrafici})
\[
t^\H_{\gr_\f}(p):=\frac{\nf\F(\F^{-1}(p))}{|\nf\F(\F^{-1}(p))|}\wedge T
\]
where, remembering~\eqref{eq:defcampizeta} and~\eqref{eq:defZETA}, we set
\[
\nf\F(w):=\nf_{k+1}\F(w)\wedge\dots\wedge\nf_{2n}\F(w)\in\bwl_{2n-k}\h_1
\]
and
\[
\nf_i\F(w):=\left(W_i+\sum_{h=1}^k(\nf\f(w))_{hi}X_h\right)(\F(w)),\qquad i=k+1,\dots,2n.
\]
For general subgroups $\W,\V$, our choice corresponds to fixing a unit tangent vector $t^\H_\W$ and declaring that $\langle t^\H_{\gr_\f}(p),t^\H_\W\rangle>0$ for $\Shaus^{Q-k}$-a.e. $p\in\gr_\f$, where $\langle\cdot ,\cdot\rangle$ is the canonical scalar product on multi-vectors.

Let us point out that, when $\W,\V$ are those in~\eqref{eq:Vfissato} and~\eqref{eq:Wfissato}, as in Remark~\ref{rem:Jacob=norm} we have  $J^\f\f=|\nf\F|$; this equality and Theorem~\ref{thm:formulaarea}, that we will prove in the subsequent Section~\ref{sec:applic}, provide the alternative representation
\begin{equation}\label{eq:defcanonicalcurrLip_ALTERNATIVA}
\curr{\gr_\f}(\omega)=C_{n,k}\int_\W\langle \nf\F\wedge T\mid\omega\circ\F\rangle\:d\leb{2n+1-k},\qquad\omega\in\DH^{2n+1-k}.
\end{equation}

\section{Applications}\label{sec:applic}
In this section we provide some consequences of Theorem~\ref{thm:rademacher}. For computational reasons it is useful to fix a convenient distance on $\H^n$:  therefore, in the present section $d$ denotes the distance $d_\infty$ introduced in~\eqref{eq:dinfty}.

We need to fix some notation. 
When $M$ is a matrix we denote by $|M| $ its Hilbert-Schmidt norm. When $M_1,M_2$ are square matrices, inequalities of the form $M_1\geq M_2$ are understood in the sense of bilinear forms; $I$ denotes the identity matrix. Recall also the notation $W_1,\dots,W_{2n}$ introduced in~\eqref{eq:defW} to denote horizontal left-invariant vector fields. Eventually, if $A$ and $B$ are given sets, $A\Delta B:=(A\setminus B)\cup(B\setminus A)$ denotes their symmetric difference.

The proof of the following lemma closely follows the one of the classical Whitney's Extension Theorem, see e.g.~\cite[Theorem~6.10]{EvansGariepyRevised}. A version of  Whitney's Extension Theorem in $\H^n$ has been proved in~\cite[Theorem~6.8]{FSSCMathAnn2001}, see also~\cite[Theorem~3.2.3]{VittonePhDThesis}.

\begin{lemma}\label{lem:quasiLusin}
For every integers $n\geq 1$ and $1\leq k\leq n$ there exist positive constant $\al_0=\al_0(n,k)$ and $c_0=c_0(n,k)$ with the following property. For every $\de>0$ and every intrinsic Lipschitz map $\f:\W\to\V$ whose intrinsic Lipschitz constant $\al$ is not greater than $\al_0$, there exists $f\in C^1_\H(\H^n,\R^k)$ such that
\begin{align}
&|W_i f(p)|\leq c_0\al\quad\text{for every }p\in\H^n\text{ and }i=k+1,\dots,2n\label{eq:1cerchiato}\\
&\big|\mathrm{col}\big[X_1f(p)|\dots|X_k f(p)\big] - I\big|\leq c_0\al\quad\text{for every }p\in\H^n\label{eq:2cerchiato}\\
& \text{the level set $S:=\{p\in\H^n:f(p)=0\}$ is a $\H$-regular submanifold }\label{eq:SHregolare}\\
&\text{$S=\gr_\psi$ for an intrinsic Lipschitz $\psi:\W\to\V$ with Lipschitz constant at most $c_0\al$}\label{eq:SintrinsicLip}\\
& \Shaus^{Q-k}((\gr_\f\,\Delta\, S)\cup\{p\in\gr_\f\cap S:\Tan^\H_{\gr_\f}(p)\neq \Tan^\H_S(p)\})<\de\label{eq:3cerchiato}\\
& \leb{2n+1-k}(\W\setminus\{\f=\psi\text{ and }\nabla^\f\f= \nabla^\psi\psi\})<\de.\label{eq:4cerchiato}
\end{align}
\end{lemma}



\begin{proof}
As in Remark~\ref{rem:Rademachersuffpianicanon}, it is not restrictive to assume that $\W$ and $\V$ are the subgroups defined  in~\eqref{eq:Vfissato} and~\eqref{eq:Wfissato}.
Fix $\de>0$ and  an intrinsic Lipschitz function $\f:\W\to\V$  whose intrinsic Lipschitz constant $\al$ is not greater than $\al_0$; the number $\al_0>0$ will be chosen later. We organize the proof in several steps.

{\em Step 0.} We establish some notation and preliminary facts. By Theorem~\ref{thm:rademacher} there exists a measurable set $A\subset\W$ such that $\leb{2n+1-k}(\W\setminus A)=0$ and $\f$ is intrinsically differentiable at every point of $A$. For every $w\in A$  let us introduce the homogeneous homomorphism $L_w:\H^n\to\R^k$
\begin{equation}\label{eq:defL_w}
L_w(p):=(p_1,\dots,p_k) - \nabla^\f\f(w)(p_{k+1},\dots, p_{2n}),\qquad p=(p_1,\dots,p_{2n+1})\in\H^n
\end{equation}
where  $\nabla^\f\f(w)$  is the intrinsic gradient (identified with a $k\times(2n-k)$ matrix)  introduced in Definition~\ref{def:differenziabilita}. Notice that $L_w$ is constructed in such a way that
\begin{equation}\label{eq:Tan=kerL_w}
\Tan^\H_{\gr_\f}(\F(w)) =\ker L_w.
\end{equation}
Moreover, for every $w\in A$
\begin{equation}\label{eq:stimaalfa}
\begin{split}
& \mathrm{col}[X_1 L_w|\dots|X_k L_w]=I\\
& |W_j(L_w)_i|=|(\nabla^\f\f(w))_{ij}|\leq\al\qquad\forall\:i=1,\dots,k,\ \forall\:j=k+1,\dots,2n
\end{split}
\end{equation}
where as usual $\F(w):=w\f(w)$ is the graph map and, in the last formula, we used Lemma~\ref{lem:nablaffijminoredialfa}. By abuse of notation, we will in the sequel identify each $L_w$  with the $k\times2n$ matrix $\nabla_\H L_w=[\ I\ |\ -\nabla^\f\f(w)\ ]$, $I$ being the $k\times k$ identity matrix.  
We also introduce the homogeneous homomorphism $\overline L:\H^n\to\R^k$ defined by
\[
\overline L(p)=(p_1,\dots,p_k),\qquad p=(p_1,\dots,p_{2n+1})\in\H^n
\]
and we  identify $\overline L$ and the $k\times2n$ matrix $[\, I\,|\,0\,]$. We can estimate the  Hilbert-Schmidt norm of $L_w-\overline L$ by
\begin{equation}\label{eq:stimaLw-Lbar}
|L_w-\overline L|\leq \sqrt{k(2n-k)}\,\al\qquad\text{for every }w\in A.
\end{equation}
Let us observe that~\eqref{eq:Tan=kerL_w}, the
Lipschitz continuity of $L_w$ and Proposition~\ref{prop:differ->convtoTanH} (and in particular~\eqref{eq:convtoTanH})
imply that for every $w\in A$
\begin{equation}\label{eq:starstarstar}
\lim_{r\to 0^+} \left(\sup\left\{\frac{L_w(\F(w)^{-1}p)}{d(\F(w),p)}: p\in\gr_\f\cap B(\F(w),r)\right\}\right)=0.
\end{equation}

By Lusin's Theorem there exists a closed set $B\subset A$ such that $\leb{2n+1-k}(\W\setminus B)<\eta$ and $\nabla^\f\f|_B$ is continuous; the number $\eta$ will be chosen later in Step 8 depending on $\de,n$ and $k$. By the Severini-Egorov Theorem, there exists a closed set $C\subset B$ such that $\leb{2n+1-k}(\W\setminus C)<\eta$ and the convergence in~\eqref{eq:starstarstar} is uniform on compact subsets of $C$. 

{\em Step 1}. Define the closed set $F:=\F(C)$ and let $U$ be the open set $U:=\heis^n\setminus F$; for every $p\in\H^n$ we set 
\[
r_p:=\frac1{20}d(p,F)=\frac1{20}\inf\{d(p,q):q\in F\}.
\]
We are going to use a variant of the classical $5r$-covering argument (see e.g.~\cite[Theorem~1.2]{Heinonen}), which  cannot be utilized here since the radii of the balls we use are not uniformly bounded. By Zorn's lemma, there exists a maximal set $\C\subset U$ such that the family of balls $\{B(p,r_p):p\in\C\}$ are pairwise disjoint; we claim that
\[
U=\bigcup_{p\in \C} B(p,5r_p).
\]
The inclusion $\supset$ in the formula above is clear by the definition of $r_p$; assume that the reverse inclusion does not hold, i.e., that there exists $q\in U\setminus \bigcup_{p\in \C} B(p,5r_p)$. By maximality of $\C$, there exists $p\in\C$ and $ p'\in U$ such that $ p'\in B(p,r_p)\cap B(q,r_q)$ and in particular
\[
20 r_q=d(q,F)\leq d(q, p') +d(p', p)+ d(p,F) \leq r_q+21 r_p.
\]
It follows that $r_q\leq \frac{21}{19}r_p$ and in turn
\[
q\in B(q, r_q)\subset B(p,r_p+r_q)\subset B(p,\tfrac{40}{19}r_p)\subset B(p,5r_p),
\]
a contradiction.

{\em Step 2}.  For every $q\in U$ we define
\[
\C_q:=\{p\in \C: B(q,10r_q)\cap B(p,10r_p)\neq\emptyset\}.
\]
We claim that $\# \C_q\leq 129^Q$ and $1/3\leq r_q/r_p\leq 3$ for every $p\in\C_q$. In fact, if $p\in\C_q$ one has
\[
|r_p-r_q|\leq \frac{1}{20}d(p,q)\leq\frac{1}{20}(10r_p+10r_q)=\frac{1}{2}(r_p+r_q).
\]
This implies that $r_p\leq 3r_q$ and $r_q\leq 3r_p$, and the bounds on $r_q/r_p$ follow.

In addition, for every $p\in\C_q$ we have
\[
d(p,q)+r_p\leq 10 (r_p+r_q)+r_p\leq 43 r_q
\]
and in particular $B(p,r_q/3)\subset B(p,r_p)\subset B(q,43r_q)$. The balls $\{B(p,r_q/3):p\in \C_q\}$ are then pairwise disjoint and contained in $B(q,43r_q)$, therefore
\[
\# \C_q\leq\frac{\leb{2n+1}(B(q,43r_q))}{\leb{2n+1}(B(0,r_q/3))}=129^Q,
\]
as claimed.

{\em Step 3}. Now let $\mu:\R\to\R$ be a smooth non-increasing function such that
$$
0\leq\mu\leq 1,\qquad \mu(t)=1\text{ for }t\leq 2^{1/4},\qquad\mu(t)=0\text{ for }t\geq 2.
$$
For every $p\in\C$ define
\[
g_p(q):=\mu\left(\frac{d_{K}(p,q)}{5r_p}\right),
\]
where, as in~\eqref{eq:dKoranyi}, $d_{K}$ is the  homogeneous Kor\'anyi (or Cygan-Kor\'anyi) distance.
Observe that $d_K(p,\cdot)$ is smooth on $\H^n\setminus\{p\}$. 
Being a homogeneous distance (see~\cite{Cygan}), $d_{K}$ is globally equivalent to $d$ and more precisely
$$
d(p',p'')\leq d_{K}(p',p'')\leq 2^{1/4}d(p',p'')\qquad\forall\:p',p''\in\H^n.
$$
It follows that $g_p\in\ci^\infty_c(\heis^n)$, $0\leq g_p\leq 1$ and
\begin{eqnarray}
&& g_p\equiv 1\quad\text{on }B(p,5r_p)\nonumber\\
&& g_p\equiv 0\quad\text{on }\heis^n\setminus B(p,10r_p).\label{gPfuoriCQ}
\end{eqnarray} 
Moreover there is a constant $M=M(n)>0$ such that $|\nabla_\H g_p|\leq M/r_p$; by Step 2, $|\nabla_\H g_p(q)|\leq 3M/r_q$ whenever $p\in \C_q$. Thanks to~\eqref{gPfuoriCQ} one has $g_p(q)=0$ if $p\in\C\setminus \C_q$, hence
\begin{equation}\label{derorizzgP}
|\nabla_\H g_p(q)|\leq 3M/r_q\qquad\text{for all }q\in\heis^n\text{ and }p\in\C.
\end{equation}
Define $\si(q):=\sum_{p\in\C} g_p(q)$ for every $q\in\heis^n$. By~\eqref{gPfuoriCQ} again, one obtains that $g_p\equiv 0$ on $B(q,10r_q)$ whenever $p\notin \C_q$,  so
$$
\si(q')=\sum_{p\in \C_q}g_p(q')\qquad\text{for every $q\in U$  and }q'\in B(q,10r_q).
$$
Observe that $\si\geq 1$ on $U$; in fact, for every $q\in U$ there exists $\tilde p$ such that $q\in B(\tilde p,5r_{\tilde p})$ and in particular $\si(q)\geq g_{\tilde p} (q)=1$. Using~\eqref{derorizzgP} and the inequality  $\# \C_q<129^{Q}$, we deduce that $\si\in\ci^\infty(U)$ and there is a constant $M'=M'(n)>0$ such that
$$
|\nabla_\H\si(q)|\leq\frac{M'}{r_q}\qquad\text{for all }q\in U.
$$
Now we define a partition of the unity $\{v_p\}_{p\in\C}$ on $U$ subordinate to the covering $\{B(p,10r_p):p\in\C)\}$ by setting
$$v_p(q):=\frac{g_p(q)}{\si(q)}.$$
Notice that $v_p\in\ci^\infty_c(U)$ and $\nabla_\H v_p=\frac{\nabla_\H g_p}{\si}-\frac{g_p\,\nabla_\H\si}{\si^2}$; eventually, we deduce that for a suitable $M''=M''(n)>0$ 
\begin{equation}\label{proprietadivP}
\sum_{p\in\C}v_p(q)=1,\qquad\sum_{p\in\C}\nabla_\H v_p(q)= 0\qquad\text{and}\qquad|\nabla_\H v_p(q)|\leq \frac{M''}{r_q}
\end{equation}
for every $q\in U$ and every $p\in\C$.

{\em Step 4}. For every $p\in\C$ choose $\bar q_p\in F$ such that $d(p,\bar q_p)=d(p,F)$; we then define $ f:\heis^n\to\R$ as 
$$
 f(q):=\left\{\begin{array}{ll}
0 & \text{if }q\in F\vspace{.2cm}\\
\displaystyle\sum_{p\in\C} v_p(q) L_{\F^{-1}(\bar q_p)}({\bar q_p\!}^{-1}q)=\sum_{p\in\C_q} v_p(q) L_{\F^{-1}(\bar q_p)}({\bar q_p\!}^{-1}q)\quad & \text{if }q\in U.
\end{array}\right.
$$
Notice that $f\in\ci^\infty(U)$ and 
\begin{equation}\label{eq:gradhf}
\gradh  f(q)=\sum_{p\in\C_q}\Big[  L_{\F^{-1}(\bar q_p)}({\bar q_p\!}^{-1}q) \otimes \nabla_\H v_p(q)  + v_p(q) L_{\F^{-1}(\bar q_p)}
\Big]\qquad\text{on $U$,}
\end{equation}
where we recall the identification between homogeneous homomorphisms $L_w$ and $k\times 2n$ matrices.

{\em Step 5}. We claim that $\gradh  f(q)=L_{\F^{-1}(q)}$ for every $q\in F$. For every such $q$ we define the compact set $H:=F\cap \overline{B(q,1)}$ and, for  $\eta>0$, 
\begin{align*}
\upsilon(\eta):=\ &\sup\left\{ \left|\frac{L_{\F^{-1}(p)}(p^{-1}p')}{d(p,p')}\right|:p,p'\in H,\ 0<d(p,p')\leq\eta \right\}\\
&+ \sup\left\{|L_{\F^{-1}(p)}-L_{\F^{-1}(p')}|:p,p'\in H,\ d(p,p')\leq\eta\right\}.
\end{align*}
The map $p\mapsto L_{\F^{-1}(p)}$ is uniformly continuous on $H$: in fact, $\F^{-1}:H\to C$ is uniformly continuous and, since $\nabla^\f\f$ is continuous on $C$, also $w\mapsto L_w$ is uniformly continuous on the compact set $\F^{-1}(H)$. This, together with the fact that the convergence in~\eqref{eq:starstarstar} is  uniform on compact subsets of $C$, implies that
\begin{equation}\label{eqstarstarstar}
\lim_{\eta\to 0^+}\upsilon(\eta)=0.
\end{equation}

For every $q'\in H$ one has
\begin{equation}\label{casoQ'inH}
\left| f(q')- f(q) - L_{\F^{-1}(q)}(q^{-1}q')\right| = \left| L_{\F^{-1}(q)}(q^{-1}q')\right|\leq \upsilon(d(q,q'))\:d(q,q').
\end{equation}
Instead, for every $q'\in U$ one has
\begin{equation}\label{whitcalcabove}
\begin{split}
&\phantom{=} \left| f(q')- f(q)  - L_{\F^{-1}(q)}(q^{-1}q')\right|\\
&= \left| f(q')-  L_{\F^{-1}(q)}(q^{-1}q')\right|\\
&\leq \sum_{p\in\C_{q'}} v_p(q') |L_{\F^{-1}(\bar q_p)}({\bar q_p\!}^{-1}q') - L_{\F^{-1}(q)}(q^{-1}q')|\\
&\leq \sum_{p\in\C_{q'}} v_p(q')\Big[ |L_{\F^{-1}(\bar q_p)}({\bar q_p\!}^{-1}q') - L_{\F^{-1}(\bar q_p)}(q^{-1}q')|+|(L_{\F^{-1}(\bar q_p)} - L_{\F^{-1}(q)})(q^{-1}q')|\Big]\\
&\leq  \sum_{p\in\C_{q'}} v_p(q')\Big[ |L_{\F^{-1}(\bar q_p)}({\bar q_p\!}^{-1}q)|+|L_{\F^{-1}(\bar q_p)}- L_{\F^{-1}(q)}|d(q,q')\Big].
\end{split}
\end{equation}
Since $20\,r_{q'}\leq d(q,q')$, for every $p\in \C_{q'}$ we obtain
\begin{equation}\label{eq:woody}
\begin{aligned}
d(q,\bar q_p)&\leq  d(q,p) + d(p,\bar q_p) \leq 2d(q,p)
\leq  2(d(q,q')+d(q',p))\\
 &\leq 2(d(q,q')+10(r_{q'}+r_p))
\leq  2(d(q,q')+40r_{q'})
\leq  6\,d(q,q').
\end{aligned}
\end{equation}
Therefore, when $d(q,q')\leq 1/6$ we  have $\bar q_p\in H$ for every $p\in\C_{q'}$; using~\eqref{whitcalcabove} and Step 2 we then obtain
$$
\left| f(q')- f(q)  - L_{\F^{-1}(q)}(q^{-1}q')\right|
 \leq 7\: \upsilon(6\,d(q,q'))\:d(q,q').
$$
Recalling~\eqref{eqstarstarstar}, this inequality and~\eqref{casoQ'inH} eventually give
$$
\left|f(q')- f(q)  - L_{\F^{-1}(q)}(q^{-1}q')\right| = o(d(q,q'))\qquad\text{as }q'\to q
$$
and the claim follows.

{\em Step 6}. Let us prove that  $f\in\ci^1_\heis(\H^n)$; since $f\in C^\infty(U)$, it suffices to prove that $\gradh f$ is continuous on $F$.  We fix $q\in F$ and $q'\in\heis^n$ with $d(q,q')\leq 1/6$ and we define $H$ and $\upsilon$  as in Step 5. If $q'\in F$ we have by Step 5
$$
|\gradh f(q')-\gradh f(q)| = |L_{\F^{-1}(p)}-L_{\F^{-1}(p')}|\leq \upsilon(d(q,q')).
$$
If $q'\in U$ we choose $\bar q\in F$ such that $d(q',\bar q)=d(q',F)=20r_{q'}$ to get
\begin{equation}\label{eqprimadi****}
\begin{split}
|\gradh f(q')-\gradh f(q)|
&\leq |\gradh f(q')-L_{\F^{-1}(\bar q)}| + |L_{\F^{-1}(\bar q)}-L_{\F^{-1}( q)}|\\
&\leq|\gradh f(q')-L_{\F^{-1}(\bar q)}| + \upsilon(2d(q,q')),
\end{split}
\end{equation}
where in the last inequality we used the estimate
$$
d(q,\bar q)\leq d(q,q')+d(q',\bar q)\leq 2d(q,q')
$$
and in particular the fact that $\bar q\in H$.
Using~\eqref{proprietadivP} and~\eqref{eq:gradhf} we estimate 
\begin{equation}\label{eq****}
\begin{split}
&|\gradh f(q')-L_{\F^{-1}(\bar q)}|\\
=\,& \bigg|\sum_{p\in\C_{q'}}\Big[ L_{\F^{-1}(\bar q_p)}({\bar q_p\!}^{-1}q') \otimes \nabla_\H v_p(q')  + v_p(q') L_{\F^{-1}(\bar q_p)} \Big] - \sum_{p\in\C_{q'}}v_p(q')L_{\F^{-1}(\bar q)} \bigg|\\
\leq\, & \sum_{p\in\C_{q'}}\big| L_{\F^{-1}(\bar q_p)}({\bar q_p\!}^{-1}\bar q) \otimes \nabla_\H v_p(q')    \big|
 +\sum_{p\in\C_{q'}}\big| \big(L_{\F^{-1}(\bar q_p) }-L_{\F^{-1}(\bar q)}\big)({\bar q\,}^{-1}q')\otimes  \nabla_\H v_p(q')    \big|\\
&+ \sum_{p\in\C_{q'}}v_p(q')|L_{\F^{-1}(\bar q_p)}-L_{\F^{-1}(\bar q)} |\\
\leq\, & \frac{M''}{r_{q'}}\sum_{p\in\C_{q'}} \upsilon(d(\bar q,\bar q_p))d(\bar q,\bar q_p)
+ \frac{M''}{r_{q'}}\sum_{p\in\C_{q'}}\upsilon(d(\bar q,\bar q_p))d(q',\bar q)
+\sum_{p\in\C_{q'}}\upsilon(d(\bar q,\bar q_p))
\end{split}
\end{equation}
where, in the last inequality, we used the fact that, as in~\eqref{eq:woody},  $\bar q_p\in H$. Recall that $r_p\leq 3 r_{q'}$ for all $p\in\C_{q'}$; therefore, for every such $p$ one can estimate 
\begin{equation}\label{eqwhitterzultima}
d(\bar q, \bar q_p) \leq d(\bar q,q')+d(q',p)+d(p,\bar q_p) \leq 20r_{q'}+10(r_{q'}+r_p)+20r_p
\leq 120 r_{q'}
\end{equation}
and in particular
\begin{equation}\label{eqwhitterzultimaBIS}
d(\bar q, \bar q_p) \leq 6d(q',\bar q)\leq 6 d(q',q)
\end{equation}
for every $p\in\C_{q'}$. Combining~\eqref{eq****}, ~\eqref{eqwhitterzultima} and~\eqref{eqwhitterzultimaBIS} one finds 
$$
|\gradh f(q')-L_{\F^{-1}(\bar q)}|
\leq (120\,M''+20M'' + 1)\sum_{p\in\C_{q'}}\upsilon(d(\bar q,\bar q_p))\leq
(140\,M'' + 1)(129)^Q\upsilon(6d(q',q))
$$
which, together with~\eqref{eqprimadi****}, gives
$$
|\gradh f(q')-\gradh f(q)|\to 0\qquad\text{as }q'\to q,
$$
as claimed.

{\em Step 7.} Let $\al_1:=\min\{1/2,\bar\al(1/2,n,k)\}$, where $\bar\al(1/2,n,k)$ is the number provided by Remark~\ref{rem:epsilonalfa}. We claim that there exists $c_1=c_1(n,k)>0$ such that, if $\al\leq\al_1$, then
\begin{equation}\label{eq:gradhfvicinoLbar}
|\gradh f(q) - \overline L|\leq c_1\,\al \qquad\text{ for every }q\in\H^n.
\end{equation}
We will later choose $\al_0$ and $c_0$ in such a way that $\al_0\leq\al_1$ and $c_0\geq c_1$; in this way, statements~\eqref{eq:1cerchiato} and~\eqref{eq:2cerchiato} are immediate consequences of~\eqref{eq:gradhfvicinoLbar}.

If $q\in F$,~\eqref{eq:gradhfvicinoLbar} follows from Step 5 and~\eqref{eq:stimaLw-Lbar}, hence we can assume $q\in U$. By~\eqref{eq:gradhf} and~\eqref{proprietadivP}
\begin{equation}\label{eq:gradhfmenoLbar}
\begin{split}
\gradh f(q) - \overline L 
=&\sum_{p\in\C_q}  L_{\F^{-1}(\bar q_p)}({\bar q_p\!}^{-1}q) \otimes \nabla_\H v_p(q)  + \sum_{p\in\C_q} v_p(q) (L_{\F^{-1}(\bar q_p)}-\overline L)\\
=& \sum_{p\in\C_q} \Big( L_{\F^{-1}(\bar q_p)}({\bar q_p\!}^{-1}q)-(q_\V-\f(q_\W)) \Big) \otimes \nabla_\H v_p(q)\\
&  + \sum_{p\in\C_q} v_p(q) (L_{\F^{-1}(\bar q_p)}-\overline L),
\end{split}
\end{equation}
where we recall that, for every $p\in\H^n$, we write $p=p_\W p_\V$ for suitable (unique)  $p_\W\in\W$ and $p_\V\in\V$. From~\eqref{eq:stimaLw-Lbar} we obtain
\begin{equation}\label{eq:gradhfmenoLbarstimasecondopezzo}
\bigg| \sum_{p\in\C_q} v_p(q) (L_{\F^{-1}(\bar q_p)}-\overline L)\bigg| \leq \sqrt{k(2n-k)}\,\al.
\end{equation}
We now estimate the first addend in the rightmost side of~\eqref{eq:gradhfmenoLbar}. First, we notice that for every $p\in\C_q$
\begin{equation}\label{eq:gradhfmenoLbarstimaprimopezzoAAA}
\begin{split}
|L_{\F^{-1}(\bar q_p)}({\bar q_p\!}^{-1}q) - \overline L({\bar q_p\!}^{-1}q)|
& \leq  |L_{\F^{-1}(\bar q_p)} - \overline L| \:d(q, \bar q_p)\\
&\leq \sqrt{k(2n-k)}\,\al\:(d(q,p)+d(p,\bar q_p))\\
& \leq \sqrt{k(2n-k)}\,\al\:(10(r_p+r_q)+20 r_p)\\
&\leq 100 \sqrt{k(2n-k)}\,\al\:r_q.
\end{split}
\end{equation}
Noticing that $\overline L({\bar q_p\!}^{-1}q)={({\bar q_p\!}^{-1}q)}_\V$ we deduce that
\[
| \overline L({\bar q_p\!}^{-1}q)- (q_\V-\f(q_\W))| 
= |q_\V-(\bar q_p)_\V+\f(q_\W)-q_\V|
=|\f(q_\W)-(\bar q_p)_\V|
= |\f(q_\W)-\f((\bar q_p)_\W)| 
\]
and by Remark~\ref{rem:sealfaminore1/2}
\begin{equation*}
\begin{split}
| \overline L({\bar q_p\!}^{-1}q)- (q_\V-\f(q_\W))| 
&\leq 2\al d(\bar q_p,\F(q_\W))\\
&\leq 2\al(d(\bar q_p,p)+d(p,q)+d(q,\F(q_\W)))\\
&=2\al(d(\bar q_p,p)+d(p,q)+|q_\V-\f(q_\W)|)\\
&\leq2\al(20 r_p + 10(r_p+r_q) + 2 d(q,\gr_\f) ), 
\end{split}
\end{equation*}
where in the last inequality we used Remark~\ref{rem:epsilonalfa} with $\ep=1/2$. Recalling that $r_p\leq 3r_q$ for every $p\in\C_q$ we deduce
\begin{equation}\label{eq:gradhfmenoLbarstimaprimopezzoBBB}
| \overline L({\bar q_p\!}^{-1}q)- (q_\V-\f(q_\W))|\leq 280\,\al\, r_q
\end{equation}
and the claim~\eqref{eq:gradhfvicinoLbar} is now a consequence of~\eqref{eq:gradhfmenoLbar},~\eqref{eq:gradhfmenoLbarstimasecondopezzo},~\eqref{eq:gradhfmenoLbarstimaprimopezzoAAA},~\eqref{eq:gradhfmenoLbarstimaprimopezzoBBB} and the last statement in~\eqref{proprietadivP}.

{\em Step 8.} Denote  by $K>0$ a constant with the property (see e.g.~\cite{GNLip}) that
\[
|g(p)-g(q)|\leq K \Big(\sup_{\H^n} |\gradh g|\Big)d(p,q)\qquad\text{for every $g\in C^1_\H(\H^n,\R^k)$ and  }p,q\in\H^n.
\]
We now fix $c_0:=\max\{c_1,2Kc_1\}$ and $\al_0:=\min\{\al_1,(2c_1)^{-1}\}$. The inequality~\eqref{eq:gradhfvicinoLbar} now implies that, if $\al\leq\al_0$, then col$[X_1f(p)|\dots|X_k f(p)]\geq \tfrac12I$, hence $\gradh f(p)$ has rank $k$ for every $p\in\H^n$. In particular, the level set $S=\{p\in\H^n:f(p)=0\}$ is a $\H$-regular submanifold and statement~\eqref{eq:SHregolare} follows. 

Theorem~\ref{teo:equivdefLip} ensures that  $S=\gr_\psi$ for an intrinsic Lipschitz $\psi:\W\to\V$. 
We claim that the intrinsic Lipschitz constant of $\psi$ is at most $2Kc_1\al$: to this aim, fix $w\in\W$ and $v\in\V\setminus\{0\}$ such that $\|w\|_\H< \|v\|_\H/(2Kc_1\al)$, so that $wv\in C_\beta$ for some $\beta>2Kc_1\al$. For every $q\in\H^n$  we have $\overline L(qw)=\overline L(q)$ and in particular
\begin{align*}
|f(qw)-f(q)|&= |(f-\overline L)(qw)-(f-\overline L)(q)|\leq K\Big(\sup_{\H^n}|\gradh f-\overline L|\Big)\|w\|_\H< \|v\|_\H/2,
\end{align*}
the last inequality following from~\eqref{eq:gradhfvicinoLbar}. Now, for every $p\in S=\gr_\psi$ we have $f(p)=0$, hence
\begin{align*}
|f(pwv)| & =|f(pwv)-f(p)|\\
&\geq |f(pwv)-f(pw)|-|f(pw)-f(p)|\\
&> \left\langle f(pwv)-f(pw),\frac{v}{|v|}\right\rangle -\frac{\|v\|_\H}{2}\ \geq\ 0,
\end{align*}
where we used Remark~\ref{rem:coercivoderivate}. This proves that $f(pwv)\neq 0$, hence $\gr_\psi\cap pC_\be=\{p\}$ for every $p\in\gr_\psi$ and all $\be>2Kc_1\al$. This implies that the intrinsic Lipschitz constant of $\psi$ is at most $2Kc_1\al$ and statement~\eqref{eq:SintrinsicLip} follows because $c_0\geq 2K c_1$.

{\em Step 9.} By construction, for every $w\in C$ we have $\F(w)\in F\subset\{f=0\}$, hence $\f(w)=\psi(w)$ and
\[
\Tan^\H_{S}(\F(w))=\ker \nabla_\H f(\F(w))=\ker L_w \stackrel{\eqref{eq:Tan=kerL_w}}{=}\Tan^\H_{\gr_\f}(\F(w)).
\]
In particular, the inclusion 
\[
C\subset \{w\in\W:\f(w)=\psi(w)\text{ and }\nabla^\f\f(w)= \nabla^\psi\psi(w)\}
\]
holds, and~\eqref{eq:4cerchiato} follows provided $\eta<\delta$. The  inclusion above also guarantees that
\begin{equation}\label{eq:sfcp1}
(\gr_\f\,\Delta\, S)\cup\{p\in\gr_\f\cap S:\Tan^\H_{\gr_\f}(p)\neq \Tan^\H_S(p)\}\subset \F(\W\setminus C)\cup \Psi(\W\setminus C),
\end{equation}
where $\Psi$ is the graph map $\Psi(w):=w\psi(w)$. The intrinsic Lipschitz constant of $\f$ and $\psi$ are both bounded by $\max\{\al_0,c_0\al_0\}$, which depends only on $n$ and $k$; then, by Remark~\ref{rem:Ahlfors}, there exists  $\kappa=\kappa(k,n)>0$ such that
\begin{equation}\label{eq:sfcp2}
\Shaus^{Q-k}\res\gr_\f
\leq
\kappa\: \F_\#(\leb{2n+1-k}\res\W)
\qtaq
\Shaus^{Q-k}\res\gr_\psi
\leq
\kappa\: \Psi_\#(\leb{2n+1-k}\res\W).
\end{equation}
Statements~\eqref{eq:sfcp1} and~\eqref{eq:sfcp2} imply that
\[
\Shaus^{Q-k}((\gr_\f\,\Delta\, S)\cup\{p\in\gr_\f\cap S:\Tan^\H_{\gr_\f}(p)\neq \Tan^\H_S(p)\})\leq 2\kappa\eta
\]
and~\eqref{eq:3cerchiato} follows provided we also impose $\eta<(2\kappa)^{-1}\de$. This concludes the proof.
\end{proof}

The following Theorem~\ref{thm:Lusin} is one of the main results of this section; recall that it implies Theorem~\ref{thm:Lusinbreve}.

\begin{theorem}\label{thm:Lusin}
Let $A\subset\W$ be an open set and $\f:A\to\V$ an intrinsic Lipschitz function. Then, for every $\ep>0$ there exists an intrinsic Lipschitz function $\psi:A\to\V$ such that $\gr_\psi$ is a $\H$-regular submanifold and 
\begin{align}
& \Shaus^{Q-k}((\gr_\f\,\Delta\, \gr_\psi)\cup\{p\in\gr_\f\cap \gr_\psi:\Tan^\H_{\gr_\f}(p)\neq \Tan^\H_{\gr_\psi}(p)\})<\ep\label{eq:intersezionegrossaS}\\
& \leb{2n+1-k}(A\setminus\{w\in A:\f(w)=\psi(w)\text{ and }\nabla^\f\f(w)= \nabla^\psi\psi(w)\})<\ep.\label{eq:intersezionegrossaW}
\end{align}
\end{theorem}
\begin{proof}
As in Remark~\ref{rem:Rademachersuffpianicanon}, it is not restrictive to assume that $\W$ and $\V$ are the subgroups defined  in~\eqref{eq:Vfissato} and~\eqref{eq:Wfissato}.
By Theorem~\ref{thm:estensione} we can without loss of generality assume that $A=\W$. Let $\al_0,c_0$ be the constants provided by Lemma~\ref{lem:quasiLusin}; up to reducing $\al_0$, we can assume that $\al_0c_0\leq 1/2$. 

Let $\al$ be the intrinsic Lipschitz constant of $\f$;  if $\al\leq\al_0$, the statement directly follows from Lemma~\ref{lem:quasiLusin}. We then assume $\al>\al_0$ and define $\la:=\al_0/\al<1$. Let us consider the $\H$-linear isomorphism $\mathcal L:\H^n\to\H^n$ defined (if $k<n$) by
\[\mathcal L(x',x'',y',y'',t):=\Big(\la x',x'',\frac{y'}\la,y'',t\Big)
\]
for every $(x',x'',y',y'',t)\in\R^k\times\R^{n-k}\times\R^k\times\R^{n-k}\times\R\equiv\H^n$,  and (if $k=n$) by $\mathcal L(x,y,t):=(\la x,y/\la,t)$ for every $(x,y,t)\in\R^n\times\R^n\times\R\equiv\H^n$. We discuss only the case $k<n$, as the case $k=n$ can be treated with some modification in the notation only.

We claim that $C_{\al_0}\subset \mathcal L(C_\al)$. To this aim, fix $w=(0,x'',y',y'',t)\in\W$ and $v=(x',0,0,0,0)\in\V$ such that $wv\in C_{\al_0}$, i.e., $\|w\|_\H\leq \|v\|_\H/\al_0$; then 
\begin{align*}
\|\mathcal L^{-1}(w)\|_\H &= \Big\|\Big(0,x'',\la y',y'',t\Big)\Big\|_\H \leq \|(0,x'',y',y'',t)\|_\H =\|w\|_\H\\
&\leq \frac{1}{\al_0}\|v\|_\H =\frac{1}{\al_0}\|(x',0,0,0,0)\|_\H=\frac{\la}{\al_0}\Big\|\Big(\frac{x'}{\la},0,0,0,0\Big)\Big\|_\H =\frac1\al\|\mathcal L^{-1}(v)\|_\H. 
\end{align*}
This proves that $wv\in\mathcal L(C_\al)$, as claimed.

For every $p\in \mathcal L(\gr_\f)$ we have
\[
\mathcal L(\gr_\f) \cap pC_{\al_0}\subset  \mathcal L(\gr_\f) \cap p\mathcal L( C_\al) = \mathcal L(\gr_\f\cap\mathcal L^{-1}(p) C_\al) =\mathcal L(\mathcal L^{-1}(p))=p,
\]
hence (see Remark~\ref{rem:nointersezalloragrafico}) the set $\mathcal L(\gr_\f)$ is  the intrinsic Lipschitz graph of some function $ \f_0:\W\to\V$ with intrinsic Lipschitz constant at most $\al_0$. By Lemma~\ref{lem:quasiLusin}, there exists $f_0\in C^1_\H(\H^n,\R^k)$ such that
\begin{align}
&|W_i f_0(p)|\leq 1/2\quad\text{for every }p\in\H^n\text{ and }i=k+1,\dots,2n\label{eq:citareloc1}\\
&\mathrm{col}\Big[X_1f_0(p)|\dots|X_k f_0(p)\Big]\geq \frac12 I\quad\text{for every }p\in\H^n\label{eq:citareloc2}\\
& \text{the level set $S_0:=\{p\in\H^n:f_0(p)=0\}$ is a $\H$-regular submanifold}\nonumber\\
&\text{$S_0=\gr_{\psi_0}$ for an intrinsic Lipschitz $\psi_0:\W\to\V$}\nonumber\\
& \Shaus^{Q-k}((\gr_{\f_0}\,\Delta\, S_0)\cup\{p\in\gr_{\f_0}\cap S_0:\Tan^\H_{\gr_{\f_0}}(p)\neq \Tan^\H_{S_0}(p)\})<\ep/M^{Q-k}\label{eq:citareloc3}\\
& \leb{2n+1-k}(\W\setminus\{w\in\W:\f_0(w)=\psi_0(w)\text{ and }\nabla^{\f_0}\f_0(w)= \nabla^{\psi_0}\psi_0(w)\})<\ep/M^{Q-k},\label{eq:citareloc4}
\end{align}
where $M>0$ denotes the Lipschitz constant of $\mathcal L^{-1}:\H^n\to\H^n$. The function $f:=f_0\circ\mathcal L$ is of class $C^1_\H(\H^n,\R^k)$ and
\begin{align*}
& X_if(p)=\la (X_i f_0)(\mathcal L(p))\quad\text{for every }i=1,\dots,k\\
& Y_if(p)=\frac1\la (Y_i f_0)(\mathcal L(p))\quad\text{for every }i=1,\dots,k\\
& W_if(p)=(W_i f_0)(\mathcal L(p))\quad\text{whenever }k+1\leq i\leq n\text{ or }n+k+1\leq i\leq 2n
\end{align*}
so that, by~\eqref{eq:citareloc1} and~\eqref{eq:citareloc2},
\begin{align*}
&|W_i f(p)|\leq \frac{1}{2\la}\quad\text{for every }p\in\H^n\text{ and }i=k+1,\dots,2n\\
&\mathrm{col}\Big[X_1f(p)|\dots|X_k f(p)\Big] \geq \frac\la 2 I\quad\text{for every }p\in\H^n.
\end{align*}
By Theorem~\ref{teo:equivdefLip} and Remarks~\ref{rem:coercivoderivate} and~\ref{rem:graficoglobale}, the level set $\{f=0\}$, which is a $\H$-regular submanifold, is the intrinsic graph of an intrinsic Lipschitz function $\psi:\W\to\V$ whose Lipschitz constant can be estimated  in terms of $\la$ (i.e., of $\al$) only, see Remark~\ref{rem:dipendenzaalfa}. Moreover, the equalities
\begin{multline*}
 (\gr_\f\,\Delta\, \gr_\psi)\cup\{p\in\gr_\f\cap \gr_\psi:\Tan^\H_{\gr_\f}(p)\neq \Tan^\H_{\gr_\psi}(p)\} \\
=\mathcal L^{-1}\big( (\gr_{\f_0}\,\Delta\, \gr_{\psi_0})\cup\{p\in\gr_{\f_0}\cap \gr_{\psi_0}:\Tan^\H_{\gr_{\f_0}}(p)\neq \Tan^\H_{\gr_{\psi_0}}(p)\}  \big)
\end{multline*}
and 
\[
\{\f=\psi\text{ and }\nabla^\f\f= \nabla^\psi\psi\}
= \mathcal L^{-1}\big( \{\f_0=\psi_0\text{ and }\nabla^{\f_0}\f_0= \nabla^{\psi_0}\psi_0\}\big)
\]
hold. 
Statements~\eqref{eq:intersezionegrossaS} and~\eqref{eq:intersezionegrossaW} now follow from~\eqref{eq:citareloc3} and~\eqref{eq:citareloc4}, and the proof is accomplished.
\end{proof}

\begin{remark}\label{rem:Lusincostantecontrollata}
Taking~\eqref{eq:SintrinsicLip} into account, a further outcome of the proof of Theorem~\ref{thm:Lusin} is the existence of a function $u:(0,+\infty)\to(0,+\infty)$ such that, if $A$, $\f$ and $\ep$ are as in the statement of Theorem~\ref{thm:Lusin} and the intrinsic Lipschitz constant of $\f$ is $\al>0$, then the function $\psi$ provided by Theorem~\ref{thm:Lusin} has intrinsic Lipschitz constant at most $u(\al)$.
\end{remark}

A first consequence of Theorem~\ref{thm:Lusin} is the equivalence between the two possible notions of $\H$-rectifiability; this was already stated in \S\ref{subsec:Hreg,Hrettif}.

\begin{corollary}\label{cor:HrectifiableC1vsLip}
Let $\f:A\to \V$ be an intrinsic Lipschitz map defined on some subset $A\subset\W$; then $\gr_\f$ is $\H$-rectifiable. 

In particular, a set $R\subset\H^n$ is $\H$-rectifiable of codimension $k$, $k\in\{1,\dots,n\}$, if and only if $\Shaus^{Q-k}(R)<\infty$ and there exists a countable family $(\f_j)_{j}$ of intrinsic Lipschitz maps $\f_j:\W_j\to\V_j$ where $\W_j,\V_j $ are homogeneous complementary subgroups of $\H^n$ with $\dim \V_j=k$ and
\[
\Shaus^{Q-k}\left(R\setminus\bigcup_j \gr_{\f_j}\right)=0.
\]
\end{corollary}
\begin{proof}
We can assume $A=\W$ because of Theorem~\ref{thm:estensione}. By Theorem~\ref{thm:Lusin}, for every $j\in\N$ there exists $\f_j:\W\to\V$ such that $\gr_{\f_j}$ is $\H$-regular and
\[
\Shaus^{Q-k}(\gr_\f\setminus \gr_{\f_j})<1/j.
\]
The first part of the statement follows as well as one implication of the second part of the statement. The other implication is a simple consequence of the fact that (Theorem~\ref{teo:Hreg_grafici_area}) $\H$-regular submanifolds are locally intrinsic Lipschitz graphs.
\end{proof}

A second consequence of Theorem~\ref{thm:Lusin} is the  area formula of Theorem~\ref{thm:formulaarea}, that we now prove. Recall that the intrinsic Jacobian determinant $J^\f\f$ was introduced in Definition~\ref{def:differenziabilita}.

\begin{proof}[Proof of Theorem~\ref{thm:formulaarea}]
Recalling Remark~\ref{rem:WVortoghonal}, it is not restrictive to assume that $\W,\V$ are the subgroups in~\eqref{eq:Wfissato} and~\eqref{eq:Vfissato}.
By Theorem~\ref{thm:estensione}, one can  assume that $A=\W$. Since every Borel non-negative function $h$ can be written as a series of characteristic functions of Borel subsets of $\gr_\f$, we can without loss of generality assume that $h=\chi_E$ for some Borel subset $E\subset\gr_\f$, and we  must prove that
\[
\Shaus^{Q-k}(E)=C_{n,k}\int_{\F^{-1}(E)}J^\f\f\:d\leb{2n+1-k}.
\]
Let $\ep>0$ be fixed; by Theorem~\ref{thm:Lusin} we can find an intrinsic Lipschitz map $\psi:\W\to\V$ such that $\gr_\psi$ is a $\H$-regular submanifold and, defining the graph map $\Psi(w):=w\psi(w)$ and  $D:=\{w\in\W:\f(w)=\psi(w)\text{ and }\nabla^\f\f(w)=\nabla^\psi\psi(w)\}$,  one has
\begin{equation}\label{eq:geronimo}
\leb{2n+1-k}(\W\setminus D)<\ep\qtaq \Shaus^{Q-k}(E\setminus \Psi(D))<\ep.
\end{equation}
Using Theorem~\ref{teo:Hreg_grafici_area}
\begin{align*}
\Shaus^{Q-k}(E) 
& = \Shaus^{Q-k}(E\cap \Psi(D))+\Shaus^{Q-k}(E\setminus \Psi(D))\\
& = C_{n,k}\int_{\Psi^{-1}(E\cap \Psi(D))} J^\psi\psi\:d\leb{2n+1-k}+\Shaus^{Q-k}(E\setminus \Psi(D))\\
& = C_{n,k}\int_{\F^{-1}(E)\cap D} J^\f\f\:d\leb{2n+1-k}+\Shaus^{Q-k}(E\setminus \Psi(D))\\
& = C_{n,k}\int_{\F^{-1}(E)} \!J^\f\f\:d\leb{2n+1-k} - C_{n,k}\int_{\F^{-1}(E)\setminus D} \!J^\f\f\:d\leb{2n+1-k}+\Shaus^{Q-k}(E\setminus \Psi(D))
\end{align*}
so that
\begin{align*}
\left| \Shaus^{Q-k}(E)- C_{n,k}\int_{\F^{-1}(E)} J^\f\f\:d\leb{2n+1-k} \right| 
 &\leq C_{n,k}\int_{\W\setminus D} J^\f\f\:d\leb{2n+1-k}+\Shaus^{Q-k}(E\setminus \Psi(D))\\
 &\leq (C_{n,k}C+1)\ep
\end{align*}
where, in the last inequality, we used~\eqref{eq:geronimo} and the fact that, by Lemma~\ref{lem:nablaffijminoredialfa}, $J^\f\f\leq C$ for some positive $C$  depending only on the intrinsic Lipschitz constant of $\f$. The arbitrariness of $\ep$ concludes the proof.
\end{proof}

Eventually, as a further consequence of Theorem~\ref{thm:Lusin} we prove the fact that, as anticipated in Section~\ref{sec:dimrademacher}, the canonical current associated with an entire intrinsic Lipschitz graph has zero boundary. We will use~\eqref {eq:defcanonicalcurrLip_ALTERNATIVA}, whose validity is now guaranteed by Theorem~\ref{thm:formulaarea}.

\begin{proposition}\label{prop:grf0bdry}
Let  $\f:\W\to\V$ be intrinsic Lipschitz. Then the  $(2n+1-k)$-Heisenberg current $\curr{\gr_\f}$ introduced in~\eqref{eq:defcanonicalcurrLip} is such that $\partial\curr{\gr_\f}=0$.
\end{proposition}
\begin{proof}
As in Remark~\ref{rem:Rademachersuffpianicanon} we can assume without loss of generality that the subgroups $\W,\V$ are those defined in~\eqref{eq:Wfissato} and~\eqref{eq:Vfissato}.
Let $\ep>0$ be fixed and let $\psi:\W\to\V$ be the map provided by Theorem~\ref{thm:Lusin}. It was shown in the proof of Theorem~\ref{thm:Lusin} that there exists $f\in C^1_\H(\H^n,\R^k)$ such that $\gr_\psi=\{f=0\}$; moreover, the construction performed in  Lemma~\ref{lem:quasiLusin} ensures that $f\in C^\infty(\H^n\setminus\gr_\psi)$. As in the proof of Theorem~\ref{teo:approssimiamoooooooo},  for every positive integer $i$ there exists $\psi_i\in C^\infty(\W,\V) $ such that
\[
\gr_{\psi_i}=\{f=(1/i,0,\dots,0)\}.
\]
Moreover, defining the graph maps $\Psi(w):=w\psi(w)$ and $\Psi_i(w):=w\psi_i(w)$, as $i\to+\infty$ one has 
\begin{equation}\label{eq:convergenzevarie}
\psi_i\to\psi,\ \Psi_i\to\Psi \text{ and }\nabla^{\psi_i}\psi_i\to\nabla^\psi\psi\qquad\text{locally uniformly on $\W$.}
\end{equation}
The last  convergence  stated in~\eqref{eq:convergenzevarie} follows from the local uniform (with respect to $w\in\W$) convergence
\[
\Tan^\H_{\gr_{\psi_i}}(\Psi_i(w))=\ker\nabla_\H f(\Psi_i(w))\to \ker\nabla_\H f(\Psi(w))=\Tan^\H_{\gr_{\psi}}(\Psi(w)).
\]
Using Proposition~\ref{prop:correnti_cl_vs_H_INTRO}, Lemma~\ref{lem:correntevistaintegrandosuW} and formulae~\eqref{eq:convergenzevarie} and~\eqref {eq:defcanonicalcurrLip_ALTERNATIVA} we deduce that for every $\omega\in\DH^{2n-k}$
\begin{align*}
0
&= \lim_{i\to\infty} \curr{\gr_{\psi_i}}(d_C\omega)\\
&= \lim_{i\to\infty} C_{n,k}\int_\W\langle \nabla^{\psi_i}\Psi_i\wedge T\mid (d_C\omega)\circ\Psi_i\rangle\:d\leb{2n+1-k}\\
&= C_{n,k}\int_\W\langle \nabla^{\psi}\Psi\wedge T\mid(d_C\omega)\circ\Psi\rangle\:d\leb{2n+1-k}\\
&=  \curr{\gr_{\psi}}(d_C\omega).
\end{align*}
In particular
\begin{align*}
|\partial\curr{\gr_{\f}}(\omega)|
&= |\curr{\gr_{\f}}(d_C\omega)-\curr{\gr_{\psi}}(d_C\omega)|\\
&= \left|\int_{\gr_\f} \langle t^\H_{\gr_\f}\mid d_C\omega\rangle\:d\Shaus^{Q-k} - \int_{\gr_\psi} \langle t^\H_{\gr_\psi}\mid d_C\omega\rangle\:d\Shaus^{Q-k} \right|\\
&\leq \Shaus^{Q-k}((\gr_\f\,\Delta\, \gr_\psi)\cup\{p\in\gr_\f\cap \gr_\psi:\Tan^\H_{\gr_\f}(p)\neq \Tan^\H_{\gr_\psi}(p)\})\;\sup_{\H^n}|d_C\omega|\\
&\leq \ep\; \sup_{\H^n}|d_C\omega|.
\end{align*}
The arbitrariness of $\ep$ implies that $\partial\curr{\gr_{\f}}(\omega)=0$ for every $\omega\in\DH^{2n-k}$, as desired.
\end{proof}


\appendix

\section{Proof of Proposition~\ref{prop:almassimo2}}\label{app:max2}
\renewcommand{\theequation}{\thesection.\arabic{equation}}
We provide here the proofs of Lemma~\ref{lem:whyrank1} and Proposition~\ref{prop:almassimo2} that were only stated in \S\ref{subsec:rankoneconnect}. The proof of Lemma~\ref{lem:whyrank1} is quite simple.

\begin{proof}[Proof of Lemma~\ref{lem:whyrank1}]
Assume that (a) holds; then either $\dim \mathscr P_1\cap\mathscr P_2=m$, in which case $\mathscr P_1=\mathscr P_2$ and (b) trivially holds, or $\dim \mathscr P_1\cap\mathscr P_2=m-1$. Let $v_1,\dots,v_{m-1}$ be a basis of $\mathscr P_1\cap\mathscr P_2$ and choose $w_1\in\mathscr P_1\setminus\mathscr P_2$ and $w_2\in\mathscr P_2\setminus\mathscr P_1$; then, for every $t_1,t_2$ as in statement (b) there exist $c_1,c_2\in\R\setminus\{0\}$ such that
\begin{align*}
&  t_1 = c_1w_1\wedge v_1\wedge\dots\wedge v_{m-1} \\
&  t_2=c_2w_2\wedge v_1\wedge\dots\wedge v_{m-1}.
\end{align*}
The difference $t_1-t_2$ is clearly a simple vector and (b) follows.

Statement (b)  implies (c), hence we have only to show that (c) implies (a). Assume  there exist $t_1,t_2$  as in (c). Let $h:=\dim\mathscr P_1\cap\mathscr P_2$; we assume $h\geq 1$, but the following argument can be easily adapted to the case $ h=0$. By assumption  $t:=t_1-t_2$ is  simple; moreover we have $\mathscr P_1\cap\mathscr P_2\subset\Span t$, because for every $v\in \mathscr P_1\cap\mathscr P_2$ we have $v\wedge t=v\wedge(t_1-t_2)=0$. In particular, after fixing  a basis $v_1,\dots, v_h$ of $\mathscr P_1\cap\mathscr P_2$ we can find $ v_{h+1},\dots,  v_m\in\Span t$, $e_1,\dots,e_{m-h}\in\mathscr P_1\setminus\mathscr P_2$ and $e_{m-h+1},\dots,e_{2(m-h)}\in\mathscr P_2\setminus \mathscr P_1$ such that
\begin{align*}
t &= v_1\wedge\dots \wedge v_h\wedge v_{h+1}\wedge\dots\wedge v_m\\
t_1&= v_1\wedge\dots \wedge v_h\wedge e_{1}\wedge\dots\wedge e_{m-h}\\
t_2&= v_1\wedge\dots \wedge v_h\wedge e_{m-h+1}\wedge\dots\wedge e_{2(m-h)}.
\end{align*}
Then, on the one side, $t':=  e_1\wedge\dots\wedge e_{m-h} -  e_{m-h+1}\wedge\dots\wedge e_{2(m-h)}$ equals the simple vector $v_{h+1}\wedge\dots\wedge v_m$, hence $\dim\Span t'=m-h$. On the other side,  $e_1,\dots,e_{2(m-h)}$ are  linearly independent and one can easily check that, in case $m-h\geq 2$, one would have
\[
\Span t'=\{v\in V:v\wedge t'=0\}=\{v\in\Span\{e_1,\dots,e_{2(m-h)}\}:v\wedge t'=0\}=\{0\}.
\]
This implies that  $h\geq m-1$, which is statement (a).
\end{proof}

On the contrary, the proof of Proposition~\ref{prop:almassimo2} is long and technical: there might well be a simpler one but the author was not able to find it. The proof provided here is based on  Lemmata~\ref{lem:vettoresemplicediRumin,b=0} and~\ref{lem:vettoresemplicediRumin}, where one essentially studies the  model cases of the planes $\Pab$ introduced in~\eqref{eq:defPab}. Proposition~\ref{prop:almassimo2} will then follow by a quite standard use of Proposition~\ref{prop:ruotiamosupianicanonici}.

Remark~\ref{rem:cosafailvettorestandard} will be utilized several times.

\begin{lemma}\label{lem:vettoresemplicediRumin,b=0}
Let $a\geq 1$ be an integer such that $n\leq 2a\leq 2n-1$; assume that $\tau\in\bwl_{2a}\h_1$ is a simple $2a$-vector such that
\[
[\tau\wedge T ]_\mathcal J = [X_1\wedge\dots\wedge X_a\wedge Y_1\wedge\dots\wedge Y_a\wedge T]_\mathcal J.
\]
Then the following statements hold:
\begin{enumerate}
\item[(i)] if $2a>n$, then $\tau= X_1\wedge\dots\wedge X_a\wedge Y_1\wedge\dots\wedge Y_a$;
\item[(ii)] if $2a=n$, then either
\[
\tau= X_1\wedge\dots\wedge X_a\wedge Y_1\wedge\dots\wedge Y_a
\]
or
\[
\tau=(-1)^a\,X_{a+1}\wedge\dots\wedge X_{2a}\wedge Y_{a+1}\wedge\dots\wedge Y_{2a}.
\]
\end{enumerate}
\end{lemma}

\begin{remark}\label{rem:nonuniqueness1}
When $2a=n$ one can use computations analogous to those in Remark~\ref{rem:cosafailvettorestandard} to check the action of $(-1)^a\,X_{a+1}\wedge\dots\wedge X_{2a}\wedge Y_{a+1}\wedge\dots\wedge Y_{2a}\wedge T$ in  duality with elements of the basis of  $\mathcal J^{2a+1}=\mathcal J^{n+1}$ provided by Proposition~\ref{prop:baseJintro}. As a matter of fact, the equality
\[
[X_1\wedge\dots\wedge X_a\wedge Y_1\wedge\dots\wedge Y_a\wedge T]_\mathcal J = [(-1)^a\,X_{a+1}\wedge\dots\wedge X_{2a}\wedge Y_{a+1}\wedge\dots\wedge Y_{2a}\wedge T]_\mathcal J.
\]
does hold. 

For instance, recalling  Example~\ref{example:baseJ3inH2} it can be easily checked that in $\H^2$
\[
\langle X_1\wedge Y_1\wedge T \mid  \la \rangle = -\langle X_2\wedge Y_2\wedge T \mid  \la \rangle\qquad\text{ for every }\la\in\mathcal J^3.
\]
\end{remark}

\begin{proof}[Proof of Lemma~\ref{lem:vettoresemplicediRumin,b=0}]
Let us write $\tau=\tau^1\wedge\dots\wedge\tau^{2a}$ for suitable horizontal vectors $\tau^i\in\h_1$. Writing
\[
\tau^i=\tau^i_1X_1+\dots+\tau^i_{2n}Y_n,\qquad i=1,\dots,2a
\] 
we introduce the matrix
\[
M:=\text{col}\big[\:\tau^1\:|\cdots|\:\tau^{2a}\:\big]
=
\left[
\begin{array}{c|c|c}
\tau^1_1 &\dots & \tau^{2a}_1\\
\vdots& &\vdots\\
\tau^1_{2n} &\dots & \tau^{2a}_{2n}
\end{array}
\right]
=\text{row}
\left[
\begin{array}{c}
M_1\\
\hline\vdots\\
\hline M_{2n}
\end{array}
\right]
,
\]
where for every $i=1,\dots,2n$ we denoted by $M_i=(\tau^1_i,\dots,\tau^{2a}_i)\in\R^{2a}$ the $i$-th row of  $M$. It is worth noticing that for every $I\subset\{1,\dots,2n\}$ with $|I|=2a$ the equality 
\[
\langle\tau\mid dz_I\rangle = \det [M_i]_{i\in I},
\]
holds, where we used notation~\eqref{eq:defdzI} and denoted by $[M_i]_{i\in I}$ the $2a\times 2a$ matrix formed by the rows $M_i$, $i\in I$, arranged in the natural order.

We now fix a subset $I\subset\{1,\dots,2n\}$ of indices having maximal cardinality among all subsets $\mathscr I\subset\{1,\dots,2n\}$ satisfying the two properties
\begin{equation}\label{eq:richiestesuI}
\begin{split}
&\forall\ i,j\in\mathscr I\quad i\not\equiv j\mod n\\
& (M_i)_{i\in\mathscr I}\text{ are linearly independent.}
\end{split}
\end{equation}
We also define
\begin{align*}
& J:=\{j\in\{1,\dots,2n\}:\exists\: i\in I\text{ such that }j\equiv i\text{ mod }n\}\setminus I\\
& K:=\{1,\dots,2n\}\setminus(I\cup J).
\end{align*}
Clearly, $I\cap J=\emptyset$.\medskip

{\em Claim 1: $a\leq |I|\leq n$ and $|J|=|I|$.}

\noindent Let us prove the first claimed inequality. We have $\tau\neq 0$, for otherwise we would obtain the contradiction
\begin{align*}
0=& \langle \tau\wedge T\mid \al_{\overline R}\wedge\theta\rangle = \langle X_1\wedge\dots\wedge X_a\wedge Y_1\wedge\dots\wedge Y_a\wedge T \mid  \al_{\overline R}\wedge\theta\rangle =1
\end{align*}
where
\begin{equation}\label{eq:defRbar}
\begin{aligned}
& {\overline R}:=
\begin{tabular}{|c|c|c|c|ccc}
\cline{1-7}
$1$ & $2$ &$\ \cdots\ $& $n-a$ &\multicolumn{1}{c|}{$ n-a+1$} & \multicolumn{1}{c|}{$\ \cdots\ $} & \multicolumn{1}{c|}{$a^{\phantom{2}}$}\\
\cline{1-7}
$a+1$ & $a+2^{\phantom{2}}$ &$\ \cdots\ $& $n$
\\
\cline{1-4}
\end{tabular}
&&\text{if }2a>n\\
& \overline R:=
\begin{tabular}{|c|c|c|c|}
\cline{1-4}
$1$ & $2$ &$\ \cdots\ $& $a$\\
\cline{1-4}
$a+1$ & $a+2^{\phantom{2}}$ &$\ \cdots\ $& $n$
\\
\cline{1-4}
\end{tabular}
&&\text{if }2a=n.
\end{aligned}
\end{equation}
In particular, the rank of $M$ is $2a$ and we can find distinct numbers $i_1,\dots,i_{2a}$ in $\unodueenne$ such that $M_{i_1},\dots,M_{i_{2a}}$ are linearly independent. It is then obvious that we can find a subset of $\{i_1,\dots,i_{2a}\}$ made by at least $a$ elements that are never congruent modulo $n$; this proves that $|I|\geq a$.

The remaining part of the claim ($|I|\leq n$ and $|J|=|I|$) is clear from the definitions of $I$ and $J$.\medskip

{\em Claim 2: for every $k\in K$, $M_k\in\Span\{M_i:i\in I\}$.}

\noindent Assume on the contrary that there exists $k\in K$ such that $(M_i)_{i\in I\cup\{k\}}$ are linearly independent; then $\mathscr I:=I\cup\{k\}$ would satisfy the two properties in \eqref{eq:richiestesuI} and this would contradict the maximality of $I$.\medskip

{\em Claim 3: there exists $J'\subset J$ such that $|J'|=2a-|I|$ and $(M_i)_{i\in I\cup J'}$ are linearly independent. In particular, $(M_i)_{i\in I\cup J'}$ form a basis of $\R^{2a}$.}

\noindent We are going to implicitly use the facts that $I\cap J=\emptyset$ and $|I|=|J|$. Let $j_1,j_2,\dots,j_{|I|}$ be an enumeration of all the elements of $J$; for $\ell=0,1,\dots,|I|$ we inductively define $J'_\ell$ by $J'_0:=\emptyset$ and
\begin{align*}
&\text{ if }M_{j_\ell}\notin\Span\{M_i:i\in I\cup J'_{\ell-1}\},\text{ then }J'_\ell:=J'_{\ell-1}\cup\{j_\ell\}\\
&\text{ if }M_{j_\ell}\in\Span\{M_i:i\in I\cup J'_{\ell-1}\},\text{ then }J'_\ell:=J'_{\ell-1}.
\end{align*}
Setting $J':=J'_{|I|}$, then the elements $(M_i)_{i\in J\cup J'}$ are linearly independent by construction; let us prove that they span the whole $\R^{2a}$,  that will also imply the equality $|J'|=2a-|I|$. Assume on the contrary that the dimension of
\[
\mathscr M:=\Span\{M_i:i\in I\cup J'\} = \Span\{M_i:i\in I\cup J\}
\]
(the second equality holding by construction) is less than $2a$; recalling Claim 2 and the fact that $I\cup J\cup K=\unodueenne$, this would imply that
\[
\forall\ i\in\unodueenne\quad M_i\in\mathscr M,
\]
and the rank of $M$ would not be $2a$. This is a contradiction, and Claim 3 is proved.\medskip

We now introduce
\begin{align*}
& A:=\{i\in\unoenne:\{i,i+n\}\subset I\cup J'\}\\
& B:=(I\cup J')\setminus (A\cup(n+A))=I\setminus (A\cup(n+A))\\
& C:= K\cap\unoenne,
\end{align*}
where for $E\subset\unoenne$ we write $n+E$ for $\{i\in\{n+1,\dots,2n\}:i-n\in E\}$. We notice the following properties:
\begin{itemize}
\item[(i)] if $i\in A$, then either ($i\in I$ and $i+n\in J'$), or  ($i\in J'$ and $i+n\in I$);
\item[(ii)] $|I|=|J|=|A|+| B|$ and $2|A|+|B|=|I|+|J'|=2a$;
\item[(iii)] $K=C\cup (n+C)$.
\end{itemize}
In particular
\[
|A|-|C| = |A|-\tfrac12 |K| = |A|-\tfrac12 (2n-|I|-|J|) = 2|A|+|B| -n =2a-n\geq 0
\]
and we can define the (non-necessarily standard) Young tableau
\[
R:=\begin{tabular}{|c|c|c|c|cc}
\cline{1-6}
$a_1$ & $a_2$ &$\ \cdots\ $& $\cdots$ &\multicolumn{1}{c|}{$ \cdots$} &   \multicolumn{1}{c|}{$a_{|A|}$}\\
\cline{1-6}
$c_1$ & $c_2$ &$\ \cdots\ $& $c_{|C|}$
\\
\cline{1-4}
\end{tabular}
\]
where $a_1,\dots,a_{|A|}$ is the increasing enumeration of the elements of $A$ and $c_1,\dots,c_{|C|}$ is the increasing enumeration of the elements of $C$; of course, when $|A|=|C|$ (or, which is the same, $2a=n$) the tableau $R$ is a $2\times |A|$ matrix.

We now consider the covector $\la:=dz_B\wedge \al_R$; it  can be easily checked that  $\la\wedge\theta\in \mathcal J^{2a+1}$.\medskip

{\em Claim 4: $\langle \tau\mid  \la\rangle\neq 0$.}

\noindent It is enough to check that
\begin{equation}\label{eq:perilclaim4}
\langle \tau\mid  \la\rangle= \langle \tau\mid  dz_B\wedge dxy_A\rangle;
\end{equation}
in fact, if \eqref{eq:perilclaim4} were true one would get
\begin{equation}\label{eq:partiamo}
\langle \tau\mid  \la\rangle= \langle \tau\mid  dz_B\wedge dxy_A\rangle = \pm \det[M_i]_{i\in A\cup(n+A)\cup B} = \pm \det[M_i]_{i\in I\cup J'}\neq0,
\end{equation}
as claimed. The signs $\pm$ in \eqref{eq:partiamo} are not relevant and depend on $A$ and $B$ only.

Let us prove \eqref{eq:perilclaim4}. By definition of $\al_R$, the covector $\la$ can be written as
\[
\la=dz_B\wedge \al_R = dz_B\wedge dxy_A + \tilde\la
\]
where $\tilde\la$ is a sum of covectors of the form $\pm dz_B\wedge dxy_{A'\cup C'}$, where $A'\subset A$ and $C'\subset C$ satisfy
\[
|A'|+|C'|=|A|\qtaq |C'|\geq 1
\]
(in particular, $|A'|<|A|$). We have
\begin{equation}\label{eq:determinantezero}
\langle \tau\mid  dz_B\wedge dxy_{A'\cup C'}\rangle = \pm \det[M_i]_{i\in B\cup A'\cup C'\cup(n+A')\cup(n+C')}
\end{equation}
and we claim that this determinant vanishes. Choose indeed $\bar a\in A\setminus A'$ and let $\bar\jmath\in J'$ be such that $\bar\jmath\equiv\bar a$ mod $n$ (recall property (i) above); by Claim  2 we obtain
\begin{align*}
\Span(M_i)_{i\in B\cup A'\cup C'\cup(n+A')\cup(n+C')} 
\subset\: & \Span(M_i)_{i\in B\cup A'\cup(n+A')} + \Span(M_i)_{i\in  C'\cup(n+C')}\\
\subset\: & \Span(M_i)_{i\in B\cup A'\cup(n+A')} + \Span(M_i)_{i\in I}\\
\subset\: & \Span(M_i)_{i\in I\cup J'\setminus\{\bar\jmath\}}.
\end{align*}
In particular, $\Span(M_i)_{i\in B\cup A'\cup C'\cup(n+A')\cup(n+C')}$ cannot be of maximal dimension $2a$: this implies that the determinant in \eqref{eq:determinantezero} is null, as claimed. The equality $\langle \tau \mid  \tilde\la\rangle=0$ then immediately follows and,  since $\langle \tau\mid  \la\rangle= \langle \tau\mid  dz_B\wedge dxy_A+\tilde\la\rangle$, equality \eqref{eq:perilclaim4} follows.\medskip

{\em Claim 5: $B=\emptyset$.}

\noindent By Remark~\ref{rem:tabellenonYoung}, the covector $\alpha_R$ can be written as a finite linear combination $\al_R=\sum_\ell c_\ell \al_{S_\ell}, c_\ell\in\R,$ of covectors $\al_{S_\ell}$ associated to standard Young tableaux $S_\ell$  containing the elements of $A\cup C$. In particular, if $B\neq\emptyset$ one would have
\[
\langle \tau\mid  \la\rangle =  \sum_\ell c_\ell\langle \tau\mid  dz_B\wedge  \al_{S_\ell}\rangle
 =  \sum_\ell c_\ell\langle X_1\wedge\dots\wedge X_a\wedge Y_1\wedge\dots\wedge Y_a\mid  dz_B\wedge  \al_{S_\ell}\rangle=0,
\]
the last equality due to Remark~\ref{rem:cosafailvettorestandard} (with $b:=0$) and the fact that $dz_B$ contains no factors of the form $dx_i$. This would contradict Claim 4, hence $B=\emptyset$ as claimed.\medskip

Claim 5 implies that $|A|=a$ and
\[
I\cup J'=A\cup(n+A)=I\cup J\qtaq C=\unoenne\setminus A.
\]
We can without loss of generality assume that $I=A$ and $J=J'=n+A$.\medskip

{\em Claim 6: $M_k=0$ for every $k\notin A\cup(n+A)$.}

\noindent Assume that, on the contrary, there exists $k\notin A\cup(n+A)$ (i.e., $k\in K=C\cup(n+C)$) such that $M_k\neq0$. By Claim 3, $(M_i)_{i\in A\cup(n+A)}$ form a basis of $\R^{2a}$, hence
\[
\Span\{M_i:i\in A\}\cap \Span\{M_i:i\in n+A\}=\{0\}
\]
which gives that
\[
\text{either}\quad M_k\notin \Span\{M_i:i\in A\}\quad\text{or}\quad  M_k\notin \Span\{M_i:i\in n+A\}.
\]
This would contradict the choice of $I$, because $\mathscr I:=A\cup\{k\}$ (in the first case) or $\mathscr I:=(n+A)\cup\{k\}$ (in the second one)  would satisfy both conditions in \eqref{eq:richiestesuI}, but $|\mathscr I|>|I|$.\medskip

{\em Claim 7: $\langle \tau\mid dxy_D\rangle=0$ for every $D\subset\unoenne$ such that $|D|=|A|$ and $D\neq A$.}

\noindent This is an immediate consequence of the equality $\langle \tau\mid dxy_D\rangle = \pm \det [M_i]_{i\in D\cup(n+D)}$ together with Claim 6.\medskip

{\em Claim 8:  $\{n-a+1,n-a+2,\dots,a\}\subset A$.}

\noindent Here we are understanding that the claim is empty when $n=2a$.   Denoting by $\overline R$ the Young tableau defined in \eqref{eq:defRbar}, one can easily check that the covector $\al_{\overline R}$ can be written as a sum of elements of the form $\pm dxy_D$ where  the subsets $D\subset \unoenne$ have cardinality $|D|=|A|=a$ and all contain $\{n-a+1,\dots,a\}$. If one had $\{n-a+1,\dots,a\}\not\subset A$, then one would get the contradiction
\[
1= \langle X_1\wedge\dots\wedge X_a\wedge Y_1\wedge\dots\wedge Y_a\mid  \al_{\overline R}\rangle = 
\langle \tau\mid  \al_{\overline R}\rangle = 0,
\]
the last equality following from Claim 7.\medskip

{\em Claim 9: if $2a>n$, then $A=\{1,\dots,a\}$.}

\noindent The condition $2a>n$ implies that $|C|=n-|A|=n-a<a=|A|$. Let $a_1<a_2<\dots<a_{|A|}$ be the elements of $A$ and $c_1<\dots<c_{|C|}$ those of $C$ and define 
\[
R_1:=\begin{tabular}{|c|c|c|c|ccc}
\cline{1-7}
$a_1$ & $a_2$ &$\ \cdots\ $& $a_{|C|}$ &\multicolumn{1}{c|}{$ a_{|C|+1}$} &\multicolumn{1}{c|}{$ \cdots$} &   \multicolumn{1}{c|}{$a_{|A|}$}
\\
\cline{1-7}
$c_1$ & $c_2$ &$\ \cdots\ $& $c_{|C|}$
\\
\cline{1-4}
\end{tabular}
\]
The Young tableau $R_1$ is not necessarily a standard one, as it might happen that $a_i>c_i$ for some $i$; however, it can be easily seen that the tableau
\[
R_2:=\begin{tabular}{|c|c|c|c|ccc}
\cline{1-7}
$m_1$ & $m_2$ &$\ \cdots\ $& $m_{|C|}$ &\multicolumn{1}{c|}{$ a_{|C|+1}$} &\multicolumn{1}{c|}{$ \cdots$} &   \multicolumn{1}{c|}{$a_{|A|}$}
\\
\cline{1-7}
$\mu_1$ & $\mu_2$ &$\ \cdots\ $& $\mu_{|C|}$
\\
\cline{1-4}
\end{tabular}
\]
defined by setting $m_i:=\min\{a_i,c_i\}, \mu_i:=\max\{a_i,c_i\}$ (or, equivalently,  by switching the positions of $a_i, c_i$ in case $a_i>c_i$) is a standard one. Notice that, since the $i$-th column of $R_1$ and the $i$-th column of $R_2$ contain the same elements $a_i,c_i$, we have $\al_{R_2}=\pm\al_{R_1}$.  Using Claim 7 we obtain
\begin{equation}\label{eq:R2ugualeadoverlineR}
\begin{aligned}
\langle X_1\wedge\dots\wedge X_a\wedge Y_1\wedge\dots\wedge Y_a\mid \al_{R_2}\rangle=&
\langle\tau\mid \al_{R_2}\rangle=\pm \langle\tau\mid \al_{R_1}\rangle\\
=&\pm \langle\tau\mid dxy_A\rangle =  \pm\det[M_i]_{i\in A\cup(n+A)}\neq 0
\end{aligned}
\end{equation}
which, by Remark \ref{rem:cosafailvettorestandard}, implies that $R_2=\overline R$. Therefore the rightmost entry in the first row of $R_2$ must be $a$; but this coincides with the rightmost entry in the first row of $R_1$, i.e., with $\max A$. Therefore $A\subset\unoenne$ is such that $|A|=a$ and $\max A=a$, and the claim $A=\{1,\dots,a\}$ follows.\medskip

{\em Claim 10: if $2a=n$, then either $A=\{1,\dots,a\}$ or $A=\{a+1,\dots,2a\}$.}

\noindent The proof is  similar to that of Claim 9. Notice that now $|C|=|A|=a$; let $a_1<a_2<\dots<a_{|A|}$ be the elements of $A$ and $c_1<\dots<c_{|A|}$ those of $C$ and define
\[
R_1:=\begin{tabular}{|c|c|c|c|}
\cline{1-4}
$a_1$ & $a_2$ &$\ \cdots\ $& $a_{|A|}$ 
\\
\cline{1-4}
$c_1$ & $c_2$ &$\ \cdots\ $& $c_{|A|}$
\\
\cline{1-4}
\end{tabular}
\]
as before. Again, let 
\[
R_2:=\begin{tabular}{|c|c|c|c|}
\cline{1-4}
$m_1$ & $m_2$ &$\ \cdots\ $& $m_{|A|}$ 
\\
\cline{1-4}
$\mu_1$ & $\mu_2$ &$\ \cdots\ $& $\mu_{|A|}$
\\
\cline{1-4}
\end{tabular}
\]
be defined by  $m_i:=\min\{a_i,c_i\}, \mu_i:=\max\{a_i,c_i\}$; $R_2$ is a standard Young tableau,  $\al_{R_2}=\pm\al_{R_1}$ and as in \eqref{eq:R2ugualeadoverlineR} we can conclude that $R_2=\overline R$.  We now distinguish two cases: 
\begin{itemize}
\item  if $1\in A$, then $a_1=m_1=1$ and, since $R_2=\overline R$, we have $\mu_1=c_1=a+1$. In particular, $C\subset\unoenne=\{1,\dots,2a\}$ is such that $|C|=a$ and $\min C=a+1$. This implies $C=\{a+1,\dots,n\}$ and in turn  $A=\{1,\dots,a\}$;
\item if $1\notin A$, then $1=m_1=c_1$ and $a+1=\mu_1=a_1$. In particular, $A\subset\unoenne=\{1,\dots,2a\}$ is such that $|A|=a$ and $\min A=a+1$, which implies $A=\{a+1,\dots,n\}$.
\end{itemize}
\medskip

The conclusion of the proof now  follows easily. If $2a>n$, by Claims 6 and 9 we have that $M_i=0$ for every $i\in\{a+1,\dots,n, n+a+1,\dots,2n\}$: this implies that $\tau=tX_1\wedge\dots\wedge X_a\wedge Y_1\wedge\dots\wedge Y_a$ for some $t\in\R$, and $t$ is forced to be 1 because $\langle X_1\wedge\dots\wedge X_a\wedge Y_1\wedge\dots\wedge Y_a\mid \al_{R_2}\rangle\neq 0 $ and
\begin{align*}
\langle X_1\wedge\dots\wedge X_a\wedge Y_1\wedge\dots\wedge Y_a\mid \al_{\overline R}\rangle =\langle \tau \mid  \al_{\overline R}\rangle
=t\langle X_1\wedge\dots\wedge X_a\wedge Y_1\wedge\dots\wedge Y_a\mid \al_{\overline R}\rangle
\end{align*}
This concludes the proof in case $n>2a$. The case $n=2a$ follows analogously on considering Claims 6 and 10.
\end{proof}

We now use Lemma~\ref{lem:vettoresemplicediRumin,b=0} to prove the following, more general result, of which Lemma~\ref{lem:vettoresemplicediRumin,b=0} represents the case $b=0$. 

\begin{lemma}\label{lem:vettoresemplicediRumin}
Let $a\geq 0$ and $b\geq 1$ be  integers such that $n\leq 2a+b\leq 2n-1$; assume that $\tau\in\bwl_{2a+b}\h_1$ is a simple $(2a+b)$-vector such that
\begin{equation}\label{eq:ipotesiuguale}
[ \tau\wedge T ]_\mathcal J = [ X_1\wedge\dots\wedge X_{a+b}\wedge Y_1\wedge\dots\wedge Y_a\wedge T]_\mathcal J.
\end{equation}
Then, the following statements hold:
\begin{enumerate}
\item[(i)] if $2a+b>n$, then $\tau= X_1\wedge\dots\wedge X_{a+b}\wedge Y_1\wedge\dots\wedge Y_a$;
\item[(ii)] if $2a+b=n$, then either
\[
\begin{split}
\tau=& X_1\wedge\dots\wedge X_{a+b}\wedge Y_1\wedge\dots\wedge Y_a
\end{split}
\]
or
\[
\begin{split}
\tau=&(-1)^{a(b+1)} \:X_{a+1}\wedge\dots\wedge X_{n}\wedge Y_{a+b+1}\wedge\dots\wedge Y_{n}
\end{split}
\]
\end{enumerate}
\end{lemma}

\begin{remark}\label{rem:nonuniqueness2}
As a matter of fact, if $2a+b=n$ the equality
\[
[X_1\wedge\dots\wedge X_{a+b}\wedge Y_1\wedge\dots\wedge Y_a\wedge T]_\mathcal J = [(-1)^{a(b+1)} \:X_{a+1}\wedge\dots\wedge X_{n}\wedge Y_{a+b+1}\wedge\dots\wedge Y_{n}\wedge T]_\mathcal J
\]
holds.  This can be proved by using Remark~\ref{rem:cosafailvettorestandard} and checking the action of $X_{a+1}\wedge\dots\wedge X_{n}\wedge Y_{a+b+1}\wedge\dots\wedge Y_{n}\wedge T$ in  duality with elements of the basis of  $\mathcal J^{2a+b+1}=\mathcal J^{n+1}$ provided by Proposition~\ref{prop:baseJintro}.

For instance, in $\H^3$ one has the equality $[X_1\wedge X_2\wedge Y_1\wedge T]_\mathcal J=[X_2\wedge X_3\wedge Y_3\wedge T]_\mathcal J$.
\end{remark}

\begin{proof}[Proof of Lemma \ref{lem:vettoresemplicediRumin}]
Let us write $\tau=\tau^1\wedge\dots\wedge\tau^{2a+b}$ for suitable horizontal vectors $\tau^i\in\h_1$; writing
\[
\tau^i=\tau^i_1X_1+\dots+\tau^i_{2n}Y_n,\qquad i=1,\dots,2a+b
\] 
we introduce the matrix
\[
M:=\text{col}\big[\:\tau^1\:|\cdots|\:\tau^{2a+b}\:\big]
=
\left[
\begin{array}{c|c|c}
\tau^1_1 &\dots & \tau^{2a+b}_1\\
\vdots& &\vdots\\
\tau^1_{2n} &\dots & \tau^{2a+b}_{2n}
\end{array}
\right]
=\text{row}
\left[
\begin{array}{c}
M_1\\
\hline\vdots\\
\hline M_{2n}
\end{array}
\right]
,
\]
where for every $i=1,\dots,2n$ we denoted by $M_i=(\tau^1_i,\dots,\tau^{2a+b}_i)\in\R^{2a+b}$ the $i$-th row of  $M$. Again, for every $I\subset\{1,\dots,2n\}$ with $|I|=2a+b$ the equality 
\[
\langle\tau\mid dz_I\rangle = \det [M_i]_{i\in I},
\]
holds, where we used notation \eqref{eq:defdzI}.\medskip

{\em Claim 1: $M_{a+1},\dots,M_{a+b}$ are linearly independent.}

\noindent Assume not: then, for every $I\subset\unodueenne$ with $|I|=2a$ we would get
\[
\langle \tau\mid  dx_{\{a+1,\dots,a+b\}}\wedge dz_I\rangle=0.
\]
In particular 
\[
\langle\tau\mid  dx_{\{a+1,\dots,a+b\}}\wedge \alpha_R\rangle =0
\quad\text{for}\quad 
R:=\begin{tabular}{|c|c|c|c|cc}
\cline{1-6}
$1$ & $2$ &$\ \cdots\ $& $n-a-b$ &\multicolumn{1}{c|}{$ \cdots$} &   \multicolumn{1}{c|}{$a$}
\\
\cline{1-6}
$a+b+1$ & $a+b+2$ &$\ \cdots\ $& $n$
\\
\cline{1-4}
\end{tabular}
\]
which would provide a contradiction since
\[
\langle\tau\mid dx_{\{a+1,\dots,a+b\}}\wedge \alpha_R\rangle =
\langle X_1\wedge\dots\wedge X_{a+b}\wedge Y_1\wedge\dots\wedge Y_a \mid  dx_{a+1,\dots,a+b}\wedge \alpha_R\rangle=\pm 1
\]
by assumption.\medskip

{\em Claim 2: up to a proper choice of $\tau^1,\dots,\tau^{2a+b}$, the $b\times(2a+b)$ sub-block $\widetilde M$ of $M$ determined by the rows $M_{a+1},\dots, M_{a+b}$ satisfies
\[
\widetilde M=\left[
\begin{array}{c}
M_{a+1}\\
\cline{1-1}\vdots\\
\cline{1-1}
M_{a+b}
\end{array}
\right]
=
\left[
\begin{array}{c|c}
&\\
0_{b\times 2a} & I_b\\
&
\end{array}
\right]
\]
where $0_{b\times 2a}$ denotes a $b\times 2a$ block with null entries and $I_b$ is the $b\times b$ identity matrix.
}

\noindent By Claim 1, the matrix $\widetilde M$ has rank $b$; up to a permutation of $\tau^1,\dots,\tau^{2a+b}$, we can assume that the rightmost $b\times b$ minor of $\widetilde M$ has nonzero determinant. Namely, calling $\widetilde M^j=(\tau^j_{a+1},\dots,\tau^j_{a+b})^T$ the $j$-th column of $\tilde M$ (the superscript $^T$ denoting transposition), the rightmost minor $\overline M:=[\widetilde M^{2a+1}|\dots|\widetilde M^{2a+b}]$ of $\widetilde M$ satisfies $\det\overline M\neq 0$. For $i=2a+1,\dots,2a+b$ we  define
\[
\si^i:=\sum_{j=1}^b (\overline M^{-1})_i^j\:\tau^{2a+j}
\]
and we notice that  $\Span\{\tau^1,\dots,\tau^{2a+b}\}=\Span\{\tau^1,\dots,\tau^{2a},\si^{2a+1},\dots,\si^{2a+b}\}$, because of the equality  $\Span\{\tau^{2a+1},\dots,\tau^{2a+b}\}=\Span\{\si^{2a+1},\dots,\si^{2a+b}\}$. Notice that
\begin{equation}\label{eq:alglin1}
\si^{2a+i}_{a+j}=\de^i_j\qquad\text{for all }i,j\in\{1,\dots,b\};
\end{equation}
in particular, when in the matrix $M$  the rightmost $b$ columns $\tau^{2a+1},\dots,\tau^{2a+b}$ are replaced by $\si^{2a+1},\dots,\si^{2a+b}$,  the block $\widetilde M$ is replaced by a block of the form 
\[
\left[
\begin{array}{c|c}
&\\
*_{b\times 2a} & I_b\\
&
\end{array}
\right]
\]
where $*_{b\times 2a}$ denotes a  $b\times 2a$ matrix. Define
\[
\si^i:=\tau^i-\sum_{j=1}^b  \tau^i_{a+j}  \si^{2a+j},\qquad i=1,\dots,2a
\]
so that
\begin{equation}\label{eq:alglin2}
\si^i_j=0\qquad \text{for all }i=1,\dots,2a,j=a+1,\dots,a+b.
\end{equation}
It is easily seen that  $\Span\{\tau^1,\dots,\tau^{2a+b}\}=\Span\{\si^1,\dots,\si^{2a+b}\}$; in particular, $\si:=\si^1\wedge\dots\wedge\si^{2a+b}$ is a non-zero multiple of $\tau$. Upon multiplying $\si^1$ by a non-zero factor, we can assume that $\si=\tau$. In this way, if one replaces $\tau^1,\dots,\tau^{2a+b}$ by $\si^1,\dots,\si^{2a+b}$, then the new matrix $M$ is such that  the block $\widetilde M$ has the  form 
\[
\left[
\begin{array}{c|c}
&\\
0_{b\times 2a} & I_b\\
&
\end{array}
\right]
\]
as wished.\medskip

In the following, we will keep on using the notation $\tau^1,\dots,\tau^{2a+b}$ for the new family $\si^1,\dots,\si^{2a+b}$.\medskip

{\em Claim 3a: if $n>2a+b$, then up to a proper choice of $\tau^1,\dots,\tau^{2a+b}$, we can assume that the matrix $M$ is of the form
\begin{equation}\label{eq:formaMa}
M=
(-1)^{ab}\left[
\begin{array}{c|c|c}
I_a & 0 & A_1\\
\cline{1-3}0 & 0 & I_b\\
\cline{1-3}0 & 0 & D_1\\
\cline{1-3}0 & I_a & A_2\\
\cline{1-3}B_1 & B_2 & C\\
\cline{1-3}0 & 0 & D_2
\end{array}
\right]
\end{equation}
where
\begin{itemize}
\item[$(i)$] the $2a+b$ columns of $M$ have been arranged into three blocks of size, respectively, $a$, $a$ and $b$;
\item[$(ii)$] the $2n$ rows of $M$ have been arranged into six blocks of size, respectively, $a$, $b$, $n-a-b$, $a$, $b$ and $n-a-b$;
\item[$(iii)$] 0 denotes null matrices of the proper size;
\item[$(iv)$] $I_a,I_b$ denote identity matrices of size $a,b$;
\item[$(v)$] $A_1,A_2,B,C_1,C_2,D_1,D_2$ denote generic matrices of the proper size.
\end{itemize}
}\medskip

{\em Claim 3b: if $n=2a+b$, then up to a proper choice of $\tau^1,\dots,\tau^{2a+b}$, we can assume that
\begin{equation}\label{eq:formaMb}
\text{either }
M=
(-1)^{ab}\left[
\begin{array}{c|c|c}
I_a & 0 & A_1\\
\cline{1-3}0 & 0 & I_b\\
\cline{1-3}0 & 0 & D_1\\
\cline{1-3}0 & I_a & A_2\\
\cline{1-3}B_1 & B_2 & C\\
\cline{1-3}0 & 0 & D_2
\end{array}
\right]
\quad\text{ or }
M=(-1)^{a(b+1)}
\left[
\begin{array}{c|c|c}
0 & 0 & D_1\\
\cline{1-3}0 & 0 & I_b\\
\cline{1-3}I_a & 0 & A_1\\
\cline{1-3}0 & 0 & D_2\\
\cline{1-3}B_1 & B_2 & C\\
\cline{1-3}0 & I_a & A_2
\end{array}
\right]
\end{equation}
where the notation is similar to Claim 3a, \upshape{($i$)-($v$)}.
}

\noindent We prove Claim 3a and 3b simultaneously; notice that there is nothing to prove in case $a=0$. Using the block subdivision of $M$ as in \eqref{eq:formaMa} and \eqref{eq:formaMb}, we already known by Claim 2 that the second block of rows  is of the form $[0|0|I_b]$. Let us prove that we can choose $\tau^1,\dots,\tau^{2a+b}$ so that  null and identity blocks $0$ and $I_a$ appear where claimed in \eqref{eq:formaMa} and \eqref{eq:formaMb}.

We consider now a $(n-b)$-th Heisenberg group and we agree that all objects associated with it will be overlined: in particular,  we denote by $\overline X_1,\dots,\overline X_{n-b},\overline{Y}_1,\dots,\overline{Y}_{n-b},\overline T$ the standard basis of left-invariant  vector fields in $\overline\H{}^{n-b}$. Define 
\[
\overline\tau^i:=\sum_{j=1}^a (\tau^i_j\overline X_j+\tau^i_{n+j}\overline Y_j)+\sum_{j=a+1}^{n-b} (\tau^i_{b+j}\overline X_j + \tau^i_{n+b+j}\overline Y_j),\qquad i=1,\dots,2a.
\]
Consider $\overline\tau:=\overline\tau^1\wedge\dots\wedge\overline\tau^{2a}$; we prove that
\begin{equation}\label{eq:taubartau}
\langle\overline\tau\wedge\overline T \mid  \overline\la\rangle = (-1)^{ab}\langle \overline X_1\wedge\dots\wedge \overline X_a\wedge \overline Y_1\wedge\dots\wedge \overline Y_a\wedge \overline T \mid  \overline \la\rangle\qquad\text{for all }\overline \la\in\overline{\mathcal J}^{2a+1},
\end{equation}
where the symbol  $\overline{\mathcal J}^{2a+1}$  stands for Rumin's space of   $(2a+1)$-covectors in $\overline\H{}^{n-b}$. Lemma~\ref{lem:vettoresemplicediRumin,b=0} then implies $\overline\tau=\overline X_1\wedge\dots\wedge \overline X_a\wedge \overline Y_1\wedge\dots\wedge \overline Y_a$ (unless $2a+b=n$, in which case we also have the possibility $\overline\tau=(-1)^a\overline X_{a+1}\wedge\dots\wedge \overline X_{n-b}\wedge \overline Y_{a+1}\wedge\dots\wedge \overline Y_{n-b}$) and this implies Claims 3a-3b up to some tedious arguments.

Consider $\overline I\subset\{1,\dots,2(n-b)\}$ such that $|\overline I|=2a$; calling $\overline M$ the $(2n-2b)\times 2a$ matrix whose columns are $\overline \tau^1,\dots,\overline\tau^{2a}$ and using for forms in $\overline\H{}^{n-b}$ the notation $d\overline z_{\overline I}$ analogous to that in  \eqref{eq:defdzI}, we have
\[
\langle \overline\tau \mid  d\overline z_{\overline I}\rangle = \det [\overline M_i]_{i\in\overline I},
\]
where of course $\overline M_i$ is the $i$-th row of $\overline M$. Since $\overline M$ is obtained from $M$ by canceling the third block of columns and the second and fifth blocks of rows (according to the arrangement in \eqref{eq:formaMa} and \eqref{eq:formaMb}) and since the second block of rows of $M$ is $[0|0|I_b]$, it is clear that
\begin{align*}
\langle \overline\tau \mid  d\overline z_{\overline I}\rangle = 
& \det [\overline M_i]_{i\in\overline I}
=(-1)^{ab}\det [M_i]_{\iota(\overline I)\cup \{a+1,\dots,a+b\}}\\
= & (-1)^{ab}\langle \tau\mid dz_{\iota(\overline I)\cup\{a+1,\dots,a+b\}}\rangle,
\end{align*}
where $\iota:\{1,\dots,2(n-b)\}\to\unodueenne$ is defined by
\begin{align*}
\iota(i):=
\left\{
\begin{array}{ll}
i & \text{if } 1\leq i\leq a\\
i+b & \text{if } a+1\leq i\leq n-b+a\\
i+2b\  & \text{if } n-b+a+1\leq i\leq 2n-2b.
\end{array}
\right.
\end{align*}
This implies that, for every disjoint subsets $\overline I,\overline J\subset \{1,\dots,n-b\}$ and every standard Young tableau $\overline R$, with rows of the proper lengths $\tfrac12(2a-|\overline I|-|\overline J|)$ and $n-b-\tfrac12(2a+|\overline I|+|\overline J|)$ and whose entries are precisely the numbers in $\{1,\dots,n-b\}\setminus(\overline  I\cup\overline J)$, we have
\begin{equation}\label{eq:seminariotrapoco}
\langle \overline\tau \mid  d\overline x_{\overline I}\wedge d\overline y_{\overline J}\wedge\al_{\overline R}\rangle 
= (-1)^{ab} \langle \tau \mid  d x_{\iota(\overline  I)\cup\{a+1,\dots,a+b\}}\wedge dy_{\overline{J}}\wedge\al_{\iota(\overline R)}\rangle,
\end{equation}
where $\iota(\overline R)$ denotes the tableau of the same form of $R$, obtained on replacing each entry, say $i$, of  $\overline R$ with $\iota(i)$. Observe that $\iota(R)$ is also a standard Young tableau because $\iota$ is increasing. Taking  Remark \ref{rem:cosafailvettorestandard}  and assumption \eqref{eq:ipotesiuguale} into account, equality \eqref{eq:seminariotrapoco} now implies \eqref{eq:taubartau}.\medskip

From now on we  assume that, if $n=2a+b$,  we are in the first case between the two displayed in \eqref{eq:formaMb}: indeed, the following  arguments  can be easily\footnote{One elegant way is to apply the $\H$-linear isomorphism associated with the Lie algebra isomorphism $\mathcal L_*(X_i)=X_{a+b+i},\  \mathcal L_*(Y_i)=Y_{a+b+i}$ if $1\leq i\leq a$, $\mathcal L_*(X_i)=X_{i},\  \mathcal L_*(Y_i)=Y_{i}$ if $a+1\leq i\leq a+b$, $\mathcal L_*(X_i)=X_{i-a-b},\  \mathcal L_*(Y_i)=Y_{i-a-b}$ if $a+b+1\leq i\leq n=2a+b$.}  generalized  to the other possible case.\medskip

{\em Claim 4: up to a proper choice of $\tau^1,\dots,\tau^{2a+b}$, we can assume that the blocks $A_1,A_2$ in \eqref{eq:formaMa}-\eqref{eq:formaMb} are 0.}

\noindent It is  enough to replace $\tau^i$, for $i=2a+1,\dots,2a+b$, with
\[
\tau^i - \sum_{j=1}^a (\tau^i_j \tau^j+\tau^i_{n+j} \tau^{a+j}).
\]
We are using in a key way the two blocks of the form $I_a$ appearing in \eqref{eq:formaMa}-\eqref{eq:formaMb}.\medskip

From now on the vectors $\tau$ are fixed; we are going to prove that the remaining blocks $B,C_1,C_2,D_1,D_2$ in \eqref{eq:formaMa}-\eqref{eq:formaMb} are 0.\medskip

{\em Claim 5: the blocks $B_1$ and $B_2$ in \eqref{eq:formaMa}-\eqref{eq:formaMb} are 0.}

\noindent We need to prove that
\begin{equation*}
\tau^{i}_{n+a+j}=0\qquad\text{for all }i=1,\dots,2a,\ j=1,\dots,b
\end{equation*}
Fix then such $i$ and $j$.
Let $\overline R$ be the standard Young tableau
\begin{equation}\label{eq:defRbarbrutto}
\begin{aligned}
& \overline  R:=
\begin{tabular}{|c|c|c|c|cc}
\cline{1-6}
$1$ & $2$ &$\ \cdots\ $& $n-a-b$   &\multicolumn{1}{c|}{$ \cdots$} &   \multicolumn{1}{c|}{$a$}
\\
\cline{1-6}
$a+b+1$ & $a+b+2$ &$\ \cdots\ $& $n$
\\
\cline{1-4}
\end{tabular}&&\text{if }n<2a+b\\
& \overline  R:=
\begin{tabular}{|c|c|c|c|}
\cline{1-4}
$1$ & $2$ &$\ \cdots\ $& $a$\\
\cline{1-4}
$a+b+1$ & $a+b+2$ &$\ \cdots\ $& $n$
\\
\cline{1-4}
\end{tabular}&&\text{if }n=2a+b.
\end{aligned}
\end{equation}
We define a new tableau $Q$ in the following way:
\begin{itemize}
\item if $1\leq i\leq a$, $Q$ is the tableau obtained from $\overline R$ on replacing, in the first row, the entry $i$ with $a+j$;
\item if $a+1\leq i\leq 2a$, $Q$ is the tableau obtained from $\overline R$ on replacing, in the first row, the entry $i-a$ with $a+j$.
\end{itemize}
Consider
\begin{align*}
& \la:=dx_{\{a+1,\dots,a+b\}\setminus\{a+j\}}\wedge dy_i\wedge\al_Q
&& \text{if }1\leq i\leq a\\
& \la:=dx_{\{i-a\}\cup\{a+1,\dots,a+b\}\setminus\{a+j\}}\wedge\al_Q
&& \text{if }a+1\leq i\leq 2a
\end{align*}
The tableau $Q$ is not a standard Young one; nonetheless, $\la\wedge\theta\in\mathcal J^{2a+b+1}$ and, using  Remark~\ref{rem:tabellenonYoung} and the assumption~\eqref{eq:ipotesiuguale}, we have
\begin{equation}\label{eq:arthgoldau}
\langle \tau\mid \la\rangle=0.
\end{equation}
Assume that $1\leq i\leq a$; we can write
\begin{equation}\label{eq:dottorTerence111}
\la=dx_{\{a+1,\dots,a+b\}\setminus\{a+j\}}\wedge dy_i\wedge\sum_S\si(S)\:dxy_S
\end{equation}
where $\si(S)\in\{1,-1\}$ is a suitable sign and the sum varies among the $2^{n-a-b}$ subsets $S\subset\{1,\dots,a,a+j,a+b+1,\dots,n\}\setminus\{i\}$ (i.e., $S$ is a subset of the entries of $Q$) of cardinality $a$ and containing exactly one element from each column of $Q$. For any such $S$ we have
\begin{equation}\label{eq:librobrutto111}
\begin{split}
\langle\tau \mid dx_{\{a+1,\dots,a+b\}\setminus\{a+j\}}\wedge dy_i\wedge dxy_S \rangle
= & \det[M_\ell]_{\ell\in(\{a+1,\dots,a+b\}\setminus\{a+j\})\cup\{n+i\}\cup S\cup(n+S)}\\
= & \pm \det [N_\ell]_{\ell\in S\cup(n+S)\cup\{n+i\}},
\end{split}
\end{equation}
where
\begin{itemize}
\item the sign $\pm$ depends only on $S,i,j$,
\item $N_\ell:=(\tau^1_\ell,\dots,\tau^{2a}_\ell,\tau^{2a+j}_\ell)$ is obtained from the $\ell$-th row $M_\ell$ of $M$ on canceling the last $b$ components except for the $(2a+j)$-th one,
\item we used that the second block of rows in \eqref{eq:formaMa}-\eqref{eq:formaMb} is $[0|0|I_b]$. 
\end{itemize}
Using Claim 4 one finds
\begin{equation}\label{eq:detSbruttobrutto111}
\text{if $S=\{1,\dots,a,a+j\}\setminus \{i\}$, then }\langle\tau \mid  dx_{\{a+1,\dots,a+b\}\setminus\{a+j\}}\wedge dy_i\wedge dxy_S \rangle = \pm\tau^{i}_{n+a+j}.
\end{equation}
Notice that the case $S=\{1,\dots,a,a+j\}\setminus \{i\}$ correspods to $S$ containing the elements in the first row of $Q$. Instead, if $S\neq\{1,\dots,a,a+j\}\setminus \{i\}$, then there exists an element $\bar\ell\in S$ belonging to the second row of $Q$, i.e., $\bar \ell\in S\cap\{a+b+1,\dots,n\}$. In this case $N_{\bar\ell}$ and $N_{n+\bar\ell}$ are linearly dependent, because by Claim 3 all their entries are null except (possibly) for the last one; in particular
\begin{equation}\label{eq:detSbruttobruttobrutto111}
\text{if $S\neq\{1,\dots,a,a+j\}\setminus \{i\}$, then }\langle\tau \mid  dx_{\{a+1,\dots,a+b\}\setminus\{a+i\}}\wedge dy_{a+i}\wedge dxy_S \rangle = 0.
\end{equation}
By \eqref{eq:arthgoldau}, \eqref{eq:dottorTerence111}, \eqref{eq:detSbruttobrutto111} and \eqref{eq:detSbruttobruttobrutto111} we obtain
\begin{equation}\label{eq:zerotaulacomponente}
0 = \langle\tau \mid  \la \rangle = \pm\tau^{i}_{n+a+j}
\end{equation}
and the claim is proved in case $1\leq i\leq a$.  If $a+1\leq i\leq 2a$, \eqref{eq:zerotaulacomponente} can be proved by a completely analogous argument that we omit. The claim is proved.\medskip

{\em Claim 6: the block $C$ in \eqref{eq:formaMa}-\eqref{eq:formaMb} is 0.}

\noindent The block $C$ is a square one, with size $b\times b$; we start by proving that the elements on the diagonal of $C$ are all null, i.e., that
\begin{equation}\label{eq:diagonale nulla}
 \tau^{2a+i}_{n+a+i}=0 \quad  \text{ for any }i=1,\dots,b.
\end{equation} 
Let then $i\in\{1,\dots,b\}$ be fixed; consider 
\[
\la:=dx_{\{a+1,\dots,a+b\}\setminus\{a+i\}}\wedge dy_{a+i}\wedge\al_{\overline R},
\]
where $\overline R$ is as in \eqref{eq:defRbarbrutto}. We can write
\begin{equation}\label{eq:dottorTerence}
\la=dx_{\{a+1,\dots,a+b\}\setminus\{a+i\}}\wedge dy_{a+i}\wedge\sum_S\si(S)\:dxy_S
\end{equation}
where $\si(S)\in\{1,-1\}$ is a suitable sign and the sum varies among the $2^{n-a-b}$ subsets $S\subset\{1,\dots,a,a+b+1,\dots,n\}$ of cardinality $a$ and containing exactly one element from each column of $\overline R$. For any such $S$ we have
\begin{equation}\label{eq:librobrutto}
\begin{split}
\langle\tau \mid  dx_{\{a+1,\dots,a+b\}\setminus\{a+i\}}\wedge dy_{a+i}\wedge dxy_S \rangle
= & \det[M_j]_{j\in(\{a+1,\dots,a+b\}\setminus\{a+i\})\cup\{n+a+i\}\cup S\cup(n+S)}\\
= & \pm \det [N_j]_{j\in S\cup(n+S)\cup\{n+a+i\}},
\end{split}
\end{equation}
where
\begin{itemize}
\item the sign $\pm$ depends only on $S$ and $i$,
\item $N_j:=(\tau^1_j,\dots,\tau^{2a}_j,\tau^{2a+i}_j)$ is obtained from the $j$-th row $M_j$ of $M$ on canceling the last $b$ components except for the $(2a+i)$-th one,
\item we used that the second block of rows in \eqref{eq:formaMa}-\eqref{eq:formaMb} is $[0|0|I_b]$. 
\end{itemize}
Using Claims 4 and 5 one finds
\begin{equation}\label{eq:detSbruttobrutto}
\text{if $S=\{1,\dots,a\}$, then }\langle\tau \mid  dx_{\{a+1,\dots,a+b\}\setminus\{a+i\}}\wedge dy_{a+i}\wedge dxy_S \rangle = \pm\tau^{2a+i}_{n+a+i}.
\end{equation}
On the contrary, if $S\neq\{1,\dots,a\}$, then there exists an element $\bar\ell\in S$ belonging to the second row of $\overline R$, i.e., $\bar \ell\in S\cap\{a+b+1,\dots,n\}$. In this case $N_{\bar\ell}$ and $N_{n+\bar\ell}$ are linearly dependent, because by Claim 3 all their entries are null except (possibly) for the last one; in particular, \eqref{eq:librobrutto} gives
\begin{equation}\label{eq:detSbruttobruttobrutto}
\text{if $S\neq\{1,\dots,a\}$, then }\langle\tau\mid  dx_{\{a+1,\dots,a+b\}\setminus\{a+i\}}\wedge dy_{a+i}\wedge dxy_S \rangle = 0.
\end{equation}
By \eqref{eq:dottorTerence}, \eqref{eq:detSbruttobrutto} and \eqref{eq:detSbruttobruttobrutto} we obtain
\[
0=\langle X_1\wedge\dots\wedge X_{a+b}\wedge Y_1\wedge\dots\wedge Y_a\mid \la\rangle = \langle\tau \mid  \la \rangle = \pm\tau^{2a+i}_{n+a+i}
\]
and \eqref{eq:diagonale nulla} is proved.

We now prove that the off-diagonal entries of $C$ are null as well: we then fix $i,j\in\{1,\dots,b\}$ with $j<i$ and prove that
\begin{equation}\label{eq:fuoridiagonalenulli}
\tau^{2a+i}_{n+a+j}=\tau^{2a+j}_{n+a+i}=0.
\end{equation}
Let us consider
\begin{align*}
& \la_1:=dx_{\{a+1,\dots,a+b\}\setminus\{a+j,a+i\}}\wedge dy_{\{n+a+j,n+a+i\}}\wedge\al_{\overline R}\\
& \la_2:=dx_{\{a+1,\dots,a+b\}\setminus\{a+j,a+i\}}\wedge \al_{Q}
\end{align*}
where $\overline R$ is as in \eqref{eq:defRbarbrutto} and $Q$ is the tableau obtained by adding a column $(a+j,a+i)$ on the left of $\overline R$, i.e.,
\begin{align*}
&  Q :=
\begin{tabular}{|c|c|c|c|ccc}
\cline{1-7}
$a+j$ & $1$ &$\ \cdots\ $& $n-a-b$  &\multicolumn{1}{c|}{$n-a-b+1$} &\multicolumn{1}{c|}{$ \cdots$} &   \multicolumn{1}{c|}{$a$}
\\
\cline{1-7}
$a+i$ & $a+b+1$ &$\ \cdots\ $& $n$
\\
\cline{1-4}
\end{tabular}&&\text{if }n<2a+b\\
&  Q :=
\begin{tabular}{|c|c|c|c|}
\cline{1-4}
$a+j$ & $1$ &$\ \cdots\ $& $a$\\
\cline{1-4}
$a+i$ & $a+b+1$ &$\ \cdots\ $& $n$
\\
\cline{1-4}
\end{tabular}&&\text{if }n=2a+b.
\end{align*}
Notice that $Q$ is not  a standard Young tableau; however, $\la_1\wedge\theta$ and $\la_2\wedge\theta$ belong to $\mathcal J^{2a+b+1}$ and, using Remark \ref{rem:tabellenonYoung} and the assumption~\eqref{eq:ipotesiuguale}, one obtains 
\begin{equation}\label{eq:launoladue}
\begin{split}
& \langle\tau\mid \la_1\rangle = \langle X_1\wedge\dots\wedge X_{a+b}\wedge Y_1\wedge\dots\wedge Y_a\mid \la_1\rangle = 0\\
& \langle\tau\mid \la_2\rangle = \langle X_1\wedge\dots\wedge X_{a+b}\wedge Y_1\wedge\dots\wedge Y_a\mid \la_2\rangle = 0.
\end{split}
\end{equation}
The  argument that follows is pretty much similar to the previous one as well as to that of Claim 5. We  write
\begin{equation}\label{eq:dottorTerence222}
\begin{aligned}
& \la_1=dx_{\{a+1,\dots,a+b\}\setminus\{a+j,a+i\}}\wedge dy_{\{n+a+j,n+a+i\}}\wedge\sum_S\si(S)\:dxy_S\\
& \la_2=dx_{\{a+1,\dots,a+b\}\setminus\{a+j,a+i\}}\wedge\sum_T\si(T)\:dxy_T
\end{aligned}
\end{equation}
where $\si(S),\si(T)\in\{1,-1\}$ are suitable signs and the sums vary among the  subsets $S,T$ of the sets of entries of $\overline R,Q$ (respectively) with cardinality (resp.) $a,a+1$  and containing exactly one element from each column of (resp.) $\overline R,Q$. For any such $S,T$ we have
\begin{equation}\label{eq:librobrutto222aaa}
\begin{split}
& \langle\tau \mid  dx_{\{a+1,\dots,a+b\}\setminus\{a+j,a+i\}}\wedge dy_{\{n+a+j,n+a+i\}}\wedge dxy_S \rangle\\
=\: & \det[M_\ell]_{\ell\in(\{a+1,\dots,a+b\}\setminus\{a+j,a+i\})\cup\{n+a+j,n+a+i\}\cup S\cup(n+S)}\\
=\: & \pm \det [O_\ell]_{\ell\in S\cup(n+S)\cup\{n+a+j,n+a+i\}},
\end{split}
\end{equation}
and
\begin{equation}\label{eq:librobrutto222bbb}
\begin{split}
\langle\tau \mid  dx_{\{a+1,\dots,a+b\}\setminus\{a+j,a+i\}}\wedge dxy_T \rangle
= & \det[M_\ell]_{\ell\in(\{a+1,\dots,a+b\}\setminus\{a+j,a+i\})\cup T\cup(n+T)}\\
= & \pm \det [O_\ell]_{\ell\in T\cup(n+T)},
\end{split}
\end{equation}
where
\begin{itemize}
\item the signs $\pm$ depend only on $S,T,i$ and $j$,
\item $O_\ell:=(\tau^1_\ell,\dots,\tau^{2a}_\ell,\tau^{2a+j}_\ell,\tau^{2a+i}_\ell)$ is obtained from the $\ell$-th row $M_\ell$ of $M$ on canceling the last $b$ components except for the $(2a+j)$-th and $(2a+i)$-th ones,
\item we used that the second block of rows in \eqref{eq:formaMa}-\eqref{eq:formaMb} is $[0|0|I_b]$. 
\end{itemize}
Using Claims 4, 5 and the fact that the diagonal of $C$ is null one finds that
\begin{itemize}
\item[(i)] if $S=\{1,\dots,a\}$, then 
\begin{align*}
&\langle\tau \mid  dx_{\{a+1,\dots,a+b\}\setminus\{a+j,a+i\}}\wedge dy_{n+a+j,n+a+i}\wedge dxy_S \rangle \\
=&\pm\det\left(\begin{array}{cc}
\tau^{2a+j}_{n+a+j} & \tau^{2a+i}_{n+a+j}\\
\tau^{2a+j}_{n+a+i} & \tau^{2a+i}_{n+a+i}
\end{array}\right)
=\pm\det\left(\begin{array}{cc}
0 & \tau^{2a+i}_{n+a+j}\\
\tau^{2a+j}_{n+a+i} & 0
\end{array}\right) 
= \pm  \tau^{2a+j}_{n+a+i}\tau^{2a+i}_{n+a+j}
\end{align*}
\item[(ii)] if $T=\{1,\dots,a,a+j\}$, then 
\begin{align*}
&
\langle\tau \mid  dx_{\{a+1,\dots,a+b\}\setminus\{a+j,a+i\}}\wedge dxy_T \rangle \\
=&\pm\det\left(\begin{array}{cc}
\tau^{2a+j}_{a+j} & \tau^{2a+i}_{a+j}\\
\tau^{2a+j}_{n+a+j} & \tau^{2a+i}_{n+a+j}
\end{array}\right)
=\pm\det\left(\begin{array}{cc}
1 & 0\\
0 & \tau^{2a+i}_{n+a+j}
\end{array}\right) 
= \pm  \tau^{2a+i}_{n+a+j}
\end{align*}
\item[(iii)] if $T=\{1,\dots,a,a+i\}$, then 
\begin{align*}
&\langle\tau \mid  dx_{\{a+1,\dots,a+b\}\setminus\{a+j,a+i\}}\wedge dxy_T \rangle \\
=&\pm\det\left(\begin{array}{cc}
\tau^{2a+j}_{a+i} & \tau^{2a+i}_{a+i}\\
\tau^{2a+j}_{n+a+i} & \tau^{2a+i}_{n+a+i}
\end{array}\right)
=\pm\det\left(\begin{array}{cc}
0 & 1\\
\tau^{2a+j}_{n+a+i} & 0
\end{array}\right) 
= \pm  \tau^{2a+j}_{n+a+i},
\end{align*}
\end{itemize}
where the $\pm$ signs will be irrelevant. 
If $S\neq\{1,\dots,a\}$, then there exists an element $\bar\ell\in  S\cap\{a+b+1,\dots,n\}$; it follows that $O_{\bar\ell},O_{n+\bar\ell},O_{n+a+j},O_{n+a+i}$ are linearly dependent, because by Claims 3 and 5 all their entries are null except (possibly) for the last two ones. In particular, \eqref{eq:librobrutto222aaa} gives
\begin{equation}\label{eq:detSbruttobruttobrutto222}
\begin{aligned}
\text{if } &S\neq\{1,\dots,a\},\\
\text{then } & \langle\tau \mid  dx_{\{a+1,\dots,a+b\}\setminus\{a+j,a+i\}}\wedge dy_{n+a+j,n+a+i}\wedge dxy_S \rangle = 0.
\end{aligned}
\end{equation}
If $T\neq\{1,\dots,a,a+j\}$ and $T\neq\{1,\dots,a,a+i\}$, then there exists an element $\bar\ell\in  T\cap\{a+b+1,\dots,n\}$; notice also that either $a+j\in T$ or $a+i\in T$. If $a+j\in T$, then $O_{\bar\ell},O_{n+\bar\ell},O_{a+j},O_{n+a+j}$ are linearly dependent, again because by Claims 3 and 5 all their entries are null except (possibly) for the last two ones. Similarly, when $a+i\in T$ one has that $O_{\bar\ell},O_{n+\bar\ell},O_{a+i},O_{n+a+i}$ are linearly dependent. We deduce by \eqref{eq:librobrutto222bbb} that
\begin{equation}\label{eq:detSbruttobruttobrutto222bbb}
\begin{aligned}
\text{if } & T\neq\{1,\dots,a,a+j\}\text{ and }T\neq\{1,\dots,a,a+i\},\\
\text{ then } & \langle\tau \mid  dx_{\{a+1,\dots,a+b\}\setminus\{a+j,a+i\}}\wedge dxy_T \rangle = 0.
\end{aligned}
\end{equation}
By \eqref{eq:launoladue}, \eqref{eq:dottorTerence222}, \eqref{eq:detSbruttobruttobrutto222}  and (i)--(ii)--(iii) above we finally achieve
\[
\tau^{2a+j}_{n+a+i}\tau^{2a+i}_{n+a+j}=0\qtaq\pm  \tau^{2a+j}_{n+a+i}\pm\tau^{2a+i}_{n+a+j}=0
\]
and \eqref{eq:fuoridiagonalenulli} is proved.
\medskip

{\em Claim 7:  the blocks $D_1$ and $D_2$ in \eqref{eq:formaMa}-\eqref{eq:formaMb} are 0.}

\noindent The claim amounts to showing that
\begin{equation}\label{eq:chiasso}
\tau^{2a+i}_{j}=0=\tau^{2a+i}_{n+j}\qquad\text{for any }1\leq i\leq b\text{ and }a+b+1\leq j\leq n.
\end{equation}
Fix $i,j$ as in \eqref{eq:chiasso} and consider the Young tableau $Q$ obtained from $\overline R$ on replacing, in the second row, the entry $j$ with $a+i$. Set
\[
\la:=dx_{\{j\}\cup\{a+1,\dots,a+b\}\setminus\{a+i\}}\wedge\al_Q
\]
and write
\begin{equation}\label{eq:dottorTerence999}
\la=dx_{\{j\}\cup\{a+1,\dots,a+b\}\setminus\{a+i\}}\wedge\sum_S\si(S)\:dxy_S
\end{equation}
where $\si(S)\in\{1,-1\}$ is a suitable sign and the sum varies among the $2^{n-a-b}$ subsets $S\subset\{a+i\}\cup\{1,\dots,a,a+b+1,\dots,n\}\setminus\{j\}$ (i.e., $S$ is a subset of the set of entries of $Q$) with cardinality $a$ and containing exactly one element from each column of $Q$. For any such $S$ we have
\begin{equation}\label{eq:librobrutto999}
\begin{split}
\langle\tau \mid  dx_{\{j\}\cup\{a+1,\dots,a+b\}\setminus\{a+i\}}\wedge dxy_S \rangle
= & \det[M_\ell]_{\ell\in(\{j\}\cup\{a+1,\dots,a+b\}\setminus\{a+i\})\cup S\cup(n+S)}\\
= & \pm \det [N_\ell]_{\ell\in \{j\}\cup S\cup(n+S)},
\end{split}
\end{equation}
where
\begin{itemize}
\item the sign $\pm$ depends only on $S,j$ and $i$,
\item $N_\ell:=(\tau^1_\ell,\dots,\tau^{2a}_\ell,\tau^{2a+i}_\ell)$ is obtained from the $\ell$-th row $M_\ell$ of $M$ on canceling the last $b$ components except for the $(2a+i)$-th one,
\item we used that the second block of rows in \eqref{eq:formaMa}-\eqref{eq:formaMb} is $[0|0|I_b]$. 
\end{itemize}
Using the previous claims we obtain
\begin{equation}\label{eq:detSbruttobrutto999}
\text{if $S=\{1,\dots,a\}$, then }\langle\tau \mid  dx_{\{j\}\cup\{a+1,\dots,a+b\}\setminus\{a+i\}}\wedge dxy_S \rangle = \pm\tau^{2a+i}_{j}.
\end{equation}
If $S\neq\{1,\dots,a\}$, then there exists an element $\bar\ell\in S$ belonging to the second row of $Q$, i.e., $\bar \ell\in S\cap(\{a+i\}\cup\{a+b+1,\dots,n\}\setminus\{j\})$. In this case $N_{\bar\ell}$ and $N_{n+\bar\ell}$ are linearly dependent, because by Claims 3, 4 and 5 all their entries are null except (possibly) for the last one; in particular, \eqref{eq:librobrutto999} gives
\begin{equation}\label{eq:detSbruttobruttobrutto999}
\text{if $S\neq\{1,\dots,a\}$, then }\langle\tau \mid  dx_{\{a+1,\dots,a+b\}\setminus\{a+i\}}\wedge dy_{a+i}\wedge dxy_S \rangle = 0.
\end{equation}
By \eqref{eq:dottorTerence999}, \eqref{eq:detSbruttobrutto999}, \eqref{eq:detSbruttobruttobrutto999} and Remark \ref{rem:tabellenonYoung} we obtain
\[
0=\langle X_1\wedge\dots\wedge X_{a+b}\wedge Y_1\wedge\dots\wedge Y_a\mid \la\rangle = \langle\tau \mid  \la \rangle = \pm\tau^{2a+i}_{j}
\]
and the first equality in~\eqref{eq:chiasso} is proved. 

We are left to show that also $\pm\tau^{2a+i}_{n+j}=0$ for any $i,j$ as in \eqref{eq:chiasso}; this can be done by considering (for the same $Q$ above)
\[
\la=dx_{\{a+1,\dots,a+b\}\setminus\{a+i\}}\wedge dy_j\wedge\al_Q
\]
and following a similar argument, that we omit. The proof of Claim 7 is then complete.\medskip

The proof of Lemma \ref{lem:vettoresemplicediRumin} now follows from the equality $\tau=\tau^1\wedge\dots\wedge\tau^{2a+b}$, the fact that the blocks $A_1,A_2,B_1,B_2,C,D_1,D_2$ in \eqref{eq:formaMa}-\eqref{eq:formaMb} are all null, and the equality
\begin{align*}
&X_1\wedge\dots\wedge X_{a+b}\wedge Y_1\wedge\dots\wedge Y_a\\
= &
(-1)^{ab}X_1\wedge\dots\wedge X_{a}\wedge Y_1\wedge\dots\wedge Y_a\wedge X_{a+1}\wedge\dots \wedge X_{a+b}.
\end{align*}
\end{proof}

\begin{proof}[Proof of Proposition~\ref{prop:almassimo2}]
Assume that there exists a $(2n+1-k)$-dimensional vertical plane $\mathscr P_1$ whose unit tangent vector $t_{\mathscr P_1}^\H$ is such that $[t_{\mathscr P_1}^\H]_\mathcal J$ is a multiple of  $\zeta$. By Proposition~\ref{prop:ruotiamosupianicanonici} there  exist non-negative integers $a$ and $b$ and a $\H$-linear isomorphim $\mathcal L:\H^n\to\H^n$ such that $a+b\leq n$, $\dim\mathscr P_1 =2a+b+1$, $\mathcal L^*(\theta)=\theta$, $\mathcal L^*(d\theta)=d\theta$ and
\[
\mathcal L(\mathscr P_1)=\Pab,
\]
where $\Pab$ is defined as in~\eqref{eq:defPab}. Let us denote by
$\mathcal L_*:\mathcal J_{2n+1-k}\to\mathcal J_{2n+1-k}$  the isomorphism defined in~\eqref{eq:L_*Rumin}. Clearly, $\mathcal L_*([t^\H_{\mathscr P_1}]_\mathcal J)=[\mathcal L_*(t^\H_{\mathscr P_1})]_\mathcal J$ is a multiple of
\[
[t^\H_{\Pab}]_\mathcal J = [X_1\wedge\dots\wedge X_{a+b}\wedge Y_1\wedge\dots\wedge Y_a\wedge T]_\mathcal J,
\]
hence $\mathcal L_*\zeta$ is a multiple of $[X_1\wedge\dots\wedge X_{a+b}\wedge Y_1\wedge\dots\wedge Y_a\wedge T]_\mathcal J$.

Assume that $\mathscr P_2$ is another vertical $(2n+1-k)$-plane in $\H^n$ such that $[t^\H_{\mathscr P_2}]_\mathcal J$ is a multiple of $\zeta$. The vertical plane $\mathscr P_3:=\mathcal L(\mathscr P_2)$ is such that $t^\H_{\mathscr P_3}$ is a multiple of $\mathcal L_*(t^\H_{\mathscr P_2})$, hence $[t^\H_{\mathscr P_3}]_\mathcal J$ is a multiple of  $\mathcal L_*([t^\H_{\mathscr P_2}]_\mathcal J)$, i.e., of $\mathcal L_*\zeta$ and eventually of $[X_1\wedge\dots\wedge X_{a+b}\wedge Y_1\wedge\dots\wedge Y_a\wedge T]_\mathcal J$. 
If $k<n$, Lemma~\ref{lem:vettoresemplicediRumin} implies that $\mathscr P_3=\Pab$, hence $\mathscr P_2=\mathscr P_1$. This proves part (i) of the statement.

If instead $k=n$, Lemma~\ref{lem:vettoresemplicediRumin} implies that either $\mathscr P_3=\Pab$ or 
\[
\mathscr P_3 = \exp(\Span\{X_{a+1},\dots, X_n,Y_{a+b+1},\dots, Y_n,T\}).
\]
Observe that, if $a=0$, then $\mathscr P_3=\Pab$. On the contrary, if $a\geq 1$ then either $\mathscr P_3=\Pab$ or $\mathscr P_3$ and $\Pab$ are not rank-one connected, because 
\[
\dim\mathscr P_3\cap\Pab = b+1\leq (2a+b+1)-2=\dim\mathscr P_3-2.
\]
All in all, we deduce that either $\mathscr P_2=\mathscr P_1$, or 
\[
\mathscr P_2=\mathcal L^{-1}(\exp(\Span\{X_{a+1},\dots, X_n,Y_{a+b+1},\dots, Y_n,T\}))
\]
and $\mathscr P_1$, $\mathscr P_2$ are not rank-one connected. This concludes the proof.
\end{proof}

\bibliographystyle{acm}
\bibliography{biblio}

\end{document}